\documentclass[11pt]{article}
\usepackage[]{amsmath,amssymb,amsthm,amsfonts,amsbsy, enumerate,eucal}
\usepackage{verbatim}
\usepackage{hyperref}
\usepackage{array}
\usepackage{xfrac}
\usepackage{color}
\usepackage{mathtools}
\usepackage[T1]{fontenc}

\usepackage{cite}
\usepackage{tikz}

\theoremstyle{definition}

\newtheorem{definition}{Definition}[section]
\newtheorem{example}[definition]{Example}

\newtheorem{remark}[definition]{Remark}

\newtheorem{note}[definition]{Note}

\newcommand\scalemath[2]{\scalebox{#1}{\mbox{\ensuremath{\displaystyle #2}}}}

\theoremstyle{plain}
\newtheorem{theorem}[definition]{Theorem}
\newtheorem{lemma}[definition]{Lemma}
\newtheorem{corollary}[definition]{Corollary}
\newtheorem{proposition}[definition]{Proposition}

\def\I{\mathbb I}

\def\C{\mathbb C}

\def\H{\hat{H}_{q}}
\def\CX{\mathbb{C}^{X}}
\def\MX{{\bf M}_{X}(\mathbb{C})}
\def\M{{\bf M}_{d+1}(\mathbb{C})}
\def\Mx{\mcal{M}\hat{x}}
\def\Mxp{\mcal{M}\hat{x}^{\perp}}
\def\MC{\mcal{M}\hat{C}}
\def\MCp{\mcal{M}\hat{C}^{\perp}}
\def\wtPp{\widetilde{\Phi}^{\perp}}
\def\diag{{\rm diag}}
\def\bdiag{{\rm blockdiag}}

\def\End{{\rm{End}}}

\def\bv{{\bf v}}

\def\X{{\bf X}}
\def\Y{{\bf Y}}
\def\A{{\bf A}}
\def\B{{\bf B}}
\def\V{{\bf V}}
\def\W{{\bf W}}
\def\T{{\bf T}}

\def\wt{\widetilde}
\def\ta{\widetilde{a}}
\def\tb{\widetilde{b}}
\def\tc{\widetilde{c}}
\def\tv{\widetilde{v}}
\def\ttht{\widetilde{\theta}}
\def\tvarphi{\widetilde{\varphi}}
\def\tphi{\widetilde{\phi}}

\def\tthtp{\widetilde{\theta}^{\perp}}
\def\tthtsp{\widetilde{\theta}^{*\perp}}
\def\tvarphip{\widetilde{\varphi}^{\perp}}
\def\tphip{\widetilde{\phi}^{\perp}}

\def\thtp{{\theta}^{\perp}}
\def\thtsp{{\theta}^{*\perp}}
\def\varphip{{\varphi}^{\perp}}
\def\phip{{\phi}^{\perp}}

\def\sp{{s}^{\perp}}
\def\ssp{{s}^{*\perp}}
\def\hp{{h}^{\perp}}
\def\hsp{{h}^{*\perp}}
\def\rp{{r}^{\perp}}
\def\tsp{\widetilde{s}^{\perp}}
\def\tssp{\widetilde{s}^{*\perp}}
\def\thp{\widetilde{h}^{\perp}}
\def\thsp{\widetilde{h}^{*\perp}}
\def\trp{\widetilde{r}^{\perp}}

\def\vp{{v}^{\perp}}

\def\ap{{a}^{\perp}}
\def\bp{{b}^{\perp}}
\def\cp{{c}^{\perp}}
\def\tap{{\widetilde a}^{\perp}}
\def\tbp{{\widetilde b}^{\perp}}
\def\tcp{{\widetilde c}^{\perp}}
\def\tvp{{\widetilde v}^{\perp}}

\newcommand{\Ga}{\ensuremath{\Gamma}}

\newcommand{\tht}{\ensuremath{\theta}}

\newcommand{\mfrk}{\ensuremath{\mathfrak}}
\newcommand{\mbb}{\ensuremath{\mathbb}}
\newcommand{\mcal}{\ensuremath{\mathcal}}

\begin{document}
\title{\bf 
$Q$-polynomial distance-regular graphs and a\\
double affine Hecke algebra of rank one
}

\author{Jae-Ho Lee}
\date{}

\maketitle
\begin{abstract}

We study a relationship between $Q$-polynomial distance-regular graphs and the double affine Hecke algebra of type $(C^{\vee}_1,C_1)$. Let $\Ga$ denote a $Q$-polynomial distance-regular graph with vertex set $X$. We assume that $\Ga$ has $q$-Racah type and contains a Delsarte clique $C$. Fix a vertex $x\in C$. We partition $X$ according to the path-length distance to both $x$ and $C$. This is an equitable partition. For each cell in this partition, consider the corresponding characteristic vector. These characteristic vectors form a basis for a $\C$-vector space $\W$.

The universal double affine Hecke algebra of type $(C^{\vee}_1,C_1)$ is the $\C$-algebra $\H$ defined by generators $\{t^{\pm1}_n\}^3_{n=0}$ and relations (i) $t_nt_n^{-1}=t_n^{-1}t_n=1$; (ii) $t_n+t_n^{-1}$ is central; (iii) $t_0t_1t_2t_3 = q^{-1/2}$. In this paper, we display an $\H$-module structure for $\W$. For this module and up to affine transformation,
\begin{itemize}
\item $t_0t_1+(t_0t_1)^{-1}$ acts as the adjacency matrix of $\Ga$;
\item $t_3t_0+(t_3t_0)^{-1}$ acts as the dual adjacency matrix of $\Ga$ with respect to $C$;
\item $t_1t_2+(t_1t_2)^{-1}$ acts as the dual adjacency matrix of $\Ga$ with respect to $x$.
\end{itemize}
To obtain our results we use the theory of Leonard systems.

 \end{abstract}

\section{Introduction}



This paper is about three classes of objects:~(i)~Leonard systems;~(ii)~$Q$-polynomial distance-regular graphs;~(iii)~double affine Hecke algebras. To motivate our results we will provide some background on each of these topics.



\medskip
\noindent
The concept of a Leonard system was introduced by Terwilliger\cite[Definition~1.4]{Terwilli}. 
To explain what this concept is, we begin with a more basic concept called a Leonard pair\cite[Definition~1.1]{Terwilli}. 
Roughly speaking, a Leonard pair consists of two diagonalizable linear transformations on a finite-dimensional vector space, each of which acts in an irreducible tridiagonal fashion on an eigenbasis for the other one. 
A Leonard system is a Leonard pair together with an appropriate ordering of their primitive idempotents.
Leonard systems are classified up to isomorphism\cite[Theorem~1.9]{Terwilli}.
This classification yields a bijection between the Leonard systems and a family of orthogonal polynomials consisting of the $q$-Racah polynomials and their relatives in the Askey scheme\cite{Terwilli:LP-qRacah}.
The Leonard systems that correspond to the $q$-Racah polynomials are said to have $q$-Racah type. This is the most general type of Leonard system. 
We will focus on the Leonard systems of $q$-Racah type.



\medskip
\noindent
In \cite{Delsarte}, Delsarte introduced the $Q$-polynomial property for distance-regular graphs.
Since then, this property has been investigated by many authors\cite{B-Ito, BCN, Caughman2, Curtin,  JTZ, LP, MP, Tanaka, Ter1}. 
A $Q$-polynomial distance-regular graph can be regarded as a discrete analogue of a rank 1 symmetric space\cite[p.~311,~312]{B-Ito}.  
Let $\Ga$ denote a $Q$-polynomial distance-regular graph with vertex set $X$.
Let $\MX$ denote the $\mbb{C}$-algebra consisting of the matrices with entries in $\mbb{C}$ whose rows and columns are indexed by $X$.
Let $\V$ denote the $\mathbb{C}$-vector space consisting of column vectors with rows indexed by $X$.
We view $\V$ as a left module for $\MX$.
Fix $x \in X$.
In \cite{T-algebra-I} Terwilliger introduced the subconstituent algebra $T=T(x)$ (or Terwilliger algebra).
The algebra $T$ is the subalgebra of $\MX$ generated by the adjacency matrix $A$ of $\Ga$ and a certain diagonal matrix $A^*=A^*(x)$, called the dual adjacency matrix of $\Ga$ with respect to $x$.
We now discuss the $T$-modules.
By a $T$-module we mean a $T$-submodule of $\V$.
The algebra $T$ is semisimple\cite[Lemma~3.4]{T-algebra-I}, so $\V$ decomposes into a direct sum of irreducible $T$-modules. 
Let $W$ denote an irreducible $T$-module.
Then $W$ is called thin whenever the intersection of $W$ with each eigenspace of $A$ and $A^*$ has dimension at most 1.
It is known that the matrices $A, A^*$ act as a Leonard pair on each thin irreducible $T$-module\cite{T-algebra-I}.
We now recall the primary $T$-module\cite[Lemma~3.6]{T-algebra-I}.
Consider the $T$-module generated by the characteristic vector of $x$. 
This module is thin and irreducible. 
We call this the primary $T$-module.
We say that $\Ga$ has $q$-Racah type whenever the Leonard system induced by the primary $T$-module has $q$-Racah type.

\medskip \noindent
In \cite{Suzuki} Suzuki extended the Terwilliger algebra concept by defining the Terwilliger algebra with respect to a set of vertices.
For our purpose the set of vertices will be a Delsarte clique.
A Delsarte clique $C$ of $\Ga$ is a nonempty set of mutually adjacency vertices of $\Ga$ that has cardinality $1-k/\tht_{\rm min}$, where $k$ is the valency of $\Ga$ and $\tht_{\rm min}$ is the minimum eigenvalue of $A$. 
The Terwilliger algebra $\wt{T}=\wt{T}(C)$ is the subalgebra of $\MX$ generated by $A$ and a certain diagonal matrix $\wt{A}^*=\wt{A}^*(C)$, called the dual adjacency matrix of $\Ga$ with respect to $C$. 
By a $\wt{T}$-module we mean a $\wt{T}$-submodule of $\V$.
As we will see in Section 4, the algebra $\wt{T}$ is semisimple. So $\V$ decomposes into a direct sum of irreducible $\wt{T}$-modules. 
Let $W$ denote an irreducible $\wt{T}$-module.
Then $W$ is called thin whenever the intersection of $W$ with each eigenspace of $A$ and $\wt{A}^*$ has dimension at most 1.
By \cite[Lemma~3.2]{T-algebra-I} the matrices $A, \wt{A}^*$ act as a Leonard pair on each thin irreducible $\wt{T}$-module.
We now recall the primary $\wt{T}$-module\cite[Section~7]{Suzuki}.
Consider the $\wt{T}$-module generated by the characteristic vector of $C$. 
It turns out that this module is thin and irreducible. We call this the primary $\wt{T}$-module.

\medskip \noindent
In this paper we introduce an algebra $\T$. 
To define $\T$ we assume that $\Ga$ contains a Delsarte clique $C$. 
Fix a vertex $x\in C$.
The algebra $\T=\T(x,C)$ is generated by $T=T(x)$ and $\wt{T}=\wt{T}(C)$. 
The algebra $\T$ is semisimple.
It turns out that the $\T$-module generated by the characteristic vector of $x$ is equal to the $\T$-module generated by the characteristic vector of $C$. 
We denote this module by $\W$.
As we will see, the $\T$-module $\W$ is irreducible.
Moreover, the $T$-module $\W$ decomposes into the direct sum of two irreducible $T$-modules, one of which is the primary $T$-module.
Also, the $\wt{T}$-module $\W$ decomposes into the direct sum of two irreducible $\wt{T}$-modules, one of which is the primary $\wt{T}$-module.
The $\T$-module $\W$ will play a role in our main result.



\medskip \noindent
In \cite{chered}, Cherednik introduced the double affine Hecke algebra (or DAHA) for a reduced root system. 
In \cite{Sahi}, Sahi extended this definition to include non-reduced root systems. 
In the present paper we consider the DAHA of type $(C^{\vee}_1,C_1)$. 
This is the most general DAHA of rank 1\cite{Oblomkov}.
In \cite{NS}, Noumi and Stokman treated this algebra in detail, in order to study the Askey-Wilson polynomials.
We now define the DAHA of type $(C^{\vee}_1,C_1)$.
\begin{definition}\cite{mac}
Fix nonzero scalars $k_0, k_1, k_2, k_3, q$ in $\mbb{C}$. Let $H=H(k_0,k_1,k_2,k_3;q)$ denote the $\mbb{C}$-algebra defined by generators $\{t_n\}^{3}_{n=0}$ and relations
\begin{align*}
(t_n-k_n)(t_n-k_n^{-1}) =0, \qquad \qquad 
t_0t_1t_2t_3 = q^{-1/2}.
\end{align*}
This algebra is called the DAHA of type $(C^{\vee}_1,C_1)$.
\end{definition}
\noindent
In \cite{Ter:AW+DAHA} Terwilliger defined a central extension $\H$ of $H$.
By definition  $\H$ has  generators $\{t^{\pm1}_n\}^3_{n=0}$ and relations (i) $t_nt_n^{-1} = t_n^{-1}t_n=1$; (ii) $t_n + t_n^{-1}$ is central; (iii) $t_0t_1t_2t_3=q^{-1/2}$. 
We call $\H$ the universal DAHA of type $(C^{\vee}_1,C_1)$.
In the present paper we will focus on $\H$ .




\medskip \noindent
We now summarize our main results. 
Let $\Ga$ denote a $Q$-polynomial distance-regular graph with vertex set $X$ and diameter $d\geq3$. 
We assume that $\Ga$ has $q$-Racah type and contains a Delsarte clique $C$. Fix a vertex $x\in C$. 
Consider the module $\W$ for $\T=\T(x,C)$ which was discussed earlier.
We show that $\W$ is an irreducible $\H$-module as well as an irreducible $\T$-module.
We show how the $\H$-action on $\W$ is related to the $\T$-action on $\W$. 
Our central result of this paper is as follows. Define
$$
\A =t_0t_1 +(t_0t_1)^{-1}, \qquad \B=t_3t_0 +(t_3t_0)^{-1}, \qquad \B^{\dagger}=t_1t_2 +(t_1t_2)^{-1}.
$$
On $\W$ and up to affine transformation,
\begin{itemize}
\item[(i)] $\A$ acts as the adjacency matrix of $\Ga$;
\item[(ii)] $\B$ acts as the dual adjacency matrix of $\Ga$ with respect to $C$;
\item[(iii)] $\B^{\dagger}$ acts as the dual adjacency matrix of $\Ga$ with respect to $x$.
\end{itemize}
Moreover on $\W$ and for appropriate $k_0, k_1 \in \mbb{C}$,
\begin{itemize}
\item[(iv)] $\frac{t_0-k_0^{-1}}{k_0-k_0^{-1}}$ acts as the projection from $\W$ onto the primary $\wt{T}$-module;
\item[(v)] $\frac{t_1-k_1^{-1}}{k_1-k_1^{-1}}$ acts as the projection from $\W$ onto the primary ${T}$-module.
\end{itemize}



\medskip \noindent
The paper is organized as follows.
It consists of two parts. In Part I, we discuss Leonard systems and $Q$-polynomial distance-regular graphs. In Part II, we discuss  the DAHA of type $(C^{\vee}_1, C_1)$.
Part I is organized as follows. 
In Section \ref{PALS} we provide some background regarding Leonard systems.
In Section \ref{q-DRG} we recall some background concerning $Q$-polynomial distance-regular graphs and the Terwilliger algebra.
In Section \ref{DelC} we review the basic properties of a Delsarte clique $C$ and study the Terwilliger algebra with respect to $C$. 
In Section \ref{subspaceW} we define the algebra $\T$ and construct a $\T$-module $\W$. 
In Sections \ref{W-T(x)}--\ref{2maps} we study the Leonard systems associated with $\W$.
In Section \ref{matrices} we identify five bases for $\W$, and display the transition matrices between certain pairs of bases among the five. 
We also display the matrices representing various linear maps in $\End(\W)$ with respect to the five bases.
These matrices will be used to prove the main results of the paper. 
Part II is organized as follows.
In Section \ref{Hq} we define the algebra $\H$ and the elements $\A, \B, \B^{\dagger}$ in $\H$.
In Section \ref{H-mod(W)} we turn $\W$ into an $\H$-module.
In Section \ref{mainsection}  we display how the $\H$-action on $\W$ is related to the $\T$-action on $\W$. 
Theorem \ref{main_thm} is the main result of the paper.
This paper ends with an Appendix in which many details are made explicit for the case of diameter 4.\\


 \ \\

\noindent
\scalemath{0.87}{
\textbf{\LARGE Part I: Leonard Systems and $Q$-polynomial Distance-Regular Graphs}}

\section{Preliminaries: Leonard systems and parameter arrays}\label{PALS}

Throughout the paper, let $d$ denote a positive integer and let $q \in \mbb{C}$ be a nonzero scalar such that $q^2 \ne 1$.
Let $\M$ denote the $\mbb{C}$-algebra consisting of all $(d+1) \times (d+1)$
matrices that have entries in $\mbb{C}$.
We index the rows and columns by $0, 1, \ldots, d$.
Throughout the paper 
$V$ denotes a vector space over $\mbb{C}$ with dimension $d+1$.
Let $\End(V)$ denote the $\mbb{C}$-algebra consisting of all 
$\mbb{C}$-linear transformations from $V$ to $V$. 
Let $I$ denote the identity of $\End(V)$.
For $A \in \End(V)$, $A$ is called {\it multiplicity-free} whenever
$A$ has $d+1$ mutually distinct eigenvalues.
Assume $A$ is multiplicity-free. Let $\{\theta_i\}^d_{i=0}$ denote
an ordering of the eigenvalues of $A$. For $0 \leq i \leq d$
let $V_i$ denote the eigenspace of $A$ associated with $\theta_i$.
Define $E_i \in \End(V)$ such that 
$(E_i-I)V_i=0$ and $E_iV_j=0$ for $j \ne i \ (0 \leq j \leq d)$.
We call $E_i$ the {\it primitive idempotent} of $A$ corresponding to 
$V_i$ (or $\theta_i$). Observe that
${\rm (i)\ } E_iE_j = \delta_{ij}E_i \ (0 \leq i,j \leq d)$;
${\rm (ii)\ } I = \sum^d_{i=0} E_i$;
${\rm (iii)\ } A = \sum^d_{i=0}\theta_iE_i$. Moreover,
\begin{equation*}
E_i = \prod_{\substack{0 \leq j \leq d \\ j \ne i}}
\frac{A-\tht_jI}{\tht_i-\tht_j}
\qquad \qquad \qquad 
(0 \leq i \leq d).
\end{equation*}
Let $\mfrk{A}$ denote the $\mbb{C}$-subalgebra of $\End(V)$
generated by $A$. Observe that each of $\{A^i\}^d_{i=0}$
and $\{E_i\}^d_{i=0}$ is a basis for $\mfrk{A}$.
We now define a Leonard system on $V$.

\begin{definition}\label{Def:LS} {\rm\cite[Definition 1.4]{Terwilli}}
By a {\it Leonard system} on ${V}$ we mean a sequence
$$\Phi := (A; A^*; \{E_i\}^d_{i=0}; \{E^*_i\}^d_{i=0})$$
that satisfies {(i)--(v)} below.
\begin{enumerate}
\item[(i)] Each of $A, A^*$ is a multiplicity-free element in $\End(V)$.
\item[(ii)] 
$\{E_i\}^d_{i=0}$ is an ordering of the primitive idempotents of $A$.
\item[(iii)] 
$\{E^*_i\}^d_{i=0}$ is an ordering 
				of the primitive idempotents of $A^*$.
\item[(iv)] For $0 \leq i, j \leq d$,
$$
E_iA^*E_j=
\begin{cases}
0 & \text{ if } |i-j|>1, \\
\neq 0 & \text{ if } |i-j| = 1.
\end{cases}
$$
\item[(v)] For $0\leq i, j \leq d$,
$$
E^*_iAE^*_j = 
\begin{cases}
0 & \text{ if } |i-j|>1, \\
\neq 0 & \text{ if } |i-j|=1.
\end{cases}
$$
\end{enumerate}
We call $d$ the {\it diameter} of $\Phi$, and say $\Phi$ is {\it over} $\mbb{C}$.
\end{definition}

\noindent
There exists a natural action of the dihedral group $D_4$
on the set of all Leonard systems.
The action is described as follows.
Let $\Phi$ denote the Leonard system from Definition \ref{Def:LS}.
 Then each of the following is a Leonard system on ${V}$:
\begin{eqnarray}
\label{LS;Phi*}
\Phi^* & := & (A^*; A; \{E^*_i\}^d_{i=0}; \{E_i\}^d_{i=0}), \\
\nonumber
\Phi^{\downarrow} & := & (A; A^*; \{E_i\}^d_{i=0}; \{E^*_{d-i}\}^d_{i=0}), \\
\nonumber
\Phi^{\Downarrow} & := & (A; A^*; \{E_{d-i}\}^d_{i=0}; \{E^*_{i}\}^d_{i=0}).
\end{eqnarray}
Viewing $*, \downarrow, \Downarrow$ as permutations on the set of all Leonard
systems, we have 
\begin{gather}
\label{D4}
*^2 =  \ \downarrow^2 \ = \ \Downarrow^2 \ = 1,  \qquad
\Downarrow * \ =\  * \downarrow, \qquad
\downarrow * \ =\  * \Downarrow, \qquad
\downarrow \Downarrow \ =\  \Downarrow \downarrow.
\end{gather}
The group generated by the symbols $*, \downarrow, \Downarrow$
subject to the relations (\ref{D4}) is the dihedral group $D_4$ of order $8$.
The permutations $*, \downarrow, \Downarrow$ induce an action of $D_4$
on the set of all Leonard systems.

\medskip \noindent
We recall the notion of  isomorphism for Leonard systems. Let 
$\Phi$ denote the Leonard system from Definition \ref{Def:LS}.
Let $V'$ denote a vector space over $\mbb{C}$ with dimension $d+1$.
Let $f: \End(V) \to \End(V')$ denote an isomorphism of $\mbb{C}$-algebras.
Write $\Phi^f=(A^f; A^{*f}; \{E^f_i\}^d_{i=0}; \{E^{*f}_i\}^d_{i=0})$ and observe
$\Phi^f$ is a Leonard system on $V'$. 
Let $\Phi$ and $\Phi'$ denote any Leonard systems over $\mbb{C}$. By an 
{\it isomorphism of Leonard systems} from $\Phi$ to $\Phi'$ we mean an isomorphism
of $\mbb{C}$-algebras $f : \End(V) \to \End(V')$ such that $\Phi^f = \Phi'$. 
We say that $\Phi$ and $\Phi'$ are {\it isomorphic}
whenever there exists an isomorphism of Leonard systems from $\Phi$ to $\Phi'$.

\medskip \noindent
Let $\Phi$ denote the Leonard system from Definition \ref{Def:LS}.
For $0 \leq i \leq d$ let $\theta_i$ (resp. $\theta^*_i$) denote the eigenvalue of
$A$ (resp. $A^*$) associated with $E_i$ (resp. $E^*_i$). 
We call $\{\tht_i\}^d_{i=0}$ (resp. $\{\tht^*_i\}^d_{i=0}$) the {\it eigenvalue
sequence} (resp. {\it dual eigenvalue sequence}) of $\Phi$.
By \cite[Theorem 3.2]{Terwilli} 
there exist nonzero scalars $\{\varphi_i\}^d_{i=1}$
and a $\mbb{C}$-algebra isomorphism 
$\natural : \End(V) \rightarrow {\bf M}_{d+1}(\mbb{C})$ such that
\begin{equation}\label{split_seq}
A^{\natural} = 
\left[
\begin{array}{cccccc}
\tht_0 &&&&& {\bf 0} \\
1 & \tht_1 &\\
& 1 & \tht_2\\
&& \cdot & \cdot \\
&&& \cdot & \cdot \\
{\bf 0} &&&&1&\tht_d
\end{array}
\right],
\qquad \qquad
A^{*\natural}=
\left[
\begin{array}{cccccc}
\tht^*_0 & \varphi_1 &&&&{\bf 0} \\
& \tht^*_1 & \varphi_2 &\\
&& \tht^*_2 & \cdot \\
&&&\cdot & \cdot & \\
&&&& \cdot & \varphi_d \\
{\bf 0} &&&&& \tht^*_d
\end{array}
\right].
\end{equation}
We call the sequence $\{\varphi_i\}^d_{i=1}$ the {\it first split sequence} of $\Phi$. 
We let $\{\phi_i\}^d_{i=0}$ denote the first split sequence of 
$\Phi^{\Downarrow}$ and call this the {\it second split sequence} of $\Phi$.
By the {\it parameter array} of $\Phi$ we mean the sequence
\begin{equation}\label{PArray}
(\{\tht_i\}^d_{i=0}, \{\tht^*_i\}^d_{i=0}, \{\varphi_i\}^d_{i=1}, \{\phi_i\}^d_{i=1}).
\end{equation}
We denote this parameter array by $p(\Phi)$.
The following theorem shows that the isomorphism class of $\Phi$ is determined by $p(\Phi)$.
\begin{theorem}{\rm\cite[Theorem 1.9]{Terwilli}}\label{thm:LS<->PA}
Let 
\begin{equation}\label{parameter_array}
(\{\tht_i\}^d_{i=0}, \{\tht^*_i\}^d_{i=0}, \{\varphi_i\}^d_{i=1}, \{\phi_i\}^d_{i=1})
\end{equation}
denote a sequence of scalars in $\mbb{C}$. Then there exists a Leonard system $\Phi$
over $\mbb{C}$ with the parameter array {\rm(\ref{parameter_array})}
if and only if the following conditions {\rm(PA1)--(PA5)} hold:
\begin{enumerate}
\item[\rm (PA1)] $\tht_i \ne \tht_j,$ \quad $\tht^*_i \ne \tht^*_j$ \quad
if \quad $i \ne j$ \qquad $(0 \leq i,j \leq d)$. 
\item[\rm(PA2)] $\varphi_i \ne 0$, \quad $\phi_i \ne 0$ \qquad $(1 \leq i \leq d)$.
\item[\rm(PA3)] 
$\varphi_i = \phi_1\sum^{i-1}_{h=0}\frac{\tht_h-\tht_{d-h}}{\tht_0-\tht_d}+
(\tht^*_i-\tht^*_0)(\tht_{i-1}-\tht_d)$ \qquad $(1 \leq i \leq d)$.
\item[\rm(PA4)] 
$\phi_i = \varphi_1\sum^{i-1}_{h=0}\frac{\tht_h-\tht_{d-h}}{\tht_0-\tht_d}+
(\tht^*_i-\tht^*_0)(\tht_{d-i+1}-\tht_0)$ \qquad $(1 \leq i \leq d)$.
\item[\rm(PA5)] The expressions
$$
\frac{\tht_{i-2}-\tht_{i+1}}{\tht_{i-1}-\tht_i}, \qquad 
\frac{\tht^*_{i-2}-\tht^*_{i+1}}{\tht^*_{i-1}-\tht^*_i}
$$
are equal and independent of $i$ for $2 \leq i \leq d-1$.
\end{enumerate}
Moreover if {\rm (PA1)--(PA5)} hold then $\Phi$ is unique up to isomorphism
of Leonard systems.
\end{theorem} 
\noindent
By a {\it parameter array over $\mbb{C}$ of diameter $d$} we mean
a sequence of complex scalars 
$(\{\tht_i\}^d_{i=0}, \{\tht^*_i\}^d_{i=0},$ $\{\varphi_i\}^d_{i=1}, \{\phi_i\}^d_{i=1})$
which satisfies the conditions (PA1)--(PA5).
By Theorem \ref{thm:LS<->PA}, the map which sends a given
Leonard system to its parameter array induces a bijection from the set of
isomorphism classes of Leonard systems over $\mbb{C}$ to the set of 
parameter arrays over $\mbb{C}$. 
In \cite{Terwilli:PA} Terwilliger displayed all the parameter arrays over $\mbb{C}$.

\medskip \noindent
Let $\Phi$ denote the Leonard system from Definition \ref{Def:LS}.
We now recall the $\Phi$-standard basis.
Let $\bv$ be a nonzero vector in $E_0V$.
By \cite[Lemma 10.2]{Terwilli:Madrid}, the sequence $\{E^*_i\bv\}^d_{i=0}$
is a basis for $V$.
This basis is said to be {\it $\Phi$-standard}.
The following is a characterization of the $\Phi$-standard basis.

\begin{lemma}{\rm\cite[Lemma 10.4]{Terwilli:Madrid}}\label{Phi-sb}
Let $\{\bv_i\}^d_{i=0}$ denote a sequence of vectors in $V$, not all $0$.
Then this sequence is a $\Phi$-standard basis for $V$ if and only if both
\begin{enumerate}
\item[\rm(i)] $\bv_i \in E^*_i V$ for $0 \leq i \leq d$;
\item[\rm(ii)] $\sum^d_{i=0}\bv_i \in E_0 V$.
\end{enumerate}
\end{lemma}

\noindent
Let $\Phi$ denote the Leonard system from Definition \ref{Def:LS} and 
$(\{\tht_i\}^d_{i=0}, \{\tht^*_i\}^d_{i=0}, \{\varphi_i\}^d_{i=1}, \{\phi_i\}^d_{i=1})$ 
denote the parameter array of  $\Phi$. For all $X \in \End(V)$ 
let $X^{\flat}$ denote the matrix in $\M$ that represents $X$ relative to
a $\Phi$-standard basis for $V$. By Definition \ref{Def:LS}  and construction 
we have
\begin{equation}\label{[A^(*flat)]}
A^{*\flat} = \diag
(\tht^*_0, \tht^*_1, \tht^*_2, \ldots, \tht^*_d).
\end{equation}
Moreover, there exist scalars
$\{a_i\}^d_{i=0},\{b_i\}^{d-1}_{i=0}, \{c_i\}^{d}_{i=1}$ in $\mathbb{C}$ such that 
\begin{equation}\label{[A^(flat)]}
A^{\flat} = 
\left[
\begin{array}{ccccc}
a_0 & b_0 &&& {\bf 0}\\
c_1 & a_1 & b_1 && \\
& c_2 & a_2 & \ddots & \\
&& \ddots &\ddots & b_{d-1} \\
{\bf 0}&&&c_d&a_d
\end{array}
\right].
\end{equation}
We call $a_i, b_i, c_i$ the {\it intersection numbers of $\Phi$}.
For notational convenience define $b_d=0$ and $c_0=0$.
By \cite[Lemma 10.5]{Terwilli:Madrid}, $A^{\flat}$ has constant row sum $\tht_0$. 
Therefore $a_i = \tht_0 - b_i - c_i \ (0 \leq i \leq d)$.
In the following two lemmas, we give the 
intersection numbers of $\Phi$ in terms of the parameter array of $\Phi$.
\begin{lemma}{\rm\cite[Theorem 23.5]{Terwilli:Madrid}}\label{Madrid(bi;ci)}
The intersection numbers $b_i, c_i$ of $\Phi$ are 
\begin{align}
\label{b_i;parray}
& b_i = 
\varphi_{i+1}\frac{(\tht^*_i-\tht^*_0)(\tht^*_i-\tht^*_1)\cdots(\tht^*_i-\tht^*_{i-1})}
			{(\tht^*_{i+1}-\tht^*_0)(\tht^*_{i+1}-\tht^*_1)\cdots(\tht^*_{i+1}-\tht^*_i)}
			\qquad \qquad (0 \leq i \leq d-1), &&& \\
\label{c_i;parray}
& c_i =
\phi_i\frac{(\tht^*_i-\tht^*_d)(\tht^*_i-\tht^*_{d-1})\cdots(\tht^*_i-\tht^*_{i+1})}
		{(\tht^*_{i-1}-\tht^*_d)(\tht^*_{i-1}-\tht^*_{d-1})\cdots(\tht^*_{i-1}-\tht^*_i)}
		\qquad \qquad (1 \leq i \leq d).	&&&
\end{align}
\end{lemma}

\begin{lemma}{\rm\cite[Theorem 23.6]{Terwilli:Madrid}}\label{Madrid(ai)}
The intersection numbers $a_i$ of $\Phi$ are
\begin{align}
\label{Madrid(ai);eq1}
a_0 & = \tht_0 + \frac{\varphi_1}{\tht^*_0-\tht^*_1}, \\
\nonumber
a_i &= \tht_i +\frac{\varphi_i}{\tht^*_i-\tht^*_{i-1}} + \frac{\varphi_{i+1}}{\tht^*_i-\tht^*_{i+1}}
\qquad \qquad (1 \leq i \leq d-1),&& \\
\nonumber
a_d & = \tht_d + \frac{\varphi_d}{\tht^*_d-\tht^*_{d-1}}.
\end{align} 
\end{lemma}

\noindent
We recall the Leonard system $\Phi^*$ from (\ref{LS;Phi*}). 
By {\it the dual intersection numbers of $\Phi$} 
we mean the intersection numbers of $\Phi^*$.
These are denoted by $a^*_i, b^*_i, c^*_i$.

\begin{lemma}{\rm\cite[Theorem 1.11]{Terwilli}},
{\rm\cite[Theorem 23.5]{Terwilli:Madrid}}\label{Madrid(bi;ci*)}
The dual intersection numbers $b^*_i, c^*_i$ of $\Phi$ are
\begin{align*}
& b^*_i = 
\varphi_{i+1}\frac{(\tht_i-\tht_0)(\tht_i-\tht_1)\cdots(\tht_i-\tht_{i-1})}
			{(\tht_{i+1}-\tht_0)(\tht_{i+1}-\tht_1)\cdots(\tht_{i+1}-\tht_i)}
			\qquad \qquad \qquad (0 \leq i \leq d-1), &&&\\
& c^*_i =
\phi_{d-i+1}\frac{(\tht_i-\tht_d)(\tht_i-\tht_{d-1})\cdots(\tht_i-\tht_{i+1})}
		{(\tht_{i-1}-\tht_d)(\tht_{i-1}-\tht_{d-1})\cdots(\tht_{i-1}-\tht_i)}
		\qquad \qquad (1 \leq i \leq d).	&&&
\end{align*}
\end{lemma}

\begin{lemma}{\rm\cite[Theorem 1.11]{Terwilli}},
{\rm\cite[Theorem 23.6]{Terwilli:Madrid}}
\label{Madrid(ai*)}
The dual intersection numbers $a^*_i$ of $\Phi$ are
\begin{align}
\nonumber
a^*_0 & = \tht^*_0 + \frac{\varphi_1}{\tht_0-\tht_1}, \\
\nonumber
a^*_i &= \tht^*_i +\frac{\varphi_i}{\tht_i-\tht_{i-1}} + \frac{\varphi_{i+1}}{\tht_i-\tht_{i+1}}
\qquad \qquad (1 \leq i \leq d-1),&&\\
\nonumber
a^*_d & = \tht^*_d + \frac{\varphi_d}{\tht_d-\tht_{d-1}}.
\end{align} 
\end{lemma}

\noindent
Let $\Phi$ denote the Leonard system from Definition \ref{Def:LS}.
Recall the parameter array $p(\Phi)$ from (\ref{PArray}).
Let $\lambda$ denote  an indeterminate and let $\mbb{C}[\lambda]$
denote the $\mbb{C}$-algebra consisting of the polynomials in $\lambda$
that have all coefficients in $\mbb{C}$.
For $0 \leq i \leq d$ define a polynomial $\mfrk{u}_i \in \mbb{C}[\lambda]$ by 
\begin{equation}\label{poly(u_i)}
\mfrk{u}_i = 
\sum^i_{n=0} \frac{(\tht^*_i-\tht^*_0)(\tht^*_i-\tht^*_1)\cdots(\tht^*_i-\tht^*_{n-1})
					(\lambda-\tht_0)(\lambda-\tht_1)\cdots(\lambda-\tht_{n-1})}
{\varphi_1\varphi_2 \cdots \varphi_n}.
\end{equation}
By \cite[Theorem 7.2]{Terwilli:PA}, 
$$
\tht_j\mfrk{u}_i(\tht_j) = c_i\mfrk{u}_{i-1}(\tht_j) + a_i\mfrk{u}_i(\tht_j) + b_i\mfrk{u}_{i+1}(\tht_j)
\qquad \qquad (0 \leq i,j \leq d),
$$
where $a_i, b_i, c_i$ are the intersection numbers of $\Phi$
and $\mfrk{u}_{-1}$ and $\mfrk{u}_{d+1}$ are indeterminates.
By \cite[Section 10]{Terwilli:PA},
\begin{equation}\label{u_i(tht_d)}
\mfrk{u}_i(\tht_d) = \frac{\phi_1\phi_2 \dotsm \phi_i}{\varphi_1\varphi_2 \dotsm \varphi_i}
\qquad (0 \leq i \leq d).
\end{equation}

 \noindent
We now recall the $q$-Racah family of parameter arrays \cite{Terwilli:PA}.
This is the most general family.

\begin{example}\label{q-rac;PA}
\cite[Example 5.3]{Terwilli:PA}
For $0 \leq i \leq d$ define
\begin{eqnarray}
\label{theta_i}
\theta_i 	& = & \theta_0+h(1-q^{i})(1-sq^{i+1})q^{-i}, \\
\label{theta^*_i}
\theta^*_i 	& = & \theta^*_0+h^*(1-q^{i})(1-s^*q^{i+1})q^{-i},
\end{eqnarray}
and for $1\leq i \leq d$ define
\begin{eqnarray}
\label{varphi_i}
\varphi_i	& = & hh^{*}q^{1-2i}(1-q^{i})(1-q^{i-d-1})(1-r_1q^i)(1-r_2q^i), \\
\label{phi_i}
\phi_i 		& = & hh^{*}q^{1-2i}(1-q^{i})(1-q^{i-d-1})(r_1-s^*q^i)(r_2-s^*q^i)/s^*,
\end{eqnarray} 
where $\tht_0$ and $\tht_0^*$ are scalars in $\mbb{C}$, and
where $h, h^*, s, s^*, r_1, r_2$ are nonzero scalars in $\mbb{C}$ such that
$r_1r_2=ss^*q^{d+1}$.
To avoid degenerate situations assume that none of 
$q^i, r_1q^i, r_2q^i, s^*q^i/r_1, s^*q^i/r_2$ is equal to 1 for $1\leq i \leq d$
and that neither of $sq^i, s^*q^i$ is equal to 1 for $2 \leq i \leq 2d$.
Then the sequence
$(\{\tht_i\}^d_{i=0}, \{\tht^*_i\}^d_{i=0},\{\varphi_i\}^d_{i=1}, \{\phi_i\}^d_{i=1})$
is a parameter array over $\mbb{C}$.
This parameter array is said to have {\it $q$-Racah type}.
\end{example}

\noindent
Let $\Phi$ denote the Leonard system from Definition \ref{Def:LS}.
We say that $\Phi$ has {\it  $q$-Racah type} whenever its parameter array has 
$q$-Racah type. Assume that $\Phi$ has $q$-Racah type
with the parameter array as in Example \ref{q-rac;PA}.
Recall the intersection numbers $\{b_i\}^{d-1}_{i=0}, \{c_i\}^d_{i=1}$ of $\Phi$ from Lemma \ref{Madrid(bi;ci)}.
Evaluating (\ref{b_i;parray}) and (\ref{c_i;parray})
using (\ref{theta_i})--(\ref{phi_i})
we find 
\begin{align}
\label{b_0;q-terms}
b_0 &= \frac{h(1-q^{-d})(1-r_1q)(1-r_2q)}{1-s^*q^2}, \\
\label{b_i;q-terms}
b_i &= \frac{h(1-q^{i-d})(1-s^*q^{i+1})(1-r_1q^{i+1})(1-r_2q^{i+1})}
{(1-s^*q^{2i+1})(1-s^*q^{2i+2})} &&(1 \leq i \leq d-1), \\
\label{c_i;q-terms}
c_i &= \frac{h(1-q^i)(1-s^*q^{i+d+1})(r_1-s^*q^i)(r_2-s^*q^i)}
{s^*q^d(1-s^*q^{2i})(1-s^*q^{2i+1})} && (1 \leq i \leq d-1), \\
\label{c_d;q-terms}
c_d &= \frac{h(1-q^d)(r_1-s^*q^d)(r_2-s^*q^d)}{s^*q^d(1-s^*q^{2d})}; 
\end{align}
see \cite[Section 24]{Terwilli:Madrid}.
To obtain the dual intersection numbers $\{b^*_i\}^{d-1}_{i=0}, \{c^*_i\}^d_{i=1}$ of $\Phi$, 
replace $(h, s^*)$ by $(h^*, s)$ in (\ref{b_0;q-terms})--(\ref{c_d;q-terms}).

\medskip \noindent
Now consider (\ref{poly(u_i)}). Pick integers $i,j\ (0 \leq i,j \leq d)$.
Evaluating (\ref{poly(u_i)}) at $\lambda=\tht_j$ and simplifying the result 
using (\ref{theta_i})--(\ref{phi_i}) we get
\begin{equation}\label{u_i(tht_j)}
\mfrk{u}_i(\tht_j) = \sum^i_{n=0}
	\frac{(q^{-i};q)_n(s^*q^{i+1};q)_n(q^{-j};q)_n(sq^{j+1};q)_nq^n}
		{(r_1q;q)_n(r_2q;q)_n(q^{-d};q)_n(q;q)_n},
\end{equation}
where 
$$
(a;q)_n := (1-a)(1-aq)(1-aq^2) \cdots (1-aq^{n-1}) \qquad n=0,1,2, \ldots
$$
From the definition of basic hypergeometric series \cite[p. 4]{Gasper}, 
the sum on the right in (\ref{u_i(tht_j)}) is 
\begin{equation}\label{u_i(tht_j)hgs}
\  _{4}\phi_{3}
\left(
\begin{matrix}
q^{-i}, s^*q^{i+1}, q^{-j}, sq^{j+1}\\
r_1q, r_2q, q^{-d}
\end{matrix}
\ \middle| \ q,q
\right).
\end{equation}
The $q$-Racah polynomials are defined in \cite{Askey}.
By (\ref{u_i(tht_j)}) and (\ref{u_i(tht_j)hgs}), the $\{\mfrk{u}_i\}^d_{i=0}$ are $q$-Racah polynomials.


\medskip
\section{Preliminaries: $Q$-polynomial distance-regular graphs} \label{q-DRG}

We now turn our attention to distance-regular graphs.
In this section we review those aspects of $Q$-polynomial distance-regular graphs
that we will need later in the paper.
For more background information we refer the reader
to Brouwer, Cohen and Neumaier \cite{BCN} and Terwilliger \cite{T-algebra-I}.

\medskip \noindent
Let $X$ denote a nonempty finite set.
Let $\MX$ denote the $\mbb{C}$-algebra consisting of the matrices
with entries in $\mbb{C}$ whose rows and columns are indexed by $X$. 
Let $\V=\CX$ denote the $\mbb{C}$-vector space consisting of column vectors
with entries in $\mbb{C}$ and rows indexed by $X$.
We view $\V$ as a left module for $\MX$ and call this the {\it standard module}.
We endow $\V$ with the Hermitean inner product
$\langle \ , \ \rangle$ that satisfies 
$\langle {u},  {v} \rangle = {u}^t{\overline{v}}$ for
${u}, {v} \in \V$, where
$t$ denotes transpose and $\ \bar{\ } $ denotes complex conjugate.
For all $x \in X$ let $\hat{x}$ denote the vector in $\V$ with
a $1$ in the $x$ coordinate and $0$ in all other coordinates.
Note that the set $\{\hat{x} \mid x \in X \}$ is an orthonormal basis for $\V$.
For a subset $Y \subseteq X$, define $\hat{Y} = \sum_{y \in Y} \hat{y}$
and call this the {\it characteristic vector of $Y$}.
We abbreviate ${\bf j} = \hat{X}$. The vector $\bf j$ has $x$-coordinate $1$ for all $x \in X$.

\medskip \noindent
Let $\Ga$ denote an undirected, connected graph, without loops or multiple edges, with vertex set $X$ and diameter $d \geq 3$.
For $x \in X$ define
\begin{equation}\label{Ga_i}
\Ga_i(x) = \{y \in X \mid \partial(x,y) = i \} \qquad (0 \leq  i \leq d),
\end{equation}
where $\partial$ denotes the path-length distance function.
We abbreviate $\Ga(x) = \Ga_1(x)$.
For an integer $k \geq 0$, $\Ga$ is said to be {\it regular with valency $k$}
whenever $|\Ga(x)| = k$ for all $x \in X$.
We say that $\Ga$ is {\it distance-regular} whenever for all integers $h, i, j$
$(0 \leq h, i, j \leq d)$ and all vertices $x,y \in X$ with $\partial(x,y) = h$,
the number $p^h_{i,j} = |\Ga_i(x) \cap \Ga_j(y)|$ is independent of $x$ and $y$.
The numbers $p^h_{i,j}$ are called the {\it intersection numbers of $\Ga$}.
Abbreviate
$$
a_i(\Ga)=p^i_{1,i} \ (0 \leq i \leq d), \qquad 
b_i(\Ga)=p^i_{1,i+1} \ (0 \leq i \leq d-1), \qquad 
c_i(\Ga)=p^i_{1,i-1} \ (1 \leq i \leq d),
$$
and define $b_d(\Ga)=0$ and $c_0(\Ga) = 0$.
By construction $a_0(\Ga)=0$ and $c_1(\Ga)=1$.
For the rest of the paper, assume that 
 $\Ga$ is distance-regular.
Observe that $\Ga$ is regular with valency $k=b_0(\Ga)$.

\medskip \noindent
We recall the Bose-Mesner algebra of $\Ga$.
For $0 \leq i \leq d$ let $A_i$ denote the matrix in $\MX$ with 
$(x,y)$-entry 
$$
(A_i)_{xy}  =
\begin{cases}
1 & \text{if } \partial(x,y) = i \\
0 & \text{if } \partial(x,y) \ne i
\end{cases}
\qquad (x, y \in X).
$$
We call $A_i$ the {\it $i$-th distance matrix of $\Ga$}.
Observe that $\sum^d_{i=0}A_i = {\bf J}$, the all-ones matrix,
and $A_iA_j = \sum^d_{h=0}p^h_{i,j}A_h$ for $0 \leq i, j \leq d$.
We abbreviate $A=A_1$ and call this the 
{\it adjacency matrix of $\Ga$}.
Let $\mcal{M}$ denote the subalgebra of $\MX$ generated by $A$.
The algebra $\mcal{M}$ is commutative.
We call $\mcal{M}$ the {\it Bose-Mesner algebra} of $\Ga$.
By \cite[p.~127]{BCN} the set $\{A_i\}^d_{i=0}$ is a basis for $\mcal{M}$.
The algebra $\mcal{M}$ is  semisimple
since it is closed under the conjugate-transpose map. By \cite[p.~45]{BCN}
$\mcal{M}$ has a basis $\{E_i\}^d_{i=0}$ such that
(i) $E_0 = |X|^{-1}{\bf J}$; (ii) $I=\sum^d_{i=0} E_i$; (iii) $\overline{E}_i = E_i~(0 \leq i \leq d)$;
(iv) $E^{t}_i = E_i~(0 \leq i \leq d)$; (v) $E_iE_j = \delta_{ij}E_i~(0 \leq i, j \leq d)$.
We call $\{E_i\}^d_{i=0}$ the {\it primitive idempotents} of $\Ga$. 
Since $\{E_i\}^d_{i=0}$ is a basis for $\mcal{M}$ there exist complex scalars
$\{\tht_i\}^d_{i=0}$ such that $A=\sum^d_{i=0}\tht_iE_i$.
Observe that $AE_i=E_iA=\tht_iE_i$ for $0 \leq i \leq d$.
By  \cite[p.~197]{B-Ito} the scalars $\{\tht_i\}^d_{i=0}$ are real.
The scalars $\{\tht_i\}^d_{i=0}$ are mutually distinct since
$A$ generates $\mcal{M}$. We call $\tht_i$ 
the {\it eigenvalue of $\Ga$} associated with $E_i$ $(0 \leq i \leq d)$.
By (i) we have $k=\tht_0$.
Note that $\V=\sum^d_{i=0}E_i\V$  (orthogonal direct sum).
For $0 \leq i \leq d$,  $E_i\V$ is the eigenspace of $A$ associated with $\tht_i$.
Let $m_i$ denote the rank of $E_i$ and 
note that $m_i$ is the dimension of $E_i\V$. We call $m_i$
the {\it multiplicity} of $E_i$ (or $\tht_i$).
Note that the vector ${\bf j}$ is a basis for $E_0\V$.

\medskip \noindent
We recall the $Q$-polynomial property. 
Let $\circ$ denote the entrywise product in $\MX$.
Observe that $A_i \circ A_j = \delta_{ij}A_i$ for $0 \leq i, j \leq d$, so
$\mcal{M}$ is closed under $\circ$. Therefore,
there exist complex scalars $q^h_{i,j} \ (0 \leq h, i, j \leq d)$ such that
$E_i \circ E_j = |X|^{-1}\sum^d_{h=0}q^{h}_{i,j}E_h$ for $0 \leq i, j \leq d$.
By \cite[p.~48,~49]{BCN} each $q^h_{i,j}$ is real and nonnegative.
The $q^h_{i,j}$ are called the {\it dual intersection numbers} 
(or {\it Krein parameters}) of $\Ga$. We abbreviate
$$
a^*_i(\Ga) = q^i_{1,i} \ (0 \leq i \leq d), \qquad
b^*_i(\Ga) = q^i_{1, i+1} \ (0 \leq i \leq d-1), \qquad
c^*_i(\Ga) = q^i_{1, i-1} \ (1 \leq i \leq d),
$$
and define $b_d^*(\Ga) = 0$ and $c^*_0(\Ga) = 0$.
By \cite[p.~193]{B-Ito} $c^*_1(\Ga)=1$.
The graph $\Ga$ is said to be {\it $Q$-polynomial} (with respect to
the ordering $\{E_i\}^d_{i=0}$ of the primitive idempotents) whenever
for $0 \leq h,i,j \leq d$, $q^h_{i,j} = 0$ (resp. $q^h_{i,j} \ne 0$)
if one of $h,i,j$ is greater than (resp. equal to)
the sum of the other two \cite[p.~235]{BCN}. 
For the rest of this paper,  assume that 
$\Ga$ is $Q$-polynomial with respect to the ordering $\{E_i\}^d_{i=0}$.
It is not necessarily the case that $\tht_0>\tht_1>\cdots>\tht_d$.
However, this ordering occurs in many examples \cite[Chapter~8]{BCN}.
To keep our discussion simple, 
we always assume that our $Q$-polynomial structure satisfies 
$\tht_0>\tht_1>\cdots>\tht_d$.

\medskip \noindent
We recall the dual Bose-Mesner algebra of $\Ga$.
For the rest of this section, fix $x \in X$. We view $x$ as a base vertex.
For $0 \leq i \leq d$ let $E^*_i = E^*_i(x)$ denote the diagonal
matrix in $\MX$ with $(y,y)$-entry
\begin{equation}\label{def(E^*(x))}
(E^*_i)_{yy} =
\begin{cases}
1 & \text{if } \partial(x,y) = i \\
0 & \text{if } \partial(x,y) \ne i
\end{cases}
\qquad (y \in X).
\end{equation}
We call $E_i^*$ the {\it $i$-th dual primitive idempotent} of $\Ga$ with respect to $x$.
Observe that (i) $I=\sum^d_{i=0}E^*_i$; (ii) $\overline{E^*_i} = E^*_i \ (0 \leq i \leq d)$; 
(iii) $(E^*_i)^{t} = E^*_i \ (0 \leq i \leq d)$; (iv) $E^*_iE^*_j = \delta_{ij}E^*_i$ $(0 \leq i, j \leq d)$.
By these facts, $\{E^*_i\}^d_{i=0}$ forms a basis for a commutative subalgebra
$\mcal{M}^*=\mcal{M}^*(x)$ of $\MX$.
We call $\mcal{M}^*$ the {\it dual Bose-Mesner algebra} of $\Ga$ with respect to $x$.
The algebra $\mcal{M}^*$ is semisimple since it is closed under the conjugate-transpose map.
For $0 \leq i \leq d$ let $A^*_i = A^*_i(x)$ denote the diagonal matrix in $\MX$
with $(y,y)$-entry
$$
(A^*_i)_{yy} = |X|(E_i)_{xy} \qquad (y \in X).
$$
We call $A^*_i$ the {\it $i$-th dual distance matrix} of $\Ga$ with respect to $x$.
By \cite[p.~379]{T-algebra-I}, $\{A^*_i\}^d_{i=0}$ is a basis for $\mcal{M}^*$.
We abbreviate $A^*=A^*_1$ and call this the {\it dual adjacency matrix} of $\Ga$
with respect to $x$. By \cite[Lemma 3.11]{T-algebra-I}, $A^*$ generates 
$\mcal{M}^*$. Since $\{E^*_i\}^d_{i=0}$ is a basis for $\mcal{M}^*$, there
exist complex scalars $\{\tht^*_i\}^d_{i=0}$ such that $A^*=\sum^d_{i=0}\tht^*_iE^*_i$.
Observe that $A^*E^*_i=E^*_iA^* = \tht^*_iE^*_i$ for $0 \leq  i \leq d$.
By \cite[Lemma 3.11]{T-algebra-I}, the $\{\tht^*_i\}^d_{i=0}$ are real. 
These scalars are mutually distinct since $A^*$ generates $\mcal{M}^*$.
We call $\tht^*_i$ the {\it dual eigenvalue} of $\Ga$ associated with $E^*_i~(0 \leq i \leq d)$.
Note that $\V=\sum^{d}_{i=0}E^*_i\V~\text{(orthogonal direct sum)}$.
From (\ref{def(E^*(x))}), for $0 \leq i \leq d$ we find $E^*_i\V = \text{Span}\{\hat{y} \mid y \in X, \partial(x,y)=i\}$. $E^*_i\V$ is the eigenspace of $A^*$ associated with $\tht^*_i$. We call $E^*_i\V$ the {\it $i$-th subconstituent} of $\Ga$ with respect to $x$.

\medskip \noindent
We now recall the Terwilliger algebra.
Let $T=T(x)$ denote the subalgebra of $\MX$ generated by $\mcal{M}$ and $\mcal{M}^*$.
$T$ is called the {\it Terwilliger algebra} (or {\it subconstituent algebra}) of $\Ga$
with respect to $x$; see \cite{T-algebra-I}. Note that $A, A^*$ generate $T$.
$T$ is finite-dimensional and noncommutative. 
Moreover, $T$ is semisimple since it is closed under the conjugate-transpose map.
By \cite[Lemma~3.2]{T-algebra-I}, the following are relations in $T$. For $0 \leq h,i,j \leq d$,
\begin{align*}
E^*_iA_hE^*_j = 0 \qquad \text{ if and only if } \qquad p^h_{ij}=0;\\
E_iA^*_hE_j = 0 \qquad \text{ if and only if } \qquad q^h_{ij}=0.
\end{align*}
Setting $h=1$ we find that for $0 \leq i, j \leq d$,
\begin{align} 
E^*_iAE^*_j = 0 \qquad \text{ if  } \qquad |i-j|>1;\nonumber \\
\label{EA*E}
E_iA^*E_j = 0 \qquad \text{ if  } \qquad |i-j|>1.
\end{align}

\noindent
By a {\it $T$-module}, we mean a subspace $W \subseteq \V$ such that
$BW \subseteq W$ for all $B \in T$. Let $W$ denote a $T$-module
and let $U$ be a $T$-submodule of $W$. Since
$T$ is closed under the conjugate-transpose map, the orthogonal complement 
of $U$ in $W$ is a $T$-module. Hence $W$ decomposes into 
an orthogonal direct sum of irreducible $T$-modules. In particular,
$\V$ is an orthogonal direct sum of irreducible $T$-modules.

\medskip \noindent
Let $W$ denote an irreducible $T$-module. 
Then $W$ is a direct sum of the nonzero spaces among $\{E^*_iW\}^d_{i=0}$,
and also a direct sum of the nonzero spaces among $\{E_iW\}^d_{i=0}$.
By the {\it endpoint} of $W$ we mean $\min\{ i \mid 0 \leq i \leq d, E^*_iW \ne 0 \}$.
By the {\it dual endpoint} of $W$ we mean $\min\{ i \mid 0 \leq i \leq d, E_iW \ne 0 \}$.
By the {\it diameter} of $W$ we mean $|\{ i \mid 0 \leq i \leq d, E^*_iW \ne 0 \}| - 1$.
By the {\it dual diameter} of $W$ we mean $|\{ i \mid 0 \leq i \leq d, E_iW \ne 0 \}| - 1$.
By \cite[Corollary 3.3]{Pascasio} the diameter of $W$ is equal to the dual diameter of $W$.
By \cite[Lemma 3.9, Lemma 3.12]{T-algebra-I}
$\dim{E^*_iW} \leq 1~(0 \leq i \leq d)$ if and only if $\dim{E_iW} \leq 1~(0 \leq i \leq d)$;
in this case $W$ is said to be {\it thin}. 

\begin{lemma}{\rm\cite[Lemma 14.8]{Boyd}}\label{LS(U)}
Let $W$ denote a thin irreducible $T$-module with  endpoint $\sigma$,
 dual endpoint $\sigma^*$ and diameter $\rho$.
Then the elements
$$(A; A^*; \{E_{i}\}^{\rho+\sigma^*}_{i=\sigma^*}; \{E^*_{i}\}^{\rho+\sigma}_{i=\sigma})$$
act on $W$ as a Leonard system.
\end{lemma}

\noindent
We give an example of a thin irreducible $T$-module.

\begin{example}\cite[Lemma 3.6]{T-algebra-I}\label{ex;Mx}
The all-ones vector $\bf{j}$
satisfies $A_i\hat{x} = E^*_i{\bf j}$ and $A^*_i{\bf j} = |X|E_i\hat{x}$
for $0 \leq i \leq d$. Therefore $\Mx = \mcal{M}^*{\bf j}$.
The space $\Mx$ is a thin irreducible $T$-module with endpoint $0$, dual endpoint $0$, and diameter $d$.
The space $\Mx$ has bases $\{A_i\hat{x}\}^d_{i=0}$ and $\{E_i\hat{x}\}^d_{i=0}$.
We call $\Mx$ the {\it primary} $T$-module.
\end{example}

\begin{lemma}{\rm\cite[Theorem 4.1]{T-algebra-II}}\label{IN(Ga,Phi)}
Let $\Phi$ denote the Leonard system on $\Mx$ from Lemma {\rm\ref{LS(U)}}.
\begin{enumerate}
\item[\rm(i)] For $0 \leq i \leq d$, the intersection numbers $a_i, b_i, c_i$ of $\Phi$ satisfy
\begin{align*}
a_i = a_i(\Ga), && b_i =  b_i(\Ga),  && c_i = c_i(\Ga). &&&&&&
\end{align*}
\item[\rm(ii)] For $0 \leq i \leq d$, the dual intersection numbers $a^*_i, b^*_i, c^*_i$ of $\Phi$
satisfy
\begin{align*}
a^*_i = a^*_i(\Ga), && b^*_i =  b^*_i(\Ga),  && c^*_i = c^*_i(\Ga). &&&&&&
\end{align*}
\end{enumerate}
\end{lemma}
\noindent
The Leonard system $\Phi$ in Lemma \ref{IN(Ga,Phi)} will be called {\it primary}, 
and its parameter array will also be called {\it primary}. 
Moreover, the intersection numbers and dual intersection numbers of $\Phi$ 
will be called {\it primary}. From now on, the notation
\begin{equation*}\label{PrimaryPA}
p(\Phi) = (\{\tht_i\}^d_{i=0}, \{\tht^*_i\}^d_{i=0}, \{\varphi_i\}^d_{i=1}, \{\phi_i\}^d_{i=1})
\end{equation*}
refers to the primary parameter array. 
Furthermore,  the notation $a_i, b_i, c_i$ (resp. $a^*_i, b^*_i, c^*_i$) 
refer to the primary intersection numbers (resp. dual intersection numbers) of $\Ga$
which are also the intersection numbers (resp. dual intersection numbers) of $\Ga$.

\medskip \noindent
We say that $\Ga$ has {\it $q$-Racah type} whenever the primary Leonard system
has $q$-Racah type. 
For the rest of the paper, assume that $\Ga$ has $q$-Racah type.
Since the primary parameter array has $q$-Racah type, 
it satisfies (\ref{theta_i})--(\ref{phi_i}) for some scalars 
$h, h^*, s, s^*, r_1, r_2$. We fix this notation for the rest of the paper.

\begin{note}\label{h,h*}
We mentioned earlier that the intersection number $c_1=1$.
In (\ref{c_i;q-terms}) we set $i=1$ and use $c_1=1$ to get
\begin{equation}\label{scalar(h)}
h = 
\frac{s^*q^d(1-s^*q^2)(1-s^*q^3)}{(1-q)(1-s^*q^{d+2})(r_1-s^*q)(r_2-s^*q)}.
\end{equation}
To get $h^*$ replace $s^*$ by $s$ in (\ref{scalar(h)}). In the resulting formula,
eliminate $s$ using $r_1r_2=ss^*q^{d+1}$ to get
\begin{equation}\label{scalar(h*)}
h^* =
\frac{q^{2d-1}(s^*-r_1r_2q^{1-d})(s^*-r_1r_2q^{2-d})}{(1-q)(s^*-r_1r_2q)(r_1-s^*q^d)(r_2-s^*q^d)}.
\end{equation}
\end{note}


\section{Delsarte cliques in $\Ga$}\label{DelC}

In this section we discuss the basic properties of a Delsarte clique $C$ of $\Ga$.
We also discuss the Terwilliger algebra associated with $C$.

\medskip \noindent
We recall the definition of a Delsarte clique.
By a {\it clique} of $\Ga$ we mean a nonempty subset of $X$ 
such that any two distinct vertices of this subset are adjacent.
Let $C$ denote a clique of $\Ga$. We mention an upper bound on $|C|$.
Recall that $k$ is the valency of $\Ga$. 
By \cite[Corollary 3.5.4]{BCN} $\tht_d \leq -1$.
By \cite[Proposition 4.4.6]{BCN}
$|C| \leq 1 - {k}/{\tht_d}$.
We say that $C$ is {\it Delsarte} whenever $|C| = 1-{k}/{\tht_d}$.
Throughout the rest of this paper 
we will assume that  $\Ga$ contains a Delsarte clique $C$.
In our study of $C$, we will use the following fact.

\begin{lemma}{\rm\cite[Lemma 2.1]{JKT}}\label{JKT}
For $0 \leq j \leq d$ and for $y,z \in X$,
\begin{equation}\label{EiEj=m,u,tht}
\langle E_j\hat{y}, E_j\hat{z} \rangle
={|X|}^{-1}m_j\mfrk{u}_i(\tht_j),
\end{equation}
where $i=\partial(y,z)$ and $\frak{u}_i$ is the polynomial from {\rm(\ref{poly(u_i)})}
attached to the primary parameter array.
Recall that $m_j$ is the multiplicity of $\tht_j$.
\end{lemma}

\begin{lemma}\label{EdC=0}
$E_j\hat{C} \ne 0$ for $0 \leq j \leq d-1$. Moreover, $E_d\hat{C} = 0$.
\end{lemma}

\begin{proof}
We evaluate $\|E_j\hat{C}\|^2$ for $0 \leq j \leq d$. 
We find
\begin{equation}\label{EdC}
\|E_j\hat{C}\|^2 = \sum_{y,z \in C} \langle E_j\hat{y}, E_j\hat{z} \rangle.
\end{equation}
For $y,z \in C$ consider the corresponding summand in (\ref{EdC}).
First assume that $y=z$. 
In (\ref{EiEj=m,u,tht}) set $i=0$ and $\frak{u}_0 = 1$ to find
$\| E_j\hat{y} \|^2 = |X|^{-1}m_j.$
Next assume that $y \ne z$. 
Then $y,z$ are adjacent. In (\ref{EiEj=m,u,tht})
set $i=1$ and $\frak{u}_1(\tht_j)=\tht_j/k$ to find
$ \langle E_j\hat{y}, E_j\hat{z} \rangle = |X|^{-1}m_j\tht_j/k. $
Evaluate (\ref{EdC}) using these comments to get
\begin{equation}\label{pf(EdC=0)}
\|E_j\hat{C}\|^2 = |C|\frac{m_j}{|X|} + |C|(|C|-1)\frac{m_j\tht_j}{|X|k}.
\end{equation}
In (\ref{pf(EdC=0)}) divide both sides by $|C|$ and use $|C|=1-k/\tht_d$ to get 
\begin{equation*}\label{EjC}
\frac{\|E_j\hat{C}\|^2}{|C|} = \frac{m_j(\tht_d-\tht_j)}{|X|\tht_d}.
\end{equation*}
The factor $\tht_d-\tht_j$ is nonzero for $0 \leq j \leq d-1$, and zero for $j=d$.
The result follows.
\end{proof}

\noindent
By a {\it partition} of $X$, we mean a set of mutually disjoint non-empty subsets of $X$ whose union is $X$.
For $y \in X$, define $\partial(y,C) = \min\{ \partial(y,z) \mid z \in C\}$.
By the {\it covering radius} of $C$ we mean
$\max\{\partial(y,C) \mid y \in X\}$.
By \cite[p.~277]{Godsil} the covering radius of $C$ is $d-1$.
For $0 \leq i \leq d-1$ define
\begin{equation}\label{C_i}
C_i := \{ y \in X \mid \partial(y,C) = i\}.
\end{equation}
Observe that $C_0 = C$. Note that $\{C_i\}^{d-1}_{i=0}$ is the partition of $X$,
and hence ${\bf j} = \sum^{d-1}_{i=0}\hat{C}_i$.
Shortly we will show that this partition is equitable in the sense of \cite[p.~75]{Godsil}. 
\begin{lemma}{\rm\cite[Corollary 4.1.2]{BCN}}\label{Sturm}
Let $\{\frak{u}_i\}^d_{i=0}$ be the polynomials from {\rm(\ref{poly(u_i)})} 
attached to the primary parameter array. For $0 \leq i \leq d-1$, we have 
$(-1)^i\frak{u}_i(\tht_d)>0$.
\end{lemma}

\noindent
For $0 \leq i \leq d-1$ and $z \in C_i$,
define
\begin{equation}\label{Def(Ni)}
N_i(z) = |\Ga_i(z) \cap C|.
\end{equation}
The following lemma shows that
$N_i(z)$ is independent of the choice of $z$.

\begin{lemma}\label{formula(Ni)}
For $0 \leq i \leq d-1$ and $z \in C_i$,
\begin{equation}\label{N_i(x)}
N_i(z) = |C|\frac{\mfrk{u}_{i+1}(\theta_d)}{\mfrk{u}_{i+1}(\theta_d)-\mfrk{u}_i(\theta_d)},
\end{equation}
where $\{\frak{u}_j\}^d_{j=0}$ are the polynomials from {\rm(\ref{poly(u_i)})} 
attached to the primary parameter array. In {\rm(\ref{N_i(x)})} the denominator 
is nonzero by Lemma {\rm\ref{Sturm}}.
\end{lemma}
\begin{proof}
For notational convenience abbreviate $N_i = N_i(z)$.
Observe that 
\begin{equation}\label{pf(Ni-formula)}
\langle E_d\hat{z}, E_d\hat{C} \rangle = 
\sum_{y \in C}\langle E_d\hat{z}, E_d\hat{y} \rangle.
\end{equation}
The left-hand side of (\ref{pf(Ni-formula)}) is zero by Lemma \ref{EdC=0}.
Concerning the right-hand side of (\ref{pf(Ni-formula)}), among $y \in C$ exactly $N_i$ are contained in $\Ga_i(z)$ by (\ref{Def(Ni)}), and for such $y$ the summand in (\ref{pf(Ni-formula)}) becomes $|X|^{-1}m_d\frak{u}_i(\tht_d)$ by (\ref{EiEj=m,u,tht}). The remaining $y \in C$ are contained in $\Ga_{i+1}(z)$, and for such $y$ the summand in (\ref{pf(Ni-formula)}) becomes $|X|^{-1}m_d\frak{u}_{i+1}(\tht_d)$ by (\ref{EiEj=m,u,tht}). By these comments, line (\ref{pf(Ni-formula)}) becomes
$$
0 =N_i{|X|}^{-1}m_d\frak{u}_i(\tht_d) + (|C|-N_i){|X|^{-1}}m_d\frak{u}_{i+1}(\tht_d).
$$
Solving the above equation for $N_i$, we obtain (\ref{N_i(x)}).
\end{proof}

\noindent
Referring to Lemma \ref{formula(Ni)}, for $0 \leq i \leq d-1$ the scalar $N_i(z)$ is independent of $z$.
Therefore we define $N_i=N_i(z)$. By construction $N_0=1$. For notational convenience, define $N_{-1}=0$. Recall the scalars $h, s, s^*, r_1, r_2$ from above Note \ref{h,h*}.

\begin{lemma}\label{formula(Ni);q,h,s,r}
For $0 \leq i \leq d-1$ both
\begin{align}
\label{N_i;q,h,s,r}
N_i & =  \frac{h(q^d-1)(r_1-s^*q^{i+1})(r_2-s^*q^{i+1})}{\tht_d s^*q^d(1-s^*q^{2i+2})}, \\
\label{|C|-N_i;q,h,s,r}
|C|-N_i &= -\frac{h(q^d-1)(1-r_1q^{i+1})(1-r_2q^{i+1})}{\tht_dq^d(1-s^*q^{2i+2})}.
\end{align}
\end{lemma}

\begin{proof}
We first show (\ref{N_i;q,h,s,r}). To do this, evaluate (\ref{N_i(x)})
using $|C|=1-{k}/{\tht_d}$ and (\ref{u_i(tht_d)}).
Simplify the result to find
\begin{equation}\label{pf(N_i;q,h,s,r)}
N_i = \frac{k-\tht_d}{\tht_d}\frac{\phi_{i+1}}{\varphi_{i+1}-\phi_{i+1}}.
\end{equation}
By \cite[Lemma 6.5]{Terwilli},
$$\varphi_{i+1}-\phi_{i+1} = (\tht^*_{i+1}-\tht^*_i)\sum^i_{h=0}(\tht_h-\tht_{d-h}).$$
Also by \cite[Lemma 10.2]{Terwilli},
$$\sum^i_{h=0}\frac{\tht_h-\tht_{d-h}}{\tht_0-\tht_d} = \frac{(q^{i+1}-1)(q^{d-i}-1)}{(q-1)(q^d-1)}.$$
Recall $k=\tht_0$. 
Evaluate (\ref{pf(N_i;q,h,s,r)}) using these comments and simplify the result to find
\begin{equation}\label{pf2(N_i;q,h,s,r)}
N_i = \frac{\phi_{i+1}(q-1)(q^d-1)}
	{\tht_d(\tht^*_{i+1}-\tht^*_i)(q^{i+1}-1)(q^{d-i}-1)}.
\end{equation}
In (\ref{pf2(N_i;q,h,s,r)}), 
evaluate $\tht^*_{i+1}-\tht^*_i$ using (\ref{theta^*_i}) and 
evaluate $\phi_{i+1}$ using (\ref{phi_i}). Simplify the result to get (\ref{N_i;q,h,s,r}).\\
We now verify (\ref{|C|-N_i;q,h,s,r}).
In the equation $|C|=1-k/\tht_d$ , evaluate the right-hand side using (\ref{theta_i}) 
and $k=\tht_0$ to get
\begin{equation}\label{|C|}
|C| = {h(1-q^d)(1-sq^{d+1})}/{\tht_dq^d}.
\end{equation}
Combine this with (\ref{N_i;q,h,s,r}) and eliminate $s$ using $r_1r_2=ss^*q^{d+1}$.
Simplify the result to get (\ref{|C|-N_i;q,h,s,r}).
\end{proof}

\begin{corollary}\label{0<Ni<C}
We have $0 < N_i < |C|$ for $0 \leq i \leq d-1$.
\end{corollary}
\begin{proof}
By (\ref{Def(Ni)}), $0 \leq N_i \leq |C|$. We show $N_i \ne 0$ and $N_i \ne |C|$. To do this we use Lemma \ref{formula(Ni);q,h,s,r}. In lines (\ref{N_i;q,h,s,r}), (\ref{|C|-N_i;q,h,s,r}) each factor is nonzero by the inequalities in Example \ref{q-rac;PA}. The result follows.
\end{proof}

\noindent
For $0 \leq  i \leq d-1$ and $z \in C_i$, define
\begin{equation}\label{int_num;C}
 \wt{c}_i(z)	= | \Ga(z) \cap C_{i-1} |, \qquad
 \ta_i(z)	= |\Ga(z) \cap C_i|, \qquad
 \tb_i(z)	= |\Ga(z) \cap C_{i+1}|,
\end{equation}
where $C_{-1}$ and $C_{d}$ are empty sets. Observe
\begin{align}\label{al+be+ga=k}
&&\tc_i(z)+\ta_i(z)+\tb_i(z) = k.  & &
\end{align}
The following theorem shows that  $\ta_i(z), \tb_i(z), \tc_i(z)$ are independent of $z$.

\begin{theorem}\label{gamma;beta} The following {\rm(i)}, {\rm(ii)} hold:
\begin{enumerate}
\item[\rm(i)] For $1\leq i \leq d-1$ and $z \in C_i$, 
\begin{equation}\label{*} \tc_i(z) = \frac{N_i}{N_{i-1}}c_i; \qquad \quad \end{equation}
\item[\rm(ii)] For $0\leq i \leq d-2$ and $z \in C_i$, \quad
\begin{equation}\label{**} \tb_i(z) = \frac{|C|-N_i}{|C|-N_{i+1}}b_{i+1}.\end{equation}
\end{enumerate}
\end{theorem}
\begin{proof}
(i) Let $m$ denote the number of ordered pairs $(y,w)$ such that $y \in \Ga_i(z) \cap C$ and
$w \in \Ga(z)\cap C_{i-1}$ and $\partial(y,w)=i-1$. We compute $m$ in two ways. 
First, there are $N_i$ choices for $y$ since $|\Ga_i(z)\cap C| = N_i$. For 
$y \in \Ga_i(z) \cap C$ there are exactly $c_i$ vertices $w \in \Ga(z)\cap C_{i-1}$ such that
$\partial(y,w)=i-1$. Therefore $m=N_ic_i$.
Secondly, there are $\tc_i(z)$ choices for $w$ since $|\Ga(z)\cap C_{i-1}|=\tc_i(z)$.
For $w \in \Ga(z)\cap C_{i-1}$ there are exactly $N_{i-1}$ vertices $y \in  \Ga_i(z) \cap C$ such that
$\partial(y,w)=i-1$. Therefore $m=\tc_i(z)N_{i-1}$.
By these comments $N_ic_i=\tc_i(z)N_{i-1}$. The result follows.

\medskip\noindent
(ii) In a similar manner to (i), compute the number of ordered pairs $(y,w)$ such that
$y\in \Ga_{i+1}(z) \cap C$ and $w \in \Ga(z)\cap C_{i+1}$ and $\partial(y,w)=i+2$. Use $|\Ga_{i+1}(z)\cap C|=|C|-N_{i}$ and $|\Ga(z) \cap C_{i+1}|=\tb_i(z)$ to get
the result.
\end{proof}

\noindent
By (\ref{al+be+ga=k}) and Theorem \ref{gamma;beta},
for $0 \leq i \leq d-1$ and $z \in C_i$ the scalars $\ta_i(z),~\tb_i(z),~\tc_i(z)$ are independent of $z$. We define
\begin{equation}\label{def(ta,tb,tc)}
\ta_i = \ta_i(z), \qquad \qquad \tb_i = \tb_i(z), \qquad \qquad \tc_i = \tc_i(z).
\end{equation}
Note that
$$
 \tc_0 = 0, \qquad  \qquad \ta_0 = |C|-1, \qquad \qquad   \tb_0 = k-|C|+1, 
$$
and $\tb_{d-1}=0$. For notational convenience, define $\wt{b}_{-1}=0$ and $\wt{c}_d=0$.

\begin{corollary}{\rm\cite[Section 11.1]{BCN}}
The partition $\{C_i\}^{d-1}_{i=0}$ is equitable.
\end{corollary}
\begin{proof}
Immediate from (\ref{al+be+ga=k}) and Theorem \ref{gamma;beta}.
\end{proof}

\begin{corollary}\label{be,ga(q-term)}
The following hold.
\begin{align}
& \tb_0  = 
\label{be_0;q-term}\frac{h(1-q^{-d+1})(1-r_1q)(1-r_2q)}
	{(1-s^*q^{3})}, & \\
& \tb_i  = 
\label{be_i;q-term}\frac{h(1-q^{i-d+1})(1-s^*q^{i+2})(1-r_1q^{i+1})(1-r_2q^{i+1})}
	{(1-s^*q^{2i+2})(1-s^*q^{2i+3})} &(1 \leq i \leq d-2),& \\
\label{ga_i;q-term}
& \tc_i = 
\frac{h(1-q^i)(1-s^*q^{i+d+1})(r_1-s^*q^{i+1})(r_2-s^*q^{i+1})}
	{s^*q^d(1-s^*q^{2i+1})(1-s^*q^{2i+2})} & (1 \leq i \leq d-2),&\\
\label{ga_(d-1);q-term}
& \tc_{d-1} = 
\frac{h(1-q^{d-1})(r_1-s^*q^{d})(r_2-s^*q^{d})}
	{s^*q^d(1-s^*q^{2d-1})}. &
\end{align}
\end{corollary}
\begin{proof}
To get (\ref{be_0;q-term}), (\ref{be_i;q-term}) evaluate (\ref{**}) using (\ref{b_i;q-terms}), (\ref{|C|-N_i;q,h,s,r})
and simplify the result. Lines (\ref{ga_i;q-term}), (\ref{ga_(d-1);q-term}) are similarly obtained using (\ref{c_i;q-terms}), (\ref{N_i;q,h,s,r}).
\end{proof}

\begin{lemma}\label{tb,tc;nonzero}
We have $ \tb_i \ne 0$ for $0 \leq i \leq d-2$ and  $\tc_i \ne 0$ for $1 \leq i \leq d-1.$
\end{lemma}
\begin{proof}
Either use Corollary \ref{0<Ni<C} and Theorem \ref{gamma;beta} or use Example \ref{q-rac;PA} and Corollary \ref{be,ga(q-term)}.
\end{proof}

\noindent
We now recall the Terwilliger algebra associated with $C$ \cite{Suzuki}.
For $0 \leq i \leq d-1$, let $\wt{E}^*_{i}$ denote the diagonal matrix in $\MX$ with $(y,y)$-entry
\begin{equation}\label{def(E(C))}
(\wt{E}^*_{i})_{yy} =
\begin{cases}
1 & \text{if } y \in C_i \\
0 & \text{if } y \notin C_i
\end{cases}
\qquad (y \in X).
\end{equation}
We call $\wt{E}^*_{i}$ the $i$-th {\it dual primitive idempotent of $\Ga$ with respect to $C$}.
Observe that (i) $I=\sum^{d-1}_{i=0}\wt{E}^*_{i}$; 
(ii) $\overline{\wt{E}^*_{i}}=\wt{E}^*_{i}~(0\leq i \leq d-1)$;
(iii) $(\wt{E}^{*}_{i})^t = \wt{E}^{*}_{i}~(0 \leq i \leq d-1)$;
(iv) $\wt{E}^{*}_{i}\wt{E}^{*}_{j} = \delta_{ij}\wt{E}^{*}_{i}~(0 \leq i,j \leq d-1)$.
By these facts, $\{\wt{E}^*_{i}\}^{d-1}_{i=0}$ forms a basis for a commutative subalgebra $\wt{\mcal{M}}^*$ of $\MX$. 
The algebra $\wt{\mcal{M}}^*$ is semisimple since it is closed under the conjugate-transpose map. 
We comment on how $\mcal{M}^*$ and $\wt{\mcal{M}}^*$ are related. 
For these subalgebras each element is a diagonal matrix. Therefore any element of $\mcal{M}^*$ commutes with any element of $\wt{\mcal{M}}^*$.

\medskip \noindent
Define the diagonal matrix $\wt{A}^* \in \MX$ by
\begin{equation}\label{def(A*C)}
\wt{A}^*= {|C|}^{-1}\sum_{y \in C}A^*(y). 
\end{equation}
We call $\wt{A}^*$ the {\it dual adjacency matrix of $\Ga$ with respect to $C$}. 

\begin{lemma}\label{wt(A)*=wt(tht)*E*} 
With the above notation,
\begin{equation}\label{A*C=sum(tht*E*C)}
\wt{A}^*=\sum^{d-1}_{i=0}\wt{\tht}^*_i\wt{E}^{*}_{i},
\end{equation}
where 
\begin{equation}\label{wt(tht)*}
\hspace{3cm} \wt{\tht}^*_i = \frac{N_i\tht^*_i + (|C|-N_i)\tht^*_{i+1}}{|C|} \qquad \qquad (0 \leq i \leq d-1).
\end{equation}
\end{lemma}

\begin{proof} For $0\leq i \leq d-1$ and  $z\in C_i$ it suffices to show that $\wt{\tht}^*_i$ is the $(z,z)$-entry of $\wt{A}^*$. Consider the right-hand side in (\ref{def(A*C)}).
Among $y\in C$ there are exactly $N_i$ with $\partial(y,z)=i$, and for such $y$ the $(z,z)$-entry of $A^*(y)$ is $\tht_i^*$. The remaining $y \in C$ satisfy $\partial(y,z)=i+1$, and for such $y$ the $(z,z)$-entry of $A^*(y)$ is $\tht_{i+1}^*$. The result follows.
\end{proof}

\begin{corollary}
$\wt{A}^* \in \wt{\mcal{M}}^*$. Moreover $\{\wt{\tht}^*_i\}^{d-1}_{i=0}$ are the eigenvalues of $\wt{A}^*$.
\end{corollary}
\begin{proof}
Immediate from Lemma \ref{wt(A)*=wt(tht)*E*}.
\end{proof}

\begin{lemma}\label{(tht)tilde*}Referring to {\rm(\ref{wt(tht)*})}, 
\begin{equation}\label{(tht)dist} 
\wt{\tht}^*_i = {\wt{\tht}}^*_0 + \wt{h}^*(1-q^i)(1-\wt{s}^{\hspace{0.05cm}*}q^{i+1})q^{-i} \qquad (0 \leq i \leq d-1),
\end{equation}
where 
\begin{align}
\label{(tht)tilde*(1)}
&\wt{s}^{\hspace{0.05cm}*} = s^*q, \qquad \qquad \wt{h}^*=\tfrac{s^*q^{-1}-r_1r_2}{s^*-r_1r_2}h^*,&&&&\\
\label{(tht)tilde*(2)}
&\wt{\tht}^*_0 = \tht^*_0 +h^*\left(\tfrac{s^*(q-1)(r_1+r_2)}{s^*-r_1r_2}+\tfrac{(s^*q^{-1}-r_1r_2)(1+s^*q^2)}{s^*-r_1r_2} - 1 - s^*q  \right).&&&&
\end{align}
Moreover $\wt{h}^*$ is nonzero.
\end{lemma}
\begin{proof}
To get (\ref{(tht)dist}) evaluate (\ref{wt(tht)*}) using (\ref{theta^*_i}) and Lemma \ref{formula(Ni);q,h,s,r} together with (\ref{|C|}), and simplify. We now show $\wt{h}^*\ne 0$. Since $h^* \ne 0$, it suffices to show that $r_1r_2 \ne s^*q^{-1}$. By Example \ref{q-rac;PA} $r_1r_2 = ss^*q^{d+1}$ and $sq^i \ne 1$ for $2 \leq i \leq 2d$. Recall $d\geq3$. The result follows.
\end{proof}

\begin{lemma} \label{tht;til;dist}
Referring to Lemma {\rm\ref{(tht)tilde*}}, 
$$
\wt{\tht}^*_i - \wt{\tht}^*_j = \wt{h}^*(1-q^{i-j})(1-s^*q^{i+j+2})q^{-i} \qquad (0 \leq i, j \leq d-1).
$$
\end{lemma}
\begin{proof}
Routine using (\ref{(tht)dist}) and the first equation in (\ref{(tht)tilde*(1)}).
\end{proof}

\begin{corollary}\label{tht;til;md}
The scalars $\{\wt{\tht}^*_i\}^{d-1}_{i=0}$ are mutually distinct. Moreover $\wt{A}^*$ generates $\wt{\mcal{M}}^*$.
\end{corollary}
\begin{proof}
To obtain the first assertion, use Lemma \ref{tht;til;dist}. By Lemma \ref{(tht)tilde*} $\wt{h}^* \ne 0$. By Example \ref{q-rac;PA} $q^i \ne 1$ for $1\leq i \leq d$ and $s^*q^i\ne1$ for $2 \leq i \leq 2d$. The second assertion follows from the first assertion.
\end{proof}

 \noindent
Let $\wt{T}$ denote the subalgebra of $\MX$ generated by $\mcal{M}$ and $\wt{\mcal{M}}^*$.
$\wt{T}$ is called the {\it Terwilliger algebra of $\Ga$ with respect to $C$} \cite{Suzuki}.
The algebra $\wt{T}$ is finite-dimensional and noncommutative.
Observe that $\wt{T}$ is generated by $A, \wt{A}^*$. $\wt{T}$ is semisimple since it is closed under the conjugate-transpose map. By a {\it $\wt{T}$-module} we mean a subspace $W \subseteq \V$ such that
$BW \subseteq W$ for all $B \in \wt{T}$. Let $W$ denote a $\wt{T}$-module. Since $\wt{T}$ is closed under the conjugate-transpose map, $W$ is an orthogonal direct sum of irreducible $\wt{T}$-modules. In particular, $\V$ is an orthogonal direct sum of irreducible $\wt{T}$-modules. For more background information on $\wt{T}$ we refer the reader to \cite{Suzuki}.

\begin{lemma}\label{prim-T(C)} The following hold:
\begin{enumerate}
\item[\rm(i)] $\hat{C}_i = \wt{E}^*_i{\bf j}  \quad (0 \leq i \leq d-1)$.
\item[\rm(ii)] $A_i \hat{C} = (|C|-N_{i-1})\hat{C}_{i-1} + N_i\hat{C}_i  \quad (0 \leq i \leq d-1)$.
\item[\rm(iii)] $A_d\hat{C} = (|C|-N_{d-1})\hat{C}_{d-1}$.
\item[\rm(iv)] $\MC = \wt{\mcal{M}}^*{\bf j}$.
\item[\rm(v)] $\MC$ has bases $\{\hat{C}_i\}^{d-1}_{i=0}$, $\{E_i\hat{C}\}^{d-1}_{i=0},$ and $\{A_i\hat{C}\}^{d-1}_{i=0}$.
\item[\rm(vi)] $\MC$ is an irreducible $\wt{T}$-module.
\end{enumerate}
We call $\MC$ the {\it primary $\wt{T}$-module}.
\end{lemma}

\begin{proof}
(i) Use (\ref{def(E(C))}).\\
(ii) Assume $i \ne 0$; otherwise the result is trivial. For $z \in C_{i-1}$, there are precisely $|C|-N_{i-1}$ vertices $y \in C$ with $\partial(y,z)=i$. For $z \in C_i$, there are precisely $N_i$ vertices $y \in C$ with $\partial(y,z) = i$. The equation follows.\\
(iii) Similar to (ii).\\
(iv) By construction $\{\wt{E}^*_i{\bf j}\}^{d-1}_{i=0}$ forms a basis for $\wt{\mcal{M}}^*{\bf j}$, so $\wt{\mcal{M}}^*{\bf j}$ has dimension $d$. Recall $\{A_i\}^d_{i=0}$ spans $\mcal{M}$ so $\{A_i\hat{C}\}^d_{i=0}$ spans $\MC$. By (i)--(iii) $A_i\hat{C} \in \wt{\mcal{M}}^*{\bf j}$ for $0 \leq i \leq d$. Therefore $\MC \subseteq \wt{\mcal{M}}^*{\bf j}$. By Lemma \ref{EdC=0} $\{E_i\hat{C}\}^{d-1}_{i=0}$ forms a basis for $\MC$. Therefore $\MC$ has dimension $d$. By these comments, $\MC=\wt{\mcal{M}}^*{\bf j}$.\\
(v) In the proof of (iv), we saw that $\{\wt{E}^*_i{\bf j}\}^{d-1}_{i=0}$ is a basis for $\MC$. By this and (i) the set $\{\hat{C}_i\}^{d-1}_{i=0}$ forms a basis for $\MC$. Also, in the proof of (iv) we saw that $\{E_i\hat{C}\}^{d-1}_{i=0}$ is a basis for $\MC$. We now show that $\{A_i\hat{C}\}^{d-1}_{i=0}$ is a basis for $\MC$. By construction Span$\{A_i\hat{C}\}^{d-1}_{i=0}$ is a subspace of $\MC$. Since $\MC$ has dimension $d$, it suffices to show that the vectors $\{A_i\hat{C}\}^{d-1}_{i=0}$ are linearly independent. Note that the $\{N_i\}^{d-1}_{i=0}$ are nonzero by Corollary \ref{0<Ni<C} and $\{\hat{C}_i\}^{d-1}_{i=0}$ are linearly independent. By these comments and (ii), the vector $\{A_i\hat{C}\}^{d-1}_{i=0}$ are linearly independent.\\
(vi) By (iv), $\MC$ is a $\wt{T}$-module. We show the $\wt{T}$-module $\MC$ is irreducible. Express $\MC$ as the orthogonal direct sum of irreducible $\wt{T}$-modules. Since $\hat{C} \in \MC$, among these irreducible $\wt{T}$-modules there exists one, denoted $W$, that is not orthogonal to $\hat{C}$. Observe $\wt{E}^*_0W \ne 0$. Also $\wt{E}^*_0W\subseteq \wt{E}^*_0\MC = \text{Span}\{\hat{C}\}$, so $\wt{E}^*_0W = \text{Span}\{\hat{C}\}$.
By this and since $W$ is a $\wt{T}$-module, we have $\hat{C} \in \wt{E}^*_0W \subseteq W$. But then $\MC \subseteq W$ and thus $\MC=W$ by the irreducibility of $W$.
\end{proof}

\begin{corollary}\label{mal-free(A,A*)} ~
\begin{enumerate}
\item[\rm (i)] $A$ is multiplicity-free on $\MC$ with eigenvalues $\{\tht_i\}^{d-1}_{i=0}$.
\item[\rm (ii)] $\wt{A}^*$ is multiplicity-free on $\MC$ with eigenvalues $\{\wt{\tht}^*_i\}^{d-1}_{i=0}$.
\end{enumerate}
\end{corollary}

\begin{proof}
(i) Follows from Lemma \ref{prim-T(C)}(v).\\
(ii) Follows from Lemma \ref{prim-T(C)}(i), (v). 
\end{proof}

\begin{lemma}\label{LS(MC)} Consider the matrices
\begin{equation}\label{LS(A;A*C)}
(A; \wt{A}^*; \{E_i\}^{d-1}_{i=0}; \{\wt{E}^{*}_{i}\}^{d-1}_{i=0}).
\end{equation}
Then the elements {\rm(\ref{LS(A;A*C)})} act on $\MC$ as a Leonard system.
\end{lemma}
\begin{proof}
We show that (\ref{LS(A;A*C)}) satisfies conditions (i)--(v) in Definition \ref{Def:LS}.
Conditions (i)--(iii) follow from Corollary \ref{mal-free(A,A*)},
condition (iv) is from (\ref{EA*E}) and (\ref{def(A*C)}),
and condition (v) follows from \cite[Lemma 4.1]{Suzuki}. 
\end{proof}

\noindent
We let $\wt{\Phi}$ denote the Leonard system on $\MC$ from Lemma \ref{LS(MC)}. We call $\wt{\Phi}$ the {\it primary Leonard system with respect to $C$}. Recall the basis $\{\hat{C}_i\}^{d-1}_{i=0}$ for $\MC$ from Lemma \ref{prim-T(C)}(v).

\begin{lemma}\label{SB(MC)}
The vectors $\{\hat{C}_i\}^{d-1}_{i=0}$ form a $\wt{\Phi}$-standard basis for $\MC$.
\end{lemma}
\begin{proof}
We invoke Lemma \ref{Phi-sb}. Abbreviate $W=\MC$. By Lemma \ref{prim-T(C)}(i) we have $\hat{C}_i \in \wt{E}^*_iW$ for $0 \leq i \leq d-1$. By Lemma \ref{prim-T(C)}(i) and the construction, $\sum^{d-1}_{j=0}\hat{C}_j = {\bf j} \in {E}_0W$. The result follows in view of Lemma \ref{Phi-sb}.
\end{proof}

\noindent
By Lemma \ref{wt(A)*=wt(tht)*E*}, the matrix representing $\wt{A}^*$ relative to the basis $\{\hat{C}_i\}^{d-1}_{i=0}$ is
\begin{equation}\label{[A^*(C)]_(C^)}
\diag(\wt{\tht}^*_0, \wt{\tht}^*_1, \wt{\tht}^*_2, \ldots, \wt{\tht}^*_{d-1}).
\end{equation}
By (\ref{int_num;C}) and (\ref{def(ta,tb,tc)}), the matrix representing $A$ relative to $\{\hat{C}_i\}^{d-1}_{i=0}$ is 
\begin{equation}\label{[A]_(Ci)}
\left[
\begin{array}{ccccc}
\vspace{0.2cm}
\ta_0 & \tb_0 & & & {\bf 0} \\
\tc_1 & \ta_1 & \tb_1 & \\
& \tc_2 & \ta_2 & \ddots \\
\vspace{0.2cm}
& & \ddots & \ddots & \tb_{d-2} \\
{\bf 0} & & & \tc_{d-1} & \ta_{d-1}
\end{array}
\right].
\end{equation}
Comparing (\ref{[A^(flat)]}) and (\ref{[A]_(Ci)}) we see that the $\wt{a}_i, \wt{b}_i, \wt{c}_i$ are the intersection numbers of $\wt{\Phi}$.
We now consider the parameter array of $\wt{\Phi}$. We denote this parameter array by
\begin{equation*}\label{PA;wt(Phi)}
p(\wt{\Phi}) = (\{\ttht_i\}^{d-1}_{i=0}; \{\ttht^*_i\}^{d-1}_{i=0}; \{\tvarphi_i\}^{d-1}_{i=1}; \{\tphi_i\}^{d-1}_{i=1}).
\end{equation*}
The scalars $\{\wt{\tht^*_i}\}^{d-1}_{i=0}$ were evaluated in Lemma \ref{(tht)tilde*}. We now consider $\{\wt{\tht}_i\}^{d-1}_{i=0}$.

\begin{lemma}\label{tht=ttht}
We have
$$\qquad \ttht_i=\tht_i \qquad \qquad (0 \leq i \leq d-1).$$
\end{lemma}
\begin{proof} 
By construction,  $AE_i = \wt{\tht}_iE_i$ on $\MC$. But $AE_i = \tht_iE_i$, so $\ttht_i=\tht_i$.
\end{proof}

\noindent
Recall the scalars $h, h^*, s, s^*, r_1, r_2$ from above Note \ref{h,h*}
\begin{theorem}\label{q-rac;til(Phi)}
For $0 \leq i \leq d-1$,
\begin{eqnarray}
\label{q-rac;ttht_i}
\ttht_i & = & \ttht_0+\wt{h}(1-q^i)(1-\wt{s}q^{i+1})q^{-i},\\
\label{q-rac;ttht*_i}
\ttht^*_i & = & \ttht^*_0 + \wt{h}^*(1-q^i)(1-\wt{s}^*q^{i+1})q^{-i},
\end{eqnarray}
and for $1 \leq i \leq d-1$,
\begin{eqnarray}
\label{q-rac;tvarphi}
\tvarphi_i & = & \wt{h}\wt{h}^*q^{1-2i}(1-q^i)(1-q^{i-d})(1-\wt{r}_1q^i)(1-\wt{r}_2q^i), \\
\label{q-rac;tphi}
\tphi_i & = & \wt{h}\wt{h}^*q^{1-2i}(1-q^i)(1-q^{i-d})(\wt{r}_1-\wt{s}^*q^{i})(\wt{r}_2-\wt{s}^*q^{i})/\wt{s}^*,
\end{eqnarray} 
where
\begin{align}
\label{th,ts...;eq(1)}
&\wt{h} = h, && \wt{s} = s, &&\wt{r}_1=r_1,&\\
\label{th,ts...;eq(2)}
&\wt{h}^* = \tfrac{s^*q^{-1}-r_1r_2}{s^*-r_1r_2}h^* ,&&  \wt{s}^* = s^*q, && \wt{r}_2=r_2,
\end{align}
and where $\ttht^*_0$ is from {\rm(\ref{(tht)tilde*(2)})}.
\end{theorem}
\begin{proof}
The Leonard system $\wt{\Phi}$ has diameter $d-1$. To verify (\ref{q-rac;ttht_i}), evaluate each term using (\ref{theta_i}) and Lemma \ref{tht=ttht}. Line (\ref{q-rac;ttht*_i}) is from Lemma \ref{(tht)tilde*}. To verify (\ref{q-rac;tvarphi}), (\ref{q-rac;tphi}) use Lemma \ref{Madrid(bi;ci)}, Lemma \ref{tht;til;dist} and (\ref{be_0;q-term})--(\ref{ga_(d-1);q-term}) along with (\ref{th,ts...;eq(1)}), (\ref{th,ts...;eq(2)}). 
\end{proof}

\begin{corollary}
The Leonard system $\wt{\Phi}$ has $q$-Racah type.
\end{corollary}
\begin{proof}
Compare Example \ref{q-rac;PA} and Theorem \ref{q-rac;til(Phi)}.
\end{proof}


\section{The subspace $\W$}\label{subspaceW}

We continue to discuss the Delsarte clique $C$ of $\Ga$. For the rest of the paper fix a vertex $x \in C$. In this section, using $x$ and $C$ we will construct a certain partition of $X$. Using this partition we construct a subspace $\W$ of $\V$ which has a module structure for both $T=T(x)$ and $\wt{T}$. 
Recall the partition $\{C_i\}^{d-1}_{i=0}$ of $X$ from (\ref{C_i}). For $0 \leq i \leq d-1$ define
\begin{equation}\label{def(C-;C+)}
C_i^{-}=C_i \cap \Gamma_i, \qquad  \qquad \quad 
	C_i^{+}=C_i \cap \Gamma_{i+1},
\end{equation}
where $\Ga_j=\Ga_j(x)$ for $0 \leq j \leq d$. For notational convenience, define
${C}^{\pm}_{-1} =\emptyset$ and ${C}^{\pm}_{d} = \emptyset.$
Observe that $C_i = C^-_i \cup C^+_i$ for $0 \leq i \leq d-1$. Also
$\Ga_i = C^+_{i-1} \cup C^-_i$ for $1 \leq i \leq d-1$ and $\Ga_0=C^-_0=\{x\}, \Ga_d = C^+_{d-1}.$
We visualize the $\{C^{\pm}_i\}^{d-1}_{i=0}$ as follows:

\begin{center}
\begin{tikzpicture}
  [scale=.8,thick,auto=left,every node/.style={circle,fill=blue!20,draw}]
  \node (n1) at (0,0) {$C^-_0$};
  \node (n2) at (2,0)  {$C^+_0$};
  \node (n3) at (2,2)  {$C^-_1$};
  \node (n4) at (4,2) {$C^+_1$};
  \node (n5) at (4,4)  {$C^-_2$};
  \node (n6) at (6,4)  {$C^+_2$};   
  \node (n7) at (6,6)  {$C^-_3$};
  \node (n8) at (8,6)  {$C^+_3$};
 
  \node (c0) at (-2.5,0) {${~}C^{~}_0$};
  \node (c1) at (-2.5,2) {${~}C^{~}_1$};
  \node (c2) at (-2.5,4) {${~}C^{~}_2$};
  \node (c3) at (-2.5,6) {${~}C^{~}_3$};
  
  \node (r0) at (0,-2.5) {${~}\Ga_0$};
  \node (r1) at (2,-2.5) {${~}\Ga_1$};
  \node (r2) at (4,-2.5) {${~}\Ga_2$};
  \node (r3) at (6,-2.5) {${~}\Ga_3$};
  \node (r4) at (8,-2.5) {${~}\Ga_4$};

  \foreach \from/\to in {n1/n2,n2/n3,n3/n4,n4/n5,n5/n6,n6/n7,n7/n8,n1/n3,n3/n5,n5/n7,n2/n4,n4/n6,n6/n8, c0/c1,c1/c2,c2/c3, r0/r1, r1/r2, r2/r3, r3/r4}
    \draw (\from) -- (\to);
    
    \draw (c0) -- (n1) [dashed];
    \draw (c1) -- (n3) [dashed];
    \draw (c2) -- (n5) [dashed];
    \draw (c3) -- (n7) [dashed];
    
    \draw (r0) -- (n1) [dashed];
    \draw (r1) -- (n2) [dashed];
    \draw (r2) -- (n4) [dashed];
    \draw (r3) -- (n6) [dashed];
    \draw (r4) -- (n8) [dashed];

\end{tikzpicture}

\medskip
{\small Example:  The sets $\{C^{\pm}_i\}^{d-1}_{i=0}$ of $X$ when $d=4$.}
\end{center}

\begin{lemma}\label{2-dim C-}
The following {\rm (i)--(iii)} hold.
\begin{enumerate}
\item[\rm(i)] For  $z \in C_0^-$,

\begin{tabular}{lcl}
$\bullet$ $z$ is adjacent to precisely & $0$ & vertices in ${C}^-_{0}$,\\

$\bullet$ $z$ is adjacent to precisely &$b_0-\tb_0$& vertices in ${C}^+_{0}$,\\

$\bullet$ $z$ is adjacent to precisely &$\tb_0$& vertices in ${C}^-_{1}$.
\end{tabular}

\item[\rm(ii)] For $1 \leq i \leq d-2$ and $z\in C_i^-$,

\begin{tabular}{lcl}
$\bullet$ $z$ is adjacent to precisely & $c_i$ & vertices in ${C}^-_{i-1}$, \\

$\bullet$ $z$ is adjacent to precisely &$\wt{c}_i-c_i$ & vertices in ${C}^+_{i-1}$,\\

$\bullet$ $z$ is adjacent to precisely &$a_i-\wt{c}_i+c_i$ & vertices in ${C}^-_{i}$,\\

$\bullet$ $z$ is adjacent to precisely &$b_i-\wt{b}_i$ & vertices in ${C}^+_{i}$,\\

$\bullet$ $z$ is adjacent to precisely &$\wt{b}_i$ & vertices in ${C}^-_{i+1}$.
\end{tabular}

\item[\rm(iii)] For  $z \in C_{d-1}^-$,

\begin{tabular}{lcl}
$\bullet$ $z$ is adjacent to precisely & $c_{d-1}$ & vertices in ${C}^-_{d-2}$, \\

$\bullet$ $z$ is adjacent to precisely &$\wt{c}_{d-1}-c_{d-1}$ & vertices in ${C}^+_{d-2}$,\\

$\bullet$ $z$ is adjacent to precisely &$a_{d-1}-\wt{c}_{d-1}+c_{d-1}$ & vertices in ${C}^-_{d-1}$,\\

$\bullet$ $z$ is adjacent to precisely &$b_{d-1}$ & vertices in ${C}^+_{d-1}$.
\end{tabular}
\end{enumerate}
\end{lemma}

\begin{proof}
Routine using (\ref{def(C-;C+)}).
\end{proof}

\begin{lemma}\label{2-dim C+}
The following {\rm (i)--(iii)} hold.
\begin{enumerate}
\item[\rm(i)] For  $z \in C_0^+$,

\begin{tabular}{lcl}
$\bullet$ $z$ is adjacent to precisely & $c_1$ & vertex in ${C}^-_{0}$,\\

$\bullet$ $z$ is adjacent to precisely &$\ta_0-c_1$& vertices in ${C}^+_{0}$,\\

$\bullet$ $z$ is adjacent to precisely &$\wt{b}_0-b_1$& vertices in ${C}^-_{1}$, \\

$\bullet$ $z$ is adjacent to precisely &$b_1$& vertices in ${C}^+_{1}$.
\end{tabular}

\item[\rm(ii)] For $1 \leq i \leq d-2$ and $z\in C_i^+$,

\begin{tabular}{lcl}
$\bullet$ $z$ is adjacent to precisely & $\wt{c}_i$ & vertices in ${C}^+_{i-1}$, \\

$\bullet$ $z$ is adjacent to precisely &$c_{i+1}-\wt{c}_i$ & vertices in ${C}^-_{i}$,\\

$\bullet$ $z$ is adjacent to precisely &$\wt{a}_i-c_{i+1}+\wt{c}_i$ & vertices in ${C}^+_{i}$,\\

$\bullet$ $z$ is adjacent to precisely &$\wt{b}_i-b_{i+1}$ & vertices in ${C}^-_{i+1}$,\\

$\bullet$ $z$ is adjacent to precisely &$b_{i+1}$ & vertices in ${C}^+_{i+1}$.
\end{tabular}

\item[\rm(iii)] For  $z \in C_{d-1}^+$,

\begin{tabular}{lcl}
$\bullet$ $z$ is adjacent to precisely & $\wt{c}_{d-1}$ & vertices in ${C}^+_{d-2}$, \\

$\bullet$ $z$ is adjacent to precisely &$c_d-\wt{c}_{d-1}$ & vertices in ${C}^-_{d-1}$,\\

$\bullet$ $z$ is adjacent to precisely &$\wt{a}_{d-1}-{c}_{d}+\wt{c}_{d-1}$ & vertices in ${C}^+_{d-1}$.
\end{tabular}
\end{enumerate}
\end{lemma}

\begin{proof}
Routine using (\ref{def(C-;C+)}).
\end{proof}

\begin{corollary}\label{|edges|} The following {\rm(i)--(iv)} hold.
\begin{enumerate}
\item[\rm(i)]   $\tb_i|C^-_i| = c_{i+1}|C^-_{i+1}|$  for $0 \leq i \leq d-2$.
\item[\rm(ii)]  $b_{i+1}|C^+_i| = \tc_{i+1}|C^+_{i+1}|$  for $0 \leq i \leq d-2$.
\item[\rm(iii)] $(b_i-\tb_i)|C^-_i| = (c_{i+1}-\tc_i)|C^+_i|$ for $0 \leq i \leq d-1$.
\item[\rm(iv)] $(\tb_i-b_{i+1})|C^+_i| = (\tc_{i+1}-c_{i+1})|C^-_{i+1}|$ for $0 \leq i \leq d-2$.
\end{enumerate}
\end{corollary}

\begin{proof}
(i) By the data of Lemma \ref{2-dim C-}, every vertex in $C^-_i$ is adjacent to precisely $\tb_i$ vertices in $C^-_{i+1}$ and every vertex in $C^-_{i+1}$ is adjacent to precisely $c_{i+1}$ vertices in $C^-_i$. The result follows.\\
(ii)--(iv) Similar to (i).
\end{proof}

\noindent
We now find the cardinality for each of $\{C^{\pm}_i\}^{d-1}_{i=0}$.
\begin{lemma}\label{cardofCi}
For $0 \leq i \leq d-1$, 
\begin{equation*} \label{|C^+-i|}
|C^{-}_{i}| = \frac{\tb_{0} \tb_{1}  \dotsm \tb_{i-1}}{c_1 c_2  \dotsm c_i},
	\qquad \qquad  
	|C^{+}_{i}| = \frac{b_1 b_2 \dotsm b_i}{\tc_1 \tc_2 \dotsm \tc_i}(|C|-1). 
\end{equation*}
\end{lemma}
\begin{proof}
To get the equation on the left, use Corollary \ref{|edges|}(i). To get the equation on the right, use Corollary \ref{|edges|}(ii).
\end{proof}

\begin{corollary}\label{C+-;nonempty}
For $0 \leq i \leq d-1$, each of $C^{\pm}_i$ is nonempty.
\end{corollary}
\begin{proof}
By Lemma \ref{tb,tc;nonzero} and Lemma \ref{cardofCi}.
\end{proof}
\noindent
By Corollary \ref{C+-;nonempty} and the construction, the $\{C^{\pm}_i\}^{d-1}_{i=0}$ is a partition of $X$. 
\begin{proposition}\label{2-dim partition}
The partition $\{C^{\pm}_i\}^{d-1}_{i=0}$ of $X$ is equitable.
\end{proposition}
\begin{proof}
By Lemma \ref{2-dim C-} and Lemma \ref{2-dim C+}.
\end{proof}

\noindent
Recall the standard module $\V$ from above line (\ref{Ga_i}). Using the equitable partition $\{C^{\pm}_i\}^{d-1}_{i=0}$ we get a subspace $\W$ of $\V$ as follows. For $0 \leq i \leq d-1$, recall the characteristic vector $\hat{C}^{\pm}_i$ of $C^{\pm}_{i}$. Let $\W$ denote the subspace of $\V$ spanned by $\{\hat{C}^{\pm}_{i}\}^{d-1}_{i=0}$. Note that $\W$ contains the vectors $\hat{x}, \hat{C}$. Recall the all $1$'s vector ${\bf j} = \sum^{d-1}_{i=0}\hat{C}_i$ from below (\ref{C_i}). The subspace $\W$ contains $\bf j$ since
\begin{equation}\label{Ci=(C-i)+(C+i)}
\hat{C}_i = \hat{C}^-_i + \hat{C}^+_i \qquad \qquad (0 \leq i \leq d-1).
\end{equation}

\begin{lemma}\label{OGB(W)}
The vectors $\{\hat{C}^{\pm}_i\}^{d-1}_{i=0}$ form an orthogonal basis for $\W$. Moreover,
\begin{equation*}
\|\hat{C}^{-}_{i}\|^2 = \frac{\tb_0 \tb_1 \dotsm \tb_{i-1}}{c_1 c_2 \dotsm c_i},
	\qquad \qquad  
	\|\hat{C}^{+}_{i}\|^2 = \frac{b_1 b_2 \dotsm b_{i}}{\tc_1 \tc_2 \dotsm \tc_i}(|C|-1). 
\end{equation*}
\end{lemma}
\begin{proof}
By construction the $\{{C}^{\pm}_i\}^{d-1}_{i=0}$ are mutually disjoint, so the $\{\hat{C}^{\pm}_{i}\}^{d-1}_{i=0}$ are mutually orthogonal.  By Corollary \ref{C+-;nonempty}, each of $\{\hat{C}^{\pm}_i\}^{d-1}_{i=0}$ is  nonzero. The first assertion follows from these comments. The second assertion follows from Lemma \ref{cardofCi}.
\end{proof}

\begin{corollary}\label{dim(W)=2d}
The dimension of $\W$ is $2d$.
\end{corollary}
\begin{proof}
Immediate from Lemma \ref{OGB(W)}.
\end{proof}

\noindent
In Lemma \ref{OGB(W)} we gave a basis for $\W$. We now give the action of $A$ on this basis. Recall $C^{\pm}_{-1}=\emptyset$ and $C^{\pm}_d=\emptyset$, so
$$
\hat{C}^-_{-1} =0, \qquad \hat{C}^+_{-1}=0, \qquad
\hat{C}^-_{d} =0, \qquad \hat{C}^+_{d}=0. 
$$

\begin{lemma}\label{action(A;W)}
The element $A$ acts on $\{\hat{C}^{\pm}_{i}\}^{d-1}_{i=0}$
as follows: for $0 \leq i \leq d-1$ both
\begin{align}
\label{A.C-}
A.\hat{C}^-_i &= \tb_{i-1} \hat{C}^-_{i-1} + (\tb_{i-1}-b_i)\hat{C}^+_{i-1}
				+ (\ta_i-b_i+\tb_i)\hat{C}^-_{i} + (c_{i+1}-\tc_i)\hat{C}^+_{i}
				+c_{i+1}\hat{C}^-_{i+1}, &&\\
\label{A.C+}
A.\hat{C}^+_i &= b_i \hat{C}^+_{i-1} + (b_i-\tb_i)\hat{C}^-_{i}
				+ (\ta_i-c_{i+1}+\tc_i)\hat{C}^+_{i} + (\tc_{i+1}-c_{i+1})\hat{C}^-_{i+1}
				+\tc_{i+1}\hat{C}^+_{i+1}. &&
\end{align}
\end{lemma}

\begin{proof}
Use Lemma \ref{2-dim C-} and Lemma \ref{2-dim C+}.
\end{proof}

\noindent
Evaluating (\ref{A.C-}), (\ref{A.C+}) using (\ref{b_0;q-terms})--(\ref{c_d;q-terms}) and (\ref{be_0;q-term})--(\ref{ga_(d-1);q-term}), we obtain the following. For $0 \leq i \leq d-1$,

\noindent
$A.\hat{C}^-_i = $
\begin{equation}\label{A_C;odd-col}
\begin{tabular}{c|l}
term & \hspace{2cm} coefficient \\
\hline \hline &\\
$\hat{C}^-_{i-1}$ & 
	$\frac{h(1-q^{i-d})(1-s^*q^{i+1})(1-r_1q^i)(1-r_2q^i)}{(1-s^*q^{2i})(1-s^*q^{2i+1})}$ 
\\ &\\

$\hat{C}^{+}_{i-1}$ & 
	$\frac{h(1-q^{i-d})(1-s^*q^{i+1})}{1-s^*q^{2i+1}} 
	\left(\frac{(1-r_1q^i)(1-r_2q^i)}{1-s^*q^{2i}} - \frac{(1-r_1q^{i+1})(1-r_2q^{i+1})}{1-s^*q^{2i+2}}\right)$
\\ &\\

$\hat{C}^{-}_{i}$ &  
	$b_0 - h\left(\frac{(1-q^i)(1-s^*q^{i+d+1})(r_1-s^*q^{i+1})(r_2-s^*q^{i+1})}
			{s^*q^d(1-s^*q^{2i+1})(1-s^*q^{2i+2})}
 + \ \frac{(1-q^{i-d})(1-s^*q^{i+1})(1-r_1q^{i+1})(1-r_2q^{i+1})}
		{(1-s^*q^{2i+1})(1-s^*q^{2i+2})} \right)$
		
\\ & \\ 

$\hat{C}^{+}_{i}$ & 
	$\frac{h(r_1-s^*q^{i+1})(r_2-s^*q^{i+1})}{s^*q^d(1-s^*q^{2i+2})} 
	\left(\frac{(1-q^{i+1})(1-s^*q^{i+d+2})}{1-s^*q^{2i+3}}-
		\frac{(1-q^i)(1-s^*q^{i+d+1})}{1-s^*q^{2i+1}}\right)$
\\ &\\

$\hat{C}^{-}_{i+1}$ & 
	$\frac{h(1-q^{i+1})(1-s^*q^{i+d+2})(r_1-s^*q^{i+1})(r_2-s^*q^{i+1})}
			{s^*q^d(1-s^*q^{2i+2})(1-s^*q^{2i+3})}$

\end{tabular}
\end{equation}

\noindent
and also

\medskip
\noindent$A.\hat{C}^+_i = $
\begin{equation}\label{A_C;even_col}
\begin{tabular}{c|l}
term & \hspace{2cm} coefficient \\
\hline \hline &\\
$\hat{C}^+_{i-1}$ & 
	$\frac{h(1-q^{i-d})(1-s^*q^{i+1})(1-r_1q^{i+1})(1-r_2q^{i+1})}
			{(1-s^*q^{2i+1})(1-s^*q^{2i+2})}$
\\ &\\

$\hat{C}^{-}_{i}$ & 
	$\frac{h(1-r_1q^{i+1})(1-r_2q^{i+1})}{1-s^*q^{2i+2}}
	\left(\frac{(1-q^{i-d})(1-s^*q^{i+1})}{1-s^*q^{2i+1}}
	-\frac{(1-q^{i-d+1})(1-s^*q^{i+2})}{1-s^*q^{2i+3}}\right)$
	
\\ &\\

$\hat{C}^{+}_{i}$ &  
	$b_0 - h\left(\frac{(1-q^{i-d+1})(1-s^*q^{i+2})(1-r_1q^{i+1})(1-r_2q^{i+1})}
			{(1-s^*q^{2i+2})(1-s^*q^{2i+3})}
			 +\frac{(1-q^{i+1})(1-s^*q^{i+d+2})(r_1-s^*q^{i+1})(r_2-s^*q^{i+1})}
		{s^*q^d(1-s^*q^{2i+2})(1-s^*q^{2i+3})} \right)$

\\ & \\ 

$\hat{C}^{-}_{i+1}$ & 
	$\frac{h(1-q^{i+1})(1-s^*q^{i+d+2})}{s^*q^d(1-s^*q^{2i+3})}
	\left(\frac{(r_1-s^*q^{i+2})(r_2-s^*q^{i+2})}{1-s^*q^{2i+4}}
	-\frac{(r_1-s^*q^{i+1})(r_2-s^*q^{i+1})}{1-s^*q^{2i+2}}\right)$
\\ &\\

$\hat{C}^{+}_{i+1}$ & 
	$\frac{h(1-q^{i+1})(1-s^*q^{i+d+2})(r_1-s^*q^{i+2})(r_2-s^*q^{i+2})}
	{s^*q^d(1-s^*q^{2i+3})(1-s^*q^{2i+4})}$

\end{tabular}
\end{equation}
The scalars $b_0$ and $h$ in the above tables are given in (\ref{b_0;q-terms}) and (\ref{scalar(h)}), respectively. For $0 \leq r \leq d$ we now give the action of $E^*_r$ on $\{\hat{C}^{\pm}_{i}\}^{d-1}_{i=0}$.

\begin{lemma}\label{action(E*;W)}
For $0 \leq r \leq d$ and $0 \leq i \leq d-1$, 
\begin{equation*}
E^*_r. \hat{C}^-_i = \delta_{ri}\hat{C}^-_i , \qquad \qquad \quad
E^*_r. \hat{C}^+_{i} = \delta_{r,{i+1}}\hat{C}^+_{i} .
\end{equation*}
\end{lemma}
\begin{proof}
Use (\ref{def(C-;C+)}).
\end{proof}

\begin{corollary}\label{action(A*;W)}
The matrix $A^*$ acts on $\{\hat{C}^{\pm}_{i}\}^{d-1}_{i=0}$ as follows. For $0 \leq i \leq d-1$,
$$
A^*.\hat{C}^-_i = \tht^*_i\hat{C}^-_i, \qquad \qquad \quad A^*.\hat{C}^+_i = \tht^*_{i+1}\hat{C}^+_i.
$$
\end{corollary}
\begin{proof} 
By Lemma \ref{action(E*;W)} and since $A^*=\sum^d_{r=0}\tht^*_rE^*_r$.
\end{proof}

\begin{lemma}\label{W=ods(E*i)}
We have
$$
\W = \sum^d_{i=0} E^*_i\W \qquad \qquad \text{\rm (orthogonal direct sum).}
$$
Moreover,
\begin{enumerate}
\item[\rm(i)] $\hat{C}^-_0$ is a basis for $E^*_0\W$.
\item[\rm(ii)] For $1\leq i\leq d-1$ the vectors $\hat{C}^{+}_{i-1}, \hat{C^-_i}$ form a basis for $E^*_i\W$.
\item[\rm(iii)] $\hat{C}^+_{d-1}$ is a basis for $E^*_d\W$. 
\end{enumerate}
\end{lemma}
\begin{proof}
Use (\ref{def(C-;C+)}) and Lemma \ref{action(E*;W)}.
\end{proof}

\begin{corollary}\label{dim(E*iW)}
For $0 \leq i \leq d$,
\begin{equation*}
\dim{E^*_i\W} = 
\begin{cases}
2 & \text{\rm \quad if \quad} 1 \leq i \leq d-1, \\
1 & \text{\rm \quad if \quad} i \in \{0, d\}.
\end{cases}
\end{equation*}
\end{corollary}
\begin{proof}
Immediate from Lemma \ref{W=ods(E*i)}.
\end{proof}

\begin{lemma}\label{W:T(x)-mod} The subspace $\W$ is a $T$-module.
\end{lemma}
\begin{proof}
Follows from Lemma \ref{action(A;W)} and Lemma \ref{action(E*;W)}.
\end{proof}

\noindent
For $0 \leq r \leq d-1$ we now give the action of $\wt{E}^*_r$ on $\{\hat{C}^{\pm}_{i}\}^{d-1}_{i=0}$. 

\begin{lemma}\label{action(E*(C);W)}
For $0 \leq i,r \leq d-1$,
\begin{equation*}
\wt{E}^{*}_{r}. \hat{C}^-_i = \delta_{ri}\hat{C}^-_i , \qquad \qquad \quad
\wt{E}^{*}_{r}. \hat{C}^+_{i} = \delta_{ri}\hat{C}^+_{i}.
\end{equation*}
\end{lemma}
\begin{proof}
Use (\ref{def(C-;C+)}).
\end{proof}

\begin{corollary}\label{action(A*C;W)}
The matrix $\wt{A}^*$ acts on $\{\hat{C}^{\pm}_i\}^{d-1}_{i=0}$ as follows. For $0 \leq i \leq d-1$,
$$
\wt{A}^*.\hat{C}^-_i = \wt{\tht}^*_i\hat{C}^-_i, \qquad \qquad \quad 
\wt{A}^*.\hat{C}^+_i = \wt{\tht}^*_i\hat{C}^+_i.
$$
\end{corollary}
\begin{proof}
By (\ref{A*C=sum(tht*E*C)}) and Lemma \ref{action(E*(C);W)}.
\end{proof}

\begin{lemma}\label{W=ods(E*(C)i)} We have
$$
\W = \sum^{d-1}_{i=0} \wt{E}^*_i\W \qquad \qquad \text{\rm (orthogonal direct sum).}
$$
Moreover, for $0\leq i\leq d-1$ the vectors $\hat{C}^{+}_{i}, \hat{C}^{-}_{i}$ form a basis for $\wt{E}^*_i\W$.
\end{lemma}
\begin{proof}
Use (\ref{def(C-;C+)}) and Lemma \ref{action(E*(C);W)}.
\end{proof}

\begin{corollary}\label{dim(E(C)*iW)}
The dimension of $\wt{E}^*_i\W$ is $2$ for $0 \leq i \leq d-1$.
\end{corollary}
\begin{proof}
Immediate from Lemma \ref{W=ods(E*(C)i)}.
\end{proof}

\begin{lemma}\label{W:T(C)-mod} The subspace $\W$ is a $\wt{T}$-module.
\end{lemma}
\begin{proof}
Follows from Lemma \ref{action(A;W)} and Lemma \ref{action(E*(C);W)}.
\end{proof}

\noindent
Motivated by Lemma \ref{W:T(x)-mod} and Lemma \ref{W:T(C)-mod}, we make a definition.

\begin{definition}\label{Algebra(T)}
Let $\T$ denote the subalgebra of $\MX$ generated by
$T$ and $\wt{T}$. The algebra $\T$ is finite-dimensional and noncommutative.
Observe that $A, A^*, \wt{A}^*$ generate $\T$. The algebra $\T$ is semisimple since it is closed under the conjugate-transpose map.
\end{definition}

\noindent
By a {\it \T-module} we mean a subspace $W \subseteq \V$ such that $BW \subseteq W$ for all $B \in \T$. 
By Lemma \ref{W:T(x)-mod} and Lemma \ref{W:T(C)-mod} $\W$ is a $\T$-module. We call $\W$ the {\it primary $\T$-module}. We now describe the $\T$-module $\W$.

\begin{lemma}\label{W=DS(E*E*)} We have
$$
\W = \sum^{d-1}_{i=0}\sum^{i+1}_{j=i}\wt{E}^*_iE^*_j\W  \qquad \qquad \text{\rm (orthogonal direct sum).}
$$
Moreover, for $0 \leq i \leq d-1$ both
\begin{enumerate}
\item[\rm(i)] The vector $\hat{C}^-_i$ is a basis for $\wt{E}^*_iE^*_i\W$,
\item[\rm(ii)] The vector $\hat{C}^+_i$ is a basis for $\wt{E}^*_iE^*_{i+1}\W$.
\end{enumerate}
\end{lemma}
\begin{proof}
Use Lemma \ref{W=ods(E*i)} and Lemma \ref{W=ods(E*(C)i)}.
\end{proof}

\begin{lemma}\label{E*E*=C+-}
For $0 \leq i \leq d-1$, 
\begin{equation*}
\wt{E}^*_iE^*_{i}{\bf j} = \hat{C}^{-}_i, \qquad \qquad \quad 
\wt{E}^*_iE^*_{i+1}{\bf j} = \hat{C}^{+}_i.
\end{equation*}
\end{lemma}
\begin{proof}
Use Lemma \ref{action(E*;W)} and Lemma \ref{action(E*(C);W)} along with  ${\bf j} = \sum^{d-1}_{i=0}(\hat{C}^{-}_i+\hat{C}^+_i)$. 
\end{proof}

\begin{lemma}\label{C+=...;C-=...}
For $0 \leq i \leq d-1$, 
\begin{equation*}
\hat{C}^-_i = \sum^i_{j=0}A_j\hat{x} - \sum^{i-1}_{j=0}\hat{C}_j,
\qquad \qquad \quad
\hat{C}^+_i = \sum^i_{j=0}\hat{C}_j - \sum^i_{j=0}A_j\hat{x}.
\end{equation*}
\end{lemma}
\begin{proof}
Use (\ref{def(C-;C+)}) and induction on $i$.
\end{proof}

\begin{corollary}\label{Mx+MC=W} The following hold.
\begin{enumerate} 
\item[\rm (i)] $\Mx + \MC = \W$.
\item[\rm (ii)] $\Mx \cap \MC = \mbb{C}{\bf j}$.
\end{enumerate}
\end{corollary}

\begin{proof}
(i) By Lemma \ref{action(A;W)} $\W$ is invariant under $\mcal{M}$. Also $\hat{x}, \hat{C} \in \W$. Thus $\Mx+\MC \subseteq \W$. Concerning the reverse inclusion, recall by Lemma \ref{prim-T(C)}(v) that $\{\hat{C}_j\}^{d-1}_{j=0}$ forms a basis for $\MC$. By Lemma \ref{C+=...;C-=...}, the set $\{\hat{C}^{\pm}_i\}^{d-1}_{i=0}$ is contained in $\Mx+\MC$. The $\{\hat{C}^{\pm}_i\}^{d-1}_{i=0}$ span $\W$ and therefore $\W$ is contained in $\Mx+\MC$. The result follows.\\
(ii) We saw earlier that ${\bf j} \in \Mx$ and ${\bf j} \in \MC$, so $\mbb{C}{\bf j} \subseteq \Mx \cap \MC$.
To finish the proof we show that the dimension of $\Mx \cap \MC$ is 1. Using Example \ref{ex;Mx}, Lemma \ref{prim-T(C)}(v), and part (i) along with Corollary \ref{dim(W)=2d},
\begin{align*}
\dim(\Mx \cap \MC) & = \dim{\Mx}+\dim{\MC} - \dim(\Mx+\MC) \\
& = d+1+d-2d  \\ &=1.
\end{align*}
The result follows.
\end{proof}

\begin{proposition}\label{W;irreducible} The $\T$-module $\W$ is irreducible.
\end{proposition}
\begin{proof}
Since $\T$ is closed under the conjugate-transpose map, $\W$ is an orthogonal direct sum of irreducible $\T$-modules. Among these $\T$-modules, there exists one, denoted $W$, that is not orthogonal to $\hat{x}$. Now $E^*_0W \ne 0$. Also $E^*_0W \subseteq E^*_0\W = \text{Span}\{\hat{x} \}$. So $\hat{x} \in E^*_0W \subseteq W$ and hence $\Mx \subseteq W$. Thus ${\bf j} \in \Mx \subseteq W$, and so $\T{\bf j} \subseteq W$. Now $\{\hat{C}^{\pm}_{i}\}^{d-1}_{i=0} \subseteq W$ by Lemma \ref{E*E*=C+-}. The $\{\hat{C}^{\pm}_i\}^{d-1}_{i=0}$ span $\W$ and therefore $\W \subseteq W$. Consequently $\W=W$. The result follows.
\end{proof}

%

\noindent
We finish this section with a comment.

\begin{lemma}\label{dim(EiW)} For $0 \leq i \leq d$,
\begin{equation*}\label{dim(EiW)eq}
\dim{E_i\W} = 
\begin{cases}
2 & \quad \text{\rm if } \quad 1 \leq i \leq d-1, \\
1 & \quad \text{\rm if } \quad i \in \{0, d\}.
\end{cases}
\end{equation*}
\end{lemma}
\begin{proof}
Using Corollary \ref{Mx+MC=W}(i), $E_i\W = E_i(\Mx+\MC) = E_i\Mx + E_i\MC$. By linear algebra,
\begin{equation}\label{dim(EiW);(1)}
\dim{E_i\W}  = \dim{E_i\Mx}+\dim{E_i\MC} - \dim{(E_i\Mx \cap E_i\MC)}.
\end{equation}
Observe that $E_i(\Mx\cap\MC)$ is contained in $E_i\Mx \cap E_i\MC$. By this and Corollary \ref{Mx+MC=W}(ii),
\begin{equation}\label{dim(EiW);(2)}
\dim{E_i(\mbb{C}{\bf j})} \leq \dim{(E_i\Mx \cap E_i\MC)}.
\end{equation}
Combining (\ref{dim(EiW);(1)}) and (\ref{dim(EiW);(2)}),
$$\dim{E_i\W}  \leq \dim{E_i\Mx}+\dim{E_i\MC} - \dim{E_i(\mbb{C}{\bf j})}.$$
By Example \ref{ex;Mx}, $\dim{E_r\Mx} =1$ for $0 \leq r \leq d$. By Lemma \ref{prim-T(C)}(v), $\dim{E_r\MC}=1$ for $0 \leq r \leq d-1$. Also ${E_d\MC}=0$ by Lemma \ref{EdC=0}. Moreover,  $\dim{E_0(\mbb{C}{\bf j})}=1$ and $E_r$ vanishes on $\mbb{C}{\bf j}$ for $1 \leq r \leq d$. By these comments the dimension of $E_i\W$ is at most $2$ for $1 \leq i \leq d-1$ and at most $1$ for $i \in \{0, d\}$. The sum of these upper bounds is $2d$. Also $\sum^d_{i=0}\dim{E_i{\W}} = \dim{\W}=2d$. By these comments the dimension of $E_i\W$ equals $2$ for $ 1 \leq i \leq d-1$ and equals $1$ for $i \in \{0, d\}$. The result follows.
\end{proof}


\medskip
\section{$\W$ as a $T$-module}\label{W-T(x)}

Recall the subspace $\W$ from above line (\ref{Ci=(C-i)+(C+i)}). In Lemma \ref{W:T(x)-mod} we saw that $\W$ is a $T$-module. Throughout this section we adopt this point of view. Recall the primary $T$-module $\Mx$ from Example \ref{ex;Mx}. Observe that $\Mx$ is an irreducible $T$-submodule of $\W$. Let $\Mxp$ denote the orthogonal complement of $\Mx$ in $\W$. Observe that $\dim{\Mxp}=d-1$, since $\dim{\W}=2d$ by Corollary \ref{dim(W)=2d} and $\dim{\Mx}=d+1$ by Example \ref{ex;Mx}. The subspace $\Mxp$ is a $T$-submodule of $\W$ since $T$ is closed under the conjugate-transpose map. By construction
\begin{equation}\label{ODS(T-mods)}
\qquad \qquad  \W = \Mx + \Mxp \qquad \qquad \text{\rm (orthogonal direct sum of $T$-modules).}
\end{equation}


\begin{lemma}\label{Mxp}
The $T$-module $\Mxp$ is irreducible and thin, with endpoint $1$, dual endpoint $1$, and diameter $d-2$. 
\end{lemma}
\begin{proof}
Abbreviate $W=\Mxp$. Using Example \ref{ex;Mx} and Corollary \ref{dim(E*iW)}, 
\begin{equation}\label{Mxp;pf(1)}
\dim{E^*_iW} = 
\begin{cases}
1 & \text{ \quad if \quad} 1 \leq i \leq d-1, \\
0 & \text{ \quad if \quad} i \in \{0, d\}.
\end{cases}
\end{equation}
Using Example \ref{ex;Mx} and Lemma \ref{dim(EiW)},
\begin{equation*}
\dim{E_iW} = 
\begin{cases}
1 & \text{ \quad if \quad} 1 \leq i \leq d-1, \\
0 & \text{ \quad if \quad} i \in \{0, d\}.
\end{cases}
\end{equation*}
To finish the proof it suffices to show that the $T$-module $W$ is irreducible. Write $W$ as a direct sum of irreducible $T$-modules. Among these $T$-modules, there exists a module $U$ with endpoint $1$ since $E^*_0W = 0$ and $E^*_1W \ne 0$ by (\ref{Mxp;pf(1)}). By \cite[Lemma 5.1]{JSC}, the dimension of $U$ is at least $d-1$. But $U \subseteq W$ and the dimension of $W$ is $d-1$, so the dimension of $U$ is at most $d-1$. Therefore $U=W$, which means $W$ is irreducible. The result follows.
\end{proof}

\noindent
Recall the primary Leonard system $\Phi$ on $\Mx$ from Lemma \ref{IN(Ga,Phi)},
with its parameter array
$$p(\Phi)=(\{\tht_i\}^{d}_{i=0}; \{\tht^*_i\}^{d}_{i=0}; \{\varphi_i\}^{d}_{i=1}; \{\phi_i\}^{d}_{i=1}).$$
We now consider the Leonard system on $\Mxp$. By Lemma \ref{LS(U)} and Lemma \ref{Mxp}, the elements
\begin{equation*}
(A; A^*; \{E_{i}\}^{d-1}_{i=1}; \{E^*_{i}\}^{d-1}_{i=1})
\end{equation*}
act on $\Mxp$ as a Leonard system; we denote this Leonard system by $\Phi^{\perp}$. We denote the  parameter array of $\Phi^{\perp}$ by
\begin{equation}\label{PA;x-perp}
p(\Phi^{\perp}) =
(\{\thtp_i\}^{d-2}_{i=0}; \{\thtsp_i\}^{d-2}_{i=0}; \{\varphip_i\}^{d-2}_{i=1}; \{\phip_i\}^{d-2}_{i=1}).
\end{equation}

\begin{lemma}\label{tht(i)=thtp(i)}
With the above notation,
\begin{align}\label{tht(i)=thtp(i);eq}
&& \tht_i^{\perp} = \tht_{i+1} , & & \tht^{*\perp}_i = \tht^*_{i+1},  && (0 \leq i \leq d-2).
\end{align}
\end{lemma}
\begin{proof}
The endpoint and dual endpoint of $\Mxp$ are both $1$.
\end{proof}

\noindent
Recall from Example \ref{ex;Mx} that the $\{A_i\hat{x}\}^d_{i=0}$ form a $\Phi$-standard basis for $\Mx$. For notational convenience define $v_i = A_i\hat{x}$ for $0 \leq i \leq d$. Observe that
\begin{equation}\label{Phi-SB}
v_0 = \hat{C}^-_0, \qquad
v_i = \hat{C}^+_{i-1}+\hat{C}^-_i \ (1 \leq i \leq d-1), \qquad
v_d = \hat{C}^+_{d-1}.
\end{equation}
We now discuss a $\Phi^{\perp}$-standard basis for $\Mxp$. Define the scalars
\begin{equation}\label{epsilon_i}
\qquad \epsilon_i := - \frac{|C^+_{i-1}|}{|C^-_{i}|}  \qquad \qquad (1 \leq i \leq d-1).
\end{equation} 
By Corollary \ref{C+-;nonempty}, $\epsilon_i$ is well-defined and nonzero for $1 \leq i \leq d-1$.

\begin{lemma} For $1 \leq i \leq d-1$,
\begin{equation}\label{epsilon_i(4)}
\epsilon_i = \frac{(1-q^i)(1-s^*q^{i+d+1})}{q^d(1-q^{i-d})(1-s^*q^{i+1})}.
\end{equation}
\end{lemma}
\begin{proof}
Using Theorem \ref{gamma;beta}, Corollary \ref{|edges|}(iv), and (\ref{epsilon_i}),
$$
\epsilon_i = \frac{\tc_i - c_i}{b_i - \tb_{i-1}} =  \frac{N_i - |C|}{N_{i-1}} \frac{c_i}{b_i}.
$$
In the above line simplify the expression on the right using (\ref{b_i;q-terms}), (\ref{c_i;q-terms}), and Lemma \ref{formula(Ni);q,h,s,r} to get (\ref{epsilon_i(4)}).
\end{proof}

\begin{lemma}\label{pre(OB;Mxp)}
For $1 \leq i \leq d-1$ the vector 
\begin{equation*}
\hat{C}^{+}_{i-1} + \epsilon_i\hat{C}^{-}_i
\end{equation*}
is a basis for $E^*_i(\Mxp)$.
\end{lemma}
\begin{proof}
By Lemma \ref{W=ods(E*i)}(ii), $0 \ne \hat{C}^{+}_{i-1} + \epsilon_i\hat{C}^{-}_i \in E^*_i\W$. By (\ref{ODS(T-mods)}) the subspace $E^*_i(\Mxp)$ is the orthogonal complement of $E^*_i\Mx$ in $E^*_i\W$. Recall $v_i$ is a basis for $E^*_i\Mx$. Using (\ref{Phi-SB}),
$$
\langle v_i, \hat{C}^{+}_{i-1} + \epsilon_i\hat{C}^{-}_i  \rangle 
	 =  \lVert \hat{C}^{+}_{i-1} \rVert^2 + \epsilon_i\lVert \hat{C}^{-}_{i} \rVert^2 
	 = |{C}^+_{i-1}| + \epsilon_i|{C}^-_i|=0.
$$
Therefore $\hat{C}^{+}_{i-1} + \epsilon_i\hat{C}^{-}_i \in E^*_i(\Mxp)$. The result follows since $E^*_i(\Mxp)$ has dimension $1$ by  (\ref{Mxp;pf(1)}).
\end{proof}

\begin{corollary}\label{OB;Mxp}
The vectors
\begin{equation*}\label{v'_i}
~ \qquad  \qquad \hat{C}^{+}_{i-1} + \epsilon_i\hat{C}^{-}_i \qquad \qquad (1 \leq i \leq d-1)
\end{equation*}
form an orthogonal basis for $\Mxp$.
\end{corollary}
\begin{proof}
Follows from Lemma \ref{pre(OB;Mxp)}.
\end{proof}

 \noindent
In Lemma \ref{OB;Mxp} we found a basis $\{\hat{C}^{+}_{i-1} + \epsilon_i\hat{C}^{-}_i\}^{d-1}_{i=1}$ for $\Mxp$. As we will see, this basis is not a $\Phi^{\perp}$-standard basis for $\Mxp$. In order to turn it into a $\Phi^{\perp}$-standard basis we make an adjustment. Pick a nonzero $w \in E_1(\Mxp)$. Define the complex scalars $\{\xi_i\}^{d-1}_{i=1}$ such that
\begin{equation}\label{w(-E1Mxp1}
w = \sum^{d-1}_{i=1} \xi_i(\hat{C}^{+}_{i-1} + \epsilon_i\hat{C}^{-}_i).
\end{equation}
For notational convenience we rewrite (\ref{w(-E1Mxp1}) as
\begin{equation*}\label{w(-E1Mxp}
w = \sum^{d-2}_{i=0} \xi_{i+1}(\hat{C}^{+}_{i} + \epsilon_{i+1}\hat{C}^{-}_{i+1}).
\end{equation*}
Define
\begin{equation}\label{vp_i}
~ \qquad \qquad \vp_i = \xi_{i+1}(\hat{C}^{+}_{i} + \epsilon_{i+1}\hat{C}^{-}_{i+1}) \qquad \qquad (0 \leq i \leq d-2).
\end{equation}
For notational convenience define $\vp_{-1} = 0$ and $\vp_{d-1} = 0$.

\begin{lemma}\label{Phip-SB}
The vectors $\{\vp_i\}^{d-2}_{i=0}$ form a $\Phi^{\perp}$-standard basis for $\Mxp$.
\end{lemma}
\begin{proof}
By Lemma \ref{pre(OB;Mxp)}, $\vp_i \in E^*_{i+1}(\Mxp)$ for $0 \leq i \leq d-2$. Also, $\sum^{d-2}_{i=0}\vp_i=w \in E_1(\Mxp)$. The result follows from these comments and Lemma \ref{Phi-sb}.
\end{proof}

\begin{corollary}
The scalars $\{\xi_i\}^{d-1}_{i=1}$ from {\rm(\ref{w(-E1Mxp1})} are all nonzero.
\end{corollary}
\begin{proof}
By Lemma \ref{Phip-SB} $\{\vp_i\}^{d-2}_{i=0}$ are all nonzero. By this and (\ref{vp_i}) the result follows.
\end{proof}

\noindent
By the comments below Lemma \ref{Phi-sb}, the matrix representing $A^*$ relative to the basis $\{\vp_i\}^{d-2}_{i=0}$ is
\begin{equation}\label{[A*]_(vp)}
\diag(\thtsp_0, \thtsp_1, \thtsp_2, \ldots, \thtsp_{d-2}).
\end{equation}
Let $\{a^{\perp}_i\}^{d-2}_{i=0}, \{b^{\perp}_i\}^{d-3}_{i=0}, \{c^{\perp}_i\}^{d-2}_{i=1}$ denote the intersection numbers of the Leonard system $\Phi^{\perp}$. By construction the matrix representing $A$ relative to $\{\vp_i\}^{d-2}_{i=0}$ is 
\begin{align}\label{[A]_(vp)}
\left[
\begin{array}{ccccc}
\vspace{0.2cm}
\ap_0 & \bp_0 && &{\bf 0}\\
\cp_1 & \ap_1 & \bp_1& \\
&\cp_2 & \ap_2 & \ddots& \\
\vspace{0.2cm}
&& \ddots & \ddots & \bp_{d-3} \\
{\bf 0}&&& \cp_{d-2} & \ap_{d-2}
\end{array}
\right]
\end{align}
For convenience, define $\bp_{-1}=0$ and $\cp_{d-1}=0$.
Recall the scalars $a_i, b_i, c_i$ from below Lemma \ref{IN(Ga,Phi)} and the  $\ta_i, \tb_i, \tc_i$ from (\ref{def(ta,tb,tc)}).

\begin{lemma}\label{6-terms} With the above notation the following {\rm(i)}--{\rm(vi)} hold.
\begin{enumerate}
\item[\rm (i)] $\bp_{i-1}\xi_{i}  =  \xi_{i+1}b_{i}$ \qquad \qquad  $(1 \leq i \leq d-2)$.
\item[\rm (ii)] $\bp_{i-1}\xi_i\epsilon_i = \xi_{i+1}(b_i-\tb_i)+ \xi_{i+1}\epsilon_{i+1}\tb_{i} \qquad \qquad (1 \leq i \leq d-2)$.
\item[\rm(iii)] $\ap_i = \ta_i-c_{i+1}+\tc_i + \epsilon_{i+1}(\tb_i-b_{i+1}) \qquad \qquad (0 \leq i \leq d-2)$.
\item[\rm(iv)] $\ap_i\epsilon_{i+1} = \tc_{i+1}-c_{i+1}+\epsilon_{i+1}(\ta_{i+1}-b_{i+1}+\tb_{i+1})$ \qquad \qquad $(0 \leq i \leq d-2)$.
\item[\rm(v)] $\cp_{i+1}\xi_{i+2} = \xi_{i+1}\tc_{i+1}+ \xi_{i+1}\epsilon_{i+1}(c_{i+2}-\tc_{i+1})$ \qquad\qquad $(0 \leq i \leq d-3)$.
\item[\rm(vi)] $\cp_{i+1}\xi_{i+2}\epsilon_{i+2} =  \xi_{i+1}\epsilon_{i+1}c_{i+2}$ \qquad \qquad $(0 \leq i \leq d-3)$.
\end{enumerate}
\end{lemma}
\begin{proof}
For $0 \leq i \leq d-2$ we evaluate $A\vp_i$ in two ways. First, using (\ref{[A]_(vp)}),
\begin{equation*}\label{A.vp_i}
A\vp_i = \bp_{i-1}\vp_{i-1} + \ap_i\vp_i + \cp_{i+1}\vp_{i+1}.
\end{equation*}
In this equation evaluate the right-hand side using (\ref{vp_i}). Secondly, in $A\vp_i$ eliminate $\vp_i$ using  (\ref{vp_i}) and evaluate the result using Lemma \ref{action(A;W)}. We have just evaluated $A\vp_i$ in two ways. Compare the results using the linear independence of $\{\hat{C}^{\pm}_j\}^{d-1}_{j=0}$. The result  follows.
\end{proof}

\begin{corollary}\label{(a,b,c)perp} Referring to Lemma {\rm\ref{6-terms}},  the following hold.
\begin{enumerate}
\item[\rm(i)] $\bp_{i} = \frac{\xi_{i+2}}{\xi_{i+1}}b_{i+1}$ \qquad $(0 \leq i \leq d-3)$.
\item[\rm(ii)] $\ap_{i} = b_0 -\tb_{i} - \tc_{i+1}$ \qquad $(0 \leq i \leq d-2)$.
\item[\rm(iii)] $\cp_i = \frac{\xi_i\epsilon_i}{\xi_{i+1}\epsilon_{i+1}}c_{i+1}$ \qquad $(1 \leq i \leq d-2)$.
\end{enumerate}
\end{corollary}

\begin{proof}
(i) Immediate from Lemma \ref{6-terms}(i).\\
(ii) In the equation from Lemma \ref{6-terms}(iii), eliminate $\epsilon_{i+1}$ using Corollary \ref{|edges|}(iv) along with (\ref{epsilon_i}) and simplify the result using $\ta_i+\tb_i+\tc_i = b_0$.\\
(iii) Immediate from Lemma \ref{6-terms}(vi).
\end{proof}

\noindent
We now compare the parameter arrays $p(\Phi)$ and $p(\Phi^{\perp})$. Recall that $\Phi$ has $q$-Racah type. Recall the scalars $h, h^*, s, s^*, r_1, r_2$ from above Note \ref{h,h*}. Recall the parameter array of $\Phi^{\perp}$ from (\ref{PA;x-perp}).

\begin{theorem}\label{q-rac(x-perp)} With the above notation, for $0 \leq i \leq d-2$,
\begin{eqnarray}
\label{thtperp}
\thtp_i 	& = & \thtp_0+\hp(1-q^{i})(1-\sp q^{i+1})q^{-i}, \\
\label{tht*perp}
\thtsp_i 	& = & \thtsp_0+\hsp(1-q^{i})(1-\ssp q^{i+1})q^{-i},
\end{eqnarray}
and for $1 \leq i \leq d-2$,
\begin{eqnarray}
\label{varphip_i}
\varphip_i	& = & \hp\hsp q^{1-2i}(1-q^{i})(1-q^{i-d+1})(1-\rp_1q^i)(1-\rp_2q^i), \\
\label{phip_i}
\phip_i 		& = & \hp\hsp q^{1-2i}(1-q^{i})(1-q^{i-d+1})(\rp_1-\ssp q^i)(\rp_2-\ssp q^i)/\ssp,
\end{eqnarray} 
where $\thtp_0$, $\thtsp_0$ are from {\rm(\ref{tht(i)=thtp(i);eq})} and
\begin{align}
\label{h,s;perp(1)}
&\hp = hq^{-1}, && \sp = sq^2, &&\rp_1=r_1q,&\\
\label{h,s;perp(2)}
&\hsp = h^*q^{-1} ,&&  \ssp = s^*q^2, &&\rp_2=r_2q.
\end{align}
\end{theorem}
\begin{proof}
Recall that $\Phi^{\perp}$ has diameter $d-2$. To get (\ref{thtperp}), (\ref{tht*perp}) use  (\ref{theta_i}), (\ref{theta^*_i}) and Lemma \ref{tht(i)=thtp(i)}. We now verify (\ref{varphip_i}), (\ref{phip_i}). To this end, recall the intersection numbers $\{\ap_j\}^{d-2}_{j=0}$ of $\Phi^{\perp}$. 
Using (\ref{Madrid(ai);eq1}), 
$\varphip_1 = (\ap_0 - \thtp_0)(\thtsp_0-\thtsp_1).$
Evaluate this using Lemma \ref{tht(i)=thtp(i)} and Corollary \ref{(a,b,c)perp}(ii)  to get
\begin{equation}\label{pf(PA;perp);eq3}
\varphip_1 = (b_0-\tb_0-\tc_1-\tht_1)(\tht^*_1-\tht^*_2).
\end{equation}
Evaluate the right-hand side of (\ref{pf(PA;perp);eq3}) using $b_0=\tht_0$, (\ref{theta_i}), (\ref{theta^*_i}), and Corollary \ref{be,ga(q-term)} and simplify the result using (\ref{h,s;perp(1)}), (\ref{h,s;perp(2)}). This yields (\ref{varphip_i}) for $i=1$.
Using this together with (\ref{thtperp}), (\ref{tht*perp}) and the condition (PA4) of Theorem \ref{thm:LS<->PA} we obtain (\ref{phip_i}) for $1 \leq i \leq d-2$. We now use (\ref{phip_i}) at $i=1$ together with (\ref{thtperp}), (\ref{tht*perp}) and the condition (PA3) of Theorem \ref{thm:LS<->PA} to obtain (\ref{varphip_i}) for $1 \leq i \leq d-2$. The result follows.
\end{proof}

\begin{corollary}
The Leonard system $\Phi^{\perp}$ has $q$-Racah type.
\end{corollary}
\begin{proof}
Compare Example \ref{q-rac;PA} and Theorem \ref{q-rac(x-perp)}.
\end{proof}


\noindent
Recall the intersection numbers $\{b^{\perp}_i\}^{d-3}_{i=0}, \{c^{\perp}_i\}^{d-2}_{i=1}$ of the Leonard system $\Phi^{\perp}$. 

\begin{corollary}\label{b,c;perp}
With the above notation, 
\begin{align}
\label{b0;h,s,...}
& \bp_0	
 =	\frac{h(1-q^{-d+2})(1-r_1q^{2})(1-r_2q^{2})}
		{q(1-s^*q^{4})}, &&  \\
\label{bp;h,s,...}
&\bp_i	
 =	\frac{h(1-q^{i-d+2})(1-s^*q^{i+3})(1-r_1q^{i+2})(1-r_2q^{i+2})}
		{q(1-s^*q^{2i+3})(1-s^*q^{2i+4})} && &&(1\leq i \leq d-3), \\
\label{cp;h,s,...}
&\cp_i
 =	\frac{h(1-q^{i})(1-s^*q^{i+d+1})(r_1-s^*q^{i+1})(r_2-s^*q^{i+1})}
		{s^*q^{d-1}(1-s^*q^{2i+2})(1-s^*q^{2i+3})} && &&  (1\leq i \leq d-3), \\
\label{c(d-2);h,s,...}
&\cp_{d-2}
 = 	\frac{h(1-q^{d-2})(r_1-s^*q^{d-1})(r_2-s^*q^{d-1})}{s^*q^{d-1}(1-s^*q^{2d-2})}.	
\end{align}
\end{corollary}

\begin{proof}
Use (\ref{b_0;q-terms})--(\ref{c_d;q-terms}) and (\ref{h,s;perp(1)}), (\ref{h,s;perp(2)}).
\end{proof}

\noindent
We finish this section with some comments about the scalars $\{\xi_i\}^{d-1}_{i=1}$ from (\ref{w(-E1Mxp1}).

\begin{lemma}\label{Lem(xi_i)}
The vector $w$ in line {\rm(\ref{w(-E1Mxp1})} can be chosen such that 
\begin{equation}\label{xi_i}
\xi_i = q^{1-i}(1-q^{i-d})(1-s^*q^{i+1}) \qquad \qquad (1 \leq i \leq d-1).
\end{equation}
\end{lemma}
\begin{proof}
Observe that the vector $w$ is defined up to multiplication by a nonzero scalar in $\mbb{C}$. Multiplying $w$ by a nonzero scalar if necessary, we may assume that
\begin{equation}\label{xi_1}
\xi_1 = (1-q^{1-d})(1-s^*q^2).
\end{equation}
Using Lemma \ref{6-terms}(i) and induction on $i$,
\begin{equation}\label{xi_1;eq(1)}
\xi_i = \frac{\bp_0\bp_1 \cdots \bp_{i-2}}{b_1b_2 \cdots b_{i-1}}\xi_1 \qquad \qquad (1 \leq i \leq d-1).
\end{equation}
Evaluate (\ref{xi_1;eq(1)}) using (\ref{b_i;q-terms}), (\ref{b0;h,s,...}), (\ref{bp;h,s,...}) and (\ref{xi_1}) to get (\ref{xi_i}). The result follows.
\end{proof}

\noindent
From now on we assume that the vector $w$ in line (\ref{w(-E1Mxp1}) has been chosen such that (\ref{xi_i}) holds. Recall the vectors $\{\vp_i\}^{d-2}_{i=0}$ from (\ref{vp_i}).

\begin{lemma} 
For $0 \leq i \leq d-2$, 
$$
\vp_i = q^{-i}(1-q^{1+i-d})(1-s^*q^{i+2})\hat{C}^{+}_{i}
+q^{-i-d}(1-q^{i+1})(1-s^*q^{i+d+2})\hat{C}^{-}_{i+1}.
$$
\end{lemma}
\begin{proof} 
Evaluate (\ref{vp_i}) using (\ref{epsilon_i(4)}) and (\ref{xi_i}).
\end{proof}


\medskip
\section{$\W$ as a $\wt{T}$-module}\label{W-T(C)}

In the previous section we discussed $\W$ as a $T$-module. Recall the algebra $\wt{T}$ from above Lemma \ref{prim-T(C)}. In Lemma \ref{W:T(C)-mod} we saw that $\W$ is a $\wt{T}$-module. Throughout this section we adopt this point of view. Recall the primary $\wt{T}$-module $\MC$ from Lemma \ref{prim-T(C)}. Observe that $\MC$ is an irreducible $\wt{T}$-submodule of $\W$. Let $\MCp$ denote the orthogonal complement of $\MC$ in $\W$. Observe that $\dim{\MCp}=d$, since $\dim{\W}=2d$ and $\dim{\MC}=d$. The subspace $\MCp$ is a $\wt{T}$-submodule of $\W$ since $\wt{T}$ is closed under the conjugate-transpose map. By construction
\begin{equation}\label{ODS(TC-mods)}
\qquad \qquad \W = \MC + \MCp \qquad \qquad \text{\rm(orthogonal direct sum of $\wt{T}$-modules).}
\end{equation}
Our next goal is to show that the $\wt{T}$-module $\MCp$ is irreducible. 

\begin{lemma}\label{dim;E,E*;MCp}
Let $W$ denote the $\wt{T}$-module $\MCp$. The following hold.
\begin{enumerate}
\item[\rm(i)] The dimension of $E_iW$ is $1$ for $1 \leq i \leq d$. Moreover, $E_0W=0$.
\item[\rm{(ii)}] The dimension of $\wt{E}^*_iW$ is $1$ for $0 \leq i \leq d-1$.
\end{enumerate}
\end{lemma}
\begin{proof}
(i) Use Corollary \ref{mal-free(A,A*)}(i), Lemma \ref{dim(EiW)}, and (\ref{ODS(TC-mods)}).\\
(ii) Use Corollary \ref{mal-free(A,A*)}(ii), Corollary \ref{dim(E(C)*iW)}, and (\ref{ODS(TC-mods)}).
\end{proof}

\begin{lemma}\label{C-perp}
The vector $\hat{C} - |C|\hat{x}$ is a basis for $\wt{E}^*_0(\MCp)$.
\end{lemma}
\begin{proof}
By (\ref{ODS(TC-mods)}) the subspace $\wt{E}^*_0(\MCp)$ is the orthogonal complement of $\wt{E}^*_0\MC$ in $\wt{E}^*_0\W$. By construction $\hat{C} - |C|\hat{x}$ is nonzero and contained in $\wt{E}^*_0\W$. Recall $\hat{C}$ is a basis for $\wt{E}^*_0\MC$. The vector $\hat{C} - |C|\hat{x}$ is orthogonal to $\hat{C}$ since
$$
\langle \hat{C} - |C|\hat{x}, \hat{C} \rangle = \| \hat{C} \|^2 -|C| =0.
$$
Therefore  $\hat{C} - |C|\hat{x} \in \wt{E}^*_0(\MCp)$. The result follows since $\wt{E}^*_0(\MCp)$ has dimension 1 by Lemma \ref{dim;E,E*;MCp}(ii).
\end{proof}

\begin{lemma}\label{basis(Mu)}
Abbreviate $u = \hat{C}-|C|\hat{x}$. 
\begin{enumerate}
\item[\rm(i)] The vectors $\{A_iu\}^{d-1}_{i=0}$ form a basis for $\MCp$.
\item[\rm(ii)] The ${\mcal{M}}$-module $\MCp$ is generated by $u$.
\end{enumerate}
\end{lemma}
\begin{proof}
(i) For $0 \leq i \leq d-1$, evaluate $A_iu$  using Lemma \ref{prim-T(C)}(ii). Express the result as a linear combination of $\{\hat{C}^{\pm}_j\}^{d-1}_{j=0}$ using (\ref{Ci=(C-i)+(C+i)}) and (\ref{Phi-SB}). We find
\begin{equation}\label{Aiu=lincombo}
A_iu = (|C|-N_{i-1})\hat{C}^-_{i-1} + (-N_{i-1})C^+_{i-1} + (N_i-|C|)\hat{C}^-_i + N_i\hat{C}^+_i.
\end{equation}
Consider the last term on the right-hand side of (\ref{Aiu=lincombo}). Recall $N_i$ is nonzero by Corollary \ref{0<Ni<C}. Therefore the $\{A_iu\}^{d-1}_{i=0}$ are linearly independent. The result follows since $\dim{\MCp}=d$.\\
(ii) Consider $\mcal{M}u = \text{Span}\{A_iu\}^d_{i=0}$. Since $u\in\MCp$, we have $\mcal{M}u\subseteq\MCp$. But $\mcal{M}u$ contains $\text{Span}\{Au_i\}^{d-1}_{i=0}$, which is equal to $\MCp$ by part (i). By these comments $\mcal{M}u = \MCp$. The result follows.
\end{proof}

\begin{proposition}\label{MCp}
The $\wt{T}$-module $\MCp$ is irreducible.
\end{proposition}
\begin{proof}
Abbreviate $u =\hat{C}-|C|\hat{x}$. Express $\MCp$ as the orthogonal direct sum of irreducible $\wt{T}$-modules. Since $u \in \MCp$, among these irreducible $\wt{T}$-modules there exists one, denoted $U$, that is not orthogonal to $u$. Observe $\wt{E}^*_0U \ne 0$. Also $\wt{E}^*_0U \subseteq \wt{E}^*_0(\MCp) = \text{Span}\{u\}$, so $\wt{E}^*_0U$ is spanned by $u$. By this and since $U$ is $\wt{T}$-module, we get $u \in \wt{E}^*_0U \subseteq U$. But then $\mcal{M}u \subseteq U$. By this and Lemma \ref{basis(Mu)}(ii) $\MCp \subseteq U$ and therefore $\MCp=U$, as desired.
\end{proof}

\noindent
We now show that the irreducible $\wt{T}$-module $\MCp$ supports a Leonard system.

\begin{lemma}\label{LS(MCp)}The following matrices
\begin{equation}\label{LS(MCp);mat}
(A; \wt{A}^*; \{E_i\}^{d}_{i=1}; \{\wt{E}^*_i\}^{d-1}_{i=0})
\end{equation}
act on $\MCp$ as a Leonard system. 
\end{lemma}
\begin{proof}
We show that (\ref{LS(MCp);mat}) satisfies conditions (i)--(v) in Definition \ref{Def:LS}. Conditions (i)--(iii) follow from (\ref{A*C=sum(tht*E*C)}), Corollary \ref{tht;til;md} and Lemma \ref{dim;E,E*;MCp}, condition (iv) follows from (\ref{EA*E}) and (\ref{def(A*C)}), and condition (v) follows from \cite[Lemma 4.1]{Suzuki}. 
\end{proof}

\noindent
We let $\wt{\Phi}^{\perp}$ denote the Leonard system on $\MCp$ from Lemma \ref{LS(MCp)}. We denote the parameter array of  $\wtPp$ by

\begin{equation}\label{PA;til(Phi)perp}
p(\wtPp) = (\{\tthtp_i\}^{d-1}_{i=0}; \{\tthtsp_i\}^{d-1}_{i=0}; \{\tvarphip_i\}^{d-1}_{i=1}; \{\tphip_i\}^{d-1}_{i=1}).
\end{equation}

\begin{lemma}\label{ttht=tthts}
With the above notation,
\begin{align} \label{ttht=tthts;eq}
&& \tthtp_i = \tht_{i+1}, && \tthtsp_i = \ttht^*_i && (0 \leq i \leq d-1).
\end{align}
\end{lemma}
\begin{proof}
By construction, $AE_{i+1} = \tthtp_iE_{i+1}$ on $\MCp$. But $AE_{i+1}=\tht_{i+1}E_{i+1}$, so $\tthtp_i=\tht_{i+1}$. Similarly we get $ \tthtsp_i = \ttht^*_i$.
\end{proof}

\noindent
In Lemma \ref{SB(MC)}, we saw that the $\{\hat{C}_i\}^{d-1}_{i=0}$ form a $\wt{\Phi}$-standard basis for $\MC$. To keep our notation consistent, define $\tv_i = \hat{C}_i$ for $0 \leq i \leq d-1$. By (\ref{Ci=(C-i)+(C+i)}),
\begin{equation}\label{SB-MC}
\qquad \qquad \tv_i = \hat{C}^-_i + \hat{C}^+_i \qquad \qquad (0 \leq i \leq d-1).
\end{equation}
Our next goal is to find a $\wtPp$-standard basis for $\MCp$. Define the scalars
\begin{equation}\label{tau_i}
\qquad \tau_i := -\frac{|C^{+}_{i}|}{|C^{-}_{i}|} \qquad \qquad (0 \leq i \leq d-1).
\end{equation}
By Corollary \ref{C+-;nonempty}, $\tau_i$ is well-defined and nonzero for $0 \leq i \leq d-1$.

\begin{lemma}\label{tau;q-term}  For $0 \leq i \leq d-1$,
\begin{equation}\label{tau;q-term;eq}
\tau_i =  \frac{s^*(1-r_1q^{i+1})(1-r_2q^{i+1})}{(r_1-s^*q^{i+1})(r_2-s^*q^{i+1})}.
\end{equation}
\end{lemma}
\begin{proof}
Evaluate (\ref{tau_i}) using Lemma \ref{cardofCi} and eliminate the scalars $\{\tb_j\}^{i-1}_{j=0}, \{\tc_j\}^{i}_{j=1}$ using Theorem \ref{gamma;beta}. Simplify the result to find
$$
\tau_i = \frac{N_i-|C|}{N_i}.
$$
Evaluate this equation using Lemma \ref{formula(Ni);q,h,s,r} to obtain the result.
\end{proof}

\begin{lemma}\label{pre(OB;MCp)}
For $0 \leq i \leq d-1$ the vector
\begin{equation*}
\tau_i\hat{C}^{-}_i + \hat{C}^{+}_{i} 
\end{equation*}
is a basis for $\wt{E}^*_i(\MCp)$.
\end{lemma}
\begin{proof}
By Lemma \ref{W=ods(E*(C)i)}, $0 \ne \tau_i\hat{C}^-_i + \hat{C}^+_i \in \wt{E}^*_i\W$. By (\ref{ODS(TC-mods)}) the subspace $\wt{E}^*_i(\MCp)$ is the orthogonal complement of $\wt{E}^*_i\MC$ in $\wt{E}^*_i\W$. Recall $\tv_i$ is a basis for $\wt{E}^*_i\MC$. Using (\ref{SB-MC}),
$$
\langle \tv_i, \tau_i\hat{C}^-_i + \hat{C}^+_i\rangle = \tau_i\|\hat{C}^-_i\|^2 + \|\hat{C}^+_i\|^2 = \tau_i|C^-_i| + |C^+_i|=0.
$$
Therefore $\tau_i\hat{C}^-_i + \hat{C}^+_i \in \wt{E}^*_i(\MCp)$. The result follows since $\wt{E}^*_i(\MCp)$ has dimension $1$ by Lemma \ref{dim;E,E*;MCp}(ii).
\end{proof}

\begin{corollary} \label{OB;MCp}
The vectors
\begin{equation*}\label{u'_i}
~ \qquad \qquad \tau_i\hat{C}^{-}_i + \hat{C}^{+}_{i} \qquad \qquad (0 \leq i \leq d-1)
\end{equation*}
form an orthogonal basis for $\MCp$.
\end{corollary}
\begin{proof}
Follows from Lemma \ref{pre(OB;MCp)}.
\end{proof}

\noindent
As we will see, the vectors $\{\tau_i\hat{C}^{-}_i + \hat{C}^{+}_{i}\}^{d-1}_{i=0}$ do not form a $\wtPp$-standard basis for $\MCp$. In order to turn it into a $\wtPp$-standard basis we make an adjustment. 
Pick a nonzero $\wt{w} \in E_1(\MCp)$. Define the complex scalars $\{\zeta_i\}^{d-1}_{i=0}$ such that
\begin{equation}\label{u=sum(SB)}
\wt{w} = \sum^{d-1}_{i=0} \zeta_i(\tau_i\hat{C}^-_i+\hat{C}^+_i).
\end{equation}
Define 
\begin{equation}\label{tilde(vp_i)}
\qquad \qquad \tvp_i = \zeta_i(\tau_i\hat{C}^-_i+\hat{C}^+_i) \qquad \qquad (0 \leq i \leq d-1).
\end{equation}
For notational convenience define $\tvp_{-1}=0$ and $\tvp_{d}=0$.

\begin{lemma}\label{SB(MCp)}
The vectors $\{\tvp_i\}^{d-1}_{i=0}$ form a $\wtPp$-standard basis for $\MCp$.
\end{lemma}
\begin{proof}
By Lemma \ref{pre(OB;MCp)}, $\tvp_i \in \wt{E}^*_{i}(\MCp)$ for $0 \leq i \leq d-1$. Also, $\sum^{d-1}_{i=0}\tvp_i = \wt{w} \in E_1(\MCp)$. The result follows from these comments and Lemma \ref{Phi-sb}.
\end{proof}

\begin{corollary} The scalars $\{\zeta_i\}^{d-1}_{i=0}$ from {\rm(\ref{u=sum(SB)})} are all nonzero.
\end{corollary}
\begin{proof}
By Lemma \ref{SB(MCp)} $\{\tvp_i\}^{d-1}_{i=0}$ are all nonzero. By this and (\ref{tilde(vp_i)}) the result follows.
\end{proof}

\noindent
The matrix representing $\wt{A}^*$ relative to the basis $\{\tvp_i\}^{d-1}_{i=0}$ is
\begin{equation}\label{[A^*(C)]_(tvp)}
\diag(\tthtsp_0, \tthtsp_1, \tthtsp_2, \ldots, \tthtsp_{d-1}).
\end{equation}
Let $\{\tap_i\}^{d-1}_{i=0}, \{\tbp_i\}^{d-2}_{i=0}, \{\tcp_i\}^{d-1}_{i=1}$ denote the intersection numbers of the Leonard system $\wtPp$. By construction the matrix representing $A$ relative to $\{\tvp_i\}^{d-1}_{i=0}$ is
\begin{equation}\label{[A]_(upi)}
\left[
\begin{array}{ccccc}
\vspace{0.2cm}
\tap_0 & \tbp_0 & & & {\bf 0} \\
\tcp_1 & \tap_1 & \tbp_1 & \\
& \tcp_2 & \tap_2 & \ddots \\
\vspace{0.2cm}
& & \ddots & \ddots & \tbp_{d-2} \\
{\bf 0} & & & \tcp_{d-1} & \tap_{d-1}
\end{array}
\right].
\end{equation}
For convenience define $\tbp_{-1}=0$ and $\tcp_{d}=0$. Recall the scalars $a_i, b_i, c_i$ from below Lemma \ref{IN(Ga,Phi)} and the  $\ta_i, \tb_i, \tc_i$ from (\ref{def(ta,tb,tc)}).

\begin{lemma}\label{6-terms;MC}
With the above notation the following {\rm(i)}--{\rm(vi)} hold.
\begin{enumerate}
\item[\rm(i)] $\zeta_{i-1}\tau_{i-1}\tbp_{i-1} = \zeta_i\tau_i\tb_{i-1}$ \qquad \qquad $(1 \leq i \leq d-1)$,
\item[\rm(ii)] $\zeta_{i-1}\tbp_{i-1} = \zeta_i\tau_i(\tb_{i-1}-b_i)+\zeta_ib_i$ \qquad \qquad $(1 \leq i \leq d-1)$,
\item[\rm(iii)] $\tap_i\tau_i = \tau_i(\ta_i-b_i+\tb_i)+(b_i-\tb_i)$ \qquad \qquad $(0 \leq i \leq d-1)$,
\item[\rm(iv)] $\tap_i = \tau_i(c_{i+1}-\tc_i)+\ta_i-c_{i+1}+\tc_i$ \qquad \qquad $(0 \leq i \leq d-1)$,
\item[\rm(v)] $\zeta_{i+1}\tau_{i+1}\tcp_{i+1} = \zeta_i\tau_ic_{i+1} + \zeta_i(\tc_{i+1}-c_{i+1})$ \qquad \qquad $(0 \leq i \leq d-2)$,
\item[\rm(vi)] $\zeta_{i+1}\tcp_{i+1} = \zeta_i\tc_{i+1}$ \qquad \qquad $(0 \leq i \leq d-2)$.
\end{enumerate}
\end{lemma}

\begin{proof}
For $0 \leq i \leq d-1$ we evaluate $A\tvp_i$ in two ways. First, using (\ref{[A]_(upi)}),
$$
A\tvp_i = \tbp_{i-1}\tvp_{i-1} + \tap_i\tvp_i + \tcp_{i+1}\tvp_{i+1}.
$$
In this equation evaluate the right-hand side using (\ref{tilde(vp_i)}). Secondly, in $A\tvp_i$ eliminate $\tvp_i$ using (\ref{tilde(vp_i)}) and evaluate the result using Lemma \ref{action(A;W)}. We have evaluated $A\tvp_i$ in two ways. Compare the results using the linear independence of $\{\hat{C}^{\pm}_j\}^{d-1}_{j=0}$. The result follows.
\end{proof}

\begin{corollary}\label{int_num(wtPp)}
Referring to Lemma {\rm\ref{6-terms;MC}}, the following holds.
\begin{enumerate}
\item[\rm(i)] $\tbp_i = \frac{\zeta_{i+1}\tau_{i+1}}{\zeta_{i}\tau_i}\tb_i$ \qquad \qquad $(0 \leq i \leq d-2)$.
\item[\rm(ii)] $\tap_i = b_0-b_i-c_{i+1}$ \qquad \qquad $(0\leq i \leq d-1)$.
\item[\rm(iii)] $\tcp_i = \frac{\zeta_{i-1}}{\zeta_i}\tc_i$ \qquad \qquad $(1 \leq i \leq d-1)$.
\end{enumerate}
\end{corollary}
\begin{proof}
(i) Follows from Lemma \ref{6-terms;MC}(i).\\
(ii) In the equation from Lemma \ref{6-terms;MC}(iv), eliminate $\tau_i$ using Corollary \ref{|edges|}(iii) along with (\ref{tau_i}) and simplify the result using $\ta_i+\tb_i+\tc_i = b_0$.\\
(iii) Follows from Lemma \ref{6-terms;MC}(vi).
\end{proof}

\noindent
Our next goal is to show that the Leonard system $\wtPp$ has  $q$-Racah type. To this end we compare the parameter arrays $p(\Phi)$ and $p(\wtPp)$. Recall the scalars $h, h^*, s, s^*, r_1, r_2$ from above Note \ref{h,h*}. Recall the parameter array $p(\wtPp)$ from (\ref{PA;til(Phi)perp}).

\begin{theorem}\label{q-rac(C-perp)}
With the above notation, for $0 \leq i \leq d-1$,
\begin{eqnarray}
\label{tthtperp}
\tthtp_i 	& = & \tthtp_0+\thp(1-q^{i})(1-\tsp q^{i+1})q^{-i}, \\
\label{ttht*perp}
\tthtsp_i 	& = & \tthtsp_0+\thsp(1-q^{i})(1-\tssp q^{i+1})q^{-i},
\end{eqnarray}
and for $1 \leq i \leq d-1$,
\begin{eqnarray}
\label{tvarphip_i}
\tvarphip_i	& = & \thp\thsp q^{1-2i}(1-q^{i})(1-q^{i-d})(1-\trp_1q^i)(1-\trp_2q^i), \\
\label{tphip_i}
\tphip_i 		& = & \thp\thsp q^{1-2i}(1-q^{i})(1-q^{i-d})(\trp_1-\tssp q^i)(\trp_2-\tssp q^i)/\tssp,
\end{eqnarray} 
where $\tthtp_0$, $\tthtsp_0$ are from {\rm(\ref{ttht=tthts;eq})} and
\begin{align}
\label{th,ts;perp(1)}
&\thp = hq^{-1}, && \tsp = sq^2, &&\trp_1=r_1q,&\\
\label{th,ts;perp(2)}
&\thsp = \tfrac{s^*q^{-1}-r_1r_2}{s^*-r_1r_2}h^*  ,&&  \tssp = s^*q, &&\trp_2=r_2q.
\end{align}
\end{theorem}
\begin{proof}
Recall that the Leonard system $\wtPp$ has diameter $d-1$. To get (\ref{tthtperp}), (\ref{ttht*perp}) use (\ref{theta_i}), (\ref{(tht)dist}) and Lemma \ref{ttht=tthts}. We now verify (\ref{tvarphip_i}), (\ref{tphip_i}). To this end, recall the intersection numbers $\{\tap_j\}^{d-1}_{j=0}$ of $\wtPp$. Using (\ref{Madrid(ai);eq1}), $\tvarphip_1=(\tap_0-\tthtp_0)(\tthtsp_0-\tthtsp_1)$. Evaluate this using Lemma \ref{ttht=tthts} and Corollary \ref{int_num(wtPp)}(ii) to get 
\begin{equation}\label{tvarphip_1}
\tvarphip_1=(c_1+\tht_1)(\ttht^*_1-\ttht^*_0).
\end{equation}
Evaluate the right-hand side of (\ref{tvarphip_1}) using (\ref{theta_i}), (\ref{c_i;q-terms}) and Lemma \ref{tht;til;dist}, and simplify the result using (\ref{th,ts;perp(1)}), (\ref{th,ts;perp(2)}). This yields (\ref{tvarphip_i}) for $i=1$. Using this together with (\ref{tthtperp}), (\ref{ttht*perp}) and the condition (PA4) of Theorem \ref{thm:LS<->PA} we obtain (\ref{tphip_i}) for $1 \leq i \leq d-1$. We now use (\ref{tphip_i}) at $i=1$ together with (\ref{tthtperp}), (\ref{ttht*perp}) and the condition (PA3) of Theorem \ref{thm:LS<->PA} to obtain (\ref{tvarphip_i}) for $1 \leq i \leq d-1$. The result follows.
\end{proof}

\begin{corollary}
The Leonard system $\wtPp$ has $q$-Racah type.
\end{corollary}
\begin{proof}
Compare Example \ref{q-rac;PA} and Theorem \ref{q-rac(C-perp)}.
\end{proof}

\begin{corollary} With the above notation,
\begin{align}
\label{tbp0;h,s,...}
& \tbp_0	
 =	\frac{h(1-q^{-d+1})(1-r_1q^{2})(1-r_2q^{2})}
		{q(1-s^*q^{3})}, &&  \\
\label{tbp;h,s,...}
&\tbp_i	
 =	\frac{h(1-q^{i-d+1})(1-s^*q^{i+2})(1-r_1q^{i+2})(1-r_2q^{i+2})}
		{q(1-s^*q^{2i+2})(1-s^*q^{2i+3})}  && &&(1\leq i \leq d-2), \\
\label{tcp;h,s,...}
&\tcp_i
 =	\frac{h(1-q^i)(1-s^*q^{i+d+1})(r_1-s^*q^i)(r_2-s^*q^i)}
		{s^*q^{d-1}(1-s^*q^{2i+1})(1-s^*q^{2i+2})} && &&  (1\leq i \leq d-2), \\
\label{tc_(d-1);h,s,...}
&\tcp_{d-1}
 = 	\frac{h(1-q^{d-1})(r_1-s^*q^{d-1})(r_2-s^*q^{d-1})}
		{s^*q^{d-1}(1-s^*q^{2d-1})}.	
\end{align}
\end{corollary}
\begin{proof}
Use (\ref{b_0;q-terms})--(\ref{c_d;q-terms}) and (\ref{th,ts;perp(1)}), (\ref{th,ts;perp(2)}).
\end{proof}

\noindent
We finish this section with some comments about the scalars $\{\zeta_i\}^{d-1}_{i=0}$ from (\ref{u=sum(SB)}).

\begin{lemma}\label{Lem(zeta_i)}
The vector $\wt{w}$ in line {\rm(\ref{u=sum(SB)})} can be chosen such that
\begin{equation}\label{zeta_i}
\qquad \zeta_i = q^{-i}(r_1-s^*q^{i+1})(r_2-s^*q^{i+1}) \qquad \qquad (0 \leq i \leq d-1).
\end{equation}
\end{lemma}
\begin{proof}
Observe that the vector $\wt{w}$ is defined up to multiplication by a nonzero scalar in $\mbb{C}$. Multiplying $\wt{w}$ by a nonzero scalar if necessary, we may assume that 
\begin{equation}\label{zeta_0}
\zeta_0 = (r_1-s^*q)(r_2-s^*q).
\end{equation}
Using Lemma \ref{6-terms;MC}(vi) and induction on $i$,
\begin{equation}\label{zeta_i;eq}
\qquad \zeta_i = \frac{\tc_1\tc_2\cdots\tc_i}{\tcp_1\tcp_2\cdots\tcp_i}\zeta_0 \qquad \qquad (0 \leq i \leq d-1).
\end{equation}
Evaluate (\ref{zeta_i;eq}) using (\ref{ga_i;q-term}), (\ref{ga_(d-1);q-term}), (\ref{tcp;h,s,...}), (\ref{tc_(d-1);h,s,...}) and (\ref{zeta_0}) to get (\ref{zeta_i}). The result follows.
\end{proof}

\noindent
From now on we assume that the vector $\wt{w}$ in line (\ref{u=sum(SB)}) has been chosen such that (\ref{zeta_i}) holds. Recall the $\wtPp$-standard basis $\{\tvp_i\}^{d-1}_{i=0}$ for $\MCp$ from (\ref{tilde(vp_i)}).

\begin{lemma} For $0 \leq i \leq d-1$,
$$
\tvp_i = q^{-i}s^*(1-r_1q^{i+1})(1-r_2q^{i+1})\hat{C}^-_i + q^{-i}(r_1-s^*q^{i+1})(r_2-s^*q^{i+1})\hat{C}^+_i.
$$
\end{lemma}
\begin{proof}
Evaluate (\ref{tilde(vp_i)}) using (\ref{tau;q-term;eq}) and (\ref{zeta_i}). 
\end{proof}


\medskip
\section{The maps $\mfrk{p}$ and $\wt{\mfrk{p}}$}\label{2maps}

Recall the subspace $\W$ from above line (\ref{Ci=(C-i)+(C+i)}). 
In this section we introduce two $\mbb{C}$-linear maps $\mfrk{p} : \W \to \W$ and $\wt{\mfrk{p}} : \W \to \W$. These maps are defined as follows. In (\ref{ODS(T-mods)}) and (\ref{ODS(TC-mods)}) we obtained the following direct sum decompositions of $\W$:
$$
\W = \Mx + \Mxp, \qquad \qquad \W=\MC + \MCp.
$$
Define a $\mbb{C}$-linear map $\mfrk{p} : \W \to \W$ such that $(\mfrk{p}-1)\Mx = 0$ and $\mfrk{p}(\Mxp)=0$. Thus $\mfrk{p}$ is the projection from $\W$ onto $\Mx$. Define a $\mbb{C}$-linear map $\wt{\mfrk{p}} : \W \to \W$ such that $(\wt{\mfrk{p}}-1)\MC = 0$ and $\wt{\mfrk{p}}(\MCp)=0$. Thus $\wt{\mfrk{p}}$ is the projection from $\W$ onto $\MC$. 

\begin{lemma}\label{Tp=pT}
For all $B \in T$ the action of $B$ on $\W$ commutes with $\mfrk{p}$. In particular, on $\W$ each of $A, A^*$ commutes with $\mfrk{p}$.
\end{lemma}
\begin{proof}
For $w \in \W$ we show $B\mfrk{p}w=\mfrk{p}Bw$. Write $w=u+v$, where $u \in \Mx$ and $v \in \Mxp$. Observe $\mfrk{p}u=u$ and $\mfrk{p}v = 0$, so $B\mfrk{p}w=Bu$. Each of $\Mx, \Mxp$ is a $T$-module, so $Bu \in \Mx$ and $Bv \in \Mxp$. Now $\mfrk{p}Bu = Bu$ and $\mfrk{p}Bv = 0$, so $\mfrk{p}Bw=Bu.$ The result follows.
\end{proof}

\begin{lemma}\label{tTtp=tptT}
For all ${B} \in \wt{T}$ the action of ${B}$ on $\W$ commutes with $\wt{\mfrk{p}}$. In particular, on $\W$ each of $A, \wt{A}^*$ commutes with $\wt{\mfrk{p}}$.
\end{lemma}
\begin{proof}
Similar to the proof of Lemma \ref{Tp=pT}.
\end{proof}

\noindent
Our next goal is to show how $\mfrk{p}$ acts on the basis $\{\hat{C}^{\pm}_i\}^{d-1}_{i=0}$ for $\W$. Recall from (\ref{Phi-SB}) the $\Phi$-standard basis $\{v_i\}^{d}_{i=0}$ for $\Mx$, and  from (\ref{vp_i}) the $\Phi^{\perp}$-standard basis $\{\vp_i\}^{d-2}_{i=0}$ for $\Mxp$. By construction,
\begin{equation}\label{p(v)=v;p(vp)=0}
\mfrk{p}v_i = v_i \quad (0 \leq i \leq d), \qquad \qquad \mfrk{p}\vp_i=0 \quad (0 \leq i \leq d-2).
\end{equation}

\begin{lemma}\label{C^pm=vi+vip} For $1 \leq i \leq d-1$,
\begin{align*}\label{eq;C^pm=vi+vip}
&\hat{C}^+_{i-1} = \tfrac{\epsilon_i}{\epsilon_i-1}v_i + \tfrac{1}{\xi_i(1-\epsilon_i)}\vp_{i-1}, \\
&\hat{C}^-_i = \tfrac{1}{1-\epsilon_i}v_i + \tfrac{1}{\xi_i(\epsilon_i-1)}\vp_{i-1}.
\end{align*}
Moreover
$$
\hat{C}^-_0 = v_0, \qquad \qquad \hat{C}^+_{d-1}=v_d.
$$
\end{lemma}
\begin{proof}
Use (\ref{Phi-SB}) and (\ref{vp_i}).
\end{proof}

\begin{lemma}\label{action(p)} The map $\mfrk{p}$ acts on $\{\hat{C}^{\pm}_{i}\}^{d-1}_{i=0}$ as follows. For $1 \leq i \leq d-1$,
\begin{equation*}
\mfrk{p}\hat{C}^-_i = \tfrac{1}{1-\epsilon_i}(\hat{C}^+_{i-1}+\hat{C}^-_i), \qquad \qquad 
\mfrk{p}\hat{C}^+_{i-1} = \tfrac{\epsilon_i}{\epsilon_i-1}(\hat{C}^+_{i-1}+\hat{C}^-_i). 
\end{equation*}
Moreover $\mfrk{p} \hat{C}^-_0 = \hat{C}^-_0$ and $\mfrk{p}\hat{C}^+_{d-1}=\hat{C}^+_{d-1}$.
\end{lemma}
\begin{proof}
By  (\ref{p(v)=v;p(vp)=0}) and Lemma \ref{C^pm=vi+vip}.
\end{proof}

\noindent
We now show how $\wt{\mfrk{p}}$ acts on the basis $\{\hat{C}^{\pm}_{i}\}^{d-1}_{i=0}$ for $\W$. Recall from (\ref{SB-MC}) the $\wt{\Phi}$-standard basis $\{\tv_i\}^{d-1}_{i=0}$ for $\MC$, and from (\ref{tilde(vp_i)}) the $\wtPp$-standard basis $\{\tvp_i\}^{d-1}_{i=0}$ for $\MCp$. By construction, 
\begin{equation}\label{tp(tv)=tv;tp(tvp)=0}
\wt{\mfrk{p}}\tv_i = \tv_i, \quad \qquad \qquad  \wt{\mfrk{p}}\tvp_i = 0 \qquad \qquad \quad(0 \leq i \leq d-1).
\end{equation}

\begin{lemma}\label{C^pm=tvi+tvip} For $0 \leq i \leq d-1$,
\begin{align*}\label{eq;C^pm=tvi+tvip}
&\hat{C}^-_i = \tfrac{1}{1-\tau_i}\tv_i + \tfrac{1}{\zeta_i(\tau_i-1)}\tvp_i, \\
&\hat{C}^+_i = \tfrac{\tau_i}{\tau_i-1}\tv_i + \tfrac{1}{\zeta_i(1-\tau_i)}\tvp_i.
\end{align*}
\end{lemma}
\begin{proof}
Use (\ref{SB-MC}) and (\ref{tilde(vp_i)}).
\end{proof}

\begin{lemma}\label{action(tp)} The map $\wt{\mfrk{p}}$ acts on $\{\hat{C}^{\pm}_{i}\}^{d-1}_{i=0}$ as follows. For $0 \leq i \leq d-1$,
\begin{equation*}
\wt{\mfrk{p}}\hat{C}^-_i = \tfrac{1}{1-\tau_i}(\hat{C}^-_i + \hat{C}^+_i), \qquad \qquad
\wt{\mfrk{p}}\hat{C}^+_i= \tfrac{\tau_i}{\tau_i-1}(\hat{C}^-_i + \hat{C}^+_i).
\end{equation*}
\end{lemma}
\begin{proof}
By (\ref{tp(tv)=tv;tp(tvp)=0}) and Lemma \ref{C^pm=tvi+tvip}.
\end{proof}


\medskip
\section{Five bases and five linear maps for $\W$}\label{matrices}
Recall the subspace $\W$ from above line (\ref{Ci=(C-i)+(C+i)}). In this section we display five bases for $\W$. We then display the transition matrices between certain pairs of bases among the five. We also display the matrix representations of  $A, A^*, \wt{A}^*, \mfrk{p}, \wt{\mfrk{p}}$ relative to these five bases. 

\medskip \noindent
We now define our five bases for $\W$. Recall the vectors $\{\hat{C}^{\pm}_i\}^{d-1}_{i=0}$  from Lemma \ref{OGB(W)}. The first basis for $\W$ is
\begin{equation}\label{basis(C)}
\mcal{C} := \{\hat{C}^-_0, \hat{C}^+_0, \hat{C}^-_1, \hat{C}^+_1,
\hat{C}^-_2, \hat{C}^+_2, \ldots, \hat{C}^-_{d-1}, \hat{C}^+_{d-1}\}.
\end{equation}
In Section \ref{W-T(x)} we saw the $\Phi$-standard basis $\{v_i\}^d_{i=0}$ for $\Mx$ and the $\Phi^{\perp}$-standard basis $\{\vp_i\}^{d-2}_{i=0}$ for $\Mxp$. The second basis for $\W$ is
\begin{equation}\label{basis(B)}
\mcal{B} := \{v_0, v_1, v_2, \ldots, v_d, \vp_0, \vp_1, \vp_2, \ldots, \vp_{d-2} \}.
\end{equation}
The third basis for $\W$ is 
\begin{equation}\label{basis(Ba)}
\mcal{B}_{alt} := \{v_0, v_1, \vp_0, v_2, \vp_1, v_3, \vp_2, \ldots, v_{d-1}, \vp_{d-2}, v_d \}.
\end{equation}
In Section \ref{W-T(C)} we saw the $\wt{\Phi}$-standard basis $\{\wt{v}_i\}^{d-1}_{i=0}$ for $\MC$ and the $\wtPp$-standard basis $\{\tvp_i\}^{d-1}_{i=0}$ for $\MCp$. The fourth basis for $\W$ is
\begin{equation}\label{basis(tB)}
\wt{\mcal{B}} := \{\tv_0, \tv_1, \tv_2, \ldots, \tv_{d-1}, \tvp_0, \tvp_1, \tvp_2, \ldots, \tvp_{d-1} \}.
\end{equation}
The fifth basis for $\W$ is
\begin{equation}\label{basis(tBa)}
\wt{\mcal{B}}_{alt} := \{ \tv_0, \tvp_0, \tv_1, \tvp_1, \tv_2, \tvp_2, \ldots, \tv_{d-1}, \tvp_{d-1} \}.
\end{equation}
We now describe the transition matrices between certain pairs of bases among the five. A pair of bases will be considered whenever they are adjacent in the following diagram:
\begin{equation}\label{diagram}
\mcal{B} \ \  -\!\!\!-\!\!\!-\!\!\!- \ \ \mcal{B}_{alt}
\ \ -\!\!\!-\!\!\!-\!\!\!- \ \ \mcal{C}
\ \  -\!\!\!-\!\!\!-\!\!\!- \ \ \wt{\mcal{B}}_{alt}
\ \ -\!\!\!-\!\!\!-\!\!\!- \ \ \wt{\mcal{B}}
\end{equation}

\begin{lemma}\label{TM(C<=>Bxa)}
The transition matrix from $\mcal{C}$ to $\mcal{B}_{alt}$ is 
\begin{equation}\label{tran(C=>Bxa)}
\bdiag\big[
{\bf T}_0, {\bf T}_1, {\bf T}_2 \dotsm, {\bf T}_{d-1}, {\bf T}_d
\big],
\end{equation}
where 
$$
{\bf T}_0 = [1], \qquad \qquad {\bf T}_d= [1],
$$ 
and for $1 \leq i \leq d-1$, 
$${\bf T}_i =
\begin{bmatrix}
1 & \xi_i \\
1 & \xi_i \epsilon_i
\end{bmatrix}, $$where
$$\xi_i 	= q^{1-i}(1-q^{i-d})(1-s^*q^{i+1}), \qquad
\xi_i\epsilon_i = q^{1-i-d}(1-q^i)(1-s^*q^{i+d+1}).$$
The transition matrix from $\mcal{B}_{alt}$ to $\mcal{C}$ is 
\begin{equation}\label{tran(Bxa=>C)}
\bdiag \big[
{\bf S}_0, {\bf S}_1, {\bf S}_2, \dotsm, {\bf S}_{d-1}, {\bf S}_d
\big],
\end{equation}
where
$$ {\bf S}_0 = [1], \qquad \qquad {\bf S}_d = [1], $$ 
and for $1 \leq i \leq d-1$,
$$
{\bf S}_i =
\begin{bmatrix}
\frac{\epsilon_i}{\epsilon_i-1} 			& \frac{1}{1-\epsilon_i} \\
\frac{1}{\xi_i(1-\epsilon_i)} 	& \frac{1}{\xi_i(\epsilon_i-1)}
\end{bmatrix},
$$where
\begin{align*}
&\tfrac{\epsilon_i}{\epsilon_i-1}	= 	\tfrac{(1-q^{i})(1-s^*q^{i+d+1})} 
								{(1-q^d)(1-s^*q^{2i+1})}, & 
&\tfrac{1}{1-\epsilon_i}	= 	\tfrac{q^d(1-q^{i-d})(1-s^*q^{i+1})}
								{(q^d-1)(1-s^*q^{2i+1})}, & \\
&\tfrac{1}{\xi_i(\epsilon_i-1)} = \tfrac{q^{d+i-1}}{(1-q^d)(1-s^*q^{2i+1})},  & 
&\tfrac{1}{\xi_i(1-\epsilon_i)}	 = 	\tfrac{q^{d+i-1}}{(q^d-1)(1-s^*q^{2i+1})}. &
\end{align*}
\end{lemma}
\begin{proof}
Use (\ref{Phi-SB}) and (\ref{vp_i}) to obtain the matrix (\ref{tran(C=>Bxa)}). Use Lemma \ref{C^pm=vi+vip} to obtain the matrix (\ref{tran(Bxa=>C)}).
\end{proof}

\begin{lemma} The transition matrix from $\mcal{C}$ to $\wt{\mcal{B}}_{alt}$ is
\begin{equation}\label{tran(C=>BCa)}
\bdiag\big[
{\bf R}_0, {\bf R}_1, {\bf R}_2 , \dotsm, {\bf R}_{d-1}
\big],
\end{equation}
where for $0 \leq i \leq d-1$, 
$$
{\bf R}_i = 
\begin{bmatrix}
1	&	\zeta_i\tau_i\\
1	&	\zeta_i
\end{bmatrix},
$$ and
$$
\zeta_i\tau_i = q^{-i}s^*(1-r_1q^{i+1})(1-r_2q^{i+1}), \qquad
\zeta_i = q^{-i}(r_1-s^*q^{i+1})(r_2-s^*q^{i+1}).
$$
The transition matrix from $\wt{\mcal{B}}_{alt}$ to $\mcal{C}$ is
\begin{equation}\label{tran(BCa=>C)}
\bdiag\big[
{\bf Q}_0, {\bf Q}_1, {\bf Q}_2,  \dotsm, {\bf Q}_{d-1}
\big],
\end{equation}
where for $0 \leq i \leq d-1$,
$$
{\bf Q}_i =
\begin{bmatrix}
\frac{1}{1-\tau_i}	&	\frac{\tau_i}{\tau_i-1}	\\
\frac{1}{\zeta_i(\tau_i-1)}	&	\frac{1}{\zeta_i(1-\tau_i)}
\end{bmatrix},
$$ and
\begin{align*}
&\tfrac{1}{1-\tau_i} = \tfrac{(r_1-s^*q^{i+1})(r_2-s^*q^{i+1})}{(r_1r_2-s^*)(1-s^*q^{2i+2})}, &
&\tfrac{\tau_i}{\tau_i-1} = 
		\tfrac{s^*(1-r_1q^{i+1})(1-r_2q^{i+1})}{(s^*-r_1r_2)(1-s^*q^{2i+2})}, & \\
&\tfrac{1}{\zeta_i(\tau_i-1)} = \tfrac{q^i}{(s^*-r_1r_2)(1-s^*q^{2i+2})}, &
&\tfrac{1}{\zeta_i(1-\tau_i)} = 
		\tfrac{q^i}{(r_1r_2-s^*)(1-s^*q^{2i+2})}. &
\end{align*}

\end{lemma}

\begin{proof}
Use (\ref{SB-MC}) and (\ref{tilde(vp_i)}) to obtain the matrix (\ref{tran(C=>BCa)}). Use Lemma \ref{C^pm=tvi+tvip} to obtain the matrix (\ref{tran(BCa=>C)}).
\end{proof}

\noindent
For $0 \leq j \leq d-1$ let ${\bf e}_j$ denote the column vector with rows indexed by $0, 1, 2, \ldots, 2d-1$, and with a $1$ in the $j$-th coordinate and $0$ in every other coordinate.

\begin{lemma}\label{TM:B<->Balt}
The transition matrix from $\mcal{B}$ to $\mcal{B}_{alt}$ is the permutation matrix 
\begin{equation*}\label{tm:Bx->Bxa}
\big[{\bf e}_0, {\bf e}_1, {\bf e}_{d+1}, {\bf e}_2, {\bf e}_{d+2}, {\bf e}_3, {\bf e}_{d+3}, \ldots, 
{\bf e}_{d-1}, {\bf e}_{2d-1}, {\bf e}_{d}\big].
\end{equation*}
The transition matrix from $\mcal{B}_{alt}$ to $\mcal{B}$ is the permutation matrix
\begin{equation*}\label{tm:Bxa->Bx}
\big[
{\bf e}_0, {\bf e}_1, {\bf e}_3, {\bf e}_5, {\bf e}_7, \ldots, {\bf e}_{2d-1}, 
			{\bf e}_2, {\bf e}_4, {\bf e}_6, \ldots, {\bf e}_{2d-2}
\big].
\end{equation*}
\end{lemma}
\begin{proof}
Compare (\ref{basis(B)}) and (\ref{basis(Ba)}).
\end{proof}

\begin{lemma}\label{TM(BC<=>BCa)}
The transition matrix from $\wt{\mcal{B}}$ to $\wt{\mcal{B}}_{alt}$ is the permutation matrix
\begin{equation*}\label{tm:BC->BCa}
\big[
{\bf e}_0, {\bf e}_d, {\bf e}_1, {\bf e}_{d+1}, {\bf e}_2, {\bf e}_{d+2}, {\bf e}_3, {\bf e}_{d+3},
\ldots, {\bf e}_{d-1}, {\bf e}_{2d-1}
\big].
\end{equation*}
The transition matrix from $\wt{\mcal{B}}_{alt}$ to $\wt{\mcal{B}}$ is the permutation matrix
\begin{equation*}\label{tm:BCa->BC}
\big[
{\bf e}_0, {\bf e}_2, {\bf e}_4, {\bf e}_6, \ldots, {\bf e}_{2d-2}, 
{\bf e}_1, {\bf e}_3, {\bf e}_5, {\bf e}_7, \ldots, {\bf e}_{2d-1}
\big].
\end{equation*}
\end{lemma}
\begin{proof}
Compare (\ref{basis(tB)}) and (\ref{basis(tBa)}).
\end{proof}

\noindent
We have now found the transition matrices between every adjacent pair of bases in (\ref{diagram}). For any pair of bases in (\ref{diagram}) the transition matrix can be computed from the above lemmas and linear algebra. For example, the transition matrix from $\mcal{B}_{alt}$ to $\wt{\mcal{B}}_{alt}$ is the product of the transition matrix from $\mcal{B}_{alt}$ to $\mathcal{C}$ given by (\ref{tran(Bxa=>C)}), and the transition matrix from $\mathcal{C}$ to $\wt{\mcal{B}}_{alt}$ given by (\ref{tran(C=>BCa)}). 

\medskip \noindent
We now display the matrices that represent 
$$
A, \qquad A^*, \qquad \wt{A}^*, \qquad \mfrk{p}, \qquad \wt{\mfrk{p}}
$$
relative to the five bases (\ref{diagram}).
We first consider the matrices representing $A$ relative to the five bases.

\begin{lemma} The matrix representing $A$ relative to the basis $\mcal{C}$ is  block tridiagonal:
\begin{equation}\label{[A]_C}
\left[{
\begin{array}{ccccc}
\vspace{0.15cm}
{\bf A}_0	& {\bf B}_0	&	&	& {\bf 0}\\
{\bf C}_1	& {\bf A}_1	& {\bf B}_1	&	& \\
		& {\bf C}_2	& {\bf A}_2 &	\ddots & \\
		&	& \ddots	& \ddots	& {\bf B}_{d-2} \vspace{0.15cm} \\
{\bf 0}	&	&	& {\bf C}_{d-1}	& {\bf A}_{d-1} 
\end{array}
}\right],
\end{equation}
where
\begin{align*}
&{\bf B}_{i}=
\begin{bmatrix}
\tb_{i}			& 0 \\
\tb_{i}-b_{i+1}	& b_{i+1}
\end{bmatrix}
&& (0 \leq i \leq d-2), &&\\
&{\bf A}_i = 
\begin{bmatrix}
\ta_i - b_i + \tb_i	& b_i-\tb_i \\
c_{i+1}-\tc_{i}	& \ta_i-c_{i+1}+\tc_{i}
\end{bmatrix}
&& (0 \leq i \leq d-1),&&\\
&{\bf C}_{i} = 
\begin{bmatrix}
c_{i}	& \tc_{i} - c_{i} \\
0	& \tc_{i}
\end{bmatrix}
&& (1 \leq i \leq d-1).&&
\end{align*}
\end{lemma}
\begin{proof}
Use Lemma \ref{action(A;W)}.
\end{proof}

\noindent
Note that the matrix (\ref{[A]_C}) is five-diagonal.

\begin{lemma}\label{lem:[A]_B}
The matrix representing $A$ relative to the basis $\mcal{B}$ is
\begin{equation*}\label{[A]_Bx}
\left[
\begin{array}{cc}
{\bf M_0} 	& {\bf 0} \\
 {\bf 0}			& {\bf M_1}
\end{array}\right],
\end{equation*}
where
\[
{\bf M_0} =
\begin{bmatrix}
\vspace{0.15cm}
a_0	& b_0 &&& \\
c_1	& a_1& b_1	&\\
	& c_2 & a_2 & \ddots	&\\ 
	& & \ddots & \ddots & b_{d-1} \vspace{0.15cm}\\
	& &	& c_d & a_d \\
\end{bmatrix},
\qquad  \qquad
{\bf M_1} = 
\begin{bmatrix}
\vspace{0.15cm}
\ap_0 & \bp_0 &&& \\
\cp_1 & \ap_1& \bp_1	&\\
	& \cp_2 & \ap_2 & \ddots	&\\ 
	& & \ddots & \ddots & \bp_{d-3} \vspace{0.15cm}\\
	& &	& \cp_{d-2} & \ap_{d-2} \\
\end{bmatrix}.
\]
\end{lemma}
\begin{proof}
By (\ref{[A^(flat)]}), (\ref{[A]_(vp)}) and (\ref{basis(B)}).
\end{proof}

\begin{lemma}
The matrix representing $A$ relative to the basis $\mcal{B}_{alt}$ is
\begin{equation*}\label{[A]_Bxa}
\left[
\begin{array}{ccccc}
\vspace{0.15cm}
{\bf x}_0 & {\bf y}_0 &&&{\bf 0}\\
{\bf z}_1 & {\bf x}_1 & {\bf y}_1 \\
&{\bf z}_2 & {\bf x}_2 & \ddots \\
&&\ddots & \ddots &{\bf y}_{d-1} \vspace{0.15cm}\\
{\bf 0}&&&{\bf z}_{d} & {\bf x}_{d}
\end{array}
\right],
\end{equation*}
where 
\begin{align*}
&{\bf y}_0 =\begin{bmatrix}b_0 & 0\end{bmatrix}, 
&&{\bf y}_{i} = \begin{bmatrix}
b_{i} & 0 \\
0 & \bp_{i-1}
\end{bmatrix} \qquad  (1 \leq i \leq d-2),
&&{\bf y}_{d-1}= \begin{bmatrix} b_{d-1} \\ 0 \end{bmatrix},&&\\
&{\bf x}_0 = \begin{bmatrix} a_0 \end{bmatrix},
&&{\bf x}_i = 
\begin{bmatrix}
a_{i} & 0 \\
0 & \ap_{i-1}
\end{bmatrix} \qquad (1 \leq i \leq d-1),
&&{\bf x}_d = \begin{bmatrix} a_d \end{bmatrix},&&\\
&{\bf z}_1 = \begin{bmatrix} c_1 \\ 0 \end{bmatrix},
&&{\bf z}_{i} = 
\begin{bmatrix}
c_{i} & 0 \\
0 & \cp_{i-1}
\end{bmatrix} \qquad (2 \leq i \leq d-1),
&&{\bf z}_d = \begin{bmatrix} c_d & 0\end{bmatrix}.
\end{align*}
\end{lemma}
\begin{proof}
Let $[A]_{\mcal{B}}$ denote the matrix representing $A$ relative to $\mcal{B}$, and let $L$ denote the transition matrix from $\mcal{B}_{alt}$ to $\mcal{B}$. By linear algebra the matrix representing $A$ relative to $\mcal{B}_{alt}$ is $L[A]_{\mcal{B}}L^{-1}$. Evaluate this matrix using Lemma \ref{TM:B<->Balt} and Lemma \ref{lem:[A]_B} to obtain the result.
\end{proof}

\begin{lemma}\label{lem:[A]_tB} The matrix representing $A$ relative to the basis $\wt{\mcal{B}}$ is
\begin{equation*}\label{[A]_BC}
\left[
\begin{array}{cc}
{\bf \wt{M}_0} 	& {\bf 0} \\
 {\bf 0}			& {\bf \wt{M}_1}
\end{array}\right],
\end{equation*}
where
\[
{\bf \wt{M}_0} =
\begin{bmatrix}
\vspace{0.15cm}
\ta_0 & \tb_0 &&& \\
\tc_1 & \ta_1& \tb_1	&\\
	& \tc_2 & \ta_2 & \ddots	&\\ 
	& & \ddots & \ddots & \tb_{d-2} \vspace{0.15cm}\\
	& &	& \tc_{d-1} & \ta_{d-1} \\
\end{bmatrix},
\qquad  \qquad
{\bf \wt{M}_1} = 
\begin{bmatrix}
\vspace{0.15cm}
\tap_0 & \tbp_0 &&& \\
\tcp_1 & \tap_1& \tbp_1	&\\
	& \tcp_2 & \tap_2 & \ddots	&\\ 
	& & \ddots & \ddots & \tbp_{d-2} \vspace{0.15cm}\\
	& &	& \tcp_{d-1} & \tap_{d-1} \\
\end{bmatrix}.
\]
\end{lemma}
\begin{proof}
By (\ref{[A]_(Ci)}), (\ref{[A]_(upi)}) and (\ref{basis(tB)}).
\end{proof}

\begin{lemma}\label{lem:[A]_tBalt}
The matrix representing $A$ relative to the basis $\wt{\mcal{B}}_{alt}$ is
\begin{equation*}\label{[A]_BCa}
\left[
\begin{array}{ccccc}
\vspace{0.15cm}
\wt{\bf x}_0 & \wt{\bf y}_0 &&&{\bf 0}\\
\wt{\bf z}_1 & \wt{\bf x}_1 & \wt{\bf y}_1 \\
& \wt{\bf z}_2 & \wt{\bf x}_2 & \ddots \\
&&\ddots & \ddots & \wt{\bf y}_{d-2} \vspace{0.15cm}\\
{\bf 0}&&& \wt{\bf z}_{d-1} & \wt{\bf x}_{d-1}
\end{array}
\right],
\end{equation*}
where 
\begin{align*}
& \wt{\bf y}_{i} = 
\begin{bmatrix}
\tb_{i} & 0 \\
0 & \tbp_{i}
\end{bmatrix} && (0 \leq i \leq d-2), && \\
& \wt{\bf x}_i = 
\begin{bmatrix}
\ta_{i} & 0 \\
0 & \tap_{i}
\end{bmatrix} && (0 \leq i \leq d-1), && \\
& \wt{\bf z}_{i} = 
\begin{bmatrix}
\tc_{i} & 0 \\
0 & \tcp_{i}
\end{bmatrix} && (1 \leq i \leq d-1). &&  
\end{align*}
\end{lemma}
\begin{proof}
Let $[A]_{\wt{\mcal{B}}}$ denote the matrix representing $A$ relative to $\wt{\mcal{B}}$, and let $L$ denote the transition matrix from $\wt{\mcal{B}}_{alt}$ to $\wt{\mcal{B}}$. By linear algebra the matrix representing $A$ relative to $\mcal{B}_{alt}$ is $L[A]_{\wt{\mcal{B}}}L^{-1}$. Evaluate this matrix using Lemma \ref{TM(BC<=>BCa)} and Lemma \ref{lem:[A]_tB} to obtain the result.
\end{proof} 

\noindent
We are done with $A$. We now consider the matrices representing $A^*$ relative to the five bases.

\begin{lemma}\label{[A*]_C}
The matrix representing $A^*$ relative to the basis $\mcal{C}$ is
$$
\diag(\tht^*_0, \tht^*_1, \tht^*_1, \tht^*_2, \tht^*_2, \ldots, \tht^*_{d-1}, \tht^*_{d-1}, \tht^*_d).
$$
\end{lemma}
\begin{proof}
From Corollary \ref{action(A*;W)} and (\ref{basis(C)}).
\end{proof}

\begin{lemma}
The matrix representing $A^*$ relative to the basis $\mcal{B}$ is
$$
\diag(\tht^*_0, \tht^*_1, \tht^*_2, \tht^*_3, \ldots, \tht^*_d, \tht^*_1, \tht^*_2, \ldots, \tht^*_{d-1}).
$$
\end{lemma}
\begin{proof}
Use (\ref{[A^(*flat)]}), (\ref{[A*]_(vp)}), (\ref{basis(B)}) along with (\ref{tht(i)=thtp(i);eq}).
\end{proof}

\begin{lemma} 
The matrix representing ${A}^*$ relative to the basis ${\mcal{B}}_{alt}$ is
$$
\diag(\tht^*_0, \tht^*_1, \tht^*_1, \tht^*_2, \tht^*_2, \ldots, \tht^*_{d-1}, \tht^*_{d-1}, \tht^*_d).
$$
\end{lemma}
\begin{proof}
Use (\ref{[A^(*flat)]}), (\ref{[A*]_(vp)}), (\ref{basis(Ba)}) along with (\ref{tht(i)=thtp(i);eq}).
\end{proof}

\begin{lemma}
The matrix representing $A^*$ relative to the basis $\wt{\mcal{B}}$ is
\begin{equation*}\label{[A*]_BC}
\left[
\begin{array}{cc}
{\bf C} & {\bf D} \\
{\bf E} & {\bf F}
\end{array}
\right],
\end{equation*} 
where 
\begin{eqnarray*}
{\bf C} & = & 
\diag \left[
\tfrac{\tht^*_0-\tau_0\tht^*_1}{1-\tau_0},
\tfrac{\tht^*_1-\tau_1\tht^*_2}{1-\tau_1},
\tfrac{\tht^*_2-\tau_2\tht^*_3}{1-\tau_2},
\dotsm, 
\tfrac{\tht^*_{d-1}-\tau_{d-1}\tht^*_{d}}{1-\tau_{d-1}}
\right],\\
{\bf D} & = &
\diag \left[
\tfrac{\zeta_0\tau_0(\tht^*_0-\tht^*_1)}{1-\tau_0},
\tfrac{\zeta_1\tau_1(\tht^*_1-\tht^*_2)}{1-\tau_1},
\tfrac{\zeta_2\tau_2(\tht^*_2-\tht^*_3)}{1-\tau_2},
\dotsm,
\tfrac{\zeta_{d-1}\tau_{d-1}(\tht^*_{d-1}-\tht^*_d)}{1-\tau_{d-1}}
\right], \\ 
{\bf E} & = &
\diag \left[
\tfrac{\tht^*_0-\tht^*_1}{\zeta_0(\tau_0-1)},
\tfrac{\tht^*_1-\tht^*_2}{\zeta_1(\tau_1-1)},
\tfrac{\tht^*_2-\tht^*_3}{\zeta_2(\tau_2-1)},
\dotsm,
\tfrac{\tht^*_{d-1}-\tht^*_d}{\zeta_{d-1}(\tau_{d-1}-1)}
\right], \\
{\bf F} & = &
\diag \left[
\tfrac{\tau_0\tht^*_0-\tht^*_1}{\tau_0-1},
\tfrac{\tau_1\tht^*_1-\tht^*_2}{\tau_1-1},
\tfrac{\tau_2\tht^*_2-\tht^*_3}{\tau_2-1},
\dotsm,
\tfrac{\tau_{d-1}\tht^*_{d-1}-\tht^*_d}{\tau_{d-1}-1}
\right],
\end{eqnarray*}
and for $0 \leq i \leq d-1$
\begin{equation*} 
\begin{array}{ccl}
\frac{\tht^*_i-\tau_i\tht^*_{i+1}}{1-\tau_i} &=&  
\tht^*_i+\frac{h^*s^*q^{-i-1}(1-q)(1-r_1q^{i+1})(1-r_2q^{i+1})}{s^*-r_1r_2},
\vspace{0.1cm}\\
\frac{\zeta_i\tau_i(\tht^*_i-\tht^*_{i+1})}{1-\tau_i}&=&
\frac{h^*s^*q^{-2i-1}(1-q)(r_1-s^*q^{i+1})(r_2-s^*q^{i+1})(1-r_1q^{i+1})(1-r_2q^{i+1})}{s^*-r_1r_2},
\vspace{0.1cm}\\ 
\frac{\tht^*_i-\tht^*_{i+1}}{\zeta_i(\tau_i-1)} &=&
\frac{h^*q^{-1}(q-1)}{s^*-r_1r_2}, \vspace{0.1cm}\\
\frac{\tau_i\tht^*_i-\tht^*_{i+1}}{\tau_i-1} &=&
\tht^*_{i+1}+\frac{h^*s^*q^{-i-1}(q-1)(1-r_1q^{i+1})(1-r_2q^{i+1})}{s^*-r_1r_2}.
\end{array}
\end{equation*}
\end{lemma}
\begin{proof}
Let $[A^*]_{\mcal{B}}$ denote the matrix representing $A^*$ relative to $\mcal{B}$, and let $L$ denote the transition matrix from $\wt{\mcal{B}}$ to $\mcal{B}$. By linear algebra the matrix representing $A^*$ relative to $\wt{\mcal{B}}$ is $L[A^*]_{\mcal{B}}L^{-1}$. The result follows.
\end{proof}

\begin{lemma}
The matrix representing $A^*$ relative to the basis $\wt{\mcal{B}}_{alt}$ is 
\begin{equation*}\label{[A*]_BCa}
\bdiag \big[
{\bf X}_0, {\bf X}_1, \dotsm, {\bf X}_{d-1}
\big],
\end{equation*}
where for $0 \leq i \leq d-1$,
$$
{\bf X}_i = 
\left[
\begin{array}{cc}
\frac{\tht^*_i-\tau_i\tht^*_{i+1}}{1-\tau_i} & 
					\frac{\zeta_i\tau_i(\tht^*_i-\tht^*_{i+1})}{1-\tau_i} \\
\frac{\tht^*_i-\tht^*_{i+1}}{\zeta_i(\tau_i-1)} &
					\frac{\tau_i\tht^*_i-\tht^*_{i+1}}{\tau_i-1}
\end{array}
\right].
$$ 
\end{lemma}

\begin{proof}
Let $[A^*]_{\mcal{B}_{alt}}$ denote the matrix representing $A^*$ relative to $\mcal{B}_{alt}$, and let $L$ denote the transition matrix from $\wt{\mcal{B}}_{alt}$ to $\mcal{B}_{alt}$. By linear algebra the matrix representing $A^*$ relative to $\wt{\mcal{B}}_{alt}$ is $L[A^*]_{\mcal{B}_{alt}}L^{-1}$. The result follows.
\end{proof}

\noindent
We are done with $A^*$.  We now consider the matrices representing $\wt{A}^*$ relative to the five bases.

\begin{lemma}\label{[A*C]_C}
The matrix representing $\wt{A}^*$ relative to the basis $\mcal{C}$ is
$$
\diag(\wt{\tht}^*_0, \wt{\tht}^*_0, \wt{\tht}^*_1, \wt{\tht}^*_1, \wt{\tht}^*_2, \wt{\tht}^*_2, \ldots, \wt{\tht}^*_{d-1}, \wt{\tht}^*_{d-1}).
$$
\end{lemma}
\begin{proof}
From Lemma \ref{action(A*C;W)} and (\ref{basis(C)}).
\end{proof}

\begin{lemma} The matrix representing $\wt{A}^*$ relative to the basis $\wt{\mcal{B}}$ is
$$
\diag(\wt{\tht}^*_0, \wt{\tht}^*_1, \wt{\tht}^*_2, \ldots , \wt{\tht}^*_{d-1}, \wt{\tht}^*_0, \wt{\tht}^*_1, \wt{\tht}^*_2, \ldots , \wt{\tht}^*_{d-1}).
$$
\end{lemma}
\begin{proof}
Use (\ref{[A^*(C)]_(C^)}), (\ref{[A^*(C)]_(tvp)}), (\ref{basis(tB)}) along with Lemma \ref{ttht=tthts}.
\end{proof}

\begin{lemma}
The matrix representing $\wt{A}^*$ relative to the basis $\wt{\mcal{B}}_{alt}$ is
$$
\diag(\wt{\tht}^*_0, \wt{\tht}^*_0, \wt{\tht}^*_1, \wt{\tht}^*_1, \wt{\tht}^*_2, \wt{\tht}^*_2, \ldots, \wt{\tht}^*_{d-1}, \wt{\tht}^*_{d-1}).
$$
\end{lemma}
\begin{proof}
Use (\ref{[A^*(C)]_(C^)}), (\ref{[A^*(C)]_(tvp)}), (\ref{basis(tBa)}) along with Lemma \ref{ttht=tthts}.
\end{proof}

\begin{lemma} The matrix representing $\wt{A}^*$ relative to the basis ${\mcal{B}}$ is
\begin{equation*}\label{[A*C]_Bx}
\left[
\begin{array}{cc}
\wt{\bf C} & \wt{\bf D} \\
\wt{\bf E} & \wt{\bf F}
\end{array}
\right],
\end{equation*}
where 
\[
\wt{\bf C}  =  \diag \left[ \wt{\tht}^*_0, 
\tfrac{\epsilon_1\wt{\tht}^*_0-\wt{\tht}^*_1}{\epsilon_1-1},
\tfrac{\epsilon_2\wt{\tht}^*_1-\wt{\tht}^*_2}{\epsilon_2-1},
\tfrac{\epsilon_3\wt{\tht}^*_2-\wt{\tht}^*_3}{\epsilon_3-1},
\dotsm,
\tfrac{\epsilon_{d-1}\wt{\tht}^*_{d-2}-\wt{\tht}^*_{d-1}}{\epsilon_{d-1}-1},
\wt{\tht}^*_{d-1}
\right],
\]
\[
\wt{\bf D} = \left[
\begin{array}{cccc}
0&0&\dotsm& 0\\
\hline
\frac{\epsilon_1\xi_1(\wt{\tht}^*_0-\wt{\tht}^*_1)}{\epsilon_1-1}&&\begin{matrix} & \\ & \end{matrix} &{\bf 0}\\
& \frac{\epsilon_2\xi_2(\wt{\tht}^*_1-\wt{\tht}^*_2)}{\epsilon_2-1} &\\
&&\ddots& \\
{\bf 0}&&\begin{matrix} & \\ & \end{matrix} &\frac{\epsilon_{d-1}\xi_{d-1}(\wt{\tht}^*_{d-2}-\wt{\tht}^*_{d-1})}{\epsilon_{d-1}-1}\\
\hline
0&0&\dotsm&0
\end{array}
\right],
\]
(in the above matrix the entries in the first row and the last row are all zero)  
\[
\wt{\bf E} = \left[
\begin{array}{c|cccc|c}
0& \frac{\wt{\tht}^*_0-\wt{\tht}^*_1}{\xi_1(1-\epsilon_1)} &  & &{\bf 0}&0\\
0& & \frac{\wt{\tht}^*_1-\wt{\tht}^*_2}{\xi_2(1-\epsilon_2)}&&&0\\
\vdots& & & \ddots & &\vdots\\
0& {\bf 0}& & &\frac{\wt{\tht}^*_{d-2}-\wt{\tht}^*_{d-1}}{\xi_{d-1}(1-\epsilon_{d-1})} &0
\end{array}
\right],
\]
(in the above matrix the entries in the first column and the last column are all zero)
\[
\wt{\bf F} =  \diag \left[
\tfrac{\wt{\tht}^*_0-\epsilon_1\wt{\tht}^*_1}{1-\epsilon_1},
\tfrac{\wt{\tht}^*_1-\epsilon_2\wt{\tht}^*_2}{1-\epsilon_2},
\tfrac{\wt{\tht}^*_2-\epsilon_3\wt{\tht}^*_3}{1-\epsilon_3},
\cdots,
\tfrac{\wt{\tht}^*_{d-2}-\epsilon_{d-1}\wt{\tht}^*_{d-1}}{1-\epsilon_{d-1}}
\right],
\]
and for $1\leq i\leq d-1$
\begin{equation*}
\begin{array}{ccl}
\frac{\epsilon_i\wt{\tht}^*_{i-1}-\wt{\tht}^*_i}{\epsilon_i-1} &=& 
	\wt{\tht}^*_{i-1} +
	\frac{\wt{h}^*(1-q)(1-q^{d-i})(1-s^*q^{i+1})}{1-q^d},\vspace{0.1cm}\\ 
\frac{\epsilon_i\xi_i(\wt{\tht}^*_{i-1}-\wt{\tht}^*_i)}{\epsilon_i-1}&=&
	\frac{\wt{h}^*q^{1-2i}(1-q)(1-q^i)(1-q^{i-d})(1-s^*q^{i+1})(1-s^*q^{d+i+1})}{q^d-1}
,\vspace{0.1cm}\\
\frac{\wt{\tht}^*_{i-1}-\wt{\tht}^*_i}{\xi_i(1-\epsilon_i)}&=&
	\frac{\wt{h}^*q^{d-1}(1-q)}{1-q^d},\vspace{0.1cm}\\
\frac{\wt{\tht}^*_{i-1}-\epsilon_i\wt{\tht}^*_i}{1-\epsilon_i}&=&
\wt{\tht}^*_{i}+
\frac{\wt{h}^*(1-q)(1-q^{d-i})(1-s^*q^{i+1})}{q^d-1}.
\end{array}
\end{equation*}
\end{lemma}
\begin{proof}
Let $[\wt{A}^*]_{\wt{\mcal{B}}}$ denote the matrix representing $\wt{A}^*$ relative to $\wt{\mcal{B}}$, and let $L$ denote the transition matrix from $\mcal{B}$ to $\wt{\mcal{B}}$. By linear algebra the matrix representing $\wt{A}^*$ relative to $\mcal{B}$ is $L[\wt{A}^*]_{\wt{\mcal{B}}}L^{-1}$. The result follows.
\end{proof}

\begin{lemma}
The matrix representing $\wt{A}^*$ relative to the basis ${\mcal{B}}_{alt}$ is 
\begin{equation*}\label{[A*C]_Bxa}
\bdiag\big[
\wt{\bf X}_0, \wt{\bf X}_1, \wt{\bf X}_2, \ldots, \wt{\bf X}_{d-1}, \wt{\bf X}_{d}
\big],
\end{equation*}
where 
$$
\wt{\bf X}_0=[\wt{\tht}^*_0],\qquad \qquad   \wt{\bf X}_{d} = [\wt{\tht}^*_{d-1}]
$$ 
and for $1 \leq i \leq d-1$
$$
\wt{\bf X}_i = \left[
\begin{array}{cc}
\frac{\epsilon_i\wt{\tht}^*_{i-1}-\wt{\tht}^*_i}{\epsilon_i-1} & \frac{\epsilon_i\xi_i(\wt{\tht}^*_{i-1}-\wt{\tht}^*_i)}{\epsilon_i-1}\\
\frac{\wt{\tht}^*_{i-1}-\wt{\tht}^*_i}{\xi_i(1-\epsilon_i)} & \frac{\wt{\tht}^*_{i-1}-\epsilon_i\wt{\tht}^*_i}{1-\epsilon_i}
\end{array}
\right].
$$
\end{lemma}
\begin{proof}
Let $[\wt{A}^*]_{\wt{\mcal{B}}_{alt}}$ denote the matrix representing $\wt{A}^*$ relative to $\wt{\mcal{B}}_{alt}$, and let $L$ denote the transition matrix from $\mcal{B}_{alt}$ to $\wt{\mcal{B}}_{alt}$. By linear algebra the matrix representing $\wt{A}^*$ relative to $\mcal{B}_{alt}$ is $L[\wt{A}^*]_{\wt{\mcal{B}}_{alt}}L^{-1}$. The result follows.
\end{proof}

\noindent
We are done with $\wt{A}^*$. We now consider the matrices representing $\mfrk{p}$ relative to the five bases.

\begin{lemma}
The matrix representing $\mfrk{p}$ relative to the basis $\mcal{C}$ is
\begin{equation}\label{[p_x]_C}
\bdiag\big[{\bf Y}_0, {\bf Y}_1, {\bf Y}_2, \ldots, {\bf Y}_{d-1}, {\bf Y}_d \big],
\end{equation}
where
$${\bf Y}_0 = [1], \qquad \qquad {\bf Y}_d= [1]$$ 
and for $1 \leq i \leq d-1$, 
$$
{\bf Y}_i = 
\left[
\begin{array}{cc}
\frac{\epsilon_i}{\epsilon_i-1} & \frac{1}{1-\epsilon_i} \\
\frac{\epsilon_i}{\epsilon_i-1} & \frac{1}{1-\epsilon_i}
\end{array}
\right].
$$
\end{lemma}
\begin{proof}
By Lemma \ref{action(p)}.
\end{proof}

\begin{lemma}
The matrix representing $\mfrk{p}$ relative to the basis $\mcal{B}$ is
\begin{equation*}\label{[p_x]_Bx}
\left[
\begin{array}{c|c}
{\bf I}_{d+1} & {\bf 0}\\
\hline
{\bf 0} & {\bf 0}
\end{array}
\right],
\end{equation*}
where ${\bf I}_{d+1}$ is the $(d+1)\times(d+1)$ identity matrix.
\end{lemma}
\begin{proof}
By (\ref{p(v)=v;p(vp)=0}) and (\ref{basis(B)}).
\end{proof}

\begin{lemma}
The matrix representing $\mfrk{p}$ relative to the basis $\mcal{B}_{alt}$ is
$$
\diag(1, 1,0, 1,0, 1,0, \ldots, 1,0, 1).
$$
\end{lemma}
\begin{proof}
By (\ref{p(v)=v;p(vp)=0}) and (\ref{basis(Ba)}).
\end{proof}

\begin{lemma}
The matrix representing $\mfrk{p}$ relative to the basis $\wt{\mcal{B}}$ is
\begin{equation*}\label{[pi_x]_BC}
\left[ 
\begin{array}{cc}
{\bf P} & {\bf Q} \\
{\bf R} & {\bf S}
\end{array}
\right],
\end{equation*}
where each matrix {\bf P}, {\bf Q}, {\bf R}, {\bf S} is $d \times d$ tridiagonal with rows and columns indexed by $0,1,2, \ldots, d-1$ such that
$$
{\bf P}_{00} = \tfrac{1-\epsilon_1(1-\tau_0)}{(1-\tau_0)(1-\epsilon_1)},
\qquad \qquad 
{\bf P}_{d-1,d-1} = \tfrac{1+\tau_{d-1}(\epsilon_{d-1}-1)}{(1-\tau_{d-1})(1-\epsilon_{d-1})},
$$
$$
{\bf P}_{ij} = 
\begin{cases}
\frac{-\epsilon_i}{(1-\tau_i)(1-\epsilon_i)} & \text{ if } j=i-1\quad (1\leq i\leq d-1) \\
\frac{1-\epsilon_{i+1}+\epsilon_{i+1}\tau_i(1-\epsilon_i)}{(1-\tau_i)(1-\epsilon_i)(1-\epsilon_{i+1})} &
		\text{ if } j=i \quad (1 \leq i \leq d-2) \\
\frac{-\tau_i}{(1-\tau_i)(1-\epsilon_{i+1})} & \text{ if } j=i+1 \quad (0 \leq i \leq d-2)
\end{cases}
$$
and
$$
{\bf Q}_{00} = \tfrac{\zeta_0\tau_0}{(1-\tau_0)(1-\epsilon_1)}, \qquad \qquad
{\bf Q}_{d-1,d-1}=\tfrac{\zeta_{d-1}\tau_{d-1}\epsilon_{d-1}}{(1-\tau_{d-1})(1-\epsilon_{d-1})},$$
$$
{\bf Q}_{ij} = 
\begin{cases}
\frac{-\zeta_{i-1}\epsilon_i}{(1-\tau_i)(1-\epsilon_i)} & \text{ if } j=i-1\quad(1 \leq i \leq d-1) \\
\frac{\zeta_i\tau_i(1-\epsilon_i\epsilon_{i+1})}{(1-\tau_i)(1-\epsilon_i)(1-\epsilon_{i+1})} &
		\text{ if } j=i \quad (1 \leq i \leq d-2)\\
\frac{-\tau_i\tau_{i+1}\zeta_{i+1}}{(1-\tau_i)(1-\epsilon_{i+1})} &
		\text{ if } j=i+1 \quad (0 \leq i \leq d-2)
\end{cases}
$$
and 
$$
{\bf R}_{00} = \tfrac{-1}{\zeta_0(1-\tau_0)(1-\epsilon_1)}, \qquad \qquad 
{\bf  R}_{d-1,d-1} = \tfrac{-\epsilon_{d-1}}{\zeta_{d-1}(1-\tau_{d-1})(1-\epsilon_{d-1})}, $$
$$
{\bf R}_{ij} = 
\begin{cases}
\frac{\epsilon_i}{\zeta_i(1-\tau_i)(1-\epsilon_i)} & \text{ if } j=i-1 \quad(1\leq i \leq d-1) \\
\frac{-1+\epsilon_i\epsilon_{i+1}}{\zeta_i(1-\tau_i)(1-\epsilon_i)(1-\epsilon_{i+1})} & 
		\text{ if } j=i \quad (1 \leq i \leq d-2) \\
\frac{1}{\zeta_i(1-\tau_i)(1-\epsilon_{i+1})} & \text{ if } j=i+1 \quad (0 \leq i \leq d-2)
\end{cases}
$$
and
$${\bf S}_{00}=\tfrac{\tau_0(\epsilon_1-1)-\epsilon_1}{(1-\tau_0)(1-\epsilon_1)}, \qquad \qquad
{\bf S}_{d-1,d-1} = \tfrac{1-\tau_{d-1}-\epsilon_{d-1}}{(1-\tau_{d-1})(1-\epsilon_{d-1})},$$ 
$$
{\bf S}_{ij} = 
\begin{cases}
\frac{\zeta_{i-1}\epsilon_i}{\zeta_i(1-\tau_i)(1-\epsilon_i)} & \text{ if } j=i-1 \quad (1 \leq i \leq d-1)\\
\frac{\tau_i(\epsilon_{i+1}-1)+\epsilon_{i+1}(\epsilon_i-1)}{(1-\tau_i)(1-\epsilon_i)(1-\epsilon_{i+1})} &
		\text{ if } j=i \quad(1 \leq i \leq d-2) \\
\frac{\zeta_{i+1}\tau_{i+1}}{\zeta_i(1-\tau_i)(1-\epsilon_{i+1})} &
		\text{ if } j=i+1 \quad (0 \leq i \leq d-2).
\end{cases}
$$
\end{lemma}
\begin{proof}
Let $[\mfrk{p}]_{\mcal{B}}$ denote the matrix representing $\mfrk{p}$ relative to $\mcal{B}$, and let $L$ denote the transition matrix from $\wt{\mcal{B}}$ to $\mcal{B}$. By linear algebra the matrix representing $\mfrk{p}$ relative to $\wt{\mcal{B}}$ is $L[\mfrk{p}]_{\mcal{B}}L^{-1}$. The result follows. 
\end{proof}

\begin{lemma}
The matrix representing $\mfrk{p}$ relative to the basis $\wt{\mcal{B}}_{alt}$ is
\begin{equation*}\label{[p_x]_BCa}
\left[
\begin{array}{ccccc}
{\bf a}_0 & {\bf b}_0 & & &{\bf 0} \vspace{0.15cm}\\
{\bf c}_1 & {\bf a}_1 & {\bf b}_1 & & \\
&{\bf c}_2 & {\bf a}_2 & \ddots \\
&&\ddots & \ddots & {\bf b}_{d-2} \vspace{0.15cm}\\
{\bf 0}&&&{\bf c}_{d-1} & {\bf a}_{d-1}
\end{array}
\right],
\end{equation*}
where 
\begin{align*}
 {\bf b}_{i} &= 
\left[
\begin{array}{cc}
\frac{-\tau_i}{(1-\tau_i)(1-\epsilon_{i+1})} &
\frac{-\tau_i\tau_{i+1}\zeta_{i+1}}{(1-\tau_i)(1-\epsilon_{i+1})}\vspace{0.1cm}\\
\frac{1}{\zeta_i(1-\tau_i)(1-\epsilon_{i+1})} &
\frac{\tau_{i+1}\zeta_{i+1}}{\zeta_i(1-\tau_i)(1-\epsilon_{i+1})}
\end{array}
\right] && (0 \leq i \leq d-2), && 
\end{align*}
\begin{align*}
{\bf a}_{0} & = 
\left[
\begin{array}{cc}
\frac{1-\epsilon_1(1-\tau_0)}{(1-\tau_0)(1-\epsilon_1)} &
\frac{\zeta_0\tau_0}{(1-\tau_0)(1-\epsilon_1)}\vspace{0.1cm}\\
\frac{-1}{\zeta_0(1-\tau_0)(1-\epsilon_1)}&
\frac{\tau_0(\epsilon_1-1)-\epsilon_1}{(1-\tau_0)(1-\epsilon_1)}
\end{array}
\right], && && \\
{\bf a}_i &= 
\left[
\begin{array}{cc}
\frac{1-\epsilon_{i+1}+\tau_i\epsilon_{i+1}(1-\epsilon_i)}{(1-\tau_i)(1-\epsilon_i)(1-\epsilon_{i+1})} &
\frac{\zeta_i\tau_i(1-\epsilon_i\epsilon_{i+1})}{(1-\tau_i)(1-\epsilon_i)(1-\epsilon_{i+1})}\vspace{0.1cm}\\
\frac{-1+\epsilon_i\epsilon_{i+1}}{\zeta_i(1-\tau_i)(1-\epsilon_i)(1-\epsilon_{i+1})} &
\frac{\tau_i(\epsilon_{i+1}-1)+\epsilon_{i+1}(\epsilon_i-1)}{(1-\tau_i)(1-\epsilon_i)(1-\epsilon_{i+1})}
\end{array}
\right] && (1 \leq i \leq d-2), &&\\
{\bf a}_{d-1} &= 
\left[
\begin{array}{cc}
\frac{1+\tau_{d-1}(\epsilon_{d-1}-1)}{(1-\tau_{d-1})(1-\epsilon_{d-1})} &
\frac{\zeta_{d-1}\tau_{d-1}\epsilon_{d-1}}{(1-\tau_{d-1})(1-\epsilon_{d-1})}\vspace{0.1cm}\\
\frac{-\epsilon_{d-1}}{\zeta_{d-1}(1-\tau_{d-1})(1-\epsilon_{d-1})}&
\frac{1-\tau_{d-1}-\epsilon_{d-1}}{(1-\tau_{d-1})(1-\epsilon_{d-1})}
\end{array}
\right], && && 
\end{align*}
\begin{align*}
{\bf c}_i &= 
\left[
\begin{array}{cc}
\frac{-\epsilon_i}{(1-\tau_i)(1-\epsilon_i)} & 
\frac{-\zeta_{i-1}\epsilon_i}{(1-\tau_i)(1-\epsilon_i)} \vspace{0.1cm} \\
\frac{\epsilon_i}{\zeta_i(1-\tau_i)(1-\epsilon_i)} &
\frac{\zeta_{i-1}\epsilon_i}{\zeta_i(1-\tau_i)(1-\epsilon_i)}
\end{array}
\right] 
&& (1 \leq i \leq d-1). &&
\end{align*}
\end{lemma}
\begin{proof}
Let $[\mfrk{p}]_{\mcal{B}_{alt}}$ denote the matrix representing $\mfrk{p}$ relative to $\mcal{B}_{alt}$, and let $L$ denote the transition matrix from $\wt{\mcal{B}}_{alt}$ to $\mcal{B}_{alt}$. By linear algebra the matrix representing $\mfrk{p}$ relative to $\wt{\mcal{B}}_{alt}$ is $L[\mfrk{p}]_{\mcal{B}_{alt}}L^{-1}$. The result follows. 
\end{proof}

\noindent
We are done with $\mfrk{p}$. We now consider the matrices representing $\wt{\mfrk{p}}$ relative to the five bases.

\begin{lemma}
The matrix representing $\wt{\mfrk{p}}$ relative to the basis ${\mcal{C}}$ is
\begin{equation}\label{[p_C]_C}
\bdiag\big[\wt{\bf Y}_0, \wt{\bf Y}_1, \wt{\bf Y}_2, \ldots, \wt{\bf Y}_{d-1}\big],
\end{equation}
where for $0 \leq i \leq d-1$,
$$
\wt{\bf Y}_i = 
\left[
\begin{array}{cc}
\frac{1}{1-\tau_i} & \frac{\tau_i}{\tau_i-1} \\
\frac{1}{1-\tau_i} & \frac{\tau_i}{\tau_i-1} 
\end{array}
\right].
$$
\end{lemma}
\begin{proof}
By Lemma \ref{action(tp)}.
\end{proof}

\begin{lemma}
The matrix representing $\wt{\mfrk{p}}$ relative to the basis $\wt{\mcal{B}}$ is
\begin{equation*}
\left[
\begin{array}{c|c}
{\bf I}_{d} & {\bf 0}\\
\hline
{\bf 0} & {\bf 0}
\end{array}
\right],
\end{equation*}
where ${\bf I}_{d}$ is the $d\times d$ identity matrix.
\end{lemma}
\begin{proof}
By (\ref{tp(tv)=tv;tp(tvp)=0}) and (\ref{basis(tB)}).
\end{proof}

\begin{lemma}
The matrix representing $\wt{\mfrk{p}}$ relative to the basis $\wt{\mcal{B}}_{alt}$ is
$$
\diag(1,0, 1,0, 1,0, \ldots, 1,0).
$$
\end{lemma}
\begin{proof}
By (\ref{tp(tv)=tv;tp(tvp)=0}) and (\ref{basis(tBa)}).
\end{proof}

\begin{lemma}
The matrix representing $\wt{\mfrk{p}}$ relative to the basis ${\mcal{B}}$ is
\begin{equation*}\label{[p_C]_Bx}
\left[
\begin{array}{cc}
\wt{\bf P} & \wt{\bf Q} \\
\wt{\bf R} & \wt{\bf S}
\end{array}
\right],
\end{equation*}
where $\wt{\bf P}, \wt{\bf Q}, \wt{\bf R}, \wt{\bf S}$ are described as follows. \\ The matrix $\wt{\bf P}$ is  $(d+1) \times (d+1)$ tridiagonal with rows and columns indexed by $0, 1, 2, \ldots, d$ such that
\[
\begin{array}{c}
\wt{\bf P}_{00} = \frac{1}{1-\tau_0}, \qquad 
\wt{\bf P}_{01} = \frac{\tau_0}{\tau_0-1}, \qquad
\wt{\bf P}_{d,d-1} = \frac{1}{1-\tau_{d-1}}, \qquad 
\wt{\bf P}_{d,d} = \frac{\tau_{d-1}}{\tau_{d-1}-1}, \vspace{0.2cm}\\

\wt{\bf P}_{ij} = 
\begin{cases}
\frac{-\epsilon_i}{(1-\tau_{i-1})(1-\epsilon_i)} & \text{ if } j=i-1 \quad (1 \leq i \leq d-1)\\
\frac{\epsilon_i\tau_{i-1}(\tau_i-1)+\tau_{i-1}-1}{(\tau_{i-1}-1)(\tau_i-1)(\epsilon_i-1)} &
			\text{ if } j=i \quad (1 \leq i \leq d-1)\\
\frac{-\tau_i}{(1-\epsilon_i)(1-\tau_i)} & \text{ if } j=i+1 \quad (1 \leq i \leq d-1).
\end{cases}\\
\end{array}
\] 
The matrix $\wt{\bf Q}$ is $(d+1) \times (d-1)$ with rows indexed by $0,1,2, \ldots, d$ and columns indexed by $0, 1, 2, \ldots , d-2$ such that
\[
\begin{array}{c}
\wt{\bf Q}_{00} = \frac{\tau_0\xi_1}{\tau_0-1}, \qquad \qquad 
\wt{\bf Q}_{d,d-2} = \frac{\xi_{d-1}\epsilon_{d-1}}{1-\tau_{d-1}}, \vspace{0.2cm}\\

\wt{\bf Q}_{ij} = 
\begin{cases}
\frac{-\xi_{i-1}\epsilon_{i-1}\epsilon_i}{(1-\tau_{i-1})(1-\epsilon_i)} & \text{ if } j = i-2 \quad(2 \leq i \leq d-1)\\
\frac{\xi_i\epsilon_i(1-\tau_{i-1}\tau_i)}{(1-\tau_{i-1})(1-\tau_i)(1-\epsilon_i)} & 
		\text{ if } j=i-1 \quad(1 \leq i \leq d-1)\\
\frac{-\tau_i\xi_{i+1}}{(1-\epsilon_i)(1-\tau_i)} & \text{ if } j=i \quad(1 \leq i \leq d-2)\\
0 &\text{ otherwise. } 
\end{cases}
\end{array}
\]
The matrix $\wt{\bf R}$ is $(d-1) \times (d+1)$ with rows indexed by $0,1,2, \ldots, d-2$ and columns indexed by $0, 1, 2, \ldots , d$ such that
\[
\wt{\bf R}_{ij} = 
\begin{cases}
\frac{1}{\xi_{i+1}(\epsilon_{i+1}-1)(\tau_{i}-1)} & \text{ if } j=i \quad(0 \leq i \leq d-2)\\
\frac{\tau_{i}\tau_{i+1}-1}{\xi_{i+1}(1-\epsilon_{i+1})(1-\tau_{i+1})(1-\tau_{i})} & \text{ if } j=i+1 \quad (0 \leq i \leq d-2)\\
\frac{\tau_{i+1}}{\xi_{i+1}(1-\epsilon_{i+1})(1-\tau_{i+1})} & \text{ if } j=i+2 \quad(0 \leq i \leq d-2)\\
0 & \text{otherwise}.
\end{cases}
\]
The matrix $\wt{\bf S}$ is $(d-1) \times (d-1)$ tridiagonal  with rows and columns indexed by $0, 1, 2, \ldots, d-2$ such that
\[
\wt{\bf S}_{ij} =
\begin{cases}
\frac{\xi_{i}\epsilon_{i}}{\xi_{i+1}(\epsilon_{i+1}-1)(\tau_{i}-1)} & \text{ if } j=i-1 \quad (1 \leq i \leq d-2)\\
\frac{\tau_{i}(\tau_{i+1}-1)+\epsilon_{i+1}(\tau_{i}-1)}{(1-\tau_{i})(1-\tau_{i+1})(1-\epsilon_{i+1})} &
						\text{ if } j=i \quad (0 \leq i \leq d-2) \\
\frac{\tau_{i+1}\xi_{i+2}}{\xi_{i+1}(1-\epsilon_{i+1})(1-\tau_{i+1})} & \text{ if } j=i+1 \quad (0\leq i \leq d-3).
\end{cases}
\]
\end{lemma}
\begin{proof}
Let $[\wt{\mfrk{p}}]_{\wt{\mcal{B}}}$ denote the matrix representing $\wt{\mfrk{p}}$ relative to $\wt{\mcal{B}}$, and let $L$ denote the transition matrix from $\mcal{B}$ to $\wt{\mcal{B}}$. By linear algebra the matrix representing $\wt{\mfrk{p}}$ relative to $\mcal{B}$ is $L[\wt{\mfrk{p}}]_{\wt{\mcal{B}}}L^{-1}$. The result follows.
\end{proof}

\begin{lemma}
The matrix representing $\wt{\mfrk{p}}$ relative to the basis ${\mcal{B}}_{alt}$ is
\begin{equation*}\label{[pi^C]_Bxa}
\left[
\begin{array}{ccccc}
\wt{\bf a}_0 & \wt{\bf b}_0 &&&{\bf 0} \vspace{0.15cm}\\
\wt{\bf c}_1 & \wt{\bf a}_1 & \wt{\bf b}_1 &&\\
&\wt{\bf c}_2 & \wt{\bf a}_2 & \ddots &\\
&& \ddots & \ddots & \wt{\bf b}_{d-1}\vspace{0.15cm} \\
{\bf 0}&&&\wt{\bf c}_{d} & \wt{\bf a}_{d}
\end{array}
\right],
\end{equation*}
where 
\begin{align*}
\wt{{\bf b}}_0 &= 
\left[
\begin{array}{cc}
\frac{\tau_0}{\tau_0-1} &
\frac{\tau_0\xi_1}{\tau_0-1}
\end{array}
\right], && && \\
\wt{\bf b}_{i} & = 
\left[
\begin{array}{c@{}c}
\frac{-\tau_i}{(1-\epsilon_i)(1-\tau_i)} &
\frac{-\tau_i\xi_{i+1}}{(1-\epsilon_i)(1-\tau_i)} \vspace{0.1cm}\\
\frac{\tau_i}{\xi_i(1-\epsilon_i)(1-\tau_i)} &
\frac{\tau_i\xi_{i+1}}{\xi_i(1-\epsilon_i)(1-\tau_i)}
\end{array}
\right] && {(1\leq i\leq d-2)},&&\\
\wt{\bf b}_{d-1} &=
\left[
\begin{array}{@{}c@{}}
\frac{-\tau_{d-1}}{(1-\epsilon_{d-1})(1-\tau_{d-1})} \vspace{0.1cm}\\
\frac{\tau_{d-1}}{\xi_{d-1}(1-\epsilon_{d-1})(1-\tau_{d-1})}
\end{array}
\right], && && 
\end{align*}
\begin{align*}
\wt{\bf a}_0 & = \left[\frac{1}{1-\tau_0}\right], && && \\
\wt{\bf a}_{i} &= 
\left[
\begin{array}{cc}
\frac{\epsilon_i\tau_{i-1}(\tau_i-1)+\tau_{i-1}-1}{(\tau_{i-1}-1)(\tau_i-1)(\epsilon_i-1)} &
\frac{\xi_i\epsilon_i(1-\tau_{i-1}\tau_i)}{(1-\tau_{i-1})(1-\tau_i)(1-\epsilon_i)}\vspace{0.1cm}\\
\frac{\tau_{i-1}\tau_i-1}{\xi_i(1-\epsilon_i)(1-\tau_i)(1-\tau_{i-1})} &
\frac{\tau_{i-1}(\tau_i-1)+\epsilon_i(\tau_{i-1}-1)}{(1-\epsilon_i)(1-\tau_i)(1-\tau_{i-1})}
\end{array}
\right] && {(1\leq i\leq d-1),}  && \\
\wt{\bf a}_{d} & = \left[\frac{\tau_{d-1}}{\tau_{d-1}-1}\right], && &&\\
\wt{\bf c}_1 & = 
\left[
\begin{array}{@{}c@{}}
\frac{-\epsilon_1}{(1-\epsilon_1)(1-\tau_0)}\vspace{0.1cm}\\
\frac{1}{\xi_1(1-\epsilon_1)(1-\tau_0)}
\end{array}
\right], && && \\
\wt{\bf c}_{i} &= 
\left[
\begin{array}{c@{}c}
\frac{-\epsilon_{i}}{(1-\tau_{i-1})(1-\epsilon_{i})} &
\frac{-\xi_{i-1}\epsilon_{i-1}\epsilon_i}{(1-\tau_{i-1})(1-\epsilon_i)}\vspace{0.1cm}\\
\frac{1}{\xi_i(1-\epsilon_i)(1-\tau_{i-1})} &
\frac{\xi_{i-1}\epsilon_{i-1}}{\xi_i(1-\epsilon_i)(1-\tau_{i-1})}
\end{array}
\right] && {(2\leq i \leq d-1)}, &&\\
\wt{\bf c}_{d} &= 
\left[
\begin{array}{@{}cc@{}}
\frac{1}{1-\tau_{d-1}} &
\frac{\xi_{d-1}\epsilon_{d-1}}{1-\tau_{d-1}}
\end{array}
\right]. && &&
\end{align*}
\end{lemma}
\begin{proof}
Let $[\wt{\mfrk{p}}]_{\wt{\mcal{B}}_{alt}}$ denote the matrix representing $\wt{\mfrk{p}}$ relative to $\wt{\mcal{B}}_{alt}$, and let $L$ denote the transition matrix from $\mcal{B}_{alt}$ to $\wt{\mcal{B}}_{alt}$. By linear algebra the matrix representing $\wt{\mfrk{p}}$ relative to ${\mcal{B}}_{alt}$ is $L[\wt{\mfrk{p}}]_{\wt{\mcal{B}}_{alt}}L^{-1}$. The result follows.
\end{proof}

\ \\

\bigskip


\noindent
\scalemath{0.87}{
\textbf{\LARGE Part II: The Universal Double Affine Hecke Algebra of Type $(C^{\vee}_1,C_1)$}}

\section{The Universal DAHA $\H$}\label{Hq}

The double affine Hecke algebra (DAHA) associated with the root system $(C^{\vee}_1, C_1)$ was discovered by Sahi\cite{Sahi}. This algebra has rank one. It involves $q$ and four additional nonzero parameters. It is known that the algebra controls the algebraic structure of the Askey-Wilson polynomials\cite{NS}. In \cite{Ter:AW+DAHA} Terwilliger introduced a central extension of that algebra, denoted $\H$ and called the universal DAHA of type $(C_1^{\vee},C_1)$. $\H$ has no parameter besides $q$ and has a larger automorphism group. Our goal for the rest of this paper is to display a representation of $\H$ using a Delsarte clique of a $Q$-polynomial distance-regular graph, and to show how this representation is related to the algebra $\bf T$ from Definition \ref{Algebra(T)}. 

\medskip \noindent
We now define the algebra $\H$. For notational convenience define a four element set
$$
\I = \{0, 1, 2, 3\}.
$$
For the rest of the paper we fix a square root $q^{1/2}$ of $q$.
\begin{definition}{\cite[Definition 3.1]{Ter:AW+DAHA}}\label{DAHA}
Let $\hat{H}_{q}$ denote the $\C$-algebra defined by 
generators $\{t^{\pm1}_n\}_{n \in \I}$
and  relations
\begin{align}
\label{daha1}&t_nt_n^{-1} = t^{-1}_nt_n = 1 \qquad n \in \I; \\
\label{daha2}&t_n + t^{-1}_{n} \text{ is central} \qquad n \in \I; \\
\label{daha3}&t_0t_1t_2t_3=q^{-1/2}. 
\end{align}
We call $\H$ the {\it universal DAHA of type $(C^{\vee}_1, C_1)$}.
We note that in \cite[Definition 3.1]{Ter:AW+DAHA}
Terwilliger defined $\H$ in a slightly different way. 
To go from his definition to ours, replace $q$ by $q^{{1}/{2}}$.
\end{definition}

\begin{definition}\label{X,Y}
We define elements $\bf X$ and $\bf Y$ in $\H$ as follows:
$$
{\bf X} = t_3t_0, \qquad \qquad {\bf Y} = t_0t_1.
$$
Note that each of $\bf X$, $\bf Y$ is invertible.
\end{definition}

\begin{definition}\label{A,B,B+}
We define elements ${\bf A}, {\bf B}, {\bf B^{\dagger}}$ in 
$\H$ as follows:
\begin{align*}
&{\bf A}  = {\bf Y} + {\bf Y}^{-1} = t_0t_1 + (t_0t_1)^{-1}, \\
&{\bf B}  = {\bf X} + {\bf X}^{-1} = t_3t_0 + (t_3t_0)^{-1}, \\
&{\bf B^{\dagger}}  = q^{1/2}{\bf X} + q^{-1/2}{\bf X}^{-1} = t_1t_2 + (t_1t_2)^{-1}.
\end{align*}
\end{definition}

\begin{note}\label{commuting}
By \cite[Lemma 3.8]{Ito-Ter} each of $\bf A$, $\bf B$ commutes with $t_0$, 
and each of $\bf A$, ${\bf B^{\dagger}}$ commutes with $t_1$.
Moreover $\bf B$, ${\bf B^{\dagger}}$ commute. 
\end{note}

\medskip
\section{An action of $\H$ on the $\T$-module $\W$}\label{H-mod(W)}

We now bring in the situation of Part I. Recall
the primary $\T$-module $\W$ from below Definition \ref{Algebra(T)}.
In this section we will show that $\W$ has the structure of a $\H$-module.
To this end we define some block diagonal matrices whose blocks are $1 \times 1$ or
$2 \times 2$. Recall the scalars $h, h^*,s, s^*, r_1, r_2$ from above Note \ref{h,h*}.
In what follows, we will encounter square roots. These are interpreted as follows. 
For the rest of the paper fix square roots
$$
s^{1/2}, \quad s^{*1/2}, \quad r_1^{1/2}, \quad r_2^{1/2}
$$
such that $r_1^{1/2}r_2^{1/2}=s^{1/2}s^{*1/2}q^{(d+1)/2}$.

\begin{definition} \label{(2x2)mat}
We define some matrices as follows:
\begin{enumerate}
\item[(a)] For $0 \leq i \leq d-1$ define
$$
\mfrk{t}_{0}(i) = \left[
\begin{array}{ccc}
\frac{1}{\sqrt{s^*r_1r_2}}
\left(
\frac{(r_1-s^*q^{i+1})(r_2-s^*q^{i+1})}{1-s^*q^{2i+2}} + s^*
\right)
& \ \quad  \ &
-\sqrt{\frac{s^*}{r_1r_2}} \frac{(1-r_1q^{i+1})(1-r_2q^{i+1})}{1-s^*q^{2i+2}} \\

\ \\

\frac{1}{\sqrt{s^*r_1r_2}}
\frac{(r_1-s^*q^{i+1})(r_2-s^*q^{i+1})}{1-s^*q^{2i+2}}
& \ \quad  \ &
\sqrt{\frac{s^*}{r_1r_2}}
\left(1-\frac{(1-r_1q^{i+1})(1-r_2q^{i+1})}{1-s^*q^{2i+2}}\right)
\end{array}
\right].
$$

\item[(b)] For $1 \leq i \leq d-1$ define
$$
\mfrk{t}_{1}(i) =\left[
\begin{array}{ccc}
\frac{q^{d/2}(1-q^{i-d})(1-s^*q^{i+1})}{1-s^*q^{2i+1}} + \frac{1}{q^{d/2}}
& \ \quad  \ &
\frac{q^{d/2}(q^{i-d}-1)(1-s^*q^{i+1})}{1-s^*q^{2i+1}}\\

\ \\

\frac{(1-q^{i})(1-s^*q^{d+i+1})}{q^{d/2}(1-s^*q^{2i+1})}
& \ \quad  \ &
\frac{(q^{i}-1)(1-s^*q^{d+i+1})}{q^{d/2}(1-s^*q^{2i+1})} +\frac{1}{q^{d/2}}
\end{array}
\right].
$$
Define  $$\mfrk{t}_{1}(0)=\left[\tfrac{1}{q^{d/2}}\right], \qquad \qquad
\mfrk{t}_{1}(d)=\left[\tfrac{1}{q^{d/2}}\right].$$

\item[(c)] For $1 \leq i \leq d-1$ define
$$
\mfrk{t}_{2}(i) =\left[
\begin{array}{ccc}
\frac{1}{\sqrt{s^*q^{d+1}}}
\left(
\frac{s^*q^{d+1}(1-q^{i})(1-q^{i-d})}{1-s^*q^{2i+1}}+1
\right)
& \ \quad  \ &
\frac{q^{i}\sqrt{s^*q^{d+1}}(1-q^{i-d})(1-s^*q^{i+1})}{1-s^*q^{2i+1}}
\\

\ \\

\frac{1}{q^{i}\sqrt{s^*q^{d+1}}}
\frac{(q^{i}-1)(1-s^*q^{d+i+1})}{1-s^*q^{2i+1}}
& \ \quad  \ &
\sqrt{s^*q^{d+1}} \left(
\frac{(q^{i}-1)(1-q^{i-d})}{1-s^*q^{2i+1}} + 1
\right)
\end{array}
\right].
$$
Define $$\mfrk{t}_{2}(0)=\left[\sqrt{s^*q^{d+1}}\right], \qquad \qquad
\mfrk{t}_{2}(d)=\left[\tfrac{1}{\sqrt{s^*q^{d+1}}}\right].$$

\item[(d)] For $0\leq i \leq d-1$ define
$$
\mfrk{t}_{3}(i) =\left[
\begin{array}{ccc}
\frac{1}{q^{i+1}\sqrt{r_1r_2}}
\left(
1-\frac{(1-r_1q^{i+1})(1-r_2q^{i+1})}{1-s^*q^{2i+2}}
\right)
& \ \quad  \ &
\frac{1}{q^{i+1}\sqrt{r_1r_2}}
\frac{(1-r_1q^{i+1})(1-r_2q^{i+1})}{1-s^*q^{2i+2}} \\

\ \\

-\frac{q^{i+1}}{\sqrt{r_1r_2}}
\frac{(r_1-s^*q^{i+1})(r_2-s^*q^{i+1})}{1-s^*q^{2i+2}}
& \ \quad  \ &
\frac{q^{i+1}}{\sqrt{r_1r_2}}
\left(
\frac{(r_1-s^*q^{i+1})(r_2-s^*q^{i+1})}{1-s^*q^{2i+2}} + s^*\right)
\end{array}
\right].
$$
\end{enumerate}
\end{definition}

\begin{definition}\label{k_0,1,2,3}
We define complex scalars $\{k_n\}_{n\in\I}$ as follows:
\begin{equation}\label{k0123}
k_0 = \sqrt{\tfrac{r_1r_2}{s^*}}, \qquad k_1 = \tfrac{1}{\sqrt{q^{d}}}, \qquad
k_2 = \sqrt{s^*q^{d+1}}, \qquad k_3 = \sqrt{\tfrac{r_2}{r_1}}.
\end{equation}
Observe that $k_0, k_1, k_2, k_3$ are nonzero.
\end{definition}

\noindent
We discuss some properties of the matrices from Definition \ref{(2x2)mat}.
We begin with the $2\times2$ matrices.

\medskip
\begin{lemma}\label{det/tr(t_i)}
For $n \in \I$ define $\varepsilon = 0$ if $n \in \{0,3\}$
and $\varepsilon = 1$ if $n \in \{1,2\}$. Referring to Definition {\rm\ref{(2x2)mat}}, 
for $\varepsilon \leq i \leq d-1$ both
\begin{enumerate}
\item[{\rm(i)}] $\det(\mfrk{t}_n(i)) = 1;$
\item[\rm(ii)] ${\rm{trace}}(\mfrk{t}_n(i)) = k_n + k^{-1}_n.$
\end{enumerate}
\end{lemma}
\begin{proof}
Routine.
\end{proof}

\begin{lemma}\label{inv(t_i)}
Referring to Lemma {\rm \ref{det/tr(t_i)}},
the matrix $\mfrk{t}_n(i)$ is invertible. Moreover
\begin{equation}\label{ti(j)+ti(j)}
\mfrk{t}_n(i) + (\mfrk{t}_n(i))^{-1} = 
(k_n + k_n^{-1})I.
\end{equation}
\end{lemma}

\begin{proof}
Use Lemma \ref{det/tr(t_i)}. 
\end{proof}

\begin{lemma}\label{ti(j)ti(j)}
With reference to Definition {\rm \ref{(2x2)mat}}
the following {\rm (i), (ii)} hold: 
\begin{enumerate}
\item[\rm(i)]
$ \mfrk{t}_3(i)\mfrk{t}_0(i) = 
\left[
\begin{array}{cc}
\frac{1}{q^{i+1}\sqrt{s^*}} & 0 \\
0 & {q^{i+1}\sqrt{s^*}}
\end{array}
\right] \qquad (0 \leq i \leq d-1).
$

\item[\rm(ii)]
$\mfrk{t}_1(i)\mfrk{t}_2(i) = 
\left[
\begin{array}{cc}
\frac{1}{q^{i}\sqrt{s^*q}} & 0 \\
0 & {q^{i}\sqrt{s^*q}}
\end{array}
\right] \qquad (1 \leq i \leq d-1).
$
\end{enumerate}
\end{lemma}
\begin{proof}
Routine.
\end{proof}

 \noindent
We have discussed the $2\times2$ matrices from Definition \ref{(2x2)mat}. 
We now consider the $1\times1$ matrices from Definition \ref{(2x2)mat}. 

\begin{lemma}\label{block(1X1)}
The following {\rm (i), (ii)} hold:
\begin{enumerate}
\item[\rm (i)] For $n \in \{1,2\}$ and $i\in \{0,d\}$,
\begin{equation*}
\mfrk{t}_n(i) + (\mfrk{t}_n(i))^{-1} = k_n + k_n^{-1}.
\end{equation*}

\item[\rm (ii)] Both
\begin{equation*}
\mfrk{t}_1(0)\mfrk{t}_2(0) = \sqrt{s^*q}, 
\qquad \qquad
\mfrk{t}_1(d)\mfrk{t}_2(d) = \frac{1}{q^d\sqrt{s^*q}}.\end{equation*}
\end{enumerate}
\end{lemma}
\begin{proof}
Immediate from Definition \ref{(2x2)mat}.
\end{proof}

\medskip
\begin{definition} \label{Mat_t_i}
With reference to Definition \ref{(2x2)mat}, for $n \in \I$ we define a $2d \times 2d$ 
block diagonal matrix $\mcal{T}_n$ as follows:
\begin{align*}
\mcal{T}_0 
	& = \text{blockdiag}\big[\mfrk{t}_0(0), \mfrk{t}_0(1), \ldots, \mfrk{t}_0(d-1)\big], \\
\mcal{T}_1 
	& = \text{blockdiag}\big[\mfrk{t}_1(0), \mfrk{t}_1(1), \ldots, \mfrk{t}_1(d-1), \mfrk{t}_1(d)\big],\\
\mcal{T}_2 
	& = \text{blockdiag}\big[\mfrk{t}_2(0), \mfrk{t}_2(1), \ldots, \mfrk{t}_2(d-1), \mfrk{t}_2(d)\big], \\
\mcal{T}_3 
	& = \text{blockdiag}\big[\mfrk{t}_3(0), \mfrk{t}_3(1), \ldots, \mfrk{t}_3(d-1)\big]. 
\end{align*}
\end{definition}

\medskip
\noindent
Referring to Definition \ref{Mat_t_i}, the matrices $\mcal{T}_0$ and $\mcal{T}_3$
take the form
$$\left(
\begin{array}{ccccccccc}
* & * & &&&&&& {\bf 0} \\
* & * & &&&&&& \\
&& * & * & &&&& \\
&& * & * & &&&& \\
&&&&\ddots &&&&\\
&&&&& * & * & & \\
&&&&& * & * & & \\
&&&&&&& * & * \\
{\bf 0}&&&&&&& * & * 
\end{array}
\right),$$
and the matrices $\mcal{T}_1$ and $\mcal{T}_2$ take the form
$$\left(
\begin{array}{ccccccccc}
* && &&&&&& {\bf 0} \\
&* & * & &&&&& \\
&* & * & &&&&& \\
&&& * & * & &&& \\
&&& * & * & &&& \\
&&&&& \ddots &&&\\
&&&&&& * & * &  \\
&&&&&& * & * & \\
{\bf 0}&&&&&&&& * 
\end{array}
\right).$$

\begin{lemma}\label{X;qX}
With reference to Definition {\rm \ref{Mat_t_i}}, both
\begin{eqnarray}\label{X;T_30}
\mcal{T}_3\mcal{T}_0 &=&
\diag\left( \tfrac{1}{q\sqrt{s^*}}, q\sqrt{s^*}, \tfrac{1}{q^2\sqrt{s^*}}, q^2\sqrt{s^*},
\ldots, 
\tfrac{1}{q^{d-1}\sqrt{s^*}}, q^{d-1}\sqrt{s^*},
\tfrac{1}{q^d\sqrt{s^*}}, q^d\sqrt{s^*}\right), \\
\label{qX;T_12}
\mcal{T}_1\mcal{T}_2 &=& 
\diag\left(\sqrt{s^*q}, 
\tfrac{1}{q\sqrt{s^*q}}, q\sqrt{s^*q},
\tfrac{1}{q^2\sqrt{s^*q}}, q^2\sqrt{s^*q},
\ldots, \tfrac{1}{q^{d-1}\sqrt{s^*q}}, q^{d-1}\sqrt{s^*q}, \tfrac{1}{q^d\sqrt{s^*q}} 
\right).
\end{eqnarray}
Moreover each of $\mcal{T}_3\mcal{T}_0, \mcal{T}_1\mcal{T}_2$
is multiplicity-free.
\end{lemma}
\begin{proof} Use Lemma \ref{ti(j)ti(j)}(i) to get (\ref{X;T_30}). Use Lemma 
\ref{ti(j)ti(j)}(ii) and Lemma \ref{block(1X1)}(ii) to get (\ref{qX;T_12}). 
To obtain the last assertion, recall by Example \ref{q-rac;PA} that
$q^i \ne 1$ for $1\leq i \leq d$ and $s^*q^i\ne1$ for $2 \leq i \leq 2d$.
\end{proof}

\medskip
\begin{lemma} \label{UDAHA(1,2,3)}
With reference to Definition {\rm\ref{Mat_t_i}},  the following {\rm (i)--(iii)} hold:
\begin{enumerate}
\item[\rm(i)] $\mcal{T}_n$ is invertible \qquad $(n \in \I)$;
\item[\rm(ii)] $\mcal{T}_n + \mcal{T}_n^{-1} = (k_n + k_n^{-1})I \qquad (n \in \I)$;
\item[\rm(iii)] $\mcal{T}_0\mcal{T}_1\mcal{T}_2\mcal{T}_3 = q^{-{1}/{2}}I$.
\end{enumerate}
\end{lemma}
\begin{proof}[Proof \rm{:}]
(i) Use Lemma \ref{inv(t_i)}.

\smallskip \noindent
(ii)  Immediate from (\ref{ti(j)+ti(j)}) and Lemma \ref{block(1X1)}(i).

\smallskip \noindent
(iii) It suffices to show that $\mcal{T}_3\mcal{T}_0\mcal{T}_1\mcal{T}_2 = q^{-1/2}I$.
This is routinely verified using Lemma \ref{X;qX}.
\end{proof}

 \noindent
We now describe an action of $\H$ on $\W$.
Recall the basis $\mcal{C}$ for $\W$ from (\ref{basis(C)}).

\begin{proposition}\label{H-mod;W}
Let $\W$ be as below Definition {\rm\ref{Algebra(T)}}. 
Then there exists an $\H$-module structure on $\W$ such that
for $n \in \I$ the matrix $\mcal{T}_n$ 
represents the generator $t_n$ relative to $\mcal{C}$.
\end{proposition}
\begin{proof} Follows from Lemma \ref{UDAHA(1,2,3)}.
\end{proof}

\noindent
It turns out that the $\H$-module $\W$ is irreducible. We will show this in Section \ref{mainsection}.

\begin{corollary}\label{X;diagonalizable}
Referring to the $\H$-module $\W$ from Proposition {\rm \ref{H-mod;W}},
the elements of $\mcal{C}$ are eigenvectors for the 
action of $\bf X$ on $\W$. Moreover the action of $\bf X$ on $\W$ is
multiplicity free.
\end{corollary}
\begin{proof}
Recall ${\bf X} = t_3t_0$. By this and Proposition \ref{H-mod;W}
the matrix $\mcal{T}_3\mcal{T}_0$ represents $\bf X$ relative to $\mcal{C}$.
The result follows in view of Lemma \ref{X;qX}.
\end{proof}

\begin{remark} In Proposition \ref{H-mod;W} we obtained
an $\H$-action on $\W$.
Let ${H}={H}(k_0,k_1,k_2,k_3;q)$
be a double affine Hecke algebra of type $(C^{\vee}_1, C_1)$,
where $k_0, k_1, k_2, k_3$ are from (\ref{k0123}).
Then there exists a surjective $\mbb{C}$-algebra homomorphism 
$\H \to {H}$ that sends $t_n \mapsto t_n$ for all $n\in \I$. 
Taking the quotient of $\H$ by the kernel of this homomorphism
we obtain an ${H}$-module structure on $\W$.
\end{remark}

\noindent
We mention a result for future use.
\begin{lemma}\label{k_0;k_1}
With reference to Definition \rm{\ref{k_0,1,2,3}} 
neither of $k_0, k_1$ is equal to $\pm1$.
\end{lemma}
\begin{proof} 
By Example \ref{q-rac;PA} $\frac{r_1r_2}{s^*} = sq^{d+1}$ and $sq^{d+1} \ne 1$, 
so $k_0^2 \ne 1$ by (\ref{k0123}).
Also by Example \ref{q-rac;PA} $q^d \ne 1$, so $k_1^2 \ne 1$ by (\ref{k0123}).
\end{proof}


\section{How the $\H$-action on $\W$ is related to $\T$}\label{mainsection}

We continue to discuss the $\T$-module $\W$.
In the previous section we saw how the algebra $\H$ acts on $\W$.
In this section we will see how the $\H$-action on $\W$ is related
to the algebra $\T$. Recall from Definition {\ref{Algebra(T)}} 
that the algebra $\T$ is generated by $A, A^*, \wt{A}^*$.
Recall the projections $\mfrk{p}, \wt{\mfrk{p}}$ from Section \ref{2maps}.
The following theorem is the main result of the paper.

\begin{theorem}\label{main_thm}
Referring to the $\H$-module $\W$ from Proposition {\rm \ref{H-mod;W}},
for each row of the table below the two displayed elements coincide on $\W$.
\begin{equation*}\label{main_table}
\begin{tabular}{c|c}
Element in $\H$ & Element in $\T$ \\
\hline \hline
{\bf A} & \quad$\frac{1}{h\sqrt{sq}}(A-(\tht_0-h-hsq)I) \begin{matrix}& \\ &\end{matrix}$ \\ 
{\bf B} & \quad $ \frac{1}{\wt{h}^*\sqrt{\wt{s}^*q}}(\wt{A}^*-(\wt{\tht}^*_0-\wt{h}^*-\wt{h}^*\wt{s}^*q)I)$ \\ 
${{\bf B^{\dagger}}}$ & \quad $\frac{1}{h^*\sqrt{s^*q}}(A^*-(\tht^*_0-h^*-h^*s^*q)I) \begin{matrix}& \\ &\end{matrix}$ \\ 
$\frac{t_0-k_0^{-1}}{k_0-k_0^{-1}}$ & $\wt{\mfrk{p}}$ \\
$\frac{t_1-k_1^{-1}}{k_1-k_1^{-1}} $ & $\mfrk{p}$
\end{tabular}
\end{equation*}
\end{theorem}

\noindent
The remainder of this section is devoted to the proof of Theorem \ref{main_thm}.
Recall from (\ref{basis(C)}) that the basis $\mcal{C}$ consists of the
vectors $\hat{C}^-_i, \hat{C}^+_i \ (0 \leq i\leq d-1)$. Recall from above Lemma \ref{action(A;W)} that
$$
\hat{C}^-_{-1} =0, \qquad \hat{C}^+_{-1}=0, \qquad
\hat{C}^-_{d} =0, \qquad \hat{C}^+_{d}=0. 
$$

\begin{lemma}\label{Y-action}
Let $\bf Y$ be as in Definition {\rm\ref{X,Y}}. Then for $0\leq i\leq d-1$, 
$\Y.\hat{C}^-_i$ and $\Y.\hat{C}^+_i$ are given as a linear combination
with the following terms and coefficients:

\medskip
 ${\bf Y}.\hat{C}^-_i =$
\begin{align}
&\begin{tabular}{l | l}
{\rm term} & \hspace{2cm} {\rm coefficient} \\
\hline \hline
 & \\
$\hat{C}^-_{i-1}$ & $\sqrt{\tfrac{s^*q^d}{r_1r_2}}
					\left(\tfrac{(1-q^{i-d})(1-s^*q^{i+1})(1-r_1q^i)(1-r_2q^i)}
					{(1-s^*q^{2i})(1-s^*q^{2i+1})}\right)$ \\
 & \\					
$\hat{C}^+_{i-1}$ & $\sqrt{\tfrac{s^*q^d}{r_1r_2}}
					\left(\tfrac{(1-q^{i-d})(1-s^*q^{i+1})}{1-s^*q^{2i+1}}\right)
					\left(\tfrac{(1-r_1q^i)(1-r_2q^i)}{1-s^*q^{2i}}-1\right)$ \\
 & \\
$\hat{C}^-_{i}$ & $\tfrac{1}{\sqrt{s^*r_1r_2q^d}}\left(
	\tfrac{(q^i-1)(1-s^*q^{d+i+1})}{1-s^*q^{2i+1}}+1\right)
	\left(\tfrac{(r_1-s^*q^{i+1})(r_2-s^*q^{i+1})}{1-s^*q^{2i+2}}+s^*\right)$ \\
 & \\
$\hat{C}^+_{i}$ & $\tfrac{1}{\sqrt{s^*r_1r_2q^d}}\left(
	\tfrac{(r_1-s^*q^{i+1})(r_2-s^*q^{i+1})}{1-s^*q^{2i+2}}\right)
	\left(\tfrac{(q^i-1)(1-s^*q^{d+i+1})}{1-s^*q^{2i+1}}+1\right)$ 
\end{tabular} &&&
\end{align}
and 

\medskip
 ${\bf Y}.\hat{C}^+_i =$
\begin{align}
&\begin{tabular}{l | l}
{\rm term} & \hspace{2cm} {\rm coefficient} \\
\hline \hline
& \\
$\hat{C}^-_{i}$ & $-\sqrt{\tfrac{s^*q^d}{r_1r_2}}\left(
 		\tfrac{(1-r_1q^{i+1})(1-r_2q^{i+1})}{1-s^*q^{2i+2}}\right)
		\left(\tfrac{(1-q^{i-d+1})(1-s^*q^{i+2})}{1-s^*q^{2i+3}}+\tfrac{1}{q^d}\right)$\\
& \\
$\hat{C}^+_{i}$ & $\sqrt{\tfrac{s^*q^d}{r_1r_2}}\left(
		1-\tfrac{(1-r_1q^{i+1})(1-r_2q^{i+1})}{1-s^*q^{2i+2}}\right)
		\left(\tfrac{(1-q^{i-d+1})(1-s^*q^{i+2})}{1-s^*q^{2i+3}}+\tfrac{1}{q^d}\right)$\\
& \\
$\hat{C}^-_{i+1}$ & $\tfrac{1}{\sqrt{s^*r_1r_2q^d}}\left(
		\tfrac{(1-q^{i+1})(1-s^*q^{d+i+2})}{1-s^*q^{2i+3}}\right)
		\left(\tfrac{(r_1-s^*q^{i+2})(r_2-s^*q^{i+2})}{1-s^*q^{2i+4}}+s^*\right)$\\
& \\
$\hat{C}^+_{i+1}$ & $\tfrac{1}{\sqrt{s^*r_1r_2q^d}}\left(
		\tfrac{(1-q^{i+1})(1-s^*q^{d+i+2})(r_1-s^*q^{i+2})(r_2-s^*q^{i+2})}
		{(1-s^*q^{2i+3})(1-s^*q^{2i+4})}\right)$
\end{tabular}&&&
\end{align}
\end{lemma}
\begin{proof} Compute $\mcal{T}_0\mcal{T}_1$ using Definition \ref{Mat_t_i} 
and the data in Definition \ref{(2x2)mat}.
\end{proof}

\begin{lemma}\label{Y^(-1)action}
Let $\bf Y$ be as in Definition {\rm\ref{X,Y}}.
For $0 \leq i \leq d-1$, $\Y^{-1}.\hat{C}^-_i$ and $\Y^{-1}.\hat{C}^+_i$ are given
as a linear combination with the following terms and coefficients:

\medskip
${\bf Y}^{-1}.\hat{C}^{-}_i=$
\begin{align}
&\begin{tabular}{l | l}
{\rm term} & \hspace{2cm} {\rm coefficient} \\
\hline \hline
& \\
$\hat{C}^{+}_{i-1}$ & $\sqrt{\tfrac{s^*q^d}{r_1r_2}}\left(
		\tfrac{(1-q^{i-d})(1-s^*q^{i+1})}{1-s^*q^{2i+1}}\right)
		\left(1-\tfrac{(1-r_1q^{i+1})(1-r_2q^{i+1})}{1-s^*q^{2i+2}}\right)$\\
& \\
$\hat{C}^{-}_{i}$ & $\sqrt{\tfrac{s^*q^d}{r_1r_2}}\left(
		1-\tfrac{(1-r_1q^{i+1})(1-r_2q^{i+1})}{1-s^*q^{2i+2}}\right)
		\left(\tfrac{(1-q^{i-d})(1-s^*q^{i+1})}{1-s^*q^{2i+1}}+\tfrac{1}{q^d}\right)$\\
& \\
$\hat{C}^{+}_{i}$ & $\tfrac{1}{\sqrt{s^*r_1r_2q^d}}\left(
		\tfrac{(r_1-s^*q^{i+1})(r_2-s^*q^{i+1})}{1-s^*q^{2i+2}}\right)
		\left(\tfrac{(1-q^{i+1})(1-s^*q^{d+i+2})}{1-s^*q^{2i+3}}-1\right)$\\
& \\
$\hat{C}^{-}_{i+1}$ & $\tfrac{1}{\sqrt{s^*r_1r_2q^d}}\left(
		\tfrac{(1-q^{i+1})(1-s^*q^{d+i+2})(r_1-s^*q^{i+1})(r_2-s^*q^{i+1})}
		{(1-s^*q^{2i+2})(1-s^*q^{2i+3})}\right)$\\
\end{tabular} &&&
\end{align} 
and

\medskip
${\bf Y}^{-1}.\hat{C}^{+}_i=$
\begin{align}
&\begin{tabular}{l | l}
{\rm term} & \hspace{2cm} {\rm coefficient} \\
\hline \hline
& \\
$\hat{C}^+_{i-1}$ & $\sqrt{\tfrac{s^*q^d}{r_1r_2}}\left( 
		\tfrac{(1-q^{i-d})(1-s^*q^{i+1})(1-r_1q^{i+1})(1-r_2q^{i+1})}
		{(1-s^*q^{2i+1})(1-s^*q^{2i+2})}\right)$\\
& \\
$\hat{C}^-_{i}$ & $\sqrt{\tfrac{s^*q^d}{r_1r_2}}\left(
		\tfrac{(1-r_1q^{i+1})(1-r_2q^{i+1})}{1-s^*q^{2i+2}}\right)
		\left(\tfrac{(1-q^{i-d})(1-s^*q^{i+1})}{1-s^*q^{2i+1}}+\tfrac{1}{q^d}\right)$\\
& \\
$\hat{C}^+_{i}$ & $\tfrac{1}{\sqrt{s^*r_1r_2q^d}}\left(
		\tfrac{(r_1-s^*q^{i+1})(r_2-s^*q^{i+1})}{1-s^*q^{2i+2}}+s^*\right)
		\left(\tfrac{(q^{i+1}-1)(1-s^*q^{d+i+2})}{1-s^*q^{2i+3}}+1\right)$\\
& \\
$\hat{C}^-_{i+1}$ & $\tfrac{1}{\sqrt{s^*r_1r_2q^d}}\left(
		\tfrac{(q^{i+1}-1)(1-s^*q^{d+i+2})}{1-s^*q^{2i+3}}\right)
		\left(\tfrac{(r_1-s^*q^{i+1})(r_2-s^*q^{i+1})}{1-s^*q^{2i+2}}+s^*\right)$\\
\end{tabular} &&&
\end{align}
\end{lemma}
\begin{proof} Similar to the proof of Lemma \ref{Y-action}.
\end{proof}

\noindent
By Lemma \ref{Y-action} and Lemma \ref{Y^(-1)action} the matrices representing ${\bf Y}, {\bf Y}^{-1}$ relative to $\mcal{C}$ take the form
\begin{equation*}
\scalemath{0.8}{
\left(
\begin{array}{cccccccccccccc}
* & * & * & \\
* & * & * & \\
& * & * & * & * \\
& * & * & * & * \\ 
&&& * & * & * & * \\ 
&&& * & * & * & * \\ 
&&&&&&& \ddots &  \\
&&&&&&&& * & * &  &  \\ 
&&&&&&&& * & * &  &  \\ 
&&&&&&&& * & * & * & * \\ 
&&&&&&&& * & * & * & * \\ 
&&&&&&&&&& * & * & * \\ 
&&&&&&&&&& * & * & * \\ 
\end{array}
\right), }
\end{equation*}
\begin{equation*}
\scalemath{0.8}{
\left(
\begin{array}{cccccccccccccc}
  * & * & \\
 * & * & * & * \\
 * & * & * & * \\ 
&& * & * & * & * \\ 
&& * & * & * & * \\ 
&&&& * & * &  &  \\
&&&& * & * &  &  \\
&&&&  &  & \ddots &  \\
&&&&&&&*&*  & * & * \\ 
&&&&&&&*&* & * & * \\ 
&&&&&&&&& * & * & * &*\\ 
&&&&&&&&& * & * & * &*\\ 
&&&&&&&&&&& * &*\\ 
\end{array}
\right)},
\end{equation*}
respectively.

\begin{proposition} \label{A-action;W}
Let $\bf A$ be as in Definition {\rm \ref{A,B,B+}}.
Then for $0 \leq i \leq d-1$, ${\bf A}.\hat{C}^-_i$ and ${\bf A}.\hat{C}^+_i$
are given as a linear combination with the following terms and coefficients:

\medskip
 ${\bf A}.\hat{C}^-_i = $
\begin{equation}\label{A-action;C-}
\begin{tabular}{l | l}
{\rm term} & \hspace{2cm} {\rm coefficient} \\
\hline \hline \\
$\hat{C}^-_{i-1}$ & 
		$\sqrt{\tfrac{s^*q^d}{r_1r_2}}
		\left(\tfrac{(1-q^{i-d})(1-s^*q^{i+1})(1-r_1q^i)(1-r_2q^i)}
				{(1-s^*q^{2i})(1-s^*q^{2i+1})}\right)$\\ \\
$\hat{C}^+_{i-1}$ & 
		$\sqrt{\frac{s^*q^d}{r_1r_2}}
		\left(\frac{(1-q^{i-d})(1-s^*q^{i+1})}{1-s^*q^{2i+1}}\right) 
		\left(\frac{(1-r_1q^i)(1-r_2q^i)}{1-s^*q^{2i}}
			- \frac{(1-r_1q^{i+1})(1-r_2q^{i+1})}{1-s^*q^{2i+2}}\right)$\\ \\
$\hat{C}^-_{i}$ & 
		$\left(\sqrt{\frac{s^*q^d}{r_1r_2}}+\sqrt{\frac{r_1r_2}{s^*q^d}}\right)
				- \sqrt{\frac{s^*q^d}{r_1r_2}}
		\times 
		\left(\frac{(1-q^i)(1-s^*q^{i+d+1})(r_1-s^*q^{i+1})(r_2-s^*q^{i+1})}
					{s^*q^d(1-s^*q^{2i+1})(1-s^*q^{2i+2})}\right. $\\ \\
		& \hspace{5cm} 
		$\left. + \ \frac{(1-q^{i-d})(1-s^*q^{i+1})(1-r_1q^{i+1})(1-r_2q^{i+1})}
					{(1-s^*q^{2i+1})(1-s^*q^{2i+2})}\right)$\\ \\
$\hat{C}^+_{i}$ & 
		$\sqrt{\frac{s^*q^d}{r_1r_2}}
		\left(\frac{(r_1-s^*q^{i+1})(r_2-s^*q^{i+1})}{s^*q^d(1-s^*q^{2i+2})}\right)
		\left(\frac{(1-q^{i+1})(1-s^*q^{i+d+2})}{1-s^*q^{2i+3}}
			- \frac{(1-q^i)(1-s^*q^{i+d+1})}{1-s^*q^{2i+1}}\right)$\\ \\
$\hat{C}^-_{i+1}$ &
		$\sqrt{\frac{s^*q^d}{r_1r_2}}
		\left(\frac{(1-q^{i+1})(1-s^*q^{i+d+2})(r_1-s^*q^{i+1})(r_2-s^*q^{i+1})}
				{s^*q^d(1-s^*q^{2i+2})(1-s^*q^{2i+3})}\right)$
\end{tabular}
\end{equation}
and

\medskip
 ${\bf A}.\hat{C}^+_i =$
\begin{equation}\label{A-action;C+}
\begin{tabular}{l | l}
{\rm term} & \hspace{2cm} {\rm coefficient} \\
\hline \hline \\
$\hat{C}^+_{i-1}$ & 
		$\sqrt{\tfrac{s^*q^d}{r_1r_2}}
		\left(\tfrac{(1-q^{i-d})(1-s^*q^{i+1})(1-r_1q^{i+1})(1-r_2q^{i+1})}
				{(1-s^*q^{2i+1})(1-s^*q^{2i+2})}\right)$\\ \\
$\hat{C}^-_{i}$ & 
		$\sqrt{\frac{s^*q^d}{r_1r_2}}
		\left(\frac{(1-r_1q^{i+1})(1-r_2q^{i+1})}{1-s^*q^{2i+2}}\right) 
		\left(\frac{(1-q^{i-d})(1-s^*q^{i+1})}{1-s^*q^{2i+1}}
			- \frac{(1-q^{i-d+1})(1-s^*q^{i+2})}{1-s^*q^{2i+3}}\right)$\\ \\
$\hat{C}^+_{i}$ & 
		$\left(\sqrt{\frac{s^*q^d}{r_1r_2}}+\sqrt{\frac{r_1r_2}{s^*q^d}}\right)
				- \sqrt{\frac{s^*q^d}{r_1r_2}}
		\times 
		\left(\frac{(1-q^{i+1})(1-s^*q^{i+d+2})(r_1-s^*q^{i+1})(r_2-s^*q^{i+1})}
					{s^*q^d(1-s^*q^{2i+2})(1-s^*q^{2i+3})}\right. $\\ \\
		& \hspace{5cm} 
		$\left. + \ \frac{(1-q^{i-d+1})(1-s^*q^{i+2})(1-r_1q^{i+1})(1-r_2q^{i+1})}
					{(1-s^*q^{2i+2})(1-s^*q^{2i+3})}\right)$\\ \\
$\hat{C}^-_{i+1}$ & 
		$\sqrt{\frac{s^*q^d}{r_1r_2}}
		\left(\frac{(1-q^{i+1})(1-s^*q^{i+d+2})}{s^*q^d(1-s^*q^{2i+3})}\right)
		\left(\frac{(r_1-s^*q^{i+2})(r_2-s^*q^{i+2})}{1-s^*q^{2i+4}}
			- \frac{(r_1-s^*q^{i+1})(r_2-s^*q^{i+1})}{1-s^*q^{2i+2}}\right)$\\ \\
$\hat{C}^+_{i+1}$ &
		$\sqrt{\frac{s^*q^d}{r_1r_2}}
		\left(\frac{(1-q^{i+1})(1-s^*q^{i+d+2})(r_1-s^*q^{i+2})(r_2-s^*q^{i+2})}
				{s^*q^d(1-s^*q^{2i+3})(1-s^*q^{2i+4})}\right)$
\end{tabular}
\end{equation}
\end{proposition}
\begin{proof}
Combine Lemma \ref{Y-action} and Lemma \ref{Y^(-1)action}.
\end{proof}

\begin{proposition}\label{B-action;W}
Let $\bf B$ be as in Definition {\rm\ref{A,B,B+}}.
Then $\bf B$ acts on the basis $\mcal{C}$ as follows:
\begin{enumerate}
\item[\rm(i)] ${\bf B}.\hat{C}^-_i = 
				\left(\frac{1}{q^{i+1}\sqrt{s^*}}+q^{i+1}\sqrt{s^*}\right)\hat{C}^-_i$
				\quad $(0 \leq i \leq d-1).$
\item[\rm(ii)] ${\bf B}.\hat{C}^+_i = 
				\left(\frac{1}{q^{i+1}\sqrt{s^*}}+q^{i+1}\sqrt{s^*}\right)\hat{C}^+_i$
				\quad $(0 \leq i \leq d-1).$
\end{enumerate}
\end{proposition}
\begin{proof} Use (\ref{X;T_30}).
\end{proof}

\begin{lemma}\label{eigval(B)}
Referring to Proposition {\rm \ref{B-action;W}}, the following scalars are mutually distinct:
\begin{align*}
&& \tfrac{1}{q^{i+1}\sqrt{s^*}}+q^{i+1}\sqrt{s^*} && (0 \leq i \leq d-1).
\end{align*}
\end{lemma}
\begin{proof}
By Example \ref{q-rac;PA}, $s^*q^i \ne 1$ for $2 \leq i \leq 2d$.
\end{proof}

\begin{corollary}
For $0 \leq i \leq d-1$,  the vectors $\hat C^-_i, \hat C^+_i$ form
a basis for an eigenspace of {\bf B} with eigenvalue 
$\frac{1}{q^{i+1}\sqrt{s^*}}+q^{i+1}\sqrt{s^*}$.
\end{corollary}
\begin{proof}
Use Proposition \ref{B-action;W} and  Lemma \ref{eigval(B)}.
\end{proof}

\begin{proposition}\label{(B+)-action;W}
Let $\bf B^{\dagger}$ be as in Definition {\rm\ref{A,B,B+}}.
Then ${\bf B}^{\dagger}$
acts on the basis $\mcal{C}$ as follows:
\begin{enumerate}
\item[\rm(i)] ${\bf B}^{\dagger}.\hat{C}^-_i = 
				\left(\frac{1}{q^i\sqrt{s^*q}}+q^i\sqrt{s^*q}\right)\hat C^-_i$
				\qquad $(0 \leq i \leq d-1)$,
\item[\rm(ii)] ${\bf B}^{\dagger}.\hat{C}^+_{i} =
				\left(\frac{1}{q^{i+1}\sqrt{s^*q}}+q^{i+1}\sqrt{s^*q}\right)\hat C^+_{i}$
				\qquad $(0 \leq i \leq d-1)$.
\end{enumerate}
\end{proposition}
\begin{proof}
Use (\ref{qX;T_12}). 
\end{proof}

\begin{lemma}\label{eigval(B+)}
Referring to Proposition {\rm \ref{(B+)-action;W}}, the following scalars are mutually distinct:
\begin{align*}
&& \tfrac{1}{q^i\sqrt{s^*q}}+q^i\sqrt{s^*q} && (0 \leq i \leq d-1).
\end{align*}
\end{lemma}
\begin{proof}
By Example \ref{q-rac;PA}, $s^*q^i \ne 1$ for $2 \leq i \leq 2d$.
\end{proof}

\begin{corollary}
For $1 \leq i \leq d-1$ the vectors $\hat C^{+}_{i-1}, \hat C^{-}_{i}$
form a basis for an eigenspace of ${\bf B}^{\dagger}$
with eigenvalue $\frac{1}{q^i\sqrt{s^*q}}+q^i\sqrt{s^*q}$.
\end{corollary}
\begin{proof}
Use Proposition \ref{(B+)-action;W} and Lemma \ref{eigval(B+)}.
\end{proof}

\noindent
Recall the scalars $k_0$ and $k_1$ from (\ref{k0123}).
Recall by Lemma \ref{k_0;k_1} that $k_0 \ne \pm1$ and $k_1 \ne \pm1$.
We now describe the action of  $\frac{t_0-k_0^{-1}}{k_0-k^{-1}_0}$
and $\frac{t_1-k_1^{-1}}{k_1-k^{-1}_1}$ on the basis $\mcal{C}$.

\begin{proposition}\label{t_0/k_0-action}
The element $\frac{t_0-k_0^{-1}}{k_0-k^{-1}_0}$
acts on the basis $\mcal{C}$ as follows:
\begin{enumerate}
\item[\rm (i)] $\left(\frac{t_0-k_0^{-1}}{k_0-k^{-1}_0}\right). \hat{C}^-_i =
				\frac{(r_1-s^*q^{i+1})(r_2-s^*q^{i+1})}{(r_1r_2-s^*)(1-s^*q^{2i+2})}
				(\hat C^-_i + \hat C^+_i)$
				\qquad $(0 \leq i \leq d-1)$.
\item[\rm (ii)]$\left(\frac{t_0-k_0^{-1}}{k_0-k^{-1}_0}\right). \hat{C}^+_i =
				\frac{s^*(1-r_1q^{i+1})(1-r_2q^{i+1})}{(s^*-r_1r_2)(1-s^*q^{2i+2})}
				(\hat C^-_i + \hat C^+_i)$
				\qquad $(0 \leq i \leq d-1)$.
\end{enumerate}
\end{proposition}
\begin{proof}
Use the $t_0$-action on the basis $\mcal{C}$ from Proposition \ref{H-mod;W}.
\end{proof}

\begin{proposition}\label{t_1/k_1-action}
The element $\frac{t_1-k_1^{-1}}{k_1-k^{-1}_1}$ 
acts on the basis $\mcal{C}$ as follows:
\begin{enumerate}
\item[\rm (i)] $\left(\tfrac{t_1-k_1^{-1}}{k_1-k^{-1}_1}\right). \hat{C}^-_0 = \hat C^-_0,$
\item[\rm (ii)]	$\left(\frac{t_1-k_1^{-1}}{k_1-k^{-1}_1}\right). \hat{C}^-_i = 
					\frac{q^d(1-q^{i-d})(1-s^*q^{i+1})}{(q^d-1)(1-s^*q^{2i+1})}
					(\hat C^+_{i-1} + \hat C^-_{i})$
					\qquad $(1 \leq i \leq d-1),$

\item[\rm (iii)] $\left(\frac{t_1-k_1^{-1}}{k_1-k^{-1}_1}\right). \hat{C}^+_{i-1} = 
					\frac{(1-q^i)(1-s^*q^{i+d+1})}{(1-q^d)(1-s^*q^{2i+1})}
					(\hat C^+_{i-1}+\hat C^-_i)$
					\qquad $(1 \leq i \leq d-1),$

\item[\rm (iv)]$\left(\tfrac{t_1-k_1^{-1}}{k_1-k^{-1}_1}\right). \hat{C}^+_{d-1} = \hat C^+_{d-1}$.	
\end{enumerate}
\end{proposition}
\begin{proof}
Use the $t_1$-action on the basis $\mcal{C}$
from Proposition \ref{H-mod;W}.
\end{proof}

\noindent
We are ready to prove Theorem \ref{main_thm}.

\begin{proof}[Proof of Theorem \ref{main_thm}{\rm:}] 
We refer to the table in the theorem statement. 
For each row we compare 
the matrices representing the two displayed elements relative to the 
basis $\mcal{C}$. In each case we show that these matrices coincide.

\bigskip \noindent
{\bf A}: From Proposition \ref{A-action;W}
we find the matrix representation of $\bf A$.
From (\ref{A_C;odd-col}) and (\ref{A_C;even_col}) we obtain the
matrix representation of $A$. From this we get the matrix representation
of $\frac{1}{h\sqrt{sq}}(A-(\tht_0-h-hsq)I)$. In this representation,
eliminate $s$ using $r_1r_2=ss^*q^{d+1}$. The result coincides with the 
matrix representation of $\bf A$.

\bigskip \noindent
{\bf B}:
From Proposition \ref{B-action;W} we find the matrix representation of $\bf B$.
From Lemma \ref{[A*C]_C} we obtain the matrix representation of $\wt{A}^*$.
From this we get the matrix representation of 
$\frac{1}{\wt{h}^*\sqrt{\wt{s}^*q}}(\wt{A}^*-(\wt{\tht}^*_0-\wt{h}^*-\wt{h}^*\wt{s}^*q)I)$.  
Evaluate this representation using (\ref{q-rac;ttht*_i}).
The result coincides with the matrix representation of $\bf B$.

\bigskip \noindent
${\bf B}^{\dagger}$: From Proposition \ref{(B+)-action;W} we find
the matrix representation of ${\bf B}^{\dagger}$.
From Lemma \ref{[A*]_C} we obtain the matrix representation of  $A^*$.
From this we get the matrix representation of 
$\tfrac{1}{h^*\sqrt{s^*q}}(A^*-(\tht^*_0-h^*-h^*s^*q)I)$.
Evaluate this representation using (\ref{theta^*_i}).
The result coincides with the matrix representation of $\B^{\dagger}$.

\bigskip \noindent
$\frac{t_0-k^{-1}_0}{k_0-k^{-1}_0}$:
From Proposition \ref{t_0/k_0-action} we find the matrix
representation of $\frac{t_0-k_0^{-1}}{k_0-k_0^{-1}}$. 
The representation coincides with the matrix (\ref{[p_C]_C}), which is 
the representation of $\wt{\mfrk{p}}$. 

\bigskip \noindent
$\frac{t_1-k^{-1}_1}{k_1-k^{-1}_1}$:
From Proposition \ref{t_1/k_1-action} we find the matrix
representation of $\frac{t_1-k_1^{-1}}{k_1-k_1^{-1}}$. 
The representation coincides with the matrix (\ref{[p_x]_C}), which is
the representation of $\mfrk{p}$. 
\end{proof}

\noindent
We finish this paper with a comment.
\begin{corollary} Referring to Theorem {\rm \ref{main_thm}}, the $\H$-module $\W$ is irreducible.
\end{corollary}
\begin{proof}
By definition \ref{Algebra(T)} the algebra $\T$ is generated by $A, A^*, \wt{A}^*$. By Proposition \ref{W;irreducible} the $\T$-module $\W$ is irreducible. The result follows in view of Theorem \ref{main_thm}.
\end{proof}

\section {Appendix}
In this Appendix we consider the case of diameter $d=4$ in great detail.
We briefly review and summarize the results of Part I and Part II.
Concerning Part I, recall that
$\Ga$ is a $Q$-polynomial distance-regular graph of $q$-Racah type
that contains a Delsarte clique $C$. Recall the semisimple algebra $\T$ 
from Definition \ref{Algebra(T)}, whose generators are $A, A^*, \wt{A}^*$.
Recall the primary $\T$-module $\W$ from below Definition \ref{Algebra(T)}. 
The $\T$-module $\W$ is decomposed in two ways:
$$
\W = \Mx + \Mxp, \qquad\qquad  \W= \MC + \MCp. \qquad \text{(orthogonal direct sum)}
$$
On each of the four summands, we found a natural Leonard system.
For each Leonard system we found the parameter array, and we described
how these parameter arrays are related. 
We recall our five linear maps in ${\rm End}(\W)$:
\begin{equation}\label{5maps}
A, \qquad A^*, \qquad \wt{A}^*, \qquad \mfrk{p}, \qquad \wt{\mfrk{p}}.
\end{equation}
Recall from (\ref{diagram}) our five bases for $\W$:
\begin{equation}\label{5bases}
\mcal{C}, \qquad \mcal{B}, \qquad \mcal{B}_{alt}, \qquad \wt{\mcal{B}}, 
\qquad \wt{\mcal{B}}_{alt}.
\end{equation}
In Section \ref{matrices}, 
we displayed the matrix representing each map in (\ref{5maps})
relative to each basis in (\ref{5bases}). We also displayed the transition matrices
between certain pairs of bases among (\ref{5bases}).

\noindent
Concerning Part II, recall the algebra $\H$ from Definition \ref{DAHA} whose
generators are $\{t^{\pm1}_n\}^{3}_{n=0}$ and relations (\ref{daha1})--(\ref{daha3}).
Recall  $\X = t_3t_0, \Y=t_0t_1$ from Definition \ref{X,Y},
and $\A, \B, \B^{\dagger}$ from Definition \ref{A,B,B+}.
In Section \ref{H-mod(W)} we constructed a $\H$-module on $\W$.
For this module and up to affine transformation we showed that
$\A$ acts as $A$ and 
$\B$ (resp. $\B^{\dagger}$) acts as $\wt{A}^*$
(resp. $A^*$). Moreover, up to affine transformation, $t_0$ (resp. $t_1$)
acts as $\wt{\mfrk{p}}$ (resp. $\mfrk{p}$). 

\medskip \noindent
To give a concrete example we now take $d=4$.
In Section \ref{Apdx:I} below, we display the matrices representing each map
in (\ref{5maps}) relative to each basis in (\ref{5bases}) and the transition
matrices between all pairs of bases among (\ref{5bases}).
In Section \ref{Apdx:II}, we display the matrices representing each generator
$\{t_n\}^3_{n=0}$ of $\H$ and 
some related elements of $\H$ with respect to the basis $\mcal{C}$ from (\ref{5bases}).

\medskip \noindent
We recall
the scalars $h, h^*, s, s^*, r_1, r_2$ from above Note \ref{h,h*}. 
The scalars $h, h^*$ satisfy (\ref{scalar(h)}), (\ref{scalar(h*)}), respectively.
Recall the formulae
$\epsilon_i$ from (\ref{epsilon_i(4)}), $\xi_i$ from (\ref{xi_i}),
$\tau_i$ from (\ref{tau;q-term;eq}), and $\zeta_i$ from (\ref{zeta_i}). For $d=4$,
these become
\begin{align}
\label{xi;epsilon}
&\epsilon_i = \frac{(1-q^i)(1-s^*q^{i+5})}{q^4(1-q^{i-4})(1-s^*q^{i+1})}, & & &
\xi_i = q^{1-i}(1-q^{i-4})(1-s^*q^{i+1}), & & &
\end{align}
for $1 \leq  i \leq 3$ and
\begin{align}
\label{zeta;tau}
&\tau_i = \frac{s^*(1-r_1q^{i+1})(1-r_2q^{i+1})}{(r_1-s^*q^{i+1})(r_2-s^*q^{i+1})}, & & &
\zeta_i = q^{-i}(r_1-s^*q^{i+1})(r_2-s^*q^{i+1}), & & &
\end{align}
for $0 \leq i \leq 3$.
All entries of each matrix in this Appendix will be expressed in terms of 
$q, s, s^*, r_1, r_2$ and their square roots.

\subsection{The matrices from Part I}\label{Apdx:I}

Consider the five bases in (\ref{5bases}). For $d=4$ these bases are
\begin{align*}
\mcal{C} & = \{
\hat{C}^-_0, \quad \hat{C}^+_0,  \quad 
\hat{C}^-_1,  \quad \hat{C}^+_1,  \quad 
\hat{C}^-_2,  \quad \hat{C}^+_2,  \quad
\hat{C}^-_3,  \quad \hat{C}^+_3\}, \\
\mcal{B} & = \{
v_0, \quad v_1, \quad v_2, \quad v_3, \quad v_4, \quad \vp_0, \quad \vp_1, \quad \vp_2\}, \\
\mcal{B}_{alt} & =\{
v_0, \quad v_1, \quad \vp_0, \quad v_2, \quad \vp_1, \quad v_3, \quad \vp_2, \quad v_4\},\\
\wt{\mcal{B}} & = \{
\tv_0, \quad \tv_1, \quad \tv_2, \quad \tv_3, \quad \tvp_0, \quad \tvp_1, \quad \tvp_2, \quad \tvp_3\},\\
\wt{\mcal{B}}_{alt} & =\{
\tv_0, \quad \tvp_0, \quad \tv_1, \quad \tvp_1, \quad \tv_2, \quad \tvp_2, \quad \tv_3, \quad \tvp_3\}.
\end{align*}
For each of the $20$ ordered pairs of bases from above,
we now display the corresponding transition matrix.
This will be done in Example \ref{ex:tm(C<->Bxa)}
through Example \ref{ex:tm(Bx<->BC)} below.

\bigskip
\begin{example}\label{ex:tm(C<->Bxa)}
The transition matrix from $\mathcal{C}$ to $\mcal{B}_{alt}$ is
\begin{equation*}
\left[
\begin{array}{c|cc|cc|cc|c}
1	&0&	0		&0	&0			&0		&0			&0\\
\hline
0	& 1	& \xi_1		&0	&0			&0		&0			&0\\
0	& 1	& \xi_1\epsilon_1	&0	&0			&0		&0			&0\\
\hline
0	&0	&0 			& 1	& \xi_2		&0		&0			&0\\
0	& 0	&0 			& 1	& \xi_2\epsilon_2	&0		&0			&0\\
\hline
0	&0	&0			&0	&0			& 1		& \xi_3		&0\\
0	&0	&0			&0	&0			& 1		& \xi_3\epsilon_3	&0\\
\hline
0	&0	&0			&0	&0			&0		&0			& 1	\\
\end{array}\right],
\end{equation*}
where $\{\xi_i\}^3_{i=1}$ and $\{\epsilon_i\}^3_{i=1}$ are from (\ref{xi;epsilon}). Moreover,
\begin{equation}\label{xi/epsilon}
\xi_i\epsilon_i	= q^{-i-3}(1-q^i)(1-s^*q^{i+5}) \qquad (1 \leq i \leq 3).
\end{equation}
The transition matrix from $\mcal{B}_{alt}$ to $\mathcal{C}$ is 
\begin{equation*}\label{ex:tm(Bxa->C)}
\left[
\begin{array}{c|cc|cc|cc|c}
1 	&0	&0	&0	&0	&0	&0	&0\\
\hline
0	&\frac{\epsilon_1}{\epsilon_1-1} & \frac{1}{1-\epsilon_1}	&0	&0	&0	&0	&0\\ 
0	&\frac{1}{\xi_1(1-\epsilon_1)} 	& \frac{1}{\xi_1(\epsilon_1-1)}	&0	&0	&0	&0	&0\\
\hline
0	&0	&0	&	\frac{\epsilon_2}{\epsilon_2-1}	& \frac{1}{1-\epsilon_2}	&0	&0	&0 \\
0	&0	&0	&	\frac{1}{\xi_2(1-\epsilon_2)} 	& \frac{1}{\xi_2(\epsilon_2-1)}	&0	&0	&0\\
\hline
0	&0	&0	&0	&0	&	\frac{\epsilon_3}{\epsilon_3-1}	& \frac{1}{1-\epsilon_3} &0\\
0	&0	&0	&0	&0	&	\frac{1}{\xi_3(1-\epsilon_3)}		& \frac{1}{\xi_3(\epsilon_3-1)}	&0\\
\hline
0	&0	&0	&0	&0	&0	&0	& 1\\
\end{array}\right].
\end{equation*}
Moreover for $1 \leq i \leq 3$,
\begin{equation}\label{xi/epsilon-1}		
\begin{tabular}{lll}
$\tfrac{\epsilon_i}{\epsilon_i-1} = 	\tfrac{(1-q^{i})(1-s^*q^{i+5})}{(1-q^4)(1-s^*q^{2i+1})}$,  
& \qquad $\begin{matrix} & \\ & \end{matrix}$ &
$\tfrac{1}{1-\epsilon_i} =  \tfrac{q^4(1-q^{i-4})(1-s^*q^{i+1})}{(q^4-1)(1-s^*q^{2i+1})},$\\

$\tfrac{1}{\xi_i(1-\epsilon_i)} = \tfrac{q^{3+i}}{(q^4-1)(1-s^*q^{2i+1})},$ 
& \qquad $\begin{matrix} & \\ & \end{matrix}$& 
$\tfrac{1}{\xi_i(\epsilon_i-1)}  = \tfrac{q^{3+i}}{(1-q^4)(1-s^*q^{2i+1})}.$
\end{tabular}
\end{equation}

\end{example}

\bigskip
\begin{example}\label{ex:tm(C<->BCa)}
The transition matrix from $\mathcal{C}$ to $\wt{\mcal{B}}_{alt}$ is
\begin{equation*}\label{ex:tm(C->BCa)}
\left[
\begin{array}{cc|cc|cc|cc}
1	&	\zeta_0\tau_0	&0	&0	&0	&0	&0	&0	\\
1	&	\zeta_0	&0	&0	&0	&0	&0	&0	\\
\hline
0	&0	&	1	&	\zeta_1\tau_1	&0	&0	&0	&0	\\
 0	&0	&	1	&	\zeta_1	&0	&0	&0	&0	\\
\hline
0	&0	&0	&0	&	1	&	\zeta_2\tau_2	&0	& 0	\\
 0	&0	&0	&0	&	1	&	\zeta_2	&0	&0	\\
\hline
0	&0	&0	&0	&0	&0	&	1	&	\zeta_3\tau_3\\
0	&0	&0	&0	&0	&0	&	1	&	\zeta_3
\end{array}
\right],
\end{equation*}
where $\{\zeta_i\}^3_{i=0}$ and $\{\tau_i\}^3_{i=0}$ are from (\ref{zeta;tau}).
Moreover,
\begin{equation}\label{zeta/tau}
\zeta_i \tau_i = q^{-i}s^*(1-r_1q^{i+1})(1-r_2q^{i+1}), \qquad (0 \leq i \leq 3)
\end{equation}

\noindent
The transition matrix from $\wt{\mcal{B}}_{alt}$ to $\mathcal{C}$ is
\begin{equation*}\label{ex:tm(BCa->C)}
\left[
\begin{array}{cc|cc|cc|cc}
\frac{1}{1-\tau_0}	&	\frac{\tau_0}{\tau_0-1}	&0	&0&0	&0	&0	&0\\
\frac{1}{\zeta_0(\tau_0-1)}	&	\frac{1}{\zeta_0(1-\tau_0)}	&0	&0	&0	&0	&0	&0\\
\hline
0&0	&	\frac{1}{1-\tau_1}	&	\frac{\tau_1}{\tau_1-1}	&0	&0	&0	&0\\
0&0	&	\frac{1}{\zeta_1(\tau_1-1)}&	\frac{1}{\zeta_1(1-\tau_1)}&0	&0	&0	&0\\
\hline
0&0	&0	&0	&\frac{1}{1-\tau_2}	&\frac{\tau_2}{\tau_2-1}	&0	&0\\
0&0	&0	&0	&\frac{1}{\zeta_2(\tau_2-1)}	&	\frac{1}{\zeta_2(1-\tau_2)}	&0	&0\\
\hline
0&0	&0	&0	&0	&0	&\frac{1}{1-\tau_3}	&\frac{\tau_3}{\tau_3-1}\\
0&0	&0	&0	&0	&0	&\frac{1}{\zeta_3(\tau_3-1)}	&\frac{1}{\zeta_3(1-\tau_3)}\\
\end{array}
\right].
\end{equation*}
Moreover for $0 \leq i \leq 3$,
\begin{equation}\label{zeta/tau-1}
\begin{tabular}{lll}
$\tfrac{1}{1-\tau_i} = \tfrac{(r_1-s^*q^{i+1})(r_2-s^*q^{i+1})}{(r_1r_2-s^*)(1-s^*q^{2i+2})}$,
& \qquad $\begin{matrix} & \\ & \end{matrix}$&
$\tfrac{\tau_i}{\tau_i-1} = \tfrac{s^*(1-r_1q^{i+1})(1-r_2q^{i+1})}{(s^*-r_1r_2)(1-s^*q^{2i+2})}$,\\
$\tfrac{1}{\zeta_i(\tau_i-1)}	  =	\tfrac{q^i}{(s^*-r_1r_2)(1-s^*q^{2i+2})},$
& \qquad $\begin{matrix} & \\ & \end{matrix}$ &
$\tfrac{1}{\zeta_i(1-\tau_i)}	 = 	\tfrac{q^i}{(r_1r_2-s^*)(1-s^*q^{2i+2})}.$
\end{tabular}
\end{equation}
\end{example}

\bigskip
\begin{example}\label{ex:tm(Bxa<->BCa)}
The transition matrix from $\mcal{B}_{alt}$ to $\wt{\mcal{B}}_{alt}$ is
\begin{equation*}\label{ex:tm(Bxa->BCa)}
\left[
\begin{array}{cc|cc|cc|cc}
1	&\zeta_0\tau_0 &0&0&0&0&0&0\\

\hline

\frac{\epsilon_1}{\epsilon_1-1}	
&\frac{\epsilon_1\zeta_0}{\epsilon_1-1}	
&\frac{1}{1-\epsilon_1}	
&\frac{\zeta_1\tau_1}{1-\epsilon_1}	
&0&0&0&0\\


\frac{1}{\xi_1(1-\epsilon_1)}	
&\frac{\zeta_0}{\xi_1(1-\epsilon_1)}	
&\frac{1}{\xi_1(\epsilon_1-1)}	
&\frac{\zeta_1\tau_1}{\xi_1(\epsilon_1-1)}	
&0&0&0&0\\

\hline

0&0&\frac{\epsilon_2}{\epsilon_2-1}	
&\frac{\epsilon_2\zeta_1}{\epsilon_2-1}	
&\frac{1}{1-\epsilon_2}	
&\frac{\zeta_2\tau_2}{1-\epsilon_2}	
&0&0\\


0&0&\frac{1}{\xi_2(1-\epsilon_2)}	
&\frac{\zeta_1}{\xi_2(1-\epsilon_2)}	
&\frac{1}{\xi_2(\epsilon_2-1)}	
&\frac{\zeta_2\tau_2}{\xi_2(\epsilon_2-1)}	
&0&0\\

\hline

0&0&0&0&\frac{\epsilon_3}{\epsilon_3-1}	
&\frac{\epsilon_3\zeta_2}{\epsilon_3-1}	
&\frac{1}{1-\epsilon_3}	
&\frac{\zeta_3\tau_3}{1-\epsilon_3}	
\\


0&0&0&0&\frac{1}{\xi_3(1-\epsilon_3)}	
&\frac{\zeta_2}{\xi_3(1-\epsilon_3)}	
&\frac{1}{\xi_3(\epsilon_3-1)}	
&\frac{\zeta_3\tau_3}{\xi_3(\epsilon_3-1)}	
\\

\hline

0&0&0&0&0&0&
1 & \zeta_3
\end{array}
\right],
\end{equation*}
where $\{\xi_i\}^3_{i=1}, \{\epsilon_i\}^3_{i=1}$ are from (\ref{xi;epsilon})
and $\{\zeta_i\}^3_{i=0}, \{\tau_i\}^3_{i=0}$ are from (\ref{zeta;tau}).
Moreover

\begin{equation}\label{xi/ep/zeta/tau}
\left.
\begin{tabular}{lll}
$\tfrac{\epsilon_i}{\epsilon_i-1} 
= \tfrac{(1-q^i)(1-s^*q^{i+5})}{(1-q^4)(1-s^*q^{2i+1})},$
& \qquad $\begin{matrix} & \\ & \end{matrix}$& 
$\tfrac{\epsilon_i\zeta_{i-1}}{\epsilon_i-1}
= \tfrac{q^{-i+1}(1-q^i)(1-s^*q^{i+5})(r_1-s^*q^i)(r_2-s^*q^i)}{(1-q^4)(1-s^*q^{2i+1})},$ \\
$\tfrac{1}{\xi_i(1-\epsilon_i)} 
= \tfrac{q^{3+i}}{(q^4-1)(1-s^*q^{2i+1})},$
& \qquad $\begin{matrix} & \\ & \end{matrix}$& 
$\tfrac{\zeta_{i-1}}{\xi_i(1-\epsilon_i)} 
= \tfrac{q^{4}(r_1-s^*q^i)(r_2-s^*q^i)}{(q^4-1)(1-s^*q^{2i+1})},$ \\
$\tfrac{1}{1-\epsilon_i} 
= \tfrac{q^4(1-q^{i-4})(1-s^*q^{i+1})}{(q^4-1)(1-s^*q^{2i+1})},$
& \qquad $\begin{matrix} & \\ & \end{matrix}$& 
$\tfrac{\zeta_i\tau_i}{1-\epsilon_i}
= \tfrac{s^*q^{4-i}(1-q^{i-4})(1-s^*q^{i+1})(1-r_1q^{i+1})(1-r_2q^{i+1})}
			{(q^4-1)(1-s^*q^{2i+1})}, $ \\
$\tfrac{1}{\xi_i(\epsilon_i-1)} 
= \tfrac{q^{3+i}}{(1-q^4)(1-s^*q^{2i+1})}, $ 
& \qquad $\begin{matrix} & \\ & \end{matrix}$& 
$\tfrac{\zeta_i\tau_i}{\xi_i(\epsilon_i-1)} 
= \tfrac{s^*q^{3}(1-r_1q^{i+1})(1-r_2q^{i+1})}{(1-q^4)(1-s^*q^{2i+1})}. $
\end{tabular}
\right\}
\end{equation}

\medskip \noindent
The transition matrix from $\wt{\mcal{B}}_{alt}$ to $\mcal{B}_{alt}$ is
\begin{equation*}\label{ex:tm(BCa->Bxa)}
\left[
\begin{array}{c|cc|cc|cc|c}
\frac{1}{1-\tau_0}
&\frac{\tau_0}{\tau_0-1}
&\frac{\tau_0\xi_1}{\tau_0-1}
&0&0&0&0&0\\

\frac{1}{\zeta_0(\tau_0-1)}
&\frac{1}{\zeta_0(1-\tau_0)}
&\frac{\xi_1}{\zeta_0(1-\tau_0)}
&0&0&0&0&0\\

\hline

0&\frac{1}{1-\tau_1}
&\frac{\xi_1\epsilon_1}{1-\tau_1}
&\frac{\tau_1}{\tau_1-1}
&\frac{\tau_1\xi_2}{\tau_1-1}
&0&0&0\\

0&\frac{1}{\zeta_1(\tau_1-1)}
&\frac{\xi_1\epsilon_1}{\zeta_1(\tau_1-1)}
&\frac{1}{\zeta_1(1-\tau_1)}
&\frac{\xi_2}{\zeta_1(1-\tau_1)}
&0&0&0\\

\hline

0&0&0&\frac{1}{1-\tau_2}
&\frac{\xi_2\epsilon_2}{1-\tau_2}
&\frac{\tau_2}{\tau_2-1}
&\frac{\tau_2\xi_3}{\tau_2-1}
&0\\

0&0&0&\frac{1}{\zeta_2(\tau_2-1)}
&\frac{\xi_2\epsilon_2}{\zeta_2(\tau_2-1)}
&\frac{1}{\zeta_2(1-\tau_2)}
&\frac{\xi_3}{\zeta_2(1-\tau_2)}
&0\\

\hline

0&0&0&0&0&\frac{1}{1-\tau_3}
&\frac{\xi_3\epsilon_3}{1-\tau_3}
&\frac{\tau_3}{\tau_3-1}
\\

0&0&0&0&0&\frac{1}{\zeta_3(\tau_3-1)}
&\frac{\xi_3\epsilon_3}{\zeta_3(\tau_3-1)}
&\frac{1}{\zeta_3(1-\tau_3)}
\\
\end{array}
\right],
\end{equation*}
where $\{\xi_i\}^3_{i=1}, \{\epsilon_i\}^3_{i=1}$ are from (\ref{xi;epsilon})
and $\{\zeta_i\}^3_{i=0}, \{\tau_i\}^3_{i=0}$ are from (\ref{zeta;tau}).
Moreover

\begin{equation}\label{zeta/tau/xi/ep}
\left.
\begin{tabular}{lll}
$\tfrac{1}{1-\tau_i} 
= \tfrac{(r_1-s^*q^{i+1})(r_2-s^*q^{i+1})}
			{(r_1r_2-s^*)(1-s^*q^{2i+2})},$
& \qquad $\begin{matrix} & \\ & \end{matrix}$& 
$\tfrac{\xi_i\epsilon_i}{1-\tau_i}
= \tfrac{q^{-i-3}(1-q^i)(1-s^*q^{i+5})(r_1-s^*q^{i+1})(r_2-s^*q^{i+1})}
	{(1-s^*q^{2i+2})(r_1r_2-s^*)},$ \\
$\tfrac{1}{\zeta_i(\tau_i-1)} 
= \tfrac{q^i}{(s^*-r_1r_2)(1-s^*q^{2i+2})}, $
& \qquad $\begin{matrix} & \\ & \end{matrix}$& 
$\tfrac{\xi_i\epsilon_i}{\zeta_i(\tau_i-1)} 
= \tfrac{q^{-3}(1-q^i)(1-s^*q^{i+5})}{(1-s^*q^{2i+2})(s^*-r_1r_2)},$ \\
$\tfrac{\tau_i}{\tau_i-1} 
 =  \tfrac{s^*(1-r_1q^{i+1})(1-r_2q^{i+1})}
				{(s^*-r_1r_2)(1-s^*q^{2i+2})},$
& \qquad $\begin{matrix} & \\ & \end{matrix}$& 
$\tfrac{\tau_i\xi_{i+1}}{\tau_i-1} 
= \tfrac{s^*q^{-i}(1-q^{i-3})(1-s^*q^{i+2})(1-r_1q^{i+1})(1-r_2q^{i+1})}
	{(1-s^*q^{2i+2})(s^*-r_1r_2)},$ \\
$\tfrac{1}{\zeta_i(1-\tau_i)} 
=	\tfrac{q^i}{(r_1r_2-s^*)(1-s^*q^{2i+2})}, $
& \qquad $\begin{matrix} & \\ & \end{matrix}$&
$\tfrac{\xi_{i+1}}{\zeta_i(1-\tau_i)} 
= \tfrac{(1-q^{i-3})(1-s^*q^{i+2})}{(1-s^*q^{2i+2})(r_1r_2-s^*)}.$
\end{tabular}
\right\}
\end{equation}
\end{example}

\medskip
\begin{example}\label{ex:tm(Bx<->Bxa)}
The transition matrix from $\mcal{B}$ to $\mcal{B}_{alt}$ is
\begin{equation*}\label{ex:tm(Bx->Bxa)}
\left[
\begin{array}{c|cc|cc|cc|c}
1 &0&0&0&0&0&0&0\\
\hline
0& 1 &0&0&0&0&0&0\\
0&0&0&1&0&0&0&0\\
\hline
0&0&0&0&0&1&0&0\\
0&0&0&0&0&0&0&1\\
\hline
0&0&1&0&0&0&0&0\\
0&0&0&0&1&0 &0&0\\
\hline
0&0&0&0&0&0&1&0
\end{array}
\right].
\end{equation*}

\noindent
The transition matrix from $\mcal{B}_{alt}$ to $\mcal{B}$ is
\begin{equation*}\label{ex:tm(Bxa->Bx)}
\left[
\begin{array}{c|cc|cc|cc|c}
1&0&0&0&0&0&0&0\\
\hline
0&1&0&0&0&0&0&0\\
0&0&0&0&0&1&0&0\\
\hline
0&0&1&0&0&0&0&0\\
0&0&0&0&0&0&1&0\\
\hline
0&0&0&1&0&0&0&0\\
0&0&0&0&0&0&0&1\\
\hline
0&0&0&0&1&0&0&0
\end{array}
\right].
\end{equation*}
\end{example}

\bigskip
\begin{example}\label{ex:tm(BC<->BCa)}
The transition matrix from $\wt{\mathcal{B}}$ to $\wt{\mcal{B}}_{alt}$ is
\begin{equation*}\label{ex:tm(BC->BCa)}
\left[
\begin{array}{cc|cc|cc|cc}
1&0&0&0&0&0&0&0\\
0&0&1&0&0&0&0&0\\
\hline
0&0&0&0&1&0&0&0\\
0&0&0&0&0&0&1&0\\
\hline
0&1&0&0&0&0&0&0\\
0&0&0&1&0&0&0&0\\
\hline
0&0&0&0&0&1&0&0\\
0&0&0&0&0&0&0&1
\end{array}
\right].
\end{equation*}
The transition matrix from $\wt{\mcal{B}}_{alt}$ to $\wt{\mathcal{B}}$ is
\begin{equation*}\label{ex:tm(BCa->BC)}
\left[
\begin{array}{cc|cc|cc|cc}
1&0&0&0&0&0&0&0\\
0&0&0&0&1&0&0&0\\
\hline
0&1&0&0&0&0&0&0\\
0&0&0&0&0&1&0&0\\
\hline
0&0&1&0&0&0&0&0\\
0&0&0&0&0&0&1&0\\
\hline
0&0&0&1&0&0&0&0\\
0&0&0&0&0&0&0&1\\
\end{array}
\right].
\end{equation*}
\end{example}

\medskip
\begin{example}\label{ex:tm(C<->Bx)}
The transition matrix from $\mathcal{C}$ to $\mcal{B}$ is
\begin{eqnarray*}
\left[
\begin{array}{c|cc|cc|cc|c}
1	&0	&0	&0	&0	&0	&0	&0\\
\hline
0	& 1	&0&0&0& \xi_1	&0	&0\\
0	& 1	&0&0&0& \xi_1\epsilon_1	&0	&0	\\
\hline
0	&0 	& 1	&0&0&0& \xi_2	&0\\
0	&0 	& 1	&0&0&0& \xi_2\epsilon_2	&0\\
\hline
0	&0	&0	& 1	&0&0&0& \xi_3		\\
0	&0	&0	& 1	&0&0&0& \xi_3\epsilon_3	\\
\hline
0	&0	&0	&0	& 1	&0&0&0\\
\end{array}
\right],
\end{eqnarray*}
where $\{\xi_i\}^3_{i=1}$ and $\{\xi_i\epsilon_i\}^3_{i=1}$ 
are from (\ref{xi;epsilon}) and (\ref{xi/epsilon}), respectively.

\medskip \noindent
The transition matrix from $\mcal{B}$ to $\mathcal{C}$ is 
\begin{align*}\label{ex:tm(Bx->C)}
\left[
\begin{array}{c|cc|cc|cc|c}
1 	&0	&0	&0	&0	&0	&0	&0\\
\hline
0	&\frac{\epsilon_1}{\epsilon_1-1} & \frac{1}{1-\epsilon_1}	&0	&0	&0	&0	&0\\ 
0	&0	&0	&	\frac{\epsilon_2}{\epsilon_2-1}	& \frac{1}{1-\epsilon_2}	&0	&0	&0 \\	
\hline
0	&0	&0	&0	&0	&	\frac{\epsilon_3}{\epsilon_3-1}	& \frac{1}{1-\epsilon_3} &0\\
0	&0	&0	&0	&0	&0	&0	& 1\\
\hline
0	&\frac{1}{\xi_1(1-\epsilon_1)} 	& \frac{1}{\xi_1(\epsilon_1-1)}	&0	&0	&0	&0	&0\\
0	&0	&0	&	\frac{1}{\xi_2(1-\epsilon_2)} 	& \frac{1}{\xi_2(\epsilon_2-1)}	&0	&0	&0\\
\hline
0	&0	&0	&0	&0	&	\frac{1}{\xi_3(1-\epsilon_3)}		& \frac{1}{\xi_3(\epsilon_3-1)}	&0
\end{array}
\right].
\end{align*}
Each entry is from (\ref{xi/epsilon-1}).
\end{example}

\begin{example}
The transition matrix from $\mathcal{C}$ to $\wt{\mathcal{B}}$ is
\begin{align*}
\left[
\begin{array}{cccc|cccc}
1	&	0&0 &0 &\zeta_0\tau_0	&0	&0&0\\
1	&	0&0 &0 &\zeta_0	&0	&0	&0	\\
\hline
0	&	1	&0 &0 &0 &\zeta_1\tau_1	&0	&0\\
0	&	1	&0 &0 &0 &\zeta_1	&0	&0\\
\hline
0	&	0	&	1	&0	&0 &0 &\zeta_2\tau_2	&0	\\
0	&	0	&	1	&0	&0 &0 &\zeta_2	&0\\
\hline
0	&	0	&0	&	1	&0 &0 &0 &	\zeta_3\tau_3\\
0	&	0	&0	&	1	&0 &0 &0 &	\zeta_3
\end{array}
\right],
\end{align*}
where $\{\zeta_i\}^3_{i=0}$ and $\{\zeta_i\tau_i\}^3_{i=0}$ 
are from (\ref{zeta;tau}) and (\ref{zeta/tau}), respectively.

\medskip \noindent
The transition matrix from $\wt{\mathcal{B}}$ to $\mathcal{C}$ is 
\begin{equation*}\label{ex:tm(BC->C)}
\left[
\begin{array}{cc|cc|cc|cc}
\frac{1}{1-\tau_0}	&	\frac{\tau_0}{\tau_0-1}	&0	&0	&0	&0	&0	&0\\
0&0	&	\frac{1}{1-\tau_1}	&	\frac{\tau_1}{\tau_1-1}	&0	&0	&0	&0\\
0&0	&0	&0	&\frac{1}{1-\tau_2}	&\frac{\tau_2}{\tau_2-1}	&0	&0\\
0&0	&0	&0	&0	&0	&\frac{1}{1-\tau_3}	&\frac{\tau_3}{\tau_3-1}\\
\hline
\frac{1}{\zeta_0(\tau_0-1)}	&	\frac{1}{\zeta_0(1-\tau_0)}	&0	&0	&0	&0	&0	&0\\
0&0	&	\frac{1}{\zeta_1(\tau_1-1)}&	\frac{1}{\zeta_1(1-\tau_1)}&0	&0	&0	&0\\
0&0	&0	&0	&\frac{1}{\zeta_2(\tau_2-1)}	&	\frac{1}{\zeta_2(1-\tau_2)}	&0	&0\\
0&0	&0	&0	&0	&0	&\frac{1}{\zeta_3(\tau_3-1)}	&\frac{1}{\zeta_3(1-\tau_3)}\\
\end{array}
\right].
\end{equation*}
Each entry is from (\ref{zeta/tau-1}).
\end{example}

\begin{example}\label{ex:tm(Bx<->BCa)}
The transition matrix from $\mcal{B}$ to $\wt{\mcal{B}}_{alt}$ is
\begin{equation*}\label{ex:tm(Bx->BCa)}
\left[
\begin{array}{cc|cc|cc|cc}
1	&\zeta_0\tau_0	
&0&0&0&0&0&0
\\
\frac{\epsilon_1}{\epsilon_1-1}	
&\frac{\epsilon_1\zeta_0}{\epsilon_1-1}	
&\frac{1}{1-\epsilon_1}	
&\frac{\zeta_1\tau_1}{1-\epsilon_1}	
&0&0&0&0
\\
0&0&\frac{\epsilon_2}{\epsilon_2-1}	
&\frac{\epsilon_2\zeta_1}{\epsilon_2-1}	
&\frac{1}{1-\epsilon_2}	
&\frac{\zeta_2\tau_2}{1-\epsilon_2}	
&0&0
\\
0&0&0&0&\frac{\epsilon_3}{\epsilon_3-1}	
&\frac{\epsilon_3\zeta_2}{\epsilon_3-1}	
&\frac{1}{1-\epsilon_3}	
&\frac{\zeta_3\tau_3}{1-\epsilon_3}
\\
0&0&0&0&0&0&
1 & \zeta_3
\\
\hline
\frac{1}{\xi_1(1-\epsilon_1)}	
&\frac{\zeta_0}{\xi_1(1-\epsilon_1)}	
&\frac{1}{\xi_1(\epsilon_1-1)}	
&\frac{\zeta_1\tau_1}{\xi_1(\epsilon_1-1)}	
&0&0&0&0
\\
0&0&\frac{1}{\xi_2(1-\epsilon_2)}	
&\frac{\zeta_1}{\xi_2(1-\epsilon_2)}	
&\frac{1}{\xi_2(\epsilon_2-1)}	
&\frac{\zeta_2\tau_2}{\xi_2(\epsilon_2-1)}	
&0&0	
\\
0&0&0&0&\frac{1}{\xi_3(1-\epsilon_3)}	
&\frac{\zeta_2}{\xi_3(1-\epsilon_3)}	
&\frac{1}{\xi_3(\epsilon_3-1)}	
&\frac{\zeta_3\tau_3}{\xi_3(\epsilon_3-1)}	
\end{array}
\right].
\end{equation*}
Each entry is from (\ref{xi/ep/zeta/tau}).

\medskip \noindent
The transition matrix from $\wt{\mcal{B}}_{alt}$ to $\mcal{B}$ is 
\begin{equation*}\label{ex:tm(BCa->Bx)}
\left[
\begin{array}{ccccc|ccc}
\frac{1}{1-\tau_0}
&\frac{\tau_0}{\tau_0-1}
&0&0&0&\frac{\tau_0\xi_1}{\tau_0-1}
&0&0
\\
\frac{1}{\zeta_0(\tau_0-1)}
&\frac{1}{\zeta_0(1-\tau_0)}
&0&0&0&\frac{\xi_1}{\zeta_0(1-\tau_0)}
&0&0
\\
\hline
0&\frac{1}{1-\tau_1}
&\frac{\tau_1}{\tau_1-1}
&0&0&\frac{\xi_1\epsilon_1}{1-\tau_1}
&\frac{\tau_1\xi_2}{\tau_1-1}
&0
\\
0&\frac{1}{\zeta_1(\tau_1-1)}
&\frac{1}{\zeta_1(1-\tau_1)}
&0&0&\frac{\xi_1\epsilon_1}{\zeta_1(\tau_1-1)}
&\frac{\xi_2}{\zeta_1(1-\tau_1)}
&0
\\
\hline
0&0&\frac{1}{1-\tau_2}
&\frac{\tau_2}{\tau_2-1}
&0&0&\frac{\xi_2\epsilon_2}{1-\tau_2}
&\frac{\tau_2\xi_3}{\tau_2-1}
\\
0&0&\frac{1}{\zeta_2(\tau_2-1)}
&\frac{1}{\zeta_2(1-\tau_2)}
&0&0&\frac{\xi_2\epsilon_2}{\zeta_2(\tau_2-1)}
&\frac{\xi_3}{\zeta_2(1-\tau_2)}
\\
\hline
0&0&0&\frac{1}{1-\tau_3}
&\frac{\tau_3}{\tau_3-1}
&0&0&\frac{\xi_3\epsilon_3}{1-\tau_3}
\\
0&0&0&\frac{1}{\zeta_3(\tau_3-1)}
&\frac{1}{\zeta_3(1-\tau_3)}
&0&0&\frac{\xi_3\epsilon_3}{\zeta_3(\tau_3-1)}
\\
\end{array}
\right].
\end{equation*}
Each entry is from (\ref{zeta/tau/xi/ep}).
\end{example}

\begin{example}\label{ex:tm(Bxa<->BC)}
The transition from $\mcal{B}_{alt}$ to $\wt{\mathcal{B}}$ is
\begin{equation*}\label{ex:tm(Bxa->BC)}
\left[
\begin{array}{cccc|cccc}
1	&0&0&0&\zeta_0\tau_0	
&0&0&0
\\ \hline
\frac{\epsilon_1}{\epsilon_1-1}	
&\frac{1}{1-\epsilon_1}
&0&0&\frac{\epsilon_1\zeta_0}{\epsilon_1-1}		
&\frac{\zeta_1\tau_1}{1-\epsilon_1}	
&0&0
\\
\frac{1}{\xi_1(1-\epsilon_1)}	
&\frac{1}{\xi_1(\epsilon_1-1)}	
&0&0&\frac{\zeta_0}{\xi_1(1-\epsilon_1)}	
&\frac{\zeta_1\tau_1}{\xi_1(\epsilon_1-1)}	
&0&0
\\
\hline
0&\frac{\epsilon_2}{\epsilon_2-1}	
&\frac{1}{1-\epsilon_2}	
&0&0&\frac{\epsilon_2\zeta_1}{\epsilon_2-1}	
&\frac{\zeta_2\tau_2}{1-\epsilon_2}	
&0
\\
0&\frac{1}{\xi_2(1-\epsilon_2)}	
&\frac{1}{\xi_2(\epsilon_2-1)}	
&0&0&\frac{\zeta_1}{\xi_2(1-\epsilon_2)}	
&\frac{\zeta_2\tau_2}{\xi_2(\epsilon_2-1)}	
&0
\\ \hline
0&0&\frac{\epsilon_3}{\epsilon_3-1}	
&\frac{1}{1-\epsilon_3}
&0&0&\frac{\epsilon_3\zeta_2}{\epsilon_3-1}		
&\frac{\zeta_3\tau_3}{1-\epsilon_3}	
\\
0&0&\frac{1}{\xi_3(1-\epsilon_3)}	
&\frac{1}{\xi_3(\epsilon_3-1)}	
&0&0&\frac{\zeta_2}{\xi_3(1-\epsilon_3)}	
&\frac{\zeta_3\tau_3}{\xi_3(\epsilon_3-1)}	
\\ \hline
0&0&0&
1 &0&0&0& \zeta_3
\end{array}
\right].
\end{equation*}
Each entry is from (\ref{xi/ep/zeta/tau}).

\medskip \noindent
The transition matrix from $\wt{\mathcal{B}}$ to $\mcal{B}_{alt}$ is
\begin{equation*}\label{ex:tm(BC->Bxa)}
\left[
\begin{array}{c|cc|cc|cc|c}
\frac{1}{1-\tau_0}
&\frac{\tau_0}{\tau_0-1}
&\frac{\tau_0\xi_1}{\tau_0-1}
&0&0&0&0&0
\\
0&\frac{1}{1-\tau_1}
&\frac{\xi_1\epsilon_1}{1-\tau_1}
&\frac{\tau_1}{\tau_1-1}
&\frac{\tau_1\xi_2}{\tau_1-1}
&0&0&0
\\
0&0&0&\frac{1}{1-\tau_2}
&\frac{\xi_2\epsilon_2}{1-\tau_2}
&\frac{\tau_2}{\tau_2-1}
&\frac{\tau_2\xi_3}{\tau_2-1}
&0
\\
0&0&0&0&0&\frac{1}{1-\tau_3}
&\frac{\xi_3\epsilon_3}{1-\tau_3}
&\frac{\tau_3}{\tau_3-1}
\\
\hline
\frac{1}{\zeta_0(\tau_0-1)}
&\frac{1}{\zeta_0(1-\tau_0)}
&\frac{\xi_1}{\zeta_0(1-\tau_0)}
&0&0&0&0&0
\\
0&\frac{1}{\zeta_1(\tau_1-1)}
&\frac{\xi_1\epsilon_1}{\zeta_1(\tau_1-1)}
&\frac{1}{\zeta_1(1-\tau_1)}
&\frac{\xi_2}{\zeta_1(1-\tau_1)}
&0&0&0
\\
0&0&0&\frac{1}{\zeta_2(\tau_2-1)}
&\frac{\xi_2\epsilon_2}{\zeta_2(\tau_2-1)}
&\frac{1}{\zeta_2(1-\tau_2)}
&\frac{\xi_3}{\zeta_2(1-\tau_2)}
&0
\\
0&0&0&0&0&\frac{1}{\zeta_3(\tau_3-1)}
&\frac{\xi_3\epsilon_3}{\zeta_3(\tau_3-1)}
&\frac{1}{\zeta_3(1-\tau_3)}
\end{array}
\right].
\end{equation*}
Each entry is from (\ref{zeta/tau/xi/ep}).
\end{example}

\medskip
\begin{example}\label{ex:tm(Bx<->BC)}
The transition matrix from $\mcal{B}$ and $\wt{\mathcal{B}}$
is
\begin{equation*}\label{ex:tm(Bx->BC)}
\left[
\begin{array}{cccc|cccc}
1	&0&0&0&\zeta_0\tau_0	
&0&0&0
\\
\frac{\epsilon_1}{\epsilon_1-1}	
&\frac{1}{1-\epsilon_1}
&0&0&\frac{\epsilon_1\zeta_0}{\epsilon_1-1}		
&\frac{\zeta_1\tau_1}{1-\epsilon_1}	
&0&0
\\
0&\frac{\epsilon_2}{\epsilon_2-1}	
&\frac{1}{1-\epsilon_2}	
&0&0&\frac{\epsilon_2\zeta_1}{\epsilon_2-1}	
&\frac{\zeta_2\tau_2}{1-\epsilon_2}	
&0
\\
0&0&\frac{\epsilon_3}{\epsilon_3-1}	
&\frac{1}{1-\epsilon_3}
&0&0&\frac{\epsilon_3\zeta_2}{\epsilon_3-1}		
&\frac{\zeta_3\tau_3}{1-\epsilon_3}	
\\
0&0&0&
1 &0&0&0& \zeta_3
\\
\hline
\frac{1}{\xi_1(1-\epsilon_1)}	
&\frac{1}{\xi_1(\epsilon_1-1)}	
&0&0&\frac{\zeta_0}{\xi_1(1-\epsilon_1)}	
&\frac{\zeta_1\tau_1}{\xi_1(\epsilon_1-1)}	
&0&0
\\
0&\frac{1}{\xi_2(1-\epsilon_2)}	
&\frac{1}{\xi_2(\epsilon_2-1)}	
&0&0&\frac{\zeta_1}{\xi_2(1-\epsilon_2)}	
&\frac{\zeta_2\tau_2}{\xi_2(\epsilon_2-1)}	
&0
\\
0&0&\frac{1}{\xi_3(1-\epsilon_3)}	
&\frac{1}{\xi_3(\epsilon_3-1)}	
&0&0&\frac{\zeta_2}{\xi_3(1-\epsilon_3)}	
&\frac{\zeta_3\tau_3}{\xi_3(\epsilon_3-1)}	
\end{array}
\right].
\end{equation*}
Each entry is from (\ref{xi/ep/zeta/tau}).

\medskip \noindent
The transition matrix from $\wt{\mathcal{B}}$ to $\mcal{B}$ is 
\begin{equation*}\label{ex:tm(BC->Bx)}
\left[
\begin{array}{ccccc|ccc}
\frac{1}{1-\tau_0}
&\frac{\tau_0}{\tau_0-1}
&0&0&0&\frac{\tau_0\xi_1}{\tau_0-1}
&0&0
\\
0&\frac{1}{1-\tau_1}
&\frac{\tau_1}{\tau_1-1}
&0&0&\frac{\xi_1\epsilon_1}{1-\tau_1}
&\frac{\tau_1\xi_2}{\tau_1-1}
&0
\\
0&0&\frac{1}{1-\tau_2}
&\frac{\tau_2}{\tau_2-1}
&0&0&\frac{\xi_2\epsilon_2}{1-\tau_2}
&\frac{\tau_2\xi_3}{\tau_2-1}
\\
0&0&0&\frac{1}{1-\tau_3}
&\frac{\tau_3}{\tau_3-1}
&0&0&\frac{\xi_3\epsilon_3}{1-\tau_3}
\\
\hline
\frac{1}{\zeta_0(\tau_0-1)}
&\frac{1}{\zeta_0(1-\tau_0)}
&0&0&0&\frac{\xi_1}{\zeta_0(1-\tau_0)}
&0&0
\\
0&\frac{1}{\zeta_1(\tau_1-1)}
&\frac{1}{\zeta_1(1-\tau_1)}
&0&0&\frac{\xi_1\epsilon_1}{\zeta_1(\tau_1-1)}
&\frac{\xi_2}{\zeta_1(1-\tau_1)}
&0
\\
0&0&\frac{1}{\zeta_2(\tau_2-1)}
&\frac{1}{\zeta_2(1-\tau_2)}
&0&0&\frac{\xi_2\epsilon_2}{\zeta_2(\tau_2-1)}
&\frac{\xi_3}{\zeta_2(1-\tau_2)}
\\
0&0&0&\frac{1}{\zeta_3(\tau_3-1)}
&\frac{1}{\zeta_3(1-\tau_3)}
&0&0&\frac{\xi_3\epsilon_3}{\zeta_3(\tau_3-1)}
\\
\end{array}
\right].
\end{equation*}
Each entry is from (\ref{zeta/tau/xi/ep}).
\end{example}
\noindent
We are done with display of transition matrices.

\medskip \noindent
For the rest of this section, we do following.
Referring to (\ref{5maps}) and (\ref{5bases}),
we give the matrix representing each map in (\ref{5maps})
relative to each basis in (\ref{5bases}).
We start with $A$. Recall the formulae
$b_i, c_i$ from (\ref{b_0;q-terms})--(\ref{c_d;q-terms}) and 
$\tb_i, \tc_i$ from (\ref{be_0;q-term})--(\ref{ga_(d-1);q-term}),
and further
$\bp_i, \cp_i$ from (\ref{b0;h,s,...})--(\ref{c(d-2);h,s,...}) and
$\tbp_i, \tcp_i$ from (\ref{tbp0;h,s,...})--(\ref{tc_(d-1);h,s,...}).

\begin{example}
The matrix representing $A$ relative to the basis $\mathcal{C}$ is 
\begin{equation*}\label{ex:[A]_C}
\footnotesize{
\left[
\begin{array}{cc|cc|cc|cc}
\ta_0 - b_0 + \tb_0	& b_0-\tb_0	& \tb_0	& 0	 & 0&0 &0 & 0\\
c_{1}-\tc_{0}		& \ta_0-c_{1}+\tc_{0}	&\tb_0-b_1	& b_1	&0 &0 &0 &0 \\
\hline
c_1	&	\tc_1-c_1	& \ta_1 - b_1 + \tb_1	& b_1-\tb_1 	& \tb_1	& 0 & 0&0\\
0	&	\tc_1	&	c_{2}-\tc_{1}	& \ta_1-c_{2}+\tc_{1}	&\tb_1-b_2	& b_2 &0&0\\
\hline
0&0 & c_2	&	\tc_2-c_2	& \ta_2 - b_2 + \tb_2	& b_2-\tb_2 	& \tb_2 & 0\\
0 &0 & 0&	\tc_2	& c_{3}-\tc_{2} & \ta_2-c_{3}+\tc_{2}	& \tb_2-b_3	& b_3 \\
\hline
0&0 &0 &0 & 	c_3	&	\tc_3-c_3	& \ta_3 - b_3 + \tb_3	& b_3-\tb_3\\ 
0 &0 &0 &0 & 0	&	\tc_3	& c_{4}-\tc_{3}		& \ta_3-c_{4}+\tc_{3}
\end{array}\right].
}
\end{equation*}

\ \\ \noindent
The matrix representing $A$ relative to the basis $\mcal{B}$ is
\begin{equation*}\label{ex:[A]_Bx}
\left[
\begin{array}{ccccc|ccc}
a_0	& b_0	&0	&0	&0	&0	&0	& 0\\
c_1		& a_1	& b_1	&0	&0	&0	&0	&0\\
0		& c_2	& a_2	& b_2	&0	&0	&0	&0\\
0		&0		& c_3	& a_3 	& b_3	&0	&0	&0\\
0		&0		&0		& c_4	& a_4	&0 	&0	&0\\
\hline
0	&0		&0		&0		&0		& \ap_0& \bp_0&0\\
0	&0		&0		&0		&0		& \cp_1	& \ap_1& \bp_1\\
0	&0		&0		&0		&0		&	0		& \cp_2	& \ap_{2} \\
\end{array}
\right].
\end{equation*}

\ \\ \noindent
The matrix representing $A$ relative to $\mcal{B}_{alt}$ is
\begin{equation*}\label{ex:[A]_Bxa}
\left[
\begin{array}{c|cc|cc|cc|c}
a_0 & b_0 & 0 & 0 &0&0&0&0\\
\hline
c_1 & a_1 & 0 & b_1&0&0&0&0\\
0 & 0 & \ap_0 & 0 & \bp_0 & 0 &0&0\\
\hline
0 & c_2 & 0 & a_2 & 0 & b_2 &0&0\\
0&0& \cp_1 & 0 & \ap_1 & 0 & \bp_1 &0\\
\hline
0&0& 0 & c_3 & 0 & a_3 & 0 & b_3 \\
0&0&0&0& \cp_2 & 0 & \ap_2 & 0\\
\hline
0&0&0&0& 0 & c_4 & 0 & a_4
\end{array}
\right].
\end{equation*}

\ \\ \noindent
The matrix representing $A$ relative to the basis $\wt{\mathcal{B}}$ is 
\begin{equation*}\label{ex:[A]_BC}
\left[
\begin{array}{cccc|cccc}
\ta_0	&	\tb_0	&0	&0	&0	&0	&0	&0\\
\tc_1	&	\ta_1	&	\tb_1	&0	&0	&0	&0	&0	\\
0		&	\tc_2	&	\ta_2	&	\tb_2	&0	&0	&0	&0\\
0		&0	&	\tc_3	&	\ta_3	&0	&0	&0	&0\\	
\hline
0		&0	&0	&0	&	\tap_0	& \tbp_0	&0	&0\\
0		&0	&0	&0	&	\tcp_1	&	\tap_1	&	\tbp_1	&0\\
0		&0	&0	&0	&0	&	\tcp_2	&	\tap_2	&	\tbp_2\\
0		&0	&0	&0	&0	&0	&	\tcp_3	&	\tap_3	
\end{array}
\right].
\end{equation*}

\ \\ \noindent
The matrix representing $A$ relative to $\wt{\mcal{B}}_{alt}$ is 
\begin{equation*}\label{ex:[A]_BCa}
\left[
\begin{array}{cc|cc|cc|cc}
\ta_0 & 0 & \tb_0 & 0&0&0&0&0\\
0 & \tap_0 & 0 & \tbp_0&0&0&0&0\\
\hline
\tc_1 & 0 & \ta_1 & 0 & \tb_1 & 0&0&0\\
0 & \tcp_1 & 0 & \tap_1 & 0 & \tbp_1&0&0\\
\hline
0&0& \tc_2 & 0 & \ta_2 & 0 & \tb_2 & 0\\
0&0& 0 & \tcp_2 & 0 & \tap_2 & 0 & \tbp_2\\
\hline
0&0&0&0& \tc_3 & 0 & \ta_3 & 0\\
0&0&0&0& 0 & \tcp_3 & 0 & \tap_3
\end{array}
\right].
\end{equation*}
\end{example}

\medskip
We are done with $A$. 
We now consider $A^*$. Recall from (\ref{theta^*_i}) the formula
\begin{equation}\label{tht^*(d=4)}
\theta^*_i  =  \theta^*_0+h^*(1-q^{i})(1-s^*q^{i+1})q^{-i},
\end{equation}
for $0 \leq  i \leq 4$.

\begin{example}
The matrix representing $A^*$ relative to the basis $\mathcal{C}$ is
\begin{equation*}\label{ex:[A*]_C}
\diag(\tht^*_0,\tht^*_1,\tht^*_1,
	\tht^*_2,\tht^*_2,\tht^*_3,\tht^*_3,\tht^*_4).
\end{equation*}

\ \\ \noindent
The matrix representing $A^*$ relative to the basis $\mcal{B}$ is 
\begin{equation*}\label{ex:[A*]_Bx}
\diag(\tht^*_0,\tht^*_1,\tht^*_2,\tht^*_3,\tht^*_4,
		\tht^*_1,\tht^*_2,\tht^*_3).
\end{equation*}

\ \\ \noindent
The matrix representing $A^*$ relative to the basis $\mcal{B}_{alt}$ is
\begin{equation*}\label{ex:[A*]_Bxa}
\diag(\tht^*_0,\tht^*_1,\tht^*_1,
			\tht^*_2,\tht^*_2,\tht^*_3,\tht^*_3,\tht^*_4).
\end{equation*}

\ \\ \noindent
The matrix representing $A^*$ relative to the basis $\wt{\mathcal{B}}$ is
\begin{equation*}\label{ex:[A*]_BC}
\left[
\begin{array}{cccc|cccc}
\frac{\tht^*_0-\tau_0\tht^*_1}{1-\tau_0}&0&0&0&
				\frac{\zeta_0\tau_0(\tht^*_0-\tht^*_1)}{1-\tau_0}&0&0&0\\
0&\frac{\tht^*_1-\tau_1\tht^*_2}{1-\tau_1}&0&0&0&
				\frac{\zeta_1\tau_1(\tht^*_1-\tht^*_2)}{1-\tau_1}&0&0\\
0&0&\frac{\tht^*_2-\tau_2\tht^*_3}{1-\tau_2}&0&0&0&
				\frac{\zeta_2\tau_2(\tht^*_2-\tht^*_3)}{1-\tau_2}&0\\
0&0&0&\frac{\tht^*_3-\tau_3\tht^*_4}{1-\tau_3}&0&0&0&
				\frac{\zeta_3\tau_3(\tht^*_3-\tht^*_4)}{1-\tau_3}\\
\hline
\frac{\tht^*_0-\tht^*_1}{\zeta_0(\tau_0-1)}&0&0&0&
				\frac{\tau_0\tht^*_0-\tht^*_1}{\tau_0-1}&0&0&0\\
0&\frac{\tht^*_1-\tht^*_2}{\zeta_1(\tau_1-1)}&0&0&0&
				\frac{\tau_1\tht^*_1-\tht^*_2}{\tau_1-1}&0&0\\
0&0&\frac{\tht^*_2-\tht^*_3}{\zeta_2(\tau_2-1)}&0&0&0&
				\frac{\tau_2\tht^*_2-\tht^*_3}{\tau_2-1}&0\\
0&0&0&\frac{\tht^*_3-\tht^*_4}{\zeta_3(\tau_3-1)}&0&0&0&
				\frac{\tau_3\tht^*_3-\tht^*_4}{\tau_3-1}
\end{array}
\right],
\end{equation*}
where 
$\{\tht^*_i\}^3_{i=0}$ are from (\ref{tht^*(d=4)}) and 
$\{\zeta_i\}^3_{i=0}, \{\tau_i\}^3_{i=0}$ are from (\ref{zeta;tau}).
Moreover,
\begin{equation}\label{tht*/zeta/tau}
\left.
\begin{tabular}{ll}
$\tfrac{\tht^*_i-\tau_i\tht^*_{i+1}}{1-\tau_i} =  
\tht^*_i+\tfrac{h^*s^*q^{-i-1}(1-q)(1-r_1q^{i+1})(1-r_2q^{i+1})}{s^*-r_1r_2},$
& $\begin{matrix} & \\ & \end{matrix}$ \\
$\tfrac{\zeta_i\tau_i(\tht^*_i-\tht^*_{i+1})}{1-\tau_i}=
\tfrac{h^*s^*q^{-2i-1}(1-q)(r_1-s^*q^{i+1})(r_2-s^*q^{i+1})(1-r_1q^{i+1})(1-r_2q^{i+1})}{s^*-r_1r_2},$ 
& $\begin{matrix} & \\ & \end{matrix}$ \\
$\tfrac{\tht^*_i-\tht^*_{i+1}}{\zeta_i(\tau_i-1)}
=\tfrac{h^*q^{-1}(1-q)}{r_1r_2-s^*}, $
& $\begin{matrix} & \\ & \end{matrix}$ \\
$\tfrac{\tau_i\tht^*_i-\tht^*_{i+1}}{\tau_i-1}
=\tht^*_{i+1}+\tfrac{h^*s^*q^{-i-1}(1-q)(1-r_1q^{i+1})(1-r_2q^{i+1})}{r_1r_2-s^*}.$
& $\begin{matrix} & \\ & \end{matrix}$
$$
\end{tabular}
\right\}
\end{equation}

\noindent
The matrix representing $A^*$ relative to the basis $\wt{\mcal{B}}_{alt}$ is
\begin{equation*}\label{ex:[A*]_BCa}
\left[
\begin{array}{cc|cc|cc|cc}
\frac{\tht^*_0-\tau_0\tht^*_1}{1-\tau_0}&\frac{\zeta_0\tau_0(\tht^*_0-\tht^*_1)}{1-\tau_0}&0&0&0&0&0&0\\
\frac{\tht^*_0-\tht^*_1}{\zeta_0(\tau_0-1)}&\frac{\tau_0\tht^*_0-\tht^*_1}{\tau_0-1}&0&0&0&0&0&0\\
\hline
0&0&\frac{\tht^*_1-\tau_1\tht^*_2}{1-\tau_1}&\frac{\zeta_1\tau_1(\tht^*_1-\tht^*_2)}{1-\tau_1}&0&0&0&0\\
0&0&\frac{\tht^*_1-\tht^*_2}{\zeta_1(\tau_1-1)}&\frac{\tau_1\tht^*_1-\tht^*_2}{\tau_1-1}&0&0&0&0\\
\hline
0&0&0&0&\frac{\tht^*_2-\tau_2\tht^*_3}{1-\tau_2}&\frac{\zeta_2\tau_2(\tht^*_2-\tht^*_3)}{1-\tau_2}&0&0\\
0&0&0&0&\frac{\tht^*_2-\tht^*_3}{\zeta_2(\tau_2-1)}&\frac{\tau_2\tht^*_2-\tht^*_3}{\tau_2-1}&0&0\\
\hline
0&0&0&0&0&0&\frac{\tht^*_3-\tau_3\tht^*_4}{1-\tau_3}&\frac{\zeta_3\tau_3(\tht^*_3-\tht^*_4)}{1-\tau_3}\\
0&0&0&0&0&0&\frac{\tht^*_3-\tht^*_4}{\zeta_3(\tau_3-1)}&\frac{\tau_3\tht^*_3-\tht^*_4}{\tau_3-1}
\end{array}
\right].
\end{equation*}
Each entry is from (\ref{tht*/zeta/tau}).
\end{example}

\medskip \noindent
We are done with $A^*$.
We now consider $\wt{A}^*$. Recall from (\ref{(tht)dist}) the formula 
\begin{equation}\label{vtht^*;d=4}
\wt{\tht}^*_i = \wt{\tht}^*_0+\wt{h}^*(1-q^{i})(1-{s}^*q^{i+2})q^{-i},
\end{equation}
for $0 \leq  i \leq 4$.

\begin{example}
The matrix representing $\wt{A}^*$ relative to the basis $\mathcal{C}$ is
\begin{equation*}\label{ex:[A*C]_C}
\diag(\wt{\tht}^*_0, \wt{\tht}^*_0, \wt{\tht}^*_1, \wt{\tht}^*_1,
			\wt{\tht}^*_2,  \wt{\tht}^*_2, \wt{\tht}^*_3, \wt{\tht}^*_3).
\end{equation*}

\ \\ \noindent
The matrix representing $\wt{A}^*$ relative to the basis $\wt{\mathcal{B}}$ is
\begin{equation*}\label{ex:[A*C]_BC}
\diag(\wt{\tht}^*_0, \wt{\tht}^*_1, \wt{\tht}^*_2, \wt{\tht}^*_3,
			\wt{\tht}^*_0,\wt{\tht}^*_1,\wt{\tht}^*_2, \wt{\tht}^*_3).
\end{equation*}

\ \\ \noindent
The matrix representing $\wt{A}^*$ relative to the basis $\wt{\mcal{B}}_{alt}$ is
\begin{equation*}\label{ex:[A*C]_BCa}
\diag(\wt{\tht}^*_0, \wt{\tht}^*_0, \wt{\tht}^*_1, \wt{\tht}^*_1,
		\wt{\tht}^*_2, \wt{\tht}^*_2, \wt{\tht}^*_3, \wt{\tht}^*_3).
\end{equation*}

\ \\ \noindent
The matrix representing $\wt{A}^*$ relative to the basis $\mcal{B}$ is
\begin{equation*}\label{ex:[A*C]_Bx}
\left[
\begin{array}{ccccc|ccc}
{\wt{\tht}}^*_0 &0&0&0&0&0&0&0\\
0& \frac{\epsilon_1\wt{\tht}^*_0-\wt{\tht}^*_1}{\epsilon_1-1} &0&0&0& \frac{\epsilon_1\xi_1(\wt{\tht}^*_0-\wt{\tht}^*_1)}{\epsilon_1-1}&0&0\\
0&0& \frac{\epsilon_2\wt{\tht}^*_1-\wt{\tht}^*_2}{\epsilon_2-1} &0&0&0& \frac{\epsilon_2\xi_2(\wt{\tht}^*_1-\wt{\tht}^*_2)}{\epsilon_2-1}&0\\
0&0&0& \frac{\epsilon_3\wt{\tht}^*_2-\wt{\tht}^*_3}{\epsilon_3-1} &0&0&0& \frac{\epsilon_3\xi_3(\wt{\tht}^*_2-\wt{\tht}^*_3)}{\epsilon_3-1}\\
0&0&0&0&\wt{\tht}^*_3&0&0&0\\
 \hline 
0&\frac{\wt{\tht}^*_0-\wt{\tht}^*_1}{\xi_1(1-\epsilon_1)}&0&0&0&\frac{\wt{\tht}^*_0-\epsilon_1\wt{\tht}^*_1}{1-\epsilon_1}&0&0\\
0&0&\frac{\wt{\tht}^*_1-\wt{\tht}^*_2}{\xi_2(1-\epsilon_2)}&0&0&0&\frac{\wt{\tht}^*_1-\epsilon_2\wt{\tht}^*_2}{1-\epsilon_2}&0\\
0&0&0&\frac{\wt{\tht}^*_2-\wt{\tht}^*_3}{\xi_3(1-\epsilon_3)}&0&0&0&\frac{\wt{\tht}^*_2-\epsilon_3\wt{\tht}^*_3}{1-\epsilon_3}
\end{array}
\right],
\end{equation*}
where $\{\wt{\tht}^*_i\}^3_{i=0}$ are from (\ref{vtht^*;d=4}) and 
$\{\xi_i\}^3_{i=1}, \{\epsilon_i\}^3_{i=1}$ are from (\ref{xi;epsilon}).
Moreover,
\begin{equation}\label{vtht*;xi;ep}
\left.
\begin{tabular}{ll}
$\tfrac{\epsilon_i{\wt{\tht}}^*_{i-1}-{\wt{\tht}}^*_i}{\epsilon_i-1} =
	{\wt{\tht}}^*_{i-1} +
	\tfrac{\wt{h}^*(1-q)(1-q^{4-i})(1-s^*q^{i+1})}{1-q^4},$
& $\begin{matrix} & \\ & \end{matrix}$ \\
$\tfrac{\epsilon_i\xi_i({\wt{\tht}}^*_{i-1}-{\wt{\tht}}^*_i)}{\epsilon_i-1} =
	\tfrac{\wt{h}^*q^{1-2i}(1-q)(1-q^i)(1-q^{i-4})(1-s^*q^{i+1})(1-s^*q^{i+5})}
		{q^4-1},$
& $\begin{matrix} & \\ & \end{matrix}$ \\
$\tfrac{{\wt{\tht}}^*_{i-1}-{\wt{\tht}}^*_i}{\xi_i(1-\epsilon_i)}=
	\tfrac{\wt{h}^*q^{3}(1-q)}{1-q^4},$
& $\begin{matrix} & \\ & \end{matrix}$ \\
$\tfrac{{\wt{\tht}}^*_{i-1}-\epsilon_i{\wt{\tht}}^*_i}{1-\epsilon_i}=
{\wt{\tht}}^*_{i}+
\tfrac{\wt{h}^*(q-1)(1-q^{4-i})(1-s^*q^{i+1})}{1-q^4}.$
& $\begin{matrix} & \\ & \end{matrix}$ 		
\end{tabular}
\right\}
\end{equation}

\ \\ \noindent
The matrix representing $\wt{A}^*$ relative to the basis $\mcal{B}_{alt}$ is
\begin{equation*}\label{ex:[A*C]_Bxa}
\left[
\begin{array}{c|cc|cc|cc|c}
\wt{\tht}^*_0 &0&0&0&0&0&0&0\\
\hline
0& \frac{\epsilon_1\wt{\tht}^*_0-\wt{\tht}^*_1}{\epsilon_1-1} & \frac{\epsilon_1\xi_1(\wt{\tht}^*_0-\wt{\tht}^*_1)}{\epsilon_1-1}&0&0&0&0&0\\
0&\frac{\wt{\tht}^*_0-\wt{\tht}^*_1}{\xi_1(1-\epsilon_1)}&\frac{\wt{\tht}^*_0-\epsilon_1\wt{\tht}^*_1}{1-\epsilon_1}&0&0&0&0&0\\
\hline
0&0&0&\frac{\epsilon_2\wt{\tht}^*_1-\wt{\tht}^*_2}{\epsilon_2-1} & \frac{\epsilon_2\xi_2(\wt{\tht}^*_1-\wt{\tht}^*_2)}{\epsilon_2-1}&0&0&0\\
0&0&0&\frac{\wt{\tht}^*_1-\wt{\tht}^*_2}{\xi_2(1-\epsilon_2)}&\frac{\wt{\tht}^*_1-\epsilon_2\wt{\tht}^*_2}{1-\epsilon_2}&0&0&0\\
\hline
0&0&0&0&0&\frac{\epsilon_3\wt{\tht}^*_2-\wt{\tht}^*_3}{\epsilon_3-1} & \frac{\epsilon_3\xi_3(\wt{\tht}^*_2-\wt{\tht}^*_3)}{\epsilon_3-1}&0\\
0&0&0&0&0&\frac{\wt{\tht}^*_2-\wt{\tht}^*_3}{\xi_3(1-\epsilon_3)}&\frac{\wt{\tht}^*_2-\epsilon_3\wt{\tht}^*_3}{1-\epsilon_3}&0\\
\hline
0&0&0&0&0&0&0&\wt{\tht}^*_3
\end{array}
\right].
\end{equation*}
Each entry is from (\ref{vtht*;xi;ep}).
\end{example}

\medskip \noindent
We are done with $\wt{A}^*$. We now consider the map $\mfrk{p}$.
\begin{example}\label{ex:[pi^x]_C}
The matrix representing $\mfrk{p}$ relative to the basis $\mathcal{C}$ is 
\begin{equation*}
\left[
\begin{array}{c|cc|cc|cc|c}
1 & 0 &0 &0 &0 &0 &0 &0\\
\hline
0 & \frac{\epsilon_1}{\epsilon_1-1} & \frac{1}{1-\epsilon_1} &0 &0 &0 &0 &0\\
0 & \frac{\epsilon_1}{\epsilon_1-1} & \frac{1}{1-\epsilon_1} &0 & 0 & 0 &0 &0\\
\hline
0 &0 &0 & \frac{\epsilon_2}{\epsilon_2-1} & \frac{1}{1-\epsilon_2} & 0 & 0 &0\\
0 &0 &0 &\frac{\epsilon_2}{\epsilon_2-1} & \frac{1}{1-\epsilon_2} & 0 & 0 &0\\
\hline
0 &0 &0 &0 &0 & \frac{\epsilon_3}{\epsilon_3-1} & \frac{1}{1-\epsilon_{3}}&0\\
0 &0 &0 &0 &0 & \frac{\epsilon_3}{\epsilon_3-1} & \frac{1}{1-\epsilon_{3}} &0\\
\hline
0 &0 &0 &0 &0 &0 &0 & 1
\end{array}
\right],
\end{equation*}
where $\{\epsilon_i\}^3_{i=1}$ are from (\ref{xi;epsilon}). Moreover,
\begin{align}\label{208}
&&\tfrac{1}{1-\epsilon_i} = \tfrac{q^4(1-q^{i-4})(1-s^*q^{i+1})}{(q^4-1)(1-s^*q^{2i+1})},
&&
\tfrac{\epsilon_i}{\epsilon_i-1} =\tfrac{(1-q^i)(1-s^*q^{i+5})}{(1-q^4)(1-s^*q^{2i+1})},
&&
\end{align}
for $1 \leq i \leq 3$.

\ \\
The matrix representing $\mfrk{p}$ relative to the basis $\mcal{B}$ is
\begin{equation*}\label{ex:[pi^x]_Bx}
\left[
\begin{array}{ccccc|ccc}
1 & 0 & 0 & 0 & 0 &0&0&0\\
0 & 1 & 0 & 0 & 0 &0&0&0\\
0 & 0 & 1 & 0 & 0 &0&0&0\\
0 & 0 & 0 & 1 & 0 &0&0&0\\
0 & 0 & 0 & 0 & 1 &0&0&0\\
\hline
0&0&0&0&0&0&0&0\\
0&0&0&0&0&0&0&0\\
0&0&0&0&0&0&0&0
\end{array}
\right].
\end{equation*}

\ \\ \noindent
The matrix representing $\mfrk{p}$ relative to the basis $\mcal{B}_{alt}$ is 
\begin{equation*}\label{ex:[pi^x]_Bxa}
\left[
\begin{array}{c|cc|cc|cc|c}
1 & 0 & 0 & 0 & 0 & 0 & 0 & 0\\
\hline
0 & 1 & 0 & 0 & 0 & 0 & 0 & 0\\
0 & 0 & 0 & 0 & 0 & 0 & 0 & 0\\
\hline
0 & 0 & 0 & 1 & 0 & 0 & 0 & 0\\
0 & 0 & 0 & 0 & 0 & 0 & 0 & 0\\
\hline
0 & 0 & 0 & 0 & 0 & 1 & 0 & 0\\
0 & 0 & 0 & 0 & 0 & 0 & 0 & 0\\
\hline
0 & 0 & 0 & 0 & 0 & 0 & 0 & 1
\end{array}
\right].
\end{equation*}

\ \\ \noindent
The matrix representing $\mfrk{p}$ relative to $\wt{\mathcal{B}}$ is
\begin{equation*}\label{ex:[pi^x]_BC}
\scalemath{0.75}{
\left[
\begin{array}{cccc|cccc}
\frac{1-\epsilon_1(1-\tau_0)}{(1-\tau_0)(1-\epsilon_1)} &
\frac{-\tau_0}{(1-\tau_0)(1-\epsilon_1)}&
0 & 0 &
\frac{\zeta_0\tau_0}{(1-\tau_0)(1-\epsilon_1)} &
\frac{-\tau_0\tau_1\zeta_1}{(1-\tau_0)(1-\epsilon_1)} &
0 & 0 \\

&&&&&&&\\

\frac{-\epsilon_1}{(1-\tau_1)(1-\epsilon_1)} &
\frac{1-\epsilon_2+\tau_1\epsilon_2(1-\epsilon_1)}{(1-\tau_1)(1-\epsilon_1)(1-\epsilon_2)} &
\frac{-\tau_1}{(1-\tau_1)(1-\epsilon_2)}&
0&
\frac{-\zeta_0\epsilon_1}{(1-\tau_1)(1-\epsilon_1)}&
\frac{\zeta_1\tau_1(1-\epsilon_1\epsilon_2)}{(1-\tau_1)(1-\epsilon_1)(1-\epsilon_2)}&
\frac{-\tau_1\tau_2\zeta_2}{(1-\tau_1)(1-\epsilon_2)} &
0    \\

&&&&&&&\\

0&
\frac{-\epsilon_2}{(1-\tau_2)(1-\epsilon_2)} &
\frac{1-\epsilon_3+\tau_2\epsilon_3(1-\epsilon_2)}{(1-\tau_2)(1-\epsilon_2)(1-\epsilon_3)} &
\frac{-\tau_2}{(1-\tau_2)(1-\epsilon_3)}&
0&
\frac{-\zeta_1\epsilon_2}{(1-\tau_2)(1-\epsilon_2)}&
\frac{\zeta_2\tau_2(1-\epsilon_2\epsilon_3)}{(1-\tau_2)(1-\epsilon_2)(1-\epsilon_3)}&
\frac{-\tau_2\tau_3\zeta_3}{(1-\tau_2)(1-\epsilon_3)}  \\

&&&&&&&\\

0&0&
\frac{-\epsilon_3}{(1-\tau_3)(1-\epsilon_3)} &
\frac{1+\tau_3(\epsilon_3-1)}{(1-\tau_3)(1-\epsilon_3)}&
0&0&
\frac{-\zeta_2\epsilon_3}{(1-\tau_3)(1-\epsilon_3)}&
\frac{\zeta_3\tau_3\epsilon_3}{(1-\tau_3)(1-\epsilon_3)} \\

&&&&&&&\\
\hline
&&&&&&&\\

\frac{-1}{\zeta_0(1-\tau_0)(1-\epsilon_1)} &
\frac{1}{\zeta_0(1-\tau_0)(1-\epsilon_1)} &
0& 0&
\frac{\tau_0(\epsilon_1-1)-\epsilon_1}{(1-\tau_0)(1-\epsilon_1)}&
\frac{\zeta_1\tau_1}{\zeta_0(1-\tau_0)(1-\epsilon_1)}&
0&0\\

&&&&&&&\\ 

\frac{\epsilon_1}{\zeta_1(1-\tau_1)(1-\epsilon_1)} &
\frac{-1+\epsilon_1\epsilon_2}{\zeta_1(1-\tau_1)(1-\epsilon_1)(1-\epsilon_2)} &
\frac{1}{\zeta_1(1-\tau_1)(1-\epsilon_2)} &
0 &
\frac{\zeta_0\epsilon_1}{\zeta_1(1-\tau_1)(1-\epsilon_1)} &
\frac{\tau_1(\epsilon_2-1)+\epsilon_2(\epsilon_1-1)}{(1-\tau_1)(1-\epsilon_1)(1-\epsilon_2)} &
\frac{\zeta_2\tau_2}{\zeta_1(1-\tau_1)(1-\epsilon_2)} &
0 \\

&&&&&&&\\ 

0&
\frac{\epsilon_2}{\zeta_2(1-\tau_2)(1-\epsilon_2)} &
\frac{-1+\epsilon_2\epsilon_3}{\zeta_2(1-\tau_2)(1-\epsilon_2)(1-\epsilon_3)} &
\frac{1}{\zeta_2(1-\tau_2)(1-\epsilon_3)} &
0&
\frac{\zeta_1\epsilon_2}{\zeta_2(1-\tau_2)(1-\epsilon_2)} &
\frac{\tau_2(\epsilon_3-1)+\epsilon_3(\epsilon_2-1)}{(1-\tau_2)(1-\epsilon_2)(1-\epsilon_3)} &
\frac{\zeta_3\tau_3}{\zeta_2(1-\tau_2)(1-\epsilon_3)} \\

&&&&&&&\\ 

0&0&
\frac{\epsilon_3}{\zeta_3(1-\tau_3)(1-\epsilon_3)} &
\frac{-\epsilon_3}{\zeta_3(1-\tau_3)(1-\epsilon_3)} &
0&0&
\frac{\zeta_2\epsilon_3}{\zeta_3(1-\tau_3)(1-\epsilon_3)} &
\frac{1-\tau_3-\epsilon_3}{(1-\tau_3)(1-\epsilon_3)} 

\end{array}
\right]},
\end{equation*}
where $\{\epsilon_i\}^3_{i=1}$ are from (\ref{xi;epsilon}) and
$\{\zeta_i\}^3_{i=0},\{\tau_i\}^3_{i=0}$ are from (\ref{zeta;tau}).

\ \\ \noindent
The matrix representing $\mfrk{p}$ relative to $\wt{\mcal{B}}_{alt}$ is
\begin{equation*}\label{ex:[pi^x]_BCa}
\scalemath{0.75}{
\left[
\begin{array}{cc|cc|cc|cc}
\frac{1-\epsilon_1(1-\tau_0)}{(1-\tau_0)(1-\epsilon_1)} &
\frac{\zeta_0\tau_0}{(1-\tau_0)(1-\epsilon_1)} &
\frac{-\tau_0}{(1-\tau_0)(1-\epsilon_1)}&
\frac{-\tau_0\tau_1\zeta_1}{(1-\tau_0)(1-\epsilon_1)} &
0 & 0 & 0 & 0 \\

&&&&&&&\\ 

\frac{-1}{\zeta_0(1-\tau_0)(1-\epsilon_1)} &
\frac{\tau_0(\epsilon_1-1)-\epsilon_1}{(1-\tau_0)(1-\epsilon_1)}&
\frac{1}{\zeta_0(1-\tau_0)(1-\epsilon_1)} &
\frac{\zeta_1\tau_1}{\zeta_0(1-\tau_0)(1-\epsilon_1)}&
0&0&0&0\\

&&&&&&& \\ 
\hline
&&&&&&&\\

\frac{-\epsilon_1}{(1-\tau_1)(1-\epsilon_1)} &
\frac{-\zeta_0\epsilon_1}{(1-\tau_1)(1-\epsilon_1)}&
\frac{1-\epsilon_2+\tau_1\epsilon_2(1-\epsilon_1)}{(1-\tau_1)(1-\epsilon_1)(1-\epsilon_2)} &
\frac{\zeta_1\tau_1(1-\epsilon_1\epsilon_2)}{(1-\tau_1)(1-\epsilon_1)(1-\epsilon_2)}&
\frac{-\tau_1}{(1-\tau_1)(1-\epsilon_2)}&
\frac{-\tau_1\tau_2\zeta_2}{(1-\tau_1)(1-\epsilon_2)} &
0&0  \\

&&&&&&& \\ 

\frac{\epsilon_1}{\zeta_1(1-\tau_1)(1-\epsilon_1)} &
\frac{\zeta_0\epsilon_1}{\zeta_1(1-\tau_1)(1-\epsilon_1)} &
\frac{-1+\epsilon_1\epsilon_2}{\zeta_1(1-\tau_1)(1-\epsilon_1)(1-\epsilon_2)} &
\frac{\tau_1(\epsilon_2-1)+\epsilon_2(\epsilon_1-1)}{(1-\tau_1)(1-\epsilon_1)(1-\epsilon_2)} &
\frac{1}{\zeta_1(1-\tau_1)(1-\epsilon_2)} &
\frac{\zeta_2\tau_2}{\zeta_1(1-\tau_1)(1-\epsilon_2)} &
0&0 \\

&&&&&&&\\ 
\hline
&&&&&&&\\

0&0&
\frac{-\epsilon_2}{(1-\tau_2)(1-\epsilon_2)} &
\frac{-\zeta_1\epsilon_2}{(1-\tau_2)(1-\epsilon_2)}&
\frac{1-\epsilon_3+\tau_2\epsilon_3(1-\epsilon_2)}{(1-\tau_2)(1-\epsilon_2)(1-\epsilon_3)} &
\frac{\zeta_2\tau_2(1-\epsilon_2\epsilon_3)}{(1-\tau_2)(1-\epsilon_2)(1-\epsilon_3)}&
\frac{-\tau_2}{(1-\tau_2)(1-\epsilon_3)}&
\frac{-\tau_2\tau_3\zeta_3}{(1-\tau_2)(1-\epsilon_3)} \\

&&&&&&&\\ 

0&0&
\frac{\epsilon_2}{\zeta_2(1-\tau_2)(1-\epsilon_2)} &
\frac{\zeta_1\epsilon_2}{\zeta_2(1-\tau_2)(1-\epsilon_2)} &
\frac{-1+\epsilon_2\epsilon_3}{\zeta_2(1-\tau_2)(1-\epsilon_2)(1-\epsilon_3)} &
\frac{\tau_2(\epsilon_3-1)+\epsilon_3(\epsilon_2-1)}{(1-\tau_2)(1-\epsilon_2)(1-\epsilon_3)} &
\frac{1}{\zeta_2(1-\tau_2)(1-\epsilon_3)} &
\frac{\zeta_3\tau_3}{\zeta_2(1-\tau_2)(1-\epsilon_3)} \\

&&&&&&&\\ 
\hline
&&&&&&&\\

0&0&0&0&
\frac{-\epsilon_3}{(1-\tau_3)(1-\epsilon_3)} &
\frac{-\zeta_2\epsilon_3}{(1-\tau_3)(1-\epsilon_3)}&
\frac{1+\tau_3(\epsilon_3-1)}{(1-\tau_3)(1-\epsilon_3)}&
\frac{\zeta_3\tau_3\epsilon_3}{(1-\tau_3)(1-\epsilon_3)} \\

&&&&&&&\\ 

0&0&0&0&
\frac{\epsilon_3}{\zeta_3(1-\tau_3)(1-\epsilon_3)} &
\frac{\zeta_2\epsilon_3}{\zeta_3(1-\tau_3)(1-\epsilon_3)} &
\frac{-\epsilon_3}{\zeta_3(1-\tau_3)(1-\epsilon_3)} &
\frac{1-\tau_3-\epsilon_3}{(1-\tau_3)(1-\epsilon_3)} 
\end{array}
\right]},
\end{equation*}
where $\{\epsilon_i\}^3_{i=1}$ are from (\ref{xi;epsilon}) and
$\{\zeta_i\}^3_{i=0},\{\tau_i\}^3_{i=0}$ are from (\ref{zeta;tau}).
\end{example}

\medskip \noindent
We are done with $\mfrk{p}$. Finally we consider the map $\wt{\mfrk{p}}$.
\begin{example}\label{ex:[pi^C]_C}
The matrix representing $\wt{\mfrk{p}}$ relative to $\mathcal{C}$ is
\begin{equation*}
\left[
\begin{array}{cc|cc|cc|cc}
\frac{1}{1-\tau_0} & \frac{\tau_0}{\tau_0-1} &0&0&0&0&0&0\\
\frac{1}{1-\tau_0} & \frac{\tau_0}{\tau_0-1} &0&0&0&0&0&0\\
\hline
0&0&\frac{1}{1-\tau_1} & \frac{\tau_1}{\tau_1-1} &0&0&0&0\\
0&0&\frac{1}{1-\tau_1} & \frac{\tau_1}{\tau_1-1} &0&0&0&0\\
\hline
0&0&0&0&\frac{1}{1-\tau_2} & \frac{\tau_2}{\tau_2-1} &0&0\\
0&0&0&0&\frac{1}{1-\tau_2} & \frac{\tau_2}{\tau_2-1} &0&0\\
\hline
0&0&0&0&0&0&\frac{1}{1-\tau_3} & \frac{\tau_3}{\tau_3-1} \\
0&0&0&0&0&0&\frac{1}{1-\tau_3} & \frac{\tau_3}{\tau_3-1}
\end{array}
\right],
\end{equation*}
where $\{\tau_i\}^3_{i=0}$ are from (\ref{zeta;tau}).
Moreover,
\begin{align}\label{209}
&& 
\tfrac{1}{1-\tau_i} = \tfrac{(r_1-s^*q^{i+1})(r_2-s^*q^{i+1})}{(r_1r_2-s^*)(1-s^*q^{2i+2})},
&&
\tfrac{\tau_i}{\tau_{i}-1}=\tfrac{s^*(1-r_1q^{i+1})(1-r_2q^{i+1})}{(s^*-r_1r_2)(1-s^*q^{2i+2})},
&&
\end{align}
for $0 \leq  i \leq 3$.

\ \\ \noindent
The matrix representing $\wt{\mfrk{p}}$ relative to $\wt{\mathcal{B}}$ is
\begin{equation*}\label{ex:[pi^C]_BC}
\left[
\begin{array}{cccc|cccc}
1&0&0&0& 0&0&0&0\\
0&1&0&0& 0&0&0&0\\
0&0&1&0& 0&0&0&0\\
0&0&0&1& 0&0&0&0\\
\hline
0&0&0&0&0&0&0&0\\
0&0&0&0&0&0&0&0\\
0&0&0&0&0&0&0&0\\
0&0&0&0&0&0&0&0
\end{array}
\right].
\end{equation*}

\ \\ \noindent
The matrix representing $\wt{\mfrk{p}}$ relative to $\wt{\mcal{B}}_{alt}$ is
\begin{equation*}\label{ex:[pi^C]_BCa}
\left[
\begin{array}{cc|cc|cc|cc}
1&0&0&0&0&0&0&0\\
0&0&0&0&0&0&0&0\\
\hline
0&0&1&0&0&0&0&0\\
0&0&0&0&0&0&0&0\\
\hline
0&0&0&0&1&0&0&0\\
0&0&0&0&0&0&0&0\\
\hline
0&0&0&0&0&0&1&0\\
0&0&0&0&0&0&0&0
\end{array}
\right].
\end{equation*}

\ \\ \noindent
The matrix representing $\wt{\mfrk{p}}$ relative to $\mcal{B}$ is
\begin{equation*}\label{ex:[pi^C]_Bx}
\scalemath{0.8}{
\left[
\begin{array}{@{}c@{}@{}c@{}@{}c@{}@{}c@{}@{}c@{}|@{}c@{}@{}c@{}@{}c@{}}
\frac{1}{1-\tau_0} &
\frac{\tau_0}{\tau_0-1} &
0 & 0 & 0 &
\frac{\tau_0\xi_1}{\tau_0-1} &
0 & 0\\ 

&&&&&&&\\ 

\frac{-\epsilon_1}{(1-\tau_0)(1-\epsilon_1)} &
\frac{\epsilon_1\tau_0(\tau_1-1)+\tau_0-1}{(\tau_0-1)(\tau_1-1)(\epsilon_1-1)}&
\frac{-\tau_1}{(1-\epsilon_1)(1-\tau_1)} &
0 & 0 &
\frac{\xi_1\epsilon_1(1-\tau_0\tau_1)}{(1-\tau_0)(1-\tau_1)(1-\epsilon_1)} &
\frac{-\tau_1\xi_2}{(1-\epsilon_1)(1-\tau_1)} &
0 \\

&&&&&&&\\ 

0 &
\frac{-\epsilon_2}{(1-\tau_1)(1-\epsilon_2)} &
\frac{\epsilon_2\tau_1(\tau_2-1)+\tau_1-1}{(\tau_1-1)(\tau_2-1)(\epsilon_2-1)}&
\frac{-\tau_2}{(1-\epsilon_2)(1-\tau_2)} &
0 &
\frac{-\xi_1\epsilon_1\epsilon_2}{(1-\tau_1)(1-\epsilon_2)} &
\frac{\xi_2\epsilon_2(1-\tau_1\tau_2)}{(1-\tau_1)(1-\tau_2)(1-\epsilon_2)} &
\frac{-\tau_2\xi_3}{(1-\epsilon_2)(1-\tau_2)} \\

&&&&&&&\\ 

0 & 0 &
\frac{-\epsilon_3}{(1-\tau_2)(1-\epsilon_3)} &
\frac{\epsilon_3\tau_2(\tau_3-1)+\tau_2-1}{(\tau_2-1)(\tau_3-1)(\epsilon_3-1)}&
\frac{-\tau_3}{(1-\epsilon_3)(1-\tau_3)} &
0 &
\frac{-\xi_2\epsilon_2\epsilon_3}{(1-\tau_2)(1-\epsilon_3)} &
\frac{\xi_3\epsilon_3(1-\tau_2\tau_3)}{(1-\tau_2)(1-\tau_3)(1-\epsilon_3)} \\

&&&&&&&\\ 

0 & 0 & 0 &
\frac{1}{1-\tau_3} &
\frac{\tau_3}{\tau_3-1} &
0 & 0 &
\frac{\xi_3\epsilon_3}{1-\tau_3} \\

&&&&&&&\\ 
\hline
&&&&&&& \\

\frac{1}{\xi_1(\epsilon_1-1)(\tau_0-1)} &
\frac{\tau_0\tau_1-1}{\xi_1(1-\epsilon_1)(1-\tau_1)(1-\tau_0)} &
\frac{\tau_1}{\xi_1(1-\epsilon_1)(1-\tau_1)} &
0 & 0 &
\frac{\tau_0(\tau_1-1)+\epsilon_1(\tau_0-1)}{(1-\epsilon_1)(1-\tau_1)(1-\tau_0)} &
\frac{\tau_1\xi_2}{\xi_1(1-\epsilon_1)(1-\tau_1)} &
0 \\

&&&&&&&\\ 

0 &
\frac{1}{\xi_2(\epsilon_2-1)(\tau_1-1)} &
\frac{\tau_1\tau_2-1}{\xi_2(1-\epsilon_2)(1-\tau_2)(1-\tau_1)} &
\frac{\tau_2}{\xi_2(1-\epsilon_2)(1-\tau_2)} &
0 &
\frac{\xi_1\epsilon_1}{\xi_2(1-\epsilon_2)(1-\tau_1)} &
\frac{\tau_1(\tau_2-1)+\epsilon_2(\tau_1-1)}{(1-\epsilon_2)(1-\tau_2)(1-\tau_1)} &
\frac{\tau_2\xi_3}{\xi_2(1-\epsilon_2)(1-\tau_2)} \\

&&&&&&&\\

0 & 0 &
\frac{1}{\xi_3(\epsilon_3-1)(\tau_2-1)} &
\frac{\tau_2\tau_3-1}{\xi_3(1-\epsilon_3)(1-\tau_3)(1-\tau_2)} &
\frac{\tau_3}{\xi_3(1-\epsilon_3)(1-\tau_3)} &
0 &
\frac{\xi_2\epsilon_2}{\xi_3(1-\epsilon_3)(1-\tau_2)} &
\frac{\tau_2(\tau_3-1)+\epsilon_3(\tau_2-1)}{(1-\epsilon_3)(1-\tau_3)(1-\tau_2)} 
\end{array}
\right]},
\end{equation*}
where $\{\xi_i\}^3_{i=1}, \{\epsilon_i\}^3_{i=1}$ are from (\ref{xi;epsilon})
and $\{\tau_i\}^3_{i=0}$ are from (\ref{zeta;tau}).

\ \\ \noindent
The matrix representing $\wt{\mfrk{p}}$ relative to $\mcal{B}_{alt}$ is
\begin{equation*}\label{ex:[pi^C]_Bxa}
\scalemath{0.8}{
\left[
\begin{array}{@{}c@{}|@{}c@{}@{}c@{}|@{}c@{}@{}c@{}|@{}c@{}@{}c@{}|@{}c@{}}
\frac{1}{1-\tau_0} &
\frac{\tau_0}{\tau_0-1} &
\frac{\tau_0\xi_1}{\tau_0-1} &
0 & 0 & 0 &
0 & 0\\ 

&&&&&&&\\ 
\hline
&&&&&&&\\

\frac{-\epsilon_1}{(1-\tau_0)(1-\epsilon_1)} &
\frac{\epsilon_1\tau_0(\tau_1-1)+\tau_0-1}{(\tau_0-1)(\tau_1-1)(\epsilon_1-1)}&
\frac{\xi_1\epsilon_1(1-\tau_0\tau_1)}{(1-\tau_0)(1-\tau_1)(1-\epsilon_1)} &
\frac{-\tau_1}{(1-\epsilon_1)(1-\tau_1)} &
\frac{-\tau_1\xi_2}{(1-\epsilon_1)(1-\tau_1)} &
0 & 0 &
0 \\

&&&&&&&\\ 

\frac{1}{\xi_1(\epsilon_1-1)(\tau_0-1)} &
\frac{\tau_0\tau_1-1}{\xi_1(1-\epsilon_1)(1-\tau_1)(1-\tau_0)} &
\frac{\tau_0(\tau_1-1)+\epsilon_1(\tau_0-1)}{(1-\epsilon_1)(1-\tau_1)(1-\tau_0)} &
\frac{\tau_1}{\xi_1(1-\epsilon_1)(1-\tau_1)} &
\frac{\tau_1\xi_2}{\xi_1(1-\epsilon_1)(1-\tau_1)} &
0 & 
0 &
0 \\

&&&&&&&\\ 
\hline
&&&&&&&\\

0 &
\frac{-\epsilon_2}{(1-\tau_1)(1-\epsilon_2)} &
\frac{-\xi_1\epsilon_1\epsilon_2}{(1-\tau_1)(1-\epsilon_2)} &
\frac{\epsilon_2\tau_1(\tau_2-1)+\tau_1-1}{(\tau_1-1)(\tau_2-1)(\epsilon_2-1)}&
\frac{\xi_2\epsilon_2(1-\tau_1\tau_2)}{(1-\tau_1)(1-\tau_2)(1-\epsilon_2)} &
\frac{-\tau_2}{(1-\epsilon_2)(1-\tau_2)} &
\frac{-\tau_2\xi_3}{(1-\epsilon_2)(1-\tau_2)} &
0 \\

&&&&&&&\\ 

0 &
\frac{1}{\xi_2(\epsilon_2-1)(\tau_1-1)} &
\frac{\xi_1\epsilon_1}{\xi_2(1-\epsilon_2)(1-\tau_1)} &
\frac{\tau_1\tau_2-1}{\xi_2(1-\epsilon_2)(1-\tau_2)(1-\tau_1)} &
\frac{\tau_1(\tau_2-1)+\epsilon_2(\tau_1-1)}{(1-\epsilon_2)(1-\tau_2)(1-\tau_1)} &
\frac{\tau_2}{\xi_2(1-\epsilon_2)(1-\tau_2)} &
\frac{\tau_2\xi_3}{\xi_2(1-\epsilon_2)(1-\tau_2)} &
0 \\

&&&&&&&\\
\hline
&&&&&&&\\

0 & 0 &
0 &
\frac{-\epsilon_3}{(1-\tau_2)(1-\epsilon_3)} &
\frac{-\xi_2\epsilon_2\epsilon_3}{(1-\tau_2)(1-\epsilon_3)} &
\frac{\epsilon_3\tau_2(\tau_3-1)+\tau_2-1}{(\tau_2-1)(\tau_3-1)(\epsilon_3-1)}&
\frac{\xi_3\epsilon_3(1-\tau_2\tau_3)}{(1-\tau_2)(1-\tau_3)(1-\epsilon_3)} &
\frac{-\tau_3}{(1-\epsilon_3)(1-\tau_3)} \\

&&&&&&&\\ 

0 & 0 &
0 &
\frac{1}{\xi_3(\epsilon_3-1)(\tau_2-1)} &
\frac{\xi_2\epsilon_2}{\xi_3(1-\epsilon_3)(1-\tau_2)} &
\frac{\tau_2\tau_3-1}{\xi_3(1-\epsilon_3)(1-\tau_3)(1-\tau_2)} &
\frac{\tau_2(\tau_3-1)+\epsilon_3(\tau_2-1)}{(1-\epsilon_3)(1-\tau_3)(1-\tau_2)} &
\frac{\tau_3}{\xi_3(1-\epsilon_3)(1-\tau_3)} \\

&&&&&&&\\ 
\hline
&&&&&&&\\

0 & 0 &
0 &
0 &
0 &
\frac{1}{1-\tau_3} &
\frac{\xi_3\epsilon_3}{1-\tau_3} &
\frac{\tau_3}{\tau_3-1}   
\end{array}
\right]},
\end{equation*}
where $\{\xi_i\}^3_{i=1}, \{\epsilon_i\}^3_{i=1}$ are from (\ref{xi;epsilon})
and $\{\tau_i\}^3_{i=0}$ are from (\ref{zeta;tau}).
\end{example}

\bigskip
\subsection{The matrices from Part II}\label{Apdx:II}

Consider the basis $\mcal{C}$ from (\ref{5bases}).
Recall that for $d=4$ this basis consists of
\begin{equation}\label{C;d=4}
\hat{C}^-_0, \quad \hat{C}^+_0,  \quad 
\hat{C}^-_1,  \quad \hat{C}^+_1,  \quad 
\hat{C}^-_2,  \quad \hat{C}^+_2,  \quad
\hat{C}^-_3,  \quad \hat{C}^+_3.
\end{equation}
In this section, for each element
$$
\{t_n\}^3_{n=0}, \quad \X^{\pm1}, \quad \Y^{\pm1}, \quad
\A, \quad \B, \quad \B^{\dagger}, \quad
\tfrac{t_0-k_0^{-1}}{k_0-k_0^{-1}}, \quad
\tfrac{t_1-k_1^{-1}}{k_1-k_1^{-1}}
$$
of $\H$
we give the matrix representing that element relative to the basis (\ref{C;d=4}).

\ \\ \noindent
The matrix representing $t_0$ relative to (\ref{C;d=4}) is block diagonal:
\begin{equation*}
\left[
\begin{array}{cccc}
\mfrk{t}_0(0) &&& {\bf 0} \\
&\mfrk{t}_0(1)&&\\
&&\mfrk{t}_0(2)&\\
{\bf 0}&&&\mfrk{t}_0(3)
\end{array}
\right],
\end{equation*}
where for $0 \leq i \leq 3$,
$$
\mfrk{t}_{0}(i) = \left[
\begin{array}{ccc}
\frac{1}{\sqrt{s^*r_1r_2}}
\left(
\frac{(r_1-s^*q^{i+1})(r_2-s^*q^{i+1})}{1-s^*q^{2i+2}} + s^*
\right)
& \ \quad  \ &
-\sqrt{\frac{s^*}{r_1r_2}} \frac{(1-r_1q^{i+1})(1-r_2q^{i+1})}{1-s^*q^{2i+2}} \\

\ \\

\frac{1}{\sqrt{s^*r_1r_2}}
\frac{(r_1-s^*q^{i+1})(r_2-s^*q^{i+1})}{1-s^*q^{2i+2}}
& \ \quad  \ &
\sqrt{\frac{s^*}{r_1r_2}}
\left(1-\frac{(1-r_1q^{i+1})(1-r_2q^{i+1})}{1-s^*q^{2i+2}}\right)
\end{array}
\right].
$$

\bigskip \noindent
The matrix representing $t_1$ relative to (\ref{C;d=4}) is block diagonal:
\begin{equation*}
\left[
\begin{array}{ccccc}
\mfrk{t}_1(0) &&&&{\bf 0} \\
& \mfrk{t}_1(1) &&& \\
&& \mfrk{t}_1(2) && \\
&&& \mfrk{t}_1(3) & \\
{\bf 0}&&&& \mfrk{t}_1(4)
\end{array}
\right],
\end{equation*}
where for $1 \leq i \leq 3$,
$$
\mfrk{t}_{1}(i) =\left[
\begin{array}{ccc}
\frac{q^{2}(1-q^{i-4})(1-s^*q^{i+1})}{1-s^*q^{2i+1}} + \frac{1}{q^{2}}
& \ \quad  \ &
\frac{q^{2}(q^{i-4}-1)(1-s^*q^{i+1})}{1-s^*q^{2i+1}}\\

\ \\

\frac{(1-q^{i})(1-s^*q^{i+5})}{q^{2}(1-s^*q^{2i+1})}
& \ \quad  \ &
\frac{(q^{i}-1)(1-s^*q^{i+5})}{q^{2}(1-s^*q^{2i+1})} +\frac{1}{q^{2}}
\end{array}
\right]
$$
and 
$$\mfrk{t}_{1}(0)=\left[{q^{-2}}\right], \qquad \qquad 
\mfrk{t}_{1}(4)=\left[{q^{-2}}\right].$$

\bigskip \noindent
The matrix representing $t_2$ relative to (\ref{C;d=4}) is 
block diagonal:
\begin{equation*}
\left[
\begin{array}{ccccc}
\mfrk{t}_2(0) &&&&{\bf 0} \\
& \mfrk{t}_2(1) &&&\\
&& \mfrk{t}_2(2) && \\
&&& \mfrk{t}_2(3) & \\
{\bf 0}&&&& \mfrk{t}_2(4)
\end{array}
\right],
\end{equation*}
where for $1 \leq i \leq 3$,
$$
\mfrk{t}_{2}(i) =\left[
\begin{array}{ccc}
\frac{1}{\sqrt{s^*q^{5}}}
\left(
\frac{s^*q^{5}(1-q^{i})(1-q^{i-4})}{1-s^*q^{2i+1}}+1
\right)
& \ \quad  \ &
\frac{q^{i}\sqrt{s^*q^{5}}(1-q^{i-4})(1-s^*q^{i+1})}{1-s^*q^{2i+1}}
\\

\ \\

\frac{1}{q^{i}\sqrt{s^*q^{5}}}
\frac{(q^{i}-1)(1-s^*q^{i+5})}{1-s^*q^{2i+1}}
& \ \quad  \ &
\sqrt{s^*q^{5}} \left(
\frac{(q^{i}-1)(1-q^{i-4})}{1-s^*q^{2i+1}} + 1
\right)
\end{array}
\right]
$$
and
$$\mfrk{t}_{2}(0)=\left[\sqrt{s^*q^{5}}\right], \qquad \qquad
\mfrk{t}_{2}(4)=\left[\tfrac{1}{\sqrt{s^*q^{5}}}\right].$$

\bigskip \noindent
The matrix representing $t_3$ relative to (\ref{C;d=4}) is block diagonal:
\begin{equation*}
\left[
\begin{array}{cccc}
\mfrk{t}_3(0) &&&{\bf 0} \\
& \mfrk{t}_3(1) && \\
&& \mfrk{t}_3(2) & \\
{\bf 0}&&& \mfrk{t}_3(3)
\end{array}
\right],
\end{equation*}
where for $0 \leq  i \leq 3$,
$$
\mfrk{t}_{3}(i) :=\left[
\begin{array}{ccc}
\frac{1}{q^{i+1}\sqrt{r_1r_2}}
\left(
1 -\frac{(1-r_1q^{i+1})(1-r_2q^{i+1})}{1-s^*q^{2i+2}}
\right)
& \ \quad  \ &
\frac{1}{q^{i+1}\sqrt{r_1r_2}}
\frac{(1-r_1q^{i+1})(1-r_2q^{i+1})}{1-s^*q^{2i+2}} \\

\ \\

-\frac{q^{i+1}}{\sqrt{r_1r_2}}
\frac{(r_1-s^*q^{i+1})(r_2-s^*q^{i+1})}{1-s^*q^{2i+2}}
& \ \quad  \ &
\frac{q^{i+1}}{\sqrt{r_1r_2}}
\left(
\frac{(r_1-s^*q^{i+1})(r_2-s^*q^{i+1})}{1-s^*q^{2i+2}} + s^*\right)
\end{array}
\right].
$$
So the matrix representations of $t_0, t_3$ 
and the matrix representations of $t_1, t_2$ take the form
$$
\left(
\begin{array}{cccccccc}
* & * & & & & & & \\
* & * & & & & & & \\
& & * & * & & & & \\
& & * & * & & & & \\
& & & & * & * & & \\
& & & & * & * & & \\
& & & & & & * & * \\
& & & & & & * & * 
\end{array}
\right)
\qquad \text{and} \qquad
\left(
\begin{array}{cccccccc}
* & & & & & & & \\
& * & * & & & & & \\
& * & * & & & & & \\
& & & * & * & & & \\
& & & * & * & & & \\
& & & & & * & * & \\
& & & & & * & * & \\
& & & & & & & *
\end{array}
\right),
$$
respectively.

\bigskip \noindent
The matrix representing $\X$ relative to (\ref{C;d=4}) is diagonal:
\begin{equation*} \diag 
\left[
\tfrac{1}{q\sqrt{s^*}},\
q\sqrt{s^*},\
\tfrac{1}{q^2\sqrt{s^*}},\
q^2\sqrt{s^*},\
\tfrac{1}{q^3\sqrt{s^*}},\
q^3\sqrt{s^*},\
\tfrac{1}{q^4\sqrt{s^*}},\
q^4\sqrt{s^*}
\right].
\end{equation*}

\bigskip \noindent
The matrix representing $\X^{-1}$ relative to (\ref{C;d=4}) is diagonal:
\begin{equation*} \diag 
\left[
q\sqrt{s^*},\
\tfrac{1}{q\sqrt{s^*}},\
q^2\sqrt{s^*},\
\tfrac{1}{q^2\sqrt{s^*}},\
q^3\sqrt{s^*},\
\tfrac{1}{q^3\sqrt{s^*}},\
q^4\sqrt{s^*}, \
\tfrac{1}{q^4\sqrt{s^*}}
\right].
\end{equation*}

\bigskip \noindent
The matrix representing $\Y$ relative to (\ref{C;d=4}) is 
$\sqrt{\frac{s^*q^4}{r_1r_2}}$
times
\begin{equation*}
\left[
\begin{array}{ccccc}
{\bf a}_0 & {\bf b}_0 &&& {\bf 0}\\
&{\bf a}_1 & {\bf b}_1 &&\\
&&{\bf a}_2 & {\bf b}_2 &\\
{\bf 0} &&&{\bf a}_3 & {\bf b}_3
\end{array}
\right],
\end{equation*}
where 
$$
{\bf a}_0 = 
\frac{1}{s^*q^4}\left[
\begin{array}{c}
\frac{(r_1-s^*q)(r_2-s^*q)}{1-s^*q^2}+s^* \\
\\
\frac{(r_1-s^*q)(r_2-s^*q)}{1-s^*q^2}
\end{array}
\right],
\qquad \qquad 
{\bf b}_3 = 
\left[
\begin{array}{c}
-\frac{(1-r_1q^{4})(1-r_2q^{4})}{q^4(1-s^*q^{8})}
 \\ \\
\frac{1}{q^4}-\frac{(1-r_1q^{4})(1-r_2q^{4})}{q^4(1-s^*q^{8})}
\end{array}
\right],
$$
and for $1 \leq i \leq 3$,
\begin{align*}
&{\bf a}_{i} =  \\ &\frac{1}{s^*q^4}
\left[
\begin{array}{cc}
\frac{(1-q^{i})(1-s^*q^{i+5})}{1-s^*q^{2i+1}}
\left(\frac{(r_1-s^*q^{i+1})(r_2-s^*q^{i+1})}{1-s^*q^{2i+2}}+s^*\right)
&
\left(1- \frac{(1-q^{i})(1-s^*q^{i+5})}{1-s^*q^{2i+1}}\right)
\left(\frac{(r_1-s^*q^{i+1})(r_2-s^*q^{i+1})}{1-s^*q^{2i+2}}+s^*\right) \\
\\
\frac{(1-q^{i})(1-s^*q^{i+5})(r_1-s^*q^{i+1})(r_2-s^*q^{i+1})}
{(1-s^*q^{2i+1})(1-s^*q^{2i+2})}
&
\left(1- \frac{(1-q^{i})(1-s^*q^{i+5})}{1-s^*q^{2i+1}}\right)
\frac{(r_1-s^*q^{i+1})(r_2-s^*q^{i+1})}{1-s^*q^{2i+2}}
\end{array}
\right],
\end{align*}
and for $0 \leq i \leq 2$,
\begin{align*}
&{\bf b}_i = \\
& \left[
\begin{array}{cc}
-\frac{(1-r_1q^{i+1})(1-r_2q^{i+1})}{1-s^*q^{2i+2}} 
\left(\frac{(1-q^{i-3})(1-s^*q^{i+2})}{1-s^*q^{2i+3}}+q^{-4}\right)
&
\frac{(1-q^{i-3})(1-s^*q^{i+2})(1-r_1q^{i+1})(1-r_2q^{i+1})}
{(1-s^*q^{2i+3})(1-s^*q^{2i+2})}
\\ \\
\left(1-\frac{(1-r_1q^{i+1})(1-r_2q^{i+1})}{1-s^*q^{2i+2}} \right) 
\left(\frac{(1-q^{i-3})(1-s^*q^{i+2})}{1-s^*q^{2i+3}}+q^{-4}\right)
&
\frac{(1-q^{i-3})(1-s^*q^{i+2})}{1-s^*q^{2i+3}}
\left(\frac{(1-r_1q^{i+1})(1-r_2q^{i+1})}{1-s^*q^{2i+2}}-1\right)
\end{array}
\right].
\end{align*}

\bigskip \noindent
The matrix representing $\Y^{-1}$ relative to (\ref{C;d=4}) is 
$\sqrt{\frac{s^*q^4}{r_1r_2}}$ times
\begin{equation*}
\left[
\begin{array}{cccc}
{\bf a}_0 & & & {\bf 0} \\
{\bf c}_1 & {\bf a}_1 & &\\
& {\bf c}_2 & {\bf a}_2 &\\
& & {\bf c}_3 & {\bf a}_3 \\
{\bf 0} & & & {\bf c}_4
\end{array}
\right],
\end{equation*}
where
\begin{align*}
&{\bf a}_0 = 
\left[
\begin{array}{ccc}
\left(1-\frac{(1-r_1q)(1-r_2q)}{1-s^*q^{2}}\right)
\left(\frac{(1-q^{-4})(1-s^*q)}{1-s^*q}+\frac{1}{q^4}\right)
&\quad &
\frac{(1-r_1q)(1-r_2q)}{1-s^*q^{2}}
\left(\frac{(1-q^{-4})(1-s^*q)}{1-s^*q}+\frac{1}{q^4}\right)
\end{array}
\right], \\
\ \\
&{\bf c}_4 = \tfrac{1}{s^*q^4}
\left[
\begin{array}{ccc}
\frac{(r_1-s^*q^4)(r_2-s^*q^4)}{1-s^*q^{8}}
\left(\frac{(1-q^{4})(1-s^*q^9)}{1-s^*q^9}-1\right)
&\quad &
\left(\frac{(r_1-s^*q^{4})(r_2-s^*q^{4})}{1-s^*q^{8}}+s^*\right)
\left(\frac{(q^{4}-1)(1-s^*q^{9})}{1-s^*q^{9}}+1\right)
\end{array}
\right],
\end{align*}
and for $1 \leq i \leq 3$,
\begin{align*}
& {\bf a}_i =
\left[ 
\begin{array}{cc}
\tfrac{(1-q^{i-4})(1-s^*q^{i+1})}{1-s^*q^{2i+1}}
\left(1-\tfrac{(1-r_1q^{i+1})(1-r_2q^{i+1})}{1-s^*q^{2i+2}}\right)
&
\tfrac{(1-q^{i-4})(1-s^*q^{i+1})(1-r_1q^{i+1})(1-r_2q^{i+1})}
		{(1-s^*q^{2i+1})(1-s^*q^{2i+2})}\\
\ \\
\left(1-\tfrac{(1-r_1q^{i+1})(1-r_2q^{i+1})}{1-s^*q^{2i+2}}\right)
\left(\tfrac{(1-q^{i-4})(1-s^*q^{i+1})}{1-s^*q^{2i+1}}+\tfrac{1}{q^4}\right)
&		
\tfrac{(1-r_1q^{i+1})(1-r_2q^{i+1})}{1-s^*q^{2i+2}}
\left(\tfrac{(1-q^{i-4})(1-s^*q^{i+1})}{1-s^*q^{2i+1}}+\tfrac{1}{q^4}\right)
\end{array}
\right], \\ \ \\
& {\bf c}_i =  \frac{1}{s^*q^4} 
\left[
\begin{array}{cc}
\tfrac{(r_1-s^*q^{i})(r_2-s^*q^{i})}{1-s^*q^{2i}}
\left(\tfrac{(1-q^{i})(1-s^*q^{i+5})}{1-s^*q^{2i+1}}-1\right) 
&
\left(\tfrac{(r_1-s^*q^{i})(r_2-s^*q^{i})}{1-s^*q^{2i}}+s^*\right)
\left(\tfrac{(q^{i}-1)(1-s^*q^{i+5})}{1-s^*q^{2i+1}}+1\right) \\
\ \\
\tfrac{(1-q^{i})(1-s^*q^{i+5})(r_1-s^*q^{i})(r_2-s^*q^{i})}
		{(1-s^*q^{2i})(1-s^*q^{2i+1})}
&
\tfrac{(q^{i}-1)(1-s^*q^{i+5})}{1-s^*q^{2i+1}}
\left(\tfrac{(r_1-s^*q^{i})(r_2-s^*q^{i})}{1-s^*q^{2i}}+s^*\right)
\end{array}
\right].
\end{align*}

\medskip \noindent
Thus the matrix representations of $\Y, \Y^{-1} $ take the form
$$
\left(
\begin{array}{cccccccc}
*&*&*&&&&& \\
*&*&*&&&&& \\
&*&*&*&*&&& \\
&*&*&*&*&&& \\
&&&*&*&*&*& \\
&&&*&*&*&*& \\
&&&&&*&*&* \\
&&&&&*&*&* 
\end{array}
\right),
\qquad \qquad 
\left(
\begin{array}{cccccccc}
*&*&&&&&& \\
*&*&*&*&&&& \\
*&*&*&*&&&& \\
&&*&*&*&*&& \\
&&*&*&*&*&& \\
&&&&*&*&*&* \\
&&&&*&*&*&* \\
&&&&&&*&* 
\end{array}
\right),
$$
respectively.

\medskip \noindent
The matrix representing $\A$ relative to (\ref{C;d=4}) is 
$\sqrt{\frac{s^*q^4}{r_1r_2}}$ times
\begin{equation*}
\left[
\begin{array}{cccc}
{\bf a}_0 & {\bf b}_0 & & {\bf 0} \\
{\bf c}_1 & {\bf a}_1 & {\bf b}_1 & \\
& {\bf c}_2 & {\bf a}_2 & {\bf b}_2 \\
{\bf 0} & & {\bf c}_3 & {\bf a}_3
\end{array}
\right],
\end{equation*}
where for $0 \leq i \leq 3$,
\begin{align*}
&{\bf a}_i = \\
&\left[
\begin{array}{l|l}
1+{\frac{r_1r_2}{s^*q^4}}
-  \frac{(1-q^i)(1-s^*q^{i+5})(r_1-s^*q^{i+1})(r_2-s^*q^{i+1})}
					{s^*q^4(1-s^*q^{2i+1})(1-s^*q^{2i+2})}
& 
	\hspace{1cm}
	\frac{(1-r_1q^{i+1})(1-r_2q^{i+1})}{1-s^*q^{2i+2}}
\\ 
\hspace{2cm} 
 - \ \frac{(1-q^{i-4})(1-s^*q^{i+1})(1-r_1q^{i+1})(1-r_2q^{i+1})}
{(1-s^*q^{2i+1})(1-s^*q^{2i+2})}
& 
	\hspace{1cm}
	\times 
	\left(\frac{(1-q^{i-4})(1-s^*q^{i+1})}{1-s^*q^{2i+1}}
		- \frac{(1-q^{i-3})(1-s^*q^{i+2})}{1-s^*q^{2i+3}}\right) 
\\ \\ \hline \\
\hspace{1cm}
\frac{(r_1-s^*q^{i+1})(r_2-s^*q^{i+1})}{s^*q^4(1-s^*q^{2i+2})} 
& 
1+{\frac{r_1r_2}{s^*q^4}}- 
		\frac{(1-q^{i+1})(1-s^*q^{i+6})(r_1-s^*q^{i+1})(r_2-s^*q^{i+1})}
					{s^*q^4(1-s^*q^{2i+2})(1-s^*q^{2i+3})}
\\ 
\hspace{1cm}
\times	\left(\frac{(1-q^{i+1})(1-s^*q^{i+6})}{1-s^*q^{2i+3}}
			- \frac{(1-q^i)(1-s^*q^{i+5})}{1-s^*q^{2i+1}}\right) 
& 
	\hspace{2cm} 
		 - \ \frac{(1-q^{i-3})(1-s^*q^{i+2})(1-r_1q^{i+1})(1-r_2q^{i+1})}
					{(1-s^*q^{2i+2})(1-s^*q^{2i+3})}
\end{array}
\right], 
\\  \\ & \text{and for $0 \leq i \leq 2$, } \\ 
&{\bf b}_i =
\left[
\begin{array}{cc}
\tfrac{(1-q^{i-3})(1-s^*q^{i+2})(1-r_1q^{i+1})(1-r_2q^{i+1})}
				{(1-s^*q^{2i+2})(1-s^*q^{2i+3})}
& 0 
\\ \\
\frac{(1-q^{i-3})(1-s^*q^{i+2})}{1-s^*q^{2i+3}}
\left(\frac{(1-r_1q^{i+1})(1-r_2q^{i+1})}{1-s^*q^{2i+2}}
	- \frac{(1-r_1q^{i+2})(1-r_2q^{i+2})}{1-s^*q^{2i+4}}\right)
& 
\tfrac{(1-q^{i-3})(1-s^*q^{i+2})(1-r_1q^{i+2})(1-r_2q^{i+2})}
		{(1-s^*q^{2i+3})(1-s^*q^{2i+4})}
\end{array}
\right],
\\  \\ & \text{and for $1 \leq i \leq 3$, } \\ 
& {\bf c}_i =
\left[
\begin{array}{ccc}
\frac{(1-q^{i})(1-s^*q^{i+5})(r_1-s^*q^{i})(r_2-s^*q^{i})}
	{s^*q^4(1-s^*q^{2i})(1-s^*q^{2i+1})}
& \quad &
	\frac{(1-q^{i})(1-s^*q^{i+5})}{s^*q^4(1-s^*q^{2i+1})}
	\left(\frac{(r_1-s^*q^{i+1})(r_2-s^*q^{i+1})}{1-s^*q^{2i+2}}
		- \frac{(r_1-s^*q^{i})(r_2-s^*q^{i})}{1-s^*q^{2i}}\right) 
\\ \\
0 
& \quad &
\frac{(1-q^{i})(1-s^*q^{i+5})(r_1-s^*q^{i+1})(r_2-s^*q^{i+1})}
	{s^*q^4(1-s^*q^{2i+1})(1-s^*q^{2i+2})}
\end{array}
\right].
\end{align*}
Thus the matrix representation of $\A$ takes the form
$$
\left(
\begin{array}{cc|cc|cc|cc}
*&*&*&&&&&\\
*&*&*&*&&&&\\
\hline
*&*&*&*&*&&&\\
&*&*&*&*&*&&\\
\hline
&&*&*&*&*&*&\\
&&&*&*&*&*&*\\
\hline
&&&&*&*&*&*\\
&&&&&*&*&*
\end{array}
\right).
$$

\bigskip \noindent
The matrix representing $\B$ relative to (\ref{C;d=4}) is diagonal:
\begin{align*}
\diag \left[
\tfrac{1}{q\sqrt{s^*}} + q\sqrt{s^*},
 \tfrac{1}{q\sqrt{s^*}} + q\sqrt{s^*},
 \tfrac{1}{q^{2}\sqrt{s^*}} + q^{2}\sqrt{s^*},
 \tfrac{1}{q^{2}\sqrt{s^*}} + q^{2}\sqrt{s^*}, \right. \hspace{2.55cm}
 \nonumber\\
\hspace{3cm}\left.
 \tfrac{1}{q^{3}\sqrt{s^*}} + q^{3}\sqrt{s^*},
 \tfrac{1}{q^{3}\sqrt{s^*}} + q^{3}\sqrt{s^*},
 \tfrac{1}{q^{4}\sqrt{s^*}} + q^{4}\sqrt{s^*},
 \tfrac{1}{q^{4}\sqrt{s^*}} + q^{4}\sqrt{s^*}
\right].
\end{align*}

\bigskip \noindent
The matrix representing $\B^{\dagger}$ relative to (\ref{C;d=4}) is diagonal:
\begin{align*}
\diag \left[
\tfrac{1}{\sqrt{s^*q}} + \sqrt{s^*q},\
\tfrac{1}{q\sqrt{s^*q}} + q\sqrt{s^*q},\
\tfrac{1}{q\sqrt{s^*q}} + q\sqrt{s^*q},\
\tfrac{1}{q^2\sqrt{s^*q}} + q^2\sqrt{s^*q},\hspace{3.1cm} \right.\nonumber\\
\left. 
\tfrac{1}{q^2\sqrt{s^*q}} + q^2\sqrt{s^*q}, \
\tfrac{1}{q^3\sqrt{s^*q}} + q^3\sqrt{s^*q},\
\tfrac{1}{q^3\sqrt{s^*q}} + q^3\sqrt{s^*q},\
\tfrac{1}{q^{4}\sqrt{s^*q}} + q^{4}\sqrt{s^*q}
\right].
\end{align*}

\bigskip \noindent
The matrix representing $\frac{t_0-k_0^{-1}}{k_0-k_0^{-1}}$ relative to (\ref{C;d=4})
is
\begin{equation*}
\left[
\begin{array}{cccccccc}
\frac{1}{1-\tau_0} & \frac{\tau_0}{\tau_0-1} &0&0&0&0&0&0\\
\frac{1}{1-\tau_0} & \frac{\tau_0}{\tau_0-1} &0&0&0&0&0&0\\
0&0&\frac{1}{1-\tau_1} & \frac{\tau_1}{\tau_1-1} &0&0&0&0\\
0&0&\frac{1}{1-\tau_1} & \frac{\tau_1}{\tau_1-1} &0&0&0&0\\
0&0&0&0&\frac{1}{1-\tau_2} & \frac{\tau_2}{\tau_2-1} &0&0\\
0&0&0&0&\frac{1}{1-\tau_2} & \frac{\tau_2}{\tau_2-1} &0&0\\
0&0&0&0&0&0&\frac{1}{1-\tau_3} & \frac{\tau_3}{\tau_3-1} \\
0&0&0&0&0&0&\frac{1}{1-\tau_3} & \frac{\tau_3}{\tau_3-1}
\end{array}
\right].
\end{equation*}
Each entry is shown in (\ref{209}).

\bigskip \noindent
The matrix representing $\frac{t_1-k_1^{-1}}{k_1-k_1^{-1}}$ relative to (\ref{C;d=4}) is
\begin{equation*}
\left[
\begin{array}{cccccccc}
1 & 0 &0 &0 &0 &0 &0 &0\\
0 & \frac{\epsilon_1}{\epsilon_1-1} & \frac{1}{1-\epsilon_1} &0 &0 &0 &0 &0\\
0 & \frac{\epsilon_1}{\epsilon_1-1} & \frac{1}{1-\epsilon_1} &0 & 0 & 0 &0 &0\\
0 &0 &0 & \frac{\epsilon_2}{\epsilon_2-1} & \frac{1}{1-\epsilon_2} & 0 & 0 &0\\
0 &0 &0 &\frac{\epsilon_2}{\epsilon_2-1} & \frac{1}{1-\epsilon_2} & 0 & 0 &0\\
0 &0 &0 &0 &0 & \frac{\epsilon_3}{\epsilon_3-1} & \frac{1}{1-\epsilon_{3}}&0\\
0 &0 &0 &0 &0 & \frac{\epsilon_3}{\epsilon_3-1} & \frac{1}{1-\epsilon_{3}} &0\\
0 &0 &0 &0 &0 &0 &0 & 1
\end{array}
\right].
\end{equation*}
Each entry is shown in (\ref{208}).

\section{Acknowledgements}
This paper was the author's Ph.D. thesis at the University of Wisconsin-Madison.
The author would like to thank his advisor, Paul Terwilliger, for offering many valuable ideas
and suggestions. The author would also like to thank Kazumasa Nomura for giving this paper a close reading and offering valuable suggestions.


\bigskip

\noindent Jae-Ho Lee \hfil\break
\noindent Department of Mathematics \hfil\break
\noindent University of Wisconsin \hfil\break
\noindent 480 Lincoln Drive \hfil\break
\noindent Madison, WI 53706-1388 USA \hfil\break
\noindent email: {\tt jhlee@math.wisc.edu }\hfil\break


\end{document}